\documentclass[11pt]{article}

\usepackage[colorlinks]{hyperref}
\usepackage{amsthm,amsmath,epic,amsfonts,latexsym,amssymb}
\usepackage[usenames,dvipsnames]{color}
\numberwithin{equation}{section}
\theoremstyle{plain} 
\newtheorem{thm}{Theorem}[section]
\newtheorem{prop}[thm]{Proposition}
\newtheorem{lem}[thm]{Lemma} 
\newtheorem{rem}{Remark}[section]
\newtheorem{cor}[thm]{Corollary}

\newtheorem{assum}{Assumption}

\newcommand{\Z}{{\mathbb{Z}}}
\newcommand{\Zd}{{\mathbb{Z}^d}}
\newcommand{\ZD}{{\mathbb{Z}^D}}
\newcommand{\NN}{{\mathbb{N}}} 
\newcommand{\R}{{\mathbb{R}}}
\newcommand{\PP}{{\mathbb{P}}}
\newcommand{\Rd}{{\mathbb{R}^d}}

\newcommand{\EE}{{\mathbb{E}}} 
\newcommand{\E}{\mathcal{E}}

\newcommand{\vep}{\varepsilon}
\newcommand{\To}{\Rightarrow}

\newcommand{\N}{\mathcal{N}}

\newcommand{\A}{\mathcal{A}} 
 
\newcommand{\cB}{\mathcal{B}}
\newcommand{\F}{\mathcal{F}}
\newcommand{\cD}{\mathcal{D}}
\newcommand{\cE}{\mathcal{E}}
\newcommand{\cF}{\mathcal{F}}
\newcommand{\cG}{\mathcal{G}}
\newcommand{\cH}{\mathcal{H}}

\newcommand{\cL}{\mathcal{L}}
\newcommand{\LD}{\mathcal{L}_D}

\newcommand{\cN}{\mathcal{N}}
\newcommand{\cP}{\mathcal{P}}

\newcommand{\cQ}{\mathcal{Q}}
 
\newcommand{\V}{\mathcal{V}}
\newcommand{\cV}{\mathcal{V}}

\newcommand{\sgn}{{\operatorname{sgn}}}
 \newcommand{\se}{\sqrt\vep}

\newcommand{\tref}[1]{{\text{\scriptsize\ref{#1}}}}

\newcommand{\email}[1]{{\it  E-mail address:}\  \textsf{#1}\\*\protect}
\newcommand{\address}[1]{\small \textsc{#1}\\*\protect}     

\newcommand{\note}[1]{\par\noindent\mbox{}\hspace{-0.5in}\textcolor{BrickRed}{%
     $\xrightarrow[\quad\quad]{}\,\,\,\,${\sf\bf\small #1}}}

\renewcommand\note[1]{}
\def\mn{\par\medskip\noindent}
\def\ep{\varepsilon}
\def\beq{\begin{equation}}
\def\eeq{\end{equation}}
\def\square{\vcenter{\vbox{\hrule height .4pt
  \hbox{\vrule width .4pt height 5pt \kern 5pt
        \vrule width .4pt} \hrule height .4pt}}}
\def\qed{\hfill$\square$}
\begin{document}

\title{Voter Model Perturbations and Reaction Diffusion Equations}
\author{J. Theodore Cox\thanks{Supported in part by NSF Grant DMS-0505439
    Part of the research was done while the author was
    visiting The University of British Columbia}, Syracuse U. \\
    Richard Durrett\thanks{Partially supported by NSF Grant DMS-0704996 from the probability program}, Duke U. \\ 
and Edwin A. Perkins\thanks{Supported in part by an NSERC Discovery Grant},  U. British Columbia}
\date{\today}
\maketitle

\begin{abstract}We consider particle systems that are perturbations of the
voter model and show that when space and time are rescaled
the system converges to a solution of a reaction diffusion
equation in dimensions $d \ge 3$. Combining this result with properties 
of the PDE, some methods arising from a low density super-Brownian limit theorem, and a block construction, we give general, and often asymptotically sharp, conditions for the existence
of non-trivial stationary distributions, and for extinction of one type. As applications, we describe the phase diagrams of  
three systems when the parameters are close to the voter model: 
(i) a stochastic spatial Lotka-Volterra model of Neuhauser and Pacala,
(ii) a model of the evolution of cooperation of
Ohtsuki, Hauert, Lieberman, and Nowak, and (iii) a continuous time version of the 
non-linear voter model of Molofsky, Durrett, Dushoff, Griffeath, and Levin.
The first application confirms a conjecture of Cox and Perkins \cite{CP-2} and the second
confirms a conjecture of Ohtsuki et al \cite{Oetal} in the context of certain infinite graphs.
An important feature of our general results is that they do not require the process to be 
attractive.
\end{abstract}

\vfill
{\footnotesize
\noindent \\
AMS 2000 {\it subject classifications}. Primary 60K35. 
Secondary 35K57, 60J68, 60F17, 92D15, 92D40. 
 \\}

 \clearpage

\section{Introduction and Statement of
  Results}\label{sec:intro}

We first describe a class of particle systems, called voter model perturbations in \cite{CP-1}.  The state
space will be $\{0,1\}^{\Z^d}$, or after rescaling $\{0,1\}^{\vep\Z^d}$, where throughout this work
we assume $d\ge 3$. The voter model part of the process will 
depend on a symmetric (i.e, $p(x)=p(-x)$), irreducible probability kernel $p:\Zd\to[0,1]$ with $p(0)=0$, covariance matrix $\sigma^2 I$, and
exponentially bounded tails so that for some $\kappa\in(0,1]$,
\beq
p(x) \le C e^{-\kappa |x|}.
\label{expbd1}
\eeq
Here and in what follows $|x| = \sup_i |x_i|$.

For $1\ge\vep>0$,  $x\in\Z^d$ and $\xi\in\{0,1\}^{\Z^d}$ define rescalings of $p$ and $\xi$ by $p_\vep(\vep x)=p(x)$, and $\xi_\vep(\vep x)=\xi(x)$, so that $\xi_\vep\in\{0,1\}^{\vep\Z^d}$.
Also define rescaled {\it local densities} $f_i^\vep$ by
\begin{equation}\label{locdens}
  f^\vep_i(\vep x,\xi_\vep) = \sum_{y\in\vep\Z^d} p_\vep(y-\vep x) 1\{\xi_\vep(y)=i\}, \qquad i=0,1.
\end{equation}
We will write $f_i(x,\xi)$ if $\vep=1$.  
For $x,\xi$ as above, introduce the voter flip rates and rapid voter flip rates given by
\begin{equation}\label{voterrates}
  c^v(x,\xi)=[(1-\xi(x))f_1(x,\xi)+\xi(x)f_0(x,\xi)],\ c_\vep^v(\vep x,\xi_\vep)=\vep^{-2}c^v(x,\xi). 
\end{equation}
The processes of interest, $\xi_t\in\{0,1\}^{\Z^d}$, are spin-flip systems with rates
\beq\label{origrates}
c_\vep^o(x,\xi)=c^v(x,\xi)+\vep^2c^*_\vep(x,\xi)\ge 0,
\eeq
where $c^*_\vep(x,\xi)$ will be a translation invariant, signed perturbation of the form
\[c^*_\vep(x,\xi)=(1-\xi(x))h_1^\vep(x,\xi) + \xi(x)h_0^\vep(x,\xi).\]  Therefore the rescaled processes, $\xi_{\vep^{-2}t}(\vep x)\in\{0,1\}^{\vep\Z^d}$, we will study have rates
\begin{align} \label{xirates}
 c_\vep(\vep x,\xi_\vep)=c_\vep^v(\vep x,\xi_\vep)+c^*_\vep(x,\xi)\ge 0.
\end{align}
We assume there is a law $q$ of $(Y^1,\dots,Y^{N_0})\in\Z^{dN_0}$ and $g^\vep_i$ on $\{0,1\}^{N_0}$, $i=0,1$, and $\vep_1\in(0,\infty], \vep_0\in(0,1]$ so that 
\begin{equation}\label{nonnegg}g_i^\vep\ge 0,
\end{equation}
and
\begin{align}\label{hgrepn}h_i^\vep(x,\xi) = -\vep_1^{-2}f_i(x,\xi)+&E_Y(g_i^\vep(\xi(x+ Y^1), \ldots \xi(x+ Y^{N_0}))), i=0,1\\ 
\nonumber&\hbox{ for all } \xi\in\{0,1\}^{\Z^d}, x\in\Z^d, \vep\in(0,\vep_0].
\end{align}
Here $E_Y$ is
expectation with respect to $q$ and in
practice the first term in the above allows us to take
$g_i^\vep\ge 0$. It is important to have $g_i^\vep$
non-negative as we will treat it as a rate in the
construction of a dual process in
Section~\ref{sec:construction}. On the other hand, in the
particular examples motivating the general theory $h^\vep_i$
will often be negative (see, e.g., \eqref{LVh} below of the
Lotka-Volterra models).

We also suppose that 
(decrease $\kappa>0$ if necessary) \beq
\label{expbd2}
P( Y^*\ge x ) \le C e^{-\kappa x}\hbox{ for }x>0 
\eeq
where $Y^* = \max\{ |Y^1|, \ldots |Y^{N_0}| \}$,
and 
 there are limiting maps $g_i:\{0,1\}^{N_0}\to \R_+$ such that 
\beq\label{gcvgce}
\lim_{\vep\downarrow 0} \Vert g_i^\vep-g_i\Vert_\infty =0,\ i=0,1.
\eeq
The condition \eqref{hgrepn} with $\vep_0<\vep_1$ (without loss of generality) easily implies the non-negativity in \eqref{xirates}.

We now show that the conditions \eqref{hgrepn}-\eqref{gcvgce} hold for general finite range convergent translation invariant perturbations without any non-negativity constraint on the $g_i^\vep$.

\begin{prop}\label{finiterangeh} Assume there are distinct points $y_1,\dots,y_{N_0}\in\Z^d$ and \hfil\break
$\hat g_i^\vep,\hat g_i:\{0,1\}^{N_0}\to\R$ such that
\begin{align} \label{finrangerep}&h_i^\vep(x,\xi)=\hat g_i^\vep(\xi(x+y_1),\dots,\xi(x+y_{N_0})),\ x\in\Z^d, \xi\in\{0,1\}^{N_0},\\
\nonumber&\{x:p(x)>0\}\subset\{y_1,\dots,y_{N_0}\},\quad\lim_{\vep\downarrow 0}\Vert\hat g_i^\vep-\hat g_i\Vert_\infty=0\ i=0,1.
\end{align}
Then \eqref{hgrepn}-\eqref{gcvgce} hold for appropriate non-negative $g_i^\vep$, $g_i$ satisfying $\Vert g_i^\vep-g_i\Vert_\infty=\Vert \hat g_i^\vep-\hat g_i\Vert_\infty$, and $Y^i=y_i$.
\end{prop}
\noindent The elementary proof is given in
Section~\ref{ssec:prelimred}. In terms of our original rates
\eqref{origrates} this shows that our class of models
include spin-flip systems $\xi_t\in\{0,1\}^{\Z^d}$, $t\ge
0$, with rates
\beq \label{finranrates}c^o_\vep(x,\xi)=c^v(x,\xi)+\vep^2 c^*(x,\xi)+\vep^2o(\vep)\ge 0,
\eeq
where $p$ (governing $c^v$) is now finite range, $c^*(x,\xi)=h(\xi(x),\xi(x+y_1),\dots,\xi(x+\xi_{N_0}))$ is a finite range, translation invariant perturbation and $o(\vep)$ means this term goes to zero with $\vep$ uniformly in $(x,\xi)$.  

On the other hand, the 
random $Y^i$'s will also allow certain natural infinite range interactions. The formulation in terms
of random locations will also simplify some of the arithmetic to come.

Let $\xi^\vep_t$, $t\ge 0$ be the unique
$\{0,1\}^{\vep\Z^d}$-valued Feller process with translation
invariant flip rates given by $c_\vep(x,\xi)$ in
\eqref{xirates} and initial state $\xi_0^\vep\in\{0,1\}^{\vep\Z^d}$.   
More formally (see Theorem B.3 in \cite{Lig99} and Section 2
of \cite{Dur95}) the generator of $\xi^\vep$ is 
\begin{align}\label{gendesc}&\hbox{the closure
of  }
\Omega_\vep g(\xi)=\sum_{x\in\vep Z^d}c_\vep(x,\xi)(g(\xi^x)-g(\xi)),\\
\nonumber&\hbox{ on the space of $g:\vep\Z^d\to\R$, depending on finitely
many coordinates.}
\end{align}
Here $\xi^x$ is $\xi$ with the
coordinate at $x$ flipped to $1-\xi(x)$.  The condition (B4) of Theorem B.3 in \cite{Lig99} is trivial
to derive from \eqref{hgrepn}.  

{\bf We stress that conditions \eqref{xirates}-\eqref{gcvgce} are in force throughout this work, and call such a process $\xi^\vep$ a voter model perturbation.}

\medskip
Given a process taking values in $\{0,1\}^{\Zd}$,  or more
generally in $\{0,1\}^{\vep\Zd}$, we say that {\it
  coexistence holds} if there is a stationary distribution
$\nu$ with
\begin{equation}\label{nucoexist}
\nu\left(\sum_x\xi(x)=\sum_x1-\xi(x)=\infty\right)=1.
\end{equation}
For voter model perturbations it is easy to see this is equivalent to both types being present $\nu$-a.s.--see Lemma~\ref{nuequiv} at the beginning of Section~\ref{sec:exist}.

We say {\it the $i$'s take over} if for all $L$,
\beq
\label{takeover}
P(\xi_t(x)=i \hbox{ for all $x\in [-L,L]^d$ }\hbox{ for } t \hbox{ large enough}) =1
\eeq
whenever the initial configuration has infinitely many sites in state $i$. 

Our main results, Theorems~\ref{thm:exist} and \ref{thm:nonexist} in Section~\ref{gensys} below, give (often sharp) conditions under which coexistence holds or one type takes over, respectively, in a voter model perturbation for small enough $\vep$. Of course these results then hold immediately for our originally unscaled processes, again for small enough $\vep$.

\subsection{Hydrodynamic limit}\label{ssec:hydro}
As $d\ge 3$, we see from Theorem V.1.8 of \cite{Lig} the
voter model with flip rates $c^v(x,\xi)=c_1^v(x,\xi)$ has a
one-parameter family of translation invariant extremal
invariant distributions $\{\PP_u:u\in[0,1]\}$ on
$\{0,1\}^{\Z^d}$ such that $\EE_u(\xi(x))=u$.  We write
$\langle g\rangle_u$ for $\EE_u(g(\xi))$. \eqref{hgrepn} and \eqref{gcvgce}  imply
\begin{align}\label{hcvgce}
&\lim_{\vep\downarrow 0}\Vert h^\vep_i-h_i\Vert_\infty=0\hbox{ where }\\
\nonumber &h_i(x,\xi)=-\vep_1^{-2}f_i(x,\xi)+E(g_i(\xi(x+ Y^1),\dots, \xi(x+Y^{N_0}))).
\end{align}
Define
\begin{equation}
\label{fdef}
f(u)=\langle (1-\xi(0))h_1(0,\xi) - \xi(0)h_0(0,\xi)\rangle_u.
\end{equation}
Then $f$ is a polynomial of degree at most $N_0+1$ (see \eqref{prodform} and Section~\ref{ssec:CPcomp} below).
The non-negativity condition \eqref{xirates}, the fact that 
\beq\label{meanvoterrts} \langle c_\vep^v(0,\xi)\rangle_0=\langle c_\vep^v(0,\xi)\rangle_1=0,
\eeq
 and the convergence \eqref{hcvgce} show that
\begin{equation}
\label{fbnd}
f(0)\ge0,\quad f(1)\le 0.
\end{equation}

Our first goal is to show that under suitable assumptions on the initial conditions,
as $\ep\to 0$ the particle systems converges
to the PDE
\begin{equation}\label{rdpde}
  {\partial u\over \partial t}={\sigma^2\over 2}\Delta
  u+f(u),\quad u(0,\cdot)=v(\cdot), 
\end{equation} 
The remark after Proposition 2.1 in \cite{AW78} implies that for any continuous
$v:R^d \to [0,1]$ the equation has a unique solution $u$, which necessarily takes values in $[0,1]$.  

%\note{I thought we need to work with piecewise continuous IC's later?}

For a continuous $v$ as above we 
will say that a family of probability measures
$\{\lambda_\vep\}$ 
of laws on $\{0,1\}^{\vep\Zd}$ has
local density $v$ if the following holds:\hfil\break

\noindent There is an
$r\in(0,1)$ such that if 
$ a_\vep = \lceil \vep^{r-1}\rceil \vep$, $Q_\vep =
[0,a_\vep)^d\cap\vep\Zd$, $|Q_\vep|=\hbox{card}\,(Q_\vep)$, and 
\begin{equation}\label{densdef0}
D(x,\xi) = \dfrac{1}{|Q_\vep|} \sum_{y\in Q_\vep}\xi(x+y)
\text{ for }x\in a_\vep\Zd, \xi\in \{0,1\}^{\vep Z^d}\,,
\end{equation}
then for all $R,\delta>0$,
\begin{equation}\label{densdef}
\lim_{\vep\to 0}\sup_{\substack{x\in a_\vep
  \Zd\\|x|\le R}}\lambda_\vep(|D(x,\xi)-v(x)|>\delta) = 0\,.
\end{equation}

The family of Bernoulli product measures $\bar\lambda_\vep$
given by
\begin{equation}\label{ICconv}
\bar\lambda_\vep(\xi(w_i)=1, i=1,\dots,n)
=\prod_{i=1}^n v(w_i) \text{ for all }n\in \NN \text{ and }
w_i\in\vep\Zd\ .
\end{equation}
certainly satisfies \eqref{densdef} for all  $r\in(0,1)$. 

\begin{thm}\label{conv}
  Assume $v:\R^d\to[0,1]$ is continuous, and the collection of
  initial conditions
  $\{\xi^\vep_0\}$ have laws $\{\lambda_\vep\}$  with
  local density $v$.  
  Let $x^k\in \R^d$ and $x_\vep^k\in\vep\Z^d$, $\,k=1,\dots
  K$ satisfy
  \begin{equation}\label{xcond}
    x^k_\vep\to x^k\hbox{ and
    }\vep^{-1}|x^k_\vep-x^{k'}_\vep|\to\infty\hbox{ as
    }\vep\to0\hbox{ for any }k\neq k'. 
  \end{equation}
  If $u$ is the solution of \eqref{rdpde}, then for any
  $\eta\in \{0,1\}^{\{1,\dots,L\}\times \{1,\dots K\}}$, $y_1,\dots,y_L\in\Z^d$
  and $T>0$, 
  \begin{multline}
    \lim_{\vep\to 0}P(\xi^\vep_T(x_\vep^k+\vep
    y_i)=\eta_{i,k},\ \ i=1,\dots,L,\, k=1,\dots K)\\ 
    =\prod_{k=1}^K\langle 1\{\xi(y_i)=\eta_{i,k},
    i=1,\dots,L\}\rangle_{u(T,x^k)}.
  \end{multline}
  In particular, if $x_\vep\in \vep\Z^d$ satisfies
 $x_\vep\to x$ as $\vep\to0$, then
  \begin{equation}\label{hydro}
    \lim_{\vep\to0} P(\xi_T^\vep(x_\vep)=1)=u(T,x)\quad\hbox{for all }T>0, x\in\R^d.
  \end{equation}
\end{thm}

De Masi, Ferrari and Lebowitz \cite{DFL86}, Durrett and Neuhauser \cite{DN94} 
and Durrett \cite{Dur02} have proved similar results for particle systems
with rapid stirring. The local equilibrium for rapid stirring is a Bernoulli product measure,
but in our setting it is the voter equilibrium.  As a result there is now dependence
between nearby sites on the microscopic scale. However, there is asymptotic
independence between sites with infinite separation on the
microscopic scale.  

It is easy to carry out a variance
calculation to improve Theorem~\ref{conv} to the following $L^2$-convergence
theorem (see the end of Section~\ref{sec:hydroproofs} for the
proof). If $\delta>0$ and $x\in\R^d$, let
$I_\delta(x)$ be the unique semi-open cube
$\prod_{i=1}^d [k_i\delta,(k_i+1)\delta)$, $k_i\in\Z$, which contains $x$.

\begin{thm}\label{thm:strongconv} Assume the hypotheses of Theorem~\ref{conv}.\\
Assume $\delta(\vep)\in\vep\NN$ decreases to zero so
that $\delta(\vep)/\vep \to\infty$ as $\vep\downarrow 0$.  If 
$$
{\tilde u}^\vep(t,x)=\sum_{y\in
I_{\delta(\vep)}(x)}\xi_t^\vep(y)(\vep/\delta(\vep))^d,
$$
then as $\vep\to0$, ${\tilde u}^\vep(t,x)\rightarrow u(t,x)$ in
$L^2$ uniformly for $x$ in compacts, for all $t>0$.
\end{thm}
\noindent A low density version of this theorem, in which the limit is random (super-Brownian motion with drift), was proved in \cite{CP-1} and is discussed in Section~\ref{ssec:CPcomp}.

To apply Theorem~\ref{conv} to the voter perturbation we will have to evaluate $f(u)$.  This is in principle straightforward thanks to the duality between the voter model and coalescing random walk which we now recall.  Let $\{\hat B^x:x\in\Z^d\}$ denote a rate $1$ coalescing random walk system on $\Z^d$ with step distribution $p$ and $\hat B^x_0=x$.  For $A,B\subset \Z^d$, let $\hat\xi^A_t=\{\hat B^x_t:x\in A\}$, $\tau(A)=\inf\{t:\hat |\xi^A_t|=1\}$ and $\tau(A,B)$ be the first time $\hat\xi^A_t\cap\hat\xi_t^B\neq\emptyset$ (it is $\infty$ if either $A$ or $B$ is empty).  The duality between $\hat B$ and the voter model (see (V.1.7) and Theorem~V.1.8 in \cite{Lig}) implies for finite $A,B\subset\Z^d$,
\begin{align}
\nonumber\langle&\prod_{y\in A}\xi(y)\prod_{z\in B}(1-\xi(z))\rangle_u\\
\label{prodform}&=\sum_{j=0}^{|A|}\sum_{k=0}^{|B|}u^j(1-u)^kP(|\hat \xi^A_\infty|=j,|\hat\xi_\infty^B|=k,\tau(A,B)=\infty).
\end{align}
The $k=0$ term is non-zero only if $B=\emptyset$ in which case the above probability is $P(|\hat\xi^A_\infty|=j)$, and similarly for the $j=0$ term. 
It follows from \eqref{prodform} and the form of the perturbation in \eqref{hgrepn} that $f(u)$ is a polynomial of degree at most $N_0+1$ with coefficients given by certain coalescing probabilities of $\hat B$ (see \eqref{fformula} below).

\subsection{PDE results}
\label{ss:PDE}

As in Durrett and Neuhauser \cite{DN94}, Theorem~\ref{conv}, in combination with results
for the PDE and a block construction, lead to theorems about the particle system.  Durrett \cite{Dur09} surveys results that have been proved by this method in the last 15 years.
To carry out
this program we will also need some low density methods taken from the superprocess
limit theorems of Cox and Perkins \cite{CP-1,CP-2}.

To prepare for the discussion of the examples, we will state the PDE results 
on which their analysis will be based. The reaction function $f:\R\to\R$ is a 
continuously differentiable function (as already noted, in our context it will be a polynomial).  Assume now, as will be the case in the examples, that $f(0)=f(1)=0$.
We let $u(t,x)$ denote the unique solution of \eqref{rdpde} with continuous initial data $v:\R^d\to[0,1]$. 

We start with a modification of a result of Aronson and Weinberger 
\cite{AW78}.

\begin{prop} \label{prop:pde1}
Suppose $f(0)=f(\alpha)=0$, $f'(0)>0$, $f'(\alpha)<0$ and $f(u)>0$ for $u \in (0,\alpha)$ with $0<\alpha\le 1$.  
There is a $w>0$ so that if the initial condition $v$ is not identically $0$, then
$$
\liminf_{t\to\infty} \inf_{|x| \le 2wt} u(t,x) \ge \alpha.
$$
\end{prop}

We also will need an exponential rate of convergence in this case under a stronger condition on the initial conditon.  We formulate it
for $f<0$ on $(0,1)$. The brief proofs of Propositions~\ref{prop:pde1} and \ref{prop:pde4} are given at the beginning of Section~\ref{lockill}.

\begin{prop}\label{prop:pde4} Assume $f<0$ on $(0,1)$ and $f'(0)<0$.  There is a $w>0$, and if $\delta>0$ there
are positive constants $L_\delta$, $c=c_\delta$, and $C=C_\delta$ so that if $L\ge L_\delta$ and
 $v(x)\le 1-\delta$ for $|x|\le L$, then
$$u(t,x)\le Ce^{-ct}\quad\hbox{ for }|x|\le L+2wt.$$
\end{prop}

There are different cases depending on the number of solutions of $f(u)=0$ in $(0,1)$.
In all cases, we suppose that $f'(0)\neq 0$ and $f'(1)\neq 0$.

\medskip
\noindent
{\it Case I: $f$ has zero roots in $(0,1)$.} In this case we can apply Propositions~\ref{prop:pde1} (with $\alpha =1$) and \ref{prop:pde4}, and their obvious analogues for $-f$. 

\medskip
\noindent
{\it Case II: $f$ has one root $\rho\in(0,1)$.} There are two
possibilities here.\hfil\break 

(i) $f'(0)> 0$ and $f'(1)<0$ and so the interior fixed point $\rho\in(0,1)$ is attracting. In this case we will also assume $f'(\rho)\neq 0$. 
Then two applications of Proposition \ref{prop:pde1} show that 
if $v\not\equiv 0$ and $v\not\equiv 1$
\beq\label{pde5}
\lim_{t\to\infty} \sup_{|x| \le wt} |u(t,x)-\rho| = 0.
\eeq

(ii) $f'(0)<0$ and $f'(1)>0$, so that 0 and 1 are locally attracting and $\rho\in(0,1)$ is unstable. In this case the limiting behavior of the PDE is determined
by the speed $r$ of the traveling wave solutions, i.e., functions $w$ with $w(-\infty)=\rho$ 
and $w(\infty)=0$ so that $u(t,x) = w(x-rt)$ solves the PDE. The next result was first proved
in $d=1$ by Fife and McLeod \cite{FM77}. See page 296 and the appendix of \cite{DN94}
for the extension to $d>1$ stated below as Proposition~\ref{prop:pde2}. The assumption there on the non-degeneracy of the interior zeros are not necessary (see Fife and McLeod \cite{FM77}). These references also show that 
\begin{equation}\label{signr} \sgn(r)=\sgn\Bigl( \int_0^1f(u)du\Bigr).
\end{equation}
$|x|_2$ will denote the Euclidean norm of $x$ and the conditions of Case II(ii) will apply in the next two propositions.

\begin{prop} \label{prop:pde2} 
Suppose $\int_0^1f(u)du<0$ and fix $\eta>0$.  
If $\delta>0$ there are positive constants $L^0_\delta$, $c_0=c_0(\delta)$, and $C_0=C_0(\delta)$ so that if $L \ge L^0_\delta$ and $v(x) \le \rho-\delta$ when $|x|_2\le L$, then
$$
u(t,x) \le C e^{-ct} \quad\hbox{for $|x|_2 \le (|r|-\eta)t$}.
$$
\end{prop}

\noindent
For the block construction it is useful to have a version of the last result for the
$L^\infty$ norm, and which adds an $L$ to the region in which the result is valid.

\begin{prop} \label{prop:pde3} 
Suppose $\int_0^1f(u)du<0$.
There is a $w>0$, and if $\delta>0$ there
are positive constants $L_\delta$, $c=c_\delta$ and $C=C_\delta$ so that if $L\ge L_\delta$ and
 $v(x)\le \rho-\delta$ for $|x|\le L$, then
$$u(t,x)\le Ce^{-ct}\quad\hbox{ for }|x|\le L+2wt.$$
\end{prop}

\noindent
The short derivation of Proposition \ref{prop:pde3} from Proposition \ref{prop:pde2} is given at the beginning of Section \ref{lockill}.

In the next three subsections we first motivate our main results  on coexistence
and domination by a single type (Theorems~\ref{thm:exist} and \ref{thm:nonexist}) by illustrating their use in three distinct families of examples.

\subsection{Lotka-Volterra systems}
\label{ss:LV}

We now apply Theorem~\ref{conv} to a stochastic spatial
Lotka-Volterra model introduced by Neuhauser and Pacala \cite{NP}.  In addition
to the kernel $p$ for the voter model, the flip rates depend on two non-negative
competition parameters, $\alpha_0$ and $\alpha_1$, and are
given by 
\begin{align}
\label{lvrates}
\nonumber  c_{LV}(x,\xi)=&f_1(f_0+\alpha_0f_1)(1-\xi(x))+f_0(f_1+\alpha_1f_0)\xi(x)\\
  =&c^v(x,\xi)+(\alpha_0-1)f_1^2(1-\xi(x))+(\alpha_1-1)f_0^2\xi(x).
\end{align}
In words, a plant of type $i$ at $x$ dies with rate
$f_i(x,\xi)+\alpha_if_{1-i}(x,\xi)$ and is immediately replaced
by the type of a randomly chosen neighboring plant, which will be
$1-i$ with probability $f_{1-i}(x,\xi)$. The death rate reflects the effects of competition from neighboring sites. The constant $\alpha_i$ represents the effect of competition on a type $i$ individual from neighbors of the opposite type.  If $\alpha_i<0$ for both $i=0,1$, then individuals prefer to be surrounded by the opposite type and ecological arguments suggest coexistence will hold. Conversely if both $\alpha_i>1$, we expect no coexistence.  

We refer to the associated Feller process $(\xi_t, t\ge 0)$
as the $LV(\alpha_0,\alpha_1)$ process.
Proposition~8.1 of \cite{CP-2} implies that
\begin{equation}\label{moncond}\hbox{if }\alpha_0\wedge\alpha_1\ge 1/2\hbox{ then }LV(\alpha_0,\alpha_1)\hbox{
    is monotone (or attractive)}.
\end{equation} 
Write $LV(\alpha)\le LV(\alpha')$ if $LV(\alpha')$
stochastically dominates $LV(\alpha)$, that is, if one can define these
processes, $\xi$ and $\xi'$, respectively, on a common probability space
so that $\xi\le \xi'$ pointwise a.s.  Then, as should be
obvious from the above interpretation of $\alpha_i$ (see
(1.3) of \cite{CP-2}),

\begin{align}\label{mona}
  &0\le \alpha_0'\le \alpha_0, 0\le \alpha_1\le
  \alpha_1',\hbox{ and either }\alpha_0\wedge\alpha_1\ge
  1/2\\\nonumber&\hbox{or }\alpha_0'\wedge\alpha_1'\ge 1/2,
  \hbox{ implies }LV(\alpha')\le LV(\alpha).
\end{align}

If $\theta_i\in\R$, let $\alpha_i=\alpha_i^\vep=1+\vep^2\theta_i$ and consider the rescaled
Lotka-Volterra process
\begin{equation}\label{resclv}
\xi^\vep_t(x)=\xi_{\vep^{-2}t}(\vep^{-1}x),\quad x\in\vep\Z^d.
\end{equation}
From \eqref{lvrates} we see that this process has rates given by \eqref{xirates} with
\beq\label{LVh} h_i^\vep(x,\xi)=h_i(x,\xi)=\theta_{1-i}f_i(x,\xi)^2,\, i=0,1,\quad x\in\Zd,\,\xi\in\{0,1\}^{\Zd}.
\eeq
To verify \eqref{hgrepn} take $0<\vep_1\le (\theta_0^-)^{-1/2}\wedge(\theta_1^-)^{-1/2}$,  
$N_0=2$, $Y^1,Y^2$ chosen independently according to $p$,
and define for $i=0,1$, 
\begin{equation}\label{LVgdef}
g^\vep_i(\eta_1,\eta_2)=g_i(\eta_1,\eta_2)=\vep_1^{-2}\eta_1(1-\eta_2)+(\vep_1^{-2}+\theta_{1-i})1(\eta_1=\eta_2=i)\ge 0.\end{equation}
Then
\begin{align*}
h^\vep_i(x,\xi)&=-\vep_1^{-2}f_i(x,\xi)+\vep_1^{-2}(f_i-f_i^2)(x,\xi)+(\vep_1^{-2}+\theta_{1-i})f_i(x,\xi)^2\\
&=-\vep_1^{-2}f_i(x,\xi)+E(g^\vep_i(\xi(x+Y^1),\xi(x+Y^2))),
\end{align*}
as required.  Therefore $\xi^\vep$ is a voter model perturbation.  

To calculate the limiting reaction function in this case consider the system of coalescing random walks $\{\hat B^x:x\in\Zd\}$ used in the duality formula \eqref{prodform}.  Let $\{e_1,e_2\}$ be i.i.d.~with law $p(\cdot)$ and independent of the $\{\hat B^x:x\in\Zd\}$. If we abuse our earlier notation and let $\langle\cdot\rangle_u$ denote expectation on the product space where $(e_1,e_2)$ and $\xi$ are independent, and $\xi$ is given the voter equilibrium with density $u$, then from \eqref{fdef}, \eqref{LVh} and the fact that $f_i(0,\xi)^2=P_e(\xi(e_1)=\xi(e_2)=i)$, we have
$$f(u) =\theta_0\langle(1-\xi(0))\xi(e_1)\xi(e_2)\rangle_u-\theta_1\langle\xi(0)(1-\xi(e_1))(1-\xi(e_2))\rangle_u.$$
In view of \eqref{prodform}
we will be interested in various coalescence probabilities. For example,
$$
p(x|y,z) = P(\exists t\ge 0 \text{ such that } \hat B^{y}_t=\hat B^{z}_t,\ \hbox{and }\forall t\ge 0,\
  \hat B^{y}_t\neq \hat B^x_t\hbox{ and }\hat
  B^{z}_t\neq \hat B^x_t )
$$
and 
$$p(x|y|z)=P(\hat B^x_t,\hat B^y_t\hbox{ and }\hat
B^z_t\hbox{ are all distinct for all }t).$$
In general walks within the same group coalesce and those separated by at least one bar do not.
If we define
\begin{equation}
\label{p2p3}
  p_2 = p(0|e_1,e_2), \quad
  p_3 = p(0|e_1|e_2),
\end{equation}
where the expected value is taken over $e_1$, $e_2$, 
then by the above formula for $f$ and \eqref{prodform},
\begin{align}\nonumber
  f(u)
  &=\theta_0u(1-u)p_2+\theta_0u^2(1-u)p_3-\theta_1(1-u)up_2-\theta_1(1-u)^2up_3\\  
  \label{lvf}&=u(1-u)[\theta_0p_2-\theta_1(p_2+p_3)+up_3(\theta_0+\theta_1)].
\end{align}

To see what this might say about the Lotka-Volterra model introduce
\begin{equation}\label{u*def}
  u^*(\theta_1/\theta_0) = \frac{ \theta_1 (p_2+p_3) - \theta_0 p_2 }{
    p_3(\theta_1 + \theta_0)}
\end{equation}
so that $f(u)=0$ for $u=0,1$ or $u^*(\theta_1/\theta_0)$. If
\begin{equation}\label{m0def}
  m_0\equiv {p_2\over p_2+p_3},
\end{equation}
then $u^*(m)$ increases from $0$ to $1$ as $m$ increases from $m_0$ to $m_0^{-1}$.  
We slightly abuse the notation and write $u^*$ for $u^*(\theta_1/\theta_0)$.

To analyze the limiting PDE we decompose the $\theta_0-\theta_1$ plane into $5$ open sectors 
drawn in Figure \ref{fig:LVpt} on which the above $0$'s are all simple. 

\begin{figure}
\begin{center}
\begin{picture}(260,260) 
\put(65,67){$R_1$}
\put(120,91){$R_3$} 
\put(87,138){$R_2$}
\put(115,137){$R_4$} 
\put(128,117){$R_5$}
\put(95,95){$\bullet$}
\put(220,88){$\theta_0$}
\put(85,230){$\theta_1$}
\put(208,210){$\hbox{slope }1$}
\put(106,239){$\hbox{slope }m_0^{-1}$}
\put(220,125){$\hbox{slope }m_0$}
\put(97,97){\line(1,1){110}}
\put(97,97){\line(-1,-4){20}}
\put(97,97){\line(1,4){35}}
\put(97,97){\line(4,1){120}}
\put(97,97){\line(-4,-1){80}}
% R1
\put(15,45){\line(1,0){60}}
\put(15,45){\line(1,1){15}}
\put(30,60){\line(1,-1){30}}
\put(60,30){\line(1,1){15}}
% R3
\put(105,60){\line(1,0){40}}
\put(105,60){\line(1,1){20}}
\put(125,80){\line(1,-1){20}}
% R2
\put(40,135){\line(1,0){40}}
\put(40,135){\line(1,-1){20}}
\put(60,115){\line(1,1){20}}
% R5
\put(160,150){\line(1,0){60}}
\put(160,150){\line(1,-1){10}}
\put(170,140){\line(1,1){30}}
\put(200,170){\line(1,-1){20}}
% R4
\put(130,200){\line(1,0){60}}
\put(130,200){\line(1,-1){20}}
\put(150,180){\line(1,1){30}}
\put(180,210){\line(1,-1){10}}
\end{picture}
\caption{Phase diagram near (1,1) for the Lotka-Volterra model
with the shape of $f$ in the regions.}
\label{fig:LVpt}
\end{center}
\end{figure}
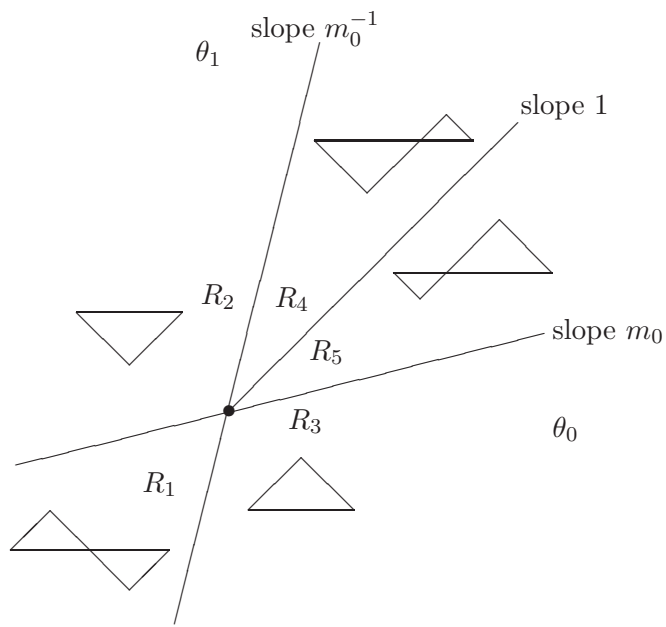

\begin{itemize}
\item If $\theta\in R_1$, $f>0$ on $(0,u^*)$, $f<0$ on $(u^*,1)$,
so $u^*\in(0,1)$ is an attracting fixed point for the ODE.
Then \eqref{pde5} shows the  PDE solutions will converge to $u^*$ given a non-trivial initial condition in $[0,1]$.
As a result
we expect coexistence in the particle system.  

\item If $\theta\in
R_2$, $f<0$ on $(0,1)$, $0$ is an attracting fixed point for
the ODE. Proposition \ref{prop:pde1} implies solutions of the PDE will converge to $0$ given
a non-trivial initial condition and we expect $0$'s to win.

\item If $\theta\in R_3$, $f>0$ on $(0,1)$, $1$ is an attracting
fixed point for the ODE and so by the reasoning from the previous case we expect $1$'s to win.

\item
On $R_4\cup R_5$, $u^*\in (0,1)$ is an unstable fixed point,
while $0$ and $1$ are attracting fixed points for the ODE.
This is case 2 of Durrett and Levin \cite{DL94}, 
so we expect the winner of the competition to be predicted 
by the direction of movement of the speed of the decreasing traveling wave solution
$u(x,t) = w(x-rt)$ with $w(-\infty)=1$ and $w(\infty)=0$. If
$r>0$ then 1's will win and if $r<0$ then 0's will
win. Symmetry dictates that the speed is 0 when
$\theta_0=\theta_1$, so this gives the dividing line between
the two cases and the monotonicity from \eqref{mona} predicts $0$'s
win on $R_4$ while $1$'s win on $R_5$.  Alternatively, by \eqref{signr} 
$r$ has the same sign as $\int_0^1
f(u)\,du$ which is positive in $R_5$ and negative in $R_4$.
\end{itemize}

Our next two results confirm these predictions for $\alpha$ close to
$(1,1)$. For $0\le \eta<1$, define regions that are versions of
$R_1$, $R_2 \cup R_4$ and $R_3 \cup R_5$ shrunken by changing the slopes of the boundary lines:
\[C^\eta=\Bigl\{(\alpha_0,\alpha_1)\in[0,1]^2: {(\alpha_0-1)(1-\eta)\over m_0}<\alpha_1-1<{m_0(\alpha_0-1)\over 1-\eta}
\Bigr\},\]
\begin{align*}\Lambda^\eta_0=\Bigl\{(\alpha_0,\alpha_1)\in(0,\infty)^2:\
  & 0<\alpha_0\le 1,\, m_0(1-\eta)(\alpha_0-1)<\alpha_1-1,\\ 
  &\hbox{or }1\le \alpha_0,\, (1+\eta)(\alpha_0-1)<\alpha_1-1\Bigr\},\\
  \Lambda^\eta_1=\Bigl\{(\alpha_0,\alpha_1)\in(0,\infty)^2:\
  &0<\alpha_0\le 1,\, \alpha_1-1<{\alpha_0-1\over
    m_0(1-\eta)},\\ 
  &\hbox{or }1\le \alpha_0,\,
  \alpha_1-1<(1-\eta)(\alpha_0-1)\Bigr\}.\end{align*}

\begin{thm}\label{thm:coex} For $0<\eta<1$ there is an
  $r_0(\eta)>0$, non-decreasing in $\eta$,  so that for the $LV(\alpha)$:

\noindent
(i) Coexistence holds for $(\alpha_0,\alpha_1)\in C^\eta$ and $1-\alpha_0<r_0(\eta)$.

\noindent
(ii) If $(\alpha_0,\alpha_1)$ is as in (i) and
$\nu_\alpha$ is a stationary distribution of
with $\nu_\alpha(\xi\equiv0\hbox{ or }\xi\equiv 1)=0$, then
  \[\sup_x\Bigl|\nu_\alpha(\xi(x)=1)-u^*\Bigl({\alpha_1-1\over
    \alpha_0-1}\Bigr)\Bigr|\le \eta.\]
\end{thm}

(i) is a consequence of Theorem 4 of \cite{CP-2}, which also applies to more general perturbations.
The main conditions of that result translate into $f'(0)\ge \eta$ and $f'(1)\ge \eta$ in our present setting (see \eqref{f'rep} in Section~\ref{ssec:CPcomp} below).
(ii) sharpens (i) by showing that if $\eta$ is small then the density of 1's in any nontrivial stationary
distribution is close to the prediction of mean-field theory. Durrett and Neuhauser \cite{DN94} prove
results of this type for some systems with fast stirring and $f(1)<0$.   Neuhauser and Pacala \cite{NP} conjectured that coexistence
holds for all $\alpha_0=\alpha_1<1$ (see Conjecture 1 of
that paper) and proved it for $\alpha_i$
sufficiently small. Hence (i) provides further evidence for
the general conjecture. 

The next result is our main contribution to the understanding of Lotka-Volterra models. It shows
that (i) of the previous result is asymptotically sharp, and verifies a
conjecture in \cite{CP-2} (after Theorem~4 in that work). We assume $p$ has finite support but believe this condition is not needed. 

\note{finrange}
\begin{thm}\label{thm:ext} Assume $p(\cdot)$ has finite support. 
  For $0<\eta<1$ there is an $r_0(\eta)>0$, non-decreasing in $\eta$, so that for the $LV(\alpha)$:

  (i) $0$'s take over for $(\alpha_0,\alpha_1)\in\Lambda
  _0^\eta$ and $0\le |\alpha_0-1|<r_0(\eta)$,

  (ii) $1$'s take over for $(\alpha_0,\alpha_1)\in\Lambda
  _1^\eta$ and $0\le |\alpha_0-1|<r_0(\eta)$.
\end{thm}
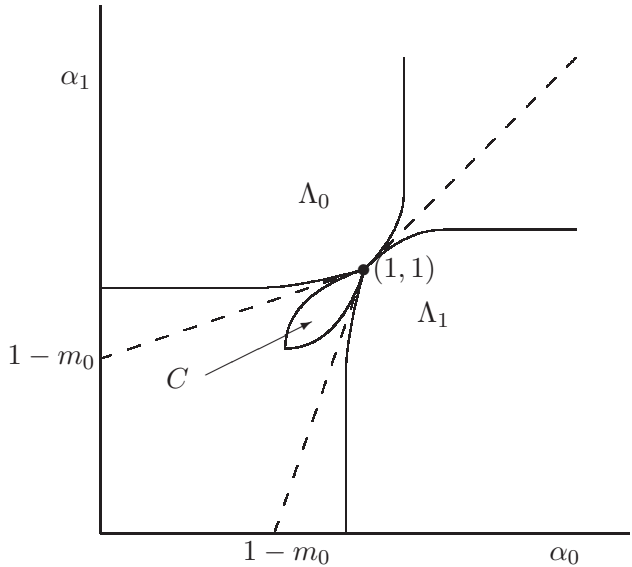
\begin{figure}
\begin{center}
  \begin{picture}(260,260) 
    % axes
    \drawline(30,30)(30,230)
    \drawline(30,30)(230,30)
    %upper diagonal
    \dashline{5}(130,130)(210,210)
    % tangent lines
    \dashline{5}(30,96)(130,130)
    \dashline{5}(96,30)(130,130)
    % intercept and axes labels
    \put(84,20){$1-m_0$} 
    \put(-5,94){$1-m_0$}
    \put(200,20){$\alpha_0$}
    \put(15,200){$\alpha_1$} 
    %names of regions
    \put(55,85){$C$} 
    \put(70,90){\vector(2,1){40}}
    \put(150,110){$\Lambda_1$} 
    \put(105,155){$\Lambda_0$}
    \put(127,127){$\bullet$} 
    \put(133, 127){$(1,1)$}
    % C_0 region
    \qbezier(100,100)(100,120)(130,130)
    \qbezier(100,100)(120,100)(130,130)
    % Lambda_1 region
    \qbezier(130,130)(145,145)(160,145)
    \qbezier(123,90)(123,105)(130,130)
    \drawline(160,145)(210,145)
    \drawline(123,90)(123,30)
    % Lambda_0 region
    \qbezier(130,130)(145,145)(145,160)
    \qbezier(90,123)(105,123)(130,130)
    \drawline(145,160)(145,210)
    \drawline(90,123)(30,123)
     % Figure name
    % \put(95,4){Figure 2}
    %\drawline(30,30)(130,130)
  \end{picture}
\caption{Coexistence on $C$, type $i$ takes over on $\Lambda_i$.}
\label{fig:lvregion}
\end{center}
\end{figure}
Conjecture 2 of \cite{NP} states that $1$'s take over for
$\alpha\in\Lambda_1^0$, $\alpha_0>1$ and $0$'s take over for
$\alpha\in\Lambda_0^0$, $\alpha_0>1$.  Theorem~\ref{thm:ext}
establishes this result asymptotically as  $\alpha$ gets close to $(1,1)$
at least for $d\ge 3$.

Together, Theorems \ref{thm:coex} and \ref{thm:ext} give a fairly complete description of the phase diagram
of the Lotka-Volterra model near the voter model. 
In Figure \ref{fig:lvregion} $C$ is the union over $\eta\in (0,1)$ of the
regions in Theorem~\ref{thm:coex}~(i) on which there is
coexistence, and $\Lambda_i$, $i=0,1$, is the
union over $\eta$ of the regions in Theorem~\ref{thm:ext}
(i) and (ii), respectively, on which $i$'s take
over, as well as other parameter values for which the same result holds by monotonicity.  For example, if $(\alpha_0,\alpha_1)\in\Lambda_1$, with $\alpha_0\wedge\alpha_1\ge 1/2$, and $(\alpha'_0,\alpha_1')$ has $\alpha'_0\ge \alpha_0$ and $\alpha'_1\le \alpha_1$, then by \eqref{mona}, $(\alpha_0',\alpha_1')\in\Lambda_1$.
Theorem~\ref{thm:coex} (i) and Theorem~\ref{thm:ext} show
that the three mutually exclusive classifications of
coexistence, $0$'s take over, and $1$'s take over, occur on the
three regions, $C$, $\Lambda_0$ and $\Lambda_1$, meeting at $(1,1)$ along mutually
tangential lines with slopes $m_0$, $1$ and $m_0^{-1}$. 

\subsection{Evolution of cooperation}\label{ssec:evolcoop}

Ohtsuki et al~\cite{Oetal} considered a system in which each site of a large ($N$ vertex) graph $G$ is occupied by a cooperator (1) or a
defector (0). Simplifying their setting a bit, we will assume that each vertex in $G$ has $k$ neighbors.  The interaction between these two types is governed by a payoff matrix with real entries
$$\begin{matrix}
 & {\bf C} & {\bf D} \\
{\bf C} & \alpha & \beta \\
{\bf D} & \gamma & \delta 
\end{matrix}
$$
This means that a cooperator receives a payoff $\alpha$ from each neighboring cooperator and a payoff $\beta$ from each neighboring defector, while for defectors the payoffs are $\gamma$ and $\delta$ from each neighboring cooperator or defector, respectively. The terminology is motivated by the particular case in which each cooperator pays a benefit $b\ge 0$ to each neighbor at a cost $c\ge 0$ per neighbor, while each defector accepts the benefit but pays no cost.  The resulting payoff matrix is then
\begin{align}\label{coopdef}
\begin{pmatrix}
\alpha & \beta \\
\gamma & \delta 
\end{pmatrix}
= \begin{pmatrix}
b-c & -c \\
b   &  0 
\end{pmatrix}.
\end{align}
In this case the payoff for $D$ always exceeds that for $C$ irregardless of the state making the payoff. 
As a result in a homogeneously mixing population cooperators will die out.
The fact that such cooperative behavior may nonetheless take over in a spatial competition is the reason for interest in these kind of models in evolutionary game theory.  In a spatial setting the intuition is that it may be possible to the $C$'s to form cooperatives which collectively have a selective advantage.

If $n_i(y)$ is the number of neighboring $i$'s for site $y\in G$, $i=0,1$, and $\xi(y)\in\{0,1\}$ is the state at site $y$, then the fitness $\rho_i(y)$ of site $y$ in state $i$ is determined by its local payoffs through
\begin{align}\label{fitnessdef}
\rho_1(y)=1-w+w(\alpha n_1(y)+\beta n_0(y))&\hbox{ if }\xi(y)=1\\
\nonumber \rho_0(y)=1-w+w(\gamma n_1(y)+\delta n_0(y))&\hbox{ if }\xi(y)=0.
\end{align}
Here $w\in[0,1]$ is a parameter determining the selection strength. Clearly for some $w_0(\alpha,\beta,\gamma,\delta,k)>0$, $\rho_i\ge 0$ for $w\in[0,w_0]$, which we assume in what follows.  For the death-birth dynamics in \cite{Oetal} a randomly chosen individual is eliminated at $x$ and its neighbors compete for the vacated site with success proportional to their fitness.  We consider the continuous time analogue which is the spin-flip system $\xi_t(x)\in\{0,1\},\ x\in G$, with rates (write $y\sim x$ if and only if $y$ and $x$ are neighbors)
\begin{align}\label{evolrates}
\nonumber c(x,\xi)&=(1-\xi(x))r_1(x,\xi)+\xi(x)r_0(x,\xi),\\
r_i(x,\xi)&=\frac{\sum_{y\sim x}\rho_{i}(y)1(\xi(y)=i)}{\sum_{y\sim x}\rho_1(y)\xi(y)+\rho_0(y)(1-\xi(y))}\in [0,1].
\end{align}

In \cite{Oetal} the authors use a non-rigorous pair approximation and diffusion approximation to argue that for the cooperator-defector model in \eqref{coopdef}, for large population size $N$ and small selection $w$, cooperators are ``favored" if and only if $b/c>k$.  Here ``favored" means that starting with a single cooperator the probability that cooperators take over is greater than $1/N$, the corresponding probability in a selectively neutral model.  They also carried out a number of simulations which showed reasonable agreement for $N\gg k$ although they noted that $b/c>k$ appeared to be necessary but not sufficient in general.  It is instructive for the reader to consider the nearest neighbor case on $\Z$ starting with cooperators to the right of $0$ and defectors to the left.  It is then easy to check that the $C/D$ interface will drift to the left, and so cooperators take over, if and only if $b/c>2$.  This was noted in \cite{Oetal} as further evidence for their $b/c>k$ rule.

Our main result here (Corollary~\ref{cor:coopdef} below) is a rigorous verification of the $b/c>k$ rule for general symmetric translation invariant graphs with vertex set $\Zd$ when $w$ is small. 
More precisely, choose a symmetric (about $0$) set
${\cal N}$ of neighbors of $0$ of size $k$, not containing $0$, and consider the graph with vertex set $\Zd$ and $x\sim y$ if and only if $x-y\in\cN$.  Assume also that the additive group generated by $\cN$ is $\Zd$ and $\sum_{x\in\cN}x_ix_j/k=\sigma^2\delta_{ij}$, so that $p(x)=k^{-1}1(x\in\cN)$ satisfies the conditions on our kernel given in, and prior to, \eqref{expbd1}. Set $w=\vep^2$.  For $x\in\Zd$ and $\xi\in\{0,1\}^{\Zd}$, let
\begin{align*}
f_i^{(2)}(x,\xi)&=k^{-1}\sum_{y\sim x}1(\xi(y)=i)f_i(x,\xi)=k^{-2}\sum_{y\sim x}\sum_{z\sim y}1(\xi(y)=\xi(z)=i)\in[0,1],\\
\theta_1(x,\xi)&=(\beta k-1)f_1(x,\xi)+k(\alpha-\beta)f_1^{(2)}(x,\xi),\\
 \theta_0(x,\xi)&=(\gamma k-1)f_0(x,\xi)+k(\delta-\gamma)f_0^{(2)}(x,\xi),\\
\phi(x,\xi)&=(\theta_0+\theta_1)(x,\xi).
\end{align*}
Using \eqref{fitnessdef} in \eqref{evolrates}, we get 
\begin{equation}\label{evolrates2}
r_i(x,\xi)=\frac{f_i+\vep^2\theta_i}{1+\vep^2\phi}(x,\xi).
\end{equation}
Note that 
\begin{equation}\label{thetabnds}
|\theta_1|\vee|\theta_0|\vee|\phi|(x,\xi)\le 2k(1+|\alpha|+|\beta|+|\gamma|+|\delta|)\equiv R,
\end{equation}
and
\begin{align*}\frac{f+\vep^2\theta}{1+\vep^2\phi}&=f+\vep^2(\theta-f\phi)+\vep^4\phi(f\phi-\theta)\Bigl[\sum_0^\infty(-\vep^2\phi)^k\Bigr]\\
&=f+\vep^2(\theta-f\phi)+\vep^4\psi_\vep(f,\phi,\theta).
\end{align*}
It follows that $\xi^\vep_t(\vep x)=\xi_{\vep^{-2}t}(x),\ x\in\Zd$ ($\xi$ has rates given by \eqref{evolrates}) 
has spin-flip rates given by \eqref{xirates} with
\begin{equation}\label{evolh}
h_i^\vep(x,\xi)=\theta_i(x,\xi)-f_i(x,\xi)\phi(x,\xi)+\vep^2\psi_\vep(f_i(x,\xi),\phi(x,\xi),\theta_i(x,\xi)).
\end{equation}
If
\begin{equation}\label{vepevol}
\vep^2<(2R)^{-1},
\end{equation}
then one easily from \eqref{thetabnds} that 
\begin{equation}\label{psivepbnd}
|\psi_\vep((f_i,\phi,\theta_i)(x,\xi))|\le 2R(R+1).
\end{equation}
From this and \eqref{evolh} it is clear that the hypotheses of Proposition~\ref{finiterangeh} hold with
\begin{equation}\label{evolgrates}\Vert \hat g^\vep_i-\hat g_i\Vert_\infty\le \vep^22R(R+1)
\end{equation}
and
$$h_i(x,\xi)=(\theta_i-f_i\phi)(x,\xi),\ i=0,1.$$
Hence $\xi^\vep_t$ is a voter model perturbation.  This also implies
\begin{equation}\label{hiden}
h_0+h_1=\theta_0+\theta_1-\phi=0.
\end{equation}
Some elementary arithmetic and \eqref{hiden} lead to
\begin{align}
\nonumber h_0(x,\xi)&=(\gamma-\beta)kf_0f_1(x,\xi)+k(\delta-\gamma)f_0^{(2)}(x,\xi)\\
\nonumber&\phantom{=(\gamma-\beta)}-kf_0(x,\xi)[(\alpha-\beta)f_1^{(2)}+(\delta-\gamma)f_0^{(2)}](x,\xi),\\
%h_1(x,\xi)=(\beta-\gamma)kf_0f_1(x,\xi)+&k(\alpha-\beta)f_1^{(2)}(x,\xi)\\
%&-kf_1(x,\xi)[(\alpha-\beta)f_1^{(2)}+(\delta-\gamma)f_0^{(2)}](x,\xi).
\label{hiden2}h_1(x,\xi)&=-h_0(x,\xi).
\end{align}

As before, let $e_1,e_2,e_3$ denote i.i.d.~ random variables with law $p$. If $P_e$ denotes averaging over the $e_i$'s then we have
\begin{align} \label{fiden}&f_i(0,\xi)=P_e(\xi(e_1)=i),\ f_i^{(2)}(0,\xi)=P_e(\xi(e_1)=i,\xi(e_1+e_2)=i),\\
\nonumber&f_{i_1}(0,\xi)f^{(2)}_{i_2}(0,\xi)=P_e(\xi(e_1)=i_1,\xi(e_2)=i_2,\xi(e_2+e_3)=i_2),
\end{align}
and similarly for higher order probabilities.  We also continue to let $\langle\cdot\rangle_u$ denote expectation on the product space where $(e_1,e_2,e_3)$ and the voter equilibrium $\xi$ are independent. If $\hat \xi=1-\xi$, then starting with \eqref{fdef} we have,
\begin{align*}
\frac{f(u)}{k}=&k^{-1}\langle \hat\xi(0)h_1(0,\xi)-\xi(0)h_0(0,\xi)\rangle_u=k^{-1}\langle h_1(0,\xi)\rangle_u,
\end{align*}
where in the last equality \eqref{hiden2} is used to see that what appears to be a quartic polynomial is actually a cubic. Using \eqref{fiden} and some arithmetic we obtain
\begin{align}\label{cubicf}
\frac{f(u)}{k}=&(\beta-\gamma)\langle\hat\xi(e_1)\xi(e_2)\rangle_u+(\alpha-\beta)\langle\hat\xi(e_1)\xi(e_2)\xi(e_2+e_3)\rangle_u\\
\nonumber&+(\gamma-\delta)\langle\xi(e_1)\hat\xi(e_2)\hat\xi(e_2+e_3)\rangle_u.
\end{align}

To simplify further we will use a simple lemma for coalescing random walk probabilities (Lemma~\ref{coalp} in Section \ref{ssec:ApplEg} below) together with the duality formula \eqref{prodform} to establish the following more explicit expression for $f$ in Section~\ref{ssec:ApplEg}.

\begin{lem}\label{evalf}$\frac{f(u)}{k}=[(\beta-\delta)+k^{-1}(\gamma-\delta)]p(0|e_1)u(1-u)$\\ 
$\phantom{Lemma 1.10.\
  \frac{f(u)}{k}==}+[(\alpha-\beta)-(\gamma-\delta)][u(1-u)(p(e_1|e_2,
e_2+e_3)+up(e_1|e_2|e_2+e_3))].$ 
\end{lem}

Rather than try to analyze this cubic as in Section~\ref{ss:PDE}, assume $\alpha-\beta=\gamma-\delta$ (which holds in our motivating example) so that $f$ becomes a quadratic with roots at $0$ and $1$. 
If $\beta-\delta>k^{-1}(\delta-\gamma)$, then $f$ is strictly positive on $(0,1)$ and so Proposition~\ref{prop:pde1} shows the PDE solutions will converge to $1$.
 If $\beta-\delta<k^{-1}(\delta-\gamma)$, then  $f$ is strictly negative on $(0,1)$ and so by symmetry the PDE solutions will converge to $0$. 
As a result for $w=\vep^2$ small, in the former case we expect $1$'s to take over and in the latter case we expect $0$'s to take over, and this is in fact the case. The following result is proved in Section~\ref{ssec:ApplEg}.

\begin{thm}\label{thm:evol1} Consider the spin-flip system on $\Zd$ ($d\ge 3$) with rates given by 
 \eqref{evolrates} where $\alpha-\beta=\gamma-\delta$.  If $\gamma-\delta>k(\delta-\beta)$, then $1$'s take over for $w>0$ sufficiently small; if $\gamma-\delta<k(\delta-\beta)$, then $0$'s take over  for $w>0$ sufficiently small.
 \end{thm}

The particular instance of  \eqref{coopdef} follows as a special case.

\begin{cor}\label{cor:coopdef} Consider the spin-flip system on $\Zd$ ($d\ge 3$) with rates given by 
 \eqref{evolrates} where the payoff matrix is given by \eqref{coopdef}.  If $b/c>k$, then the cooperators take over for $w>0$ sufficiently small, and if $b/c<k$, then the defectors take over for $w>0$ sufficiently small.
 \end{cor}
 \begin{proof} In this case $\alpha-\beta=\gamma-\delta=b$, $\delta-\beta=c$, and so we have $$\gamma-\delta>k(\delta-\beta)\hbox{ iff }b>kc\hbox{ iff }b/c>k.
 $$
 \end{proof}

\subsection{Nonlinear voter models}\label{ssec:nlv}

Molofsky et al.~\cite{Metal} considered a discrete time particle system on $\Z^2$ in which
each site is in state 0 or 1 and
$$
P( \xi_{n+1}(x,y) = 1 | \xi_n ) = p_k
$$
if $k$ of the sites $(x,y), (x+1,y), (x-1,y), (x,y+1), (x,y-1)$ are in state 1.
They assumed that $p_0=0$ and $p_5=1$, so that all 0's and all 1's were absorbing states
and $p_1=1-p_4$ and $p_2=1-p_3$, so that the model was symmetric under interchange of 0's
and 1's. If the states of adjacent sites were independent then the density would evolve according
to the mean field dynamics
\begin{align*}
x_{t+1} = h(x_t) & = p_1 \cdot 5x_t(1-x_t)^4 + p_2 \cdot 10x_t^2(1-x_t)^3 \\
& + (1-p_2) \cdot 10x_t^3(1-x_t)^2 + (1-p_1) \cdot 5x_t^4(1-x_t) + x_t^5
\end{align*}

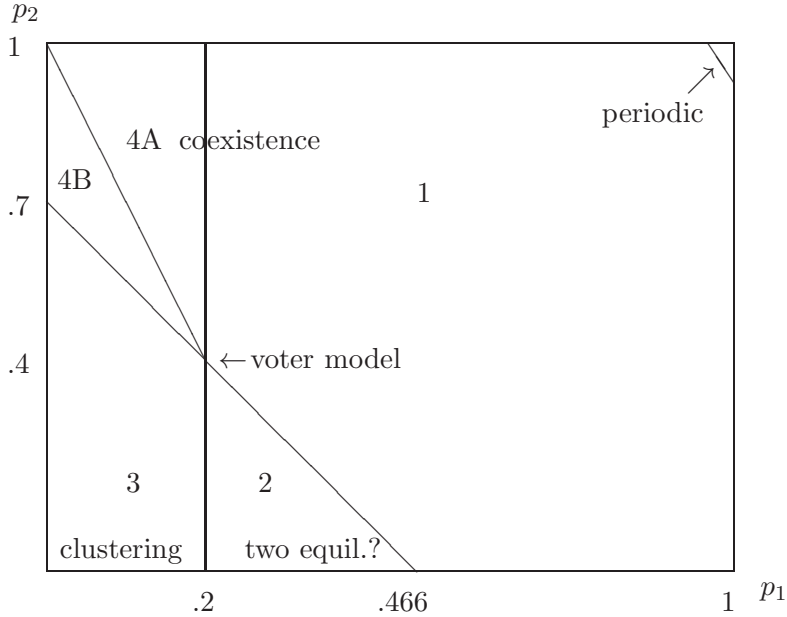
\begin{figure}[ht]
\begin{center}
\begin{picture}(320,260)
\put(30,30){\line(1,0){260}}
\put(30,30){\line(0,1){200}}
\put(30,230){\line(1,0){260}}
\put(290,30){\line(0,1){200}}
\put(30,170){\line(1,-1){140}}
\put(90,30){\line(0,1){200}}
\put(90,110){\line(-1,2){60}}
\put(170,170){1}
\put(110,60){2}
\put(60,60){3}
\put(60,190){4A}
\put(15,165){.7}
\put(34,175){4B}
\put(15,105){.4}
\put(15,225){1}
\put(17,240){$p_2$}
\put(85,15){.2}
\put(155,15){.466}
\put(285,15){1}
\put(300,21){$p_1$}
\put(95,107){$\leftarrow$}
\put(107,107){voter model}
\put(35,35){clustering}
\put(105,35){two equil.?}
\put(80,190){coexistence}
\put(290,215){\line(-2,3){10}}
\put(272,213){$\nearrow$}
\put(240,200){periodic}
\end{picture}
\caption{Conjectured phase diagram for the discrete time two-dimensional nonlinear voter model
of \cite{Metal}.}
\label{fig:nlvpt}
\end{center}
\end{figure}

Based on simulations and an analysis of the mean-field equation, Molofsky et al~\cite{Metal} predicted the
phase diagram given in Figure \ref{fig:nlvpt}. To explain this, $h(x)=x$ is a fifth degree equation with 0, 1/2, and 1 as roots. 
$h'(0)=h'(1)=5p_1$ so 0 and 1 are locally attracting if $5p_1<1$ and unstable if $5p_1>1$.
$h'(1/2) = (15-15p_1-10p_2)/8$, so 1/2 is locally attracting if $15p_1+10p_2>7$ and
unstable if $15p_1+10p_2<7$. From the stability properties of 0, 1/2, and 1, it is easy to determine when there are 
additional roots $\alpha$ and $1-\alpha$ in the unit interval and whether or not they are stable.
The four shapes are given in Figure \ref{fig:quint}. To make the drawing easier we have represented the quintic as
a piecewise linear function.

\begin{figure}[ht]
\begin{center}
\begin{picture}(250,105)
\put(30,25){3}
\put(37,27){$\bullet$}
\put(40,30){\line(1,0){60}}
\put(40,30){\line(1,-1){15}}
\put(55,15){\line(1,1){30}}
\put(85,45){\line(1,-1){15}}
\put(97,27){$\bullet$}
\put(50,32){$\leftarrow$}
\put(80,22){$\rightarrow$}
\put(30,70){1}
\put(40,75){\line(1,0){60}}
\put(40,75){\line(1,1){15}}
\put(55,90){\line(1,-1){30}}
\put(67,72){$\bullet$}
\put(85,60){\line(1,1){15}}
\put(50,67){$\rightarrow$}
\put(80,77){$\leftarrow$}
\put(130,25){4}
\put(137,27){$\bullet$}
\put(140,30){\line(1,0){80}}
\put(140,30){\line(1,-1){15}}
\put(155,15){\line(1,1){20}}
\put(175,35){\line(1,-1){10}}
\put(177,27){$\bullet$}
\put(185,25){\line(1,1){20}}
\put(205,45){\line(1,-1){15}}
\put(217,27){$\bullet$}
\put(150,32){$\leftarrow$}
\put(200,22){$\rightarrow$}
\put(130,70){2}
\put(140,75){\line(1,0){80}}
\put(140,75){\line(1,1){15}}
\put(155,90){\line(1,-1){20}}
\put(167,72){$\bullet$}
\put(175,70){\line(1,1){10}}
\put(185,80){\line(1,-1){20}}
\put(187,72){$\bullet$}
\put(205,60){\line(1,1){15}}
\put(150,67){$\rightarrow$}
\put(200,77){$\leftarrow$}
\end{picture}
\caption{Four possible shapes of the symmetric quintic $f$. Black dots indicate the locations
of stable fixed points.}
\label{fig:quint}
\end{center}
\end{figure}
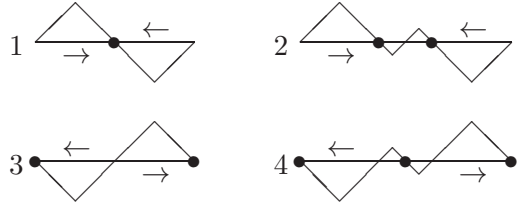

The implications of the shape of $f(u)$ ($=h(u)-u$ in the above) for the behavior 
for the system will be discussed below in the context of a similar system in continuous time.
There we will see that the division between 4A and 4B is dictated by the speed of traveling
waves for the PDE. Here we have drawn the ``Levin line" $6p_1+2p_2=2$ which comes from
computing the expected number of 1's at time 1 when we have two adjacent 1's at time 0.
Simulations suggest that the true boundary curve exits the square at $(0.024,1)$,
see page 280 in \cite{Metal}.

For our continuous time model the perturbation rate from the voter model
is determined by four points chosen at random from $x + {\cal N}$ where 
${\cal N}$ is the set of integer lattice points in $([-L,L]^d - \{0\})$. Let $a(i)\ge 0$ be the flip rate at a given site
when $i$ (randomly chosen) neighbors have a type disagreeing with that of the site and suppose $a(0)=0$.
Let $(Y_1, \ldots Y_4)$ be chosen at random and without replacement from ${\cal N}$.  Then we consider the spin-flip system $\xi(x), x\in\Zd$ with rates
\begin{align*}c(x,\xi)=c^v(x,\xi)+\vep^2[&(1-\xi(x))E_Y(g_1(\xi(x+Y^1),\dots,\xi(x+Y^4)))\\
&+\xi(x)E_Y(g_0(\xi(x+Y^1),\dots,\xi(x+Y^4)))],
\end{align*}
where
$$g_1(\xi_1,\dots,\xi_4)=a\Bigl(\sum_1^4\xi_i\Bigr), \quad g_0(\xi_1,\dots,\xi_4)=a\Bigl(4-\sum_1^4\xi_i\Bigr).$$
Then the rescaled system $\xi_t^\vep(\vep x)=\xi_{t\vep^{-2}}(x),\, x\in\Zd$ is a voter model perturbation since the required conditions are trivial (clearly \eqref{hgrepn} holds with $\vep_1=\infty$). 
We call $\xi$ the nonlinear voter model.  General models of this type were introduced and studied in \cite{CD91}.

We abuse our notation as before and incorporate expectation with respect to an independent copy of $Y=(Y^1,\dots,Y^4)$ in our voter equilibrium expectation $\langle\cdot\rangle_u$.  If $Y^0\equiv0$, then a short calculation shows that our reaction function in \eqref{fdef} is now 
\beq \label{nlvf}f(u)=\sum_{j=1}^4a(j)(q_j(u)-q_j(1-u)),\eeq
where
\[q_j(u)={4\choose j}\langle\prod_{i=0}^{4-j}(1-\xi(Y^i))\prod_{i=5-j}^4 \xi(Y^i)\rangle_u.\]
Clearly
\beq\label{froots} f(0)=f(1)=f(1/2)=0\hbox{ and }f(u)=-f(1-u).
\eeq
It does not seem easy to calculate $f$ explicitly, but if $L$ is large, 
most of the sum comes from $Y^i$ that are well separated and so the $\xi$ values at the above sites should be nearly independent.  To make this precise let 
$A=\{Y^{5-j},\dots, Y^4\}$, $B=\{Y^0,\dots,Y^{4-j}\}$, ($1\le j\le 4$), and note by \eqref{prodform} that
\begin{align*} 
q_j(u)&={4\choose j}\sum_{i=1}^j\sum_{k=1}^{5-j}u^i(1-u)^k P(|\hat\xi^A_\infty|=i, |\hat\xi_\infty^B|=k,\tau(A,B)=\infty)\\
&={4\choose j}u^j(1-u)^{5-j}+\hat q_j(u),
\end{align*}
where
\begin{align*}
\hat q_j(u)&={4\choose j}\Bigl[-u^j(1-u)^{5-j}P(|\hat \xi^{A\cup B}_\infty|<5)\\
&\phantom{={4\choose j}\Bigl[-}+\sum_{i=1}^j\sum_{k=1}^{5-j}1(i+k<5)u^i(1-u)^k P(|\hat\xi^A_\infty|=i, |\hat\xi_\infty^B|=k,\tau(A,B)=\infty)\Bigr]\\
&=\sum_{i=1}^5 d_i(j,L)u^i.
\end{align*}
If $\eta_0(L)=P(|\hat \xi_\infty^{A\cup B}|<5)$, that is the probability that there is a coalescence among the random walks starting at $Y^0,\dots,Y^4$, then it follows easily from the above that
$
|d_i(j,L)|\le c_0\eta_0(L)$.
Use this in \eqref{nlvf} to conclude that $f(u)=f_1(u)+f_2(u)$, where
$f_2$ includes the (smaller) contributions from the $\hat q_j$'s.  That is 
\beq
\nonumber f_2(u)=\sum_{j=1}^5 e(j,L)u^j,\eeq
where
\beq\label{f2cobnds}\sup_{1\le j\le 5}|e(j,L)|\le c_1\eta_0(L),
\eeq
and 
\begin{align*}
f_1(u) = & - u [ a(4) (1-u)^4 + a(3) \cdot 4u(1-u)^3 + a(2) \cdot 6u^2 (1-u)^2 + a(1) \cdot 4u^3(1-u) ] \\ 
& {} + (1-u) [ a(4) u^4 + a(3) \cdot 4u^3(1-u) + a(2) \cdot 6u^2 (1-u)^2 + a(1) \cdot 4u(1-u)^3 ] \\ 
= & b_1 u(1-u)^4 + b_2 u^2 (1-u)^3 - b_2 u^3(1-u)^2 - b_1 u^4(1-u)
\end{align*}
where $b_1=4a(1)-a(4)$ and $b_2=6a(2)-4a(3)$. By symmetry we have
\beq\label{f1roots} f_1(0)=f_1(1)=f_1(1/2)=0\hbox{ and }f_1(u)=-f_1(1-u).\eeq
Clearly $\eta_0(L)\to 0$ as $L\to\infty$, in fact well-known return estimates (such as Lemma~\ref{lem:notcrowd}(a) below with $t_0=0$, $r_0=1$ and $p$ large) and a simple optimization argument show that 
\beq\label{eta0bound}
\eta_0(L)\le C_\delta L^{{-[d(d-2)/(2(d-1))}]+\delta},\ \delta>0.\eeq
To prepare for the next analysis
we note that 
\begin{align*}
f_1'(u) = & b_1[ (1-u)^4 - 4 u(1-u)^3 ] + b_2 [ 2u(1-u)^3 - 3 u^2(1-u)^2] \\
& - b_2[ 3u^2(1-u)^2 - 2u^3(1-u) ] - b_1 [ 4u^3(1-u) - u^4 ],
\end{align*}
and so we have $f_1'(0)=f_1'(1)=b_1$ and $f_1'(1/2) = -(6b_1+2b_2)/16$. A little calculus,
left for the reader, shows
\beq
\int_0^{1/2} f_1(u) \, du = \frac{5b_1+b_2}{192}=-\int_{1/2}^1f_1(u)\, du.
\label{nlvint}
\eeq

\begin{figure}
\begin{center}
\begin{picture}(200,210)
\put(100,30){\line(0,1){150}}
\put(170,30){\line(-1,1){150}}
\put(100,100){\line(-1,2){40}}
\put(140,140){1}
\put(60,60){3}
\put(120,40){2}
\put(45,160){4B}
\put(75,160){4A}
\put(85,20){$b_1=0$}
\put(30,185){$5b_1+b_2=0$}
\put(150,55){$3b_1+b_2=0$}
\end{picture}
\caption{Phase diagram for the continuous time nonlinear voter model with large range in $d\ge 3$.}
\label{fig:nlvpt2}
\end{center}
\end{figure}
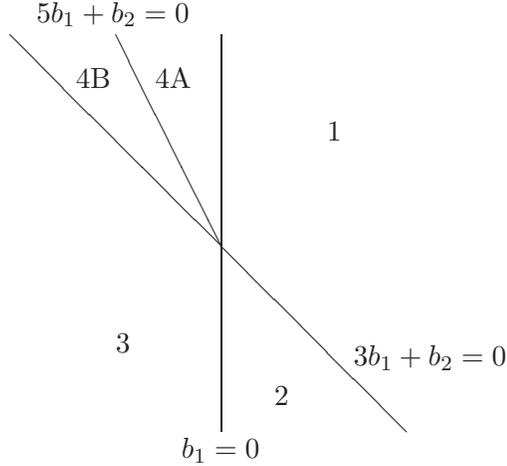

We are now ready to describe the phase diagram for the nonlinear voter. Consult
Figure \ref{fig:nlvpt2} for a picture. Note that in what follows when $L$ is chosen large, it is understood that how large depends on $\bar a=(a(1),\dots,a(4))$.

\begin{enumerate}
  \item  $f_1'(0)>0$, $f_1'(1/2)<0$ and so by \eqref{f2cobnds} for large enough $L$ the same is true for $f$. In this case 0, 1/2, and 1 are the only roots of $f$ (all simple) and
  1/2 is an attracting fixed point for the ODE. An application of Proposition~\ref{prop:pde1} on $[0,1/2]$ and a comparison principle, showing that solutions depend monotonically on their initial data (see Proposition 2.1 of \cite{AW78}), to reduce to the case where $v\in[0,1/2]$, shows that any non-trivial solution $u$ of the PDE \eqref{rdpde} satisfies $\liminf_{t\to\infty} \inf_{|x|\le 2wt}u(t,x)\ge 1/2$ for some $w>0$. The same reasoning with $0$ and $1$ reversed shows the corresponding upper bound of $1/2$. Therefore any non-trivial solution of  \eqref{rdpde} will converge to 1/2 and we expect coexistence.
  \item  $f_1'(0)>0$, $f_1'(1/2)>0$ and so by \eqref{f2cobnds} for large enough $L$ the same is true for $f$. In this case 0, 1/2, 1 are unstable fixed points for the ODE
  and there are attracting fixed points for the ODE at $a$ and $1-a$ for some $a \in (0,1/2)$. All are simple zeros of $f$.
 Another double application of Proposition~\ref{prop:pde1} now shows that any non-trivial solution $u(t,x)$ to the PDE will have $\liminf_{t\to\infty} \inf_{|x|\le 2wt} u(t,x) \ge a$ and $\limsup_{t\to\infty}\sup_{|x|\le 2wt} u(t,x)  \le 1-a$, so we expect
  coexistence. Simulations in \cite{Metal}, see Figure 7 and the discussion on page 278 of that work, suggest that in this case there
  may be two nontrivial stationary distributions: one with density near $a$ and the other
  with density near $1-a$. The symmetry of $f$ about $1/2$ is essential for this last possibility as we note below (see Theorem~\ref{thm:nlv2}).
  \item  $f_1'(0)<0$, $f_1'(1/2)>0$ and so by \eqref{f2cobnds} for large enough $L$ the same is true for $f$. In this case 0, 1/2, and 1 are the only roots of $f$ (all simple) and
  1/2 is an unstable fixed point while 0 and 1 are attracting. In this bistable
  case the winner is dictated by the sign of the speed of the traveling wave,
  but by symmetry (recall \eqref{froots}) the speed is 0. One would guess that clustering occurs in this
  case and there are only trivial stationary distributions, but our method yields no result.
  \item $f_1'(0)<0$, $f_1'(1/2)<0$ and so by \eqref{f2cobnds} for large enough $L$ the same is true for $f$. In this case 0, 1/2, 1 are attracting fixed points and 
 there are unstable fixed points at $a$ and $1-a$ for some $a \in (0,1/2)$ (all simple zeros of $f$).
By the discussion in Case II in Section~\ref{ss:PDE} (with $[0,1/2]$ and $[1/2,1]$ in place of the unit interval) there are traveling wave solutions $w_i(x-c_it)$, $i=1,2$ with $w_1(-\infty)=1$, $w_1(\infty)=w_2(-\infty)=1/2$
  and $w_2(\infty)=0$. Symmetry implies $c_2=-c_1$, but we can have $c_1<0<c_2$ (Case 4A)
  in which case Proposition~\ref{prop:pde3} and its mirror image show that solutions to the PDE will converge to 1/2 providing that the initial condition is bounded away from $0$ and $1$ on a large enough set. We again use the comparison principle as in Case 1 to assume the initial data takes values in the appropriate interval, $[0,1/2]$ or  $[1/2,1]$, and assume $L$ is large enough so that the integrals of $f_1$ and $f$ on $[0,1/2]$ (and hence on $[1/2,1]$) have the same sign. Hence we expect coexistence in Case 4A and all invariant distributions to have density near $1/2$.  If $c_1 > 0 > c_2$ (Case 4B)  and $L$ is large enough, there is a standing 
  wave solution $w_0(x)$ of the PDE in $d=1$ with $w_0(-\infty) =0$, $w_0(\infty)=1$ (see p. 284 in \cite{FM81}), and our method
  yields no result. 
  
\end{enumerate}
\begin{thm} \label{thm:nlv} Assume $(b_1,b_2)$ are as in Case 1, 2 or 4A. If $L$ is sufficiently large (depending on $\bar a$) then:

\noindent(a) Coexistence holds for $\vep$ small enough (depending on $L$ and $\bar a$).

\noindent(b) In Case 1 or 4A if $\eta>0$ there is an $\vep_0(\eta,L,\bar a)$ so that if $0<\vep\le \vep_0$ and $\nu$ is any stationary distribution for the nonlinear voter model satisfying\\ 
$\nu(\xi\equiv 0\hbox{ or }\xi\equiv 1)=0$, then 
$$\sup_x\Bigl|\nu(\xi(x)=1)-\frac{1}{2}\Bigr|\le \eta.$$
\end{thm} 

\noindent{\bf Remark.} The proof is given in Section~\ref{ssec:nlv2} below.  Case 4A is of particular interest as there is coexistence even though $f'(0)<0$. Here the low density limit theorem in \cite{CP-1} shows convergence to super-Brownian motion with drift $f'(0)<0$ (see the discussion in Section~\ref{ssec:CPcomp} below).  From this one might incorrectly guess (after an exchange of limits) that there is a.s. extinction of $0$'s for the nonlinear voter model, while our proof of coexistence will show that there is positive probability of survival of $0$'s even starting with a single $0$. 

\medskip

We next show that more can be said in Case 2 if we break the symmetry. This also demonstrates how one can handle higher degree reaction functions in the pde and still apply the general results in the next Section.  Consider the nonlinear voter model $\xi$ as before but now for $\lambda>0$ replace $g_1$ with
\beq
 g_{1,\lambda}(\xi_1,\dots,\xi_4)) = (1+\lambda) a\Bigl(\sum_1^4\xi_i\Bigr),
\label{nlvgdef2}
\eeq
while $g_0$ is unchanged.
To avoid trivialities we assume $\sum_1^4a(j)>0$.  A short calculation now shows that if $f$ is as in \eqref{nlvf}, then our reaction function in \eqref{fdef} becomes
\beq\label{nlvfdef2}
f_{(\lambda)}(u)=f(u)+\lambda\sum_{j=1}^4a(j)q_j(u)\equiv f(u)+\lambda f_0(u)>f(u) \hbox{ on }(0,1).
\eeq
Decomposing $q_j(u)$ as before we get 
\[f_{(\lambda)}(u)=f_{1,\lambda}(u)+f_{2,\lambda}(u),\]
where
\begin{align}\label{flambda1}
f_{1,\lambda}(u) & =  (b_1+4\lambda a(1)) u(1-u)^4 + (b_2+6\lambda a(2)) u^2 (1-u)^3 \\
\nonumber& \quad - (b_2-4\lambda a(3)) u^3(1-u)^2 - (b_1-\lambda a(4)) u^4(1-u)\\
\nonumber& = f_1(u)+\lambda f_3(u)>f_1(u)\hbox{ on }(0,1),\\
\label{flambda2}
f_{2,\lambda}(u)&=\sum_{j=1}^5 e(j,L,\lambda)u^j,\ \hbox{and }\sup_{1\le j\le 5}|e(j,L,\lambda)|\le c_2(\lambda +1)\eta_0(L).
\end{align}
We also have 
\begin{align}\label{flambdader}&f_{1,\lambda}'(0)=b_1+4\lambda a(1), \quad f_{1,\lambda}'(1)=b_1-\lambda a(4).
% \hbox{ and }\\
%\nonumber&f_{1,\lambda}'(1/2) = -\frac{6b_1+2b_2}{16} + \frac{\lambda}{16}(-12a(1)-6a(2)+4a(3)+3a(4))\\
%\nonumber&\phantom{f_{1,\lambda}'(1/2)}  = - \frac{3b_1+b_2}{16} (2+\lambda).
\end{align}

\begin{thm} \label{thm:nlv2} 
Suppose $b_1>0$ and $0<\lambda < b_1/a(4)$. \\
\noindent (a) Coexistence holds for large $L$ and small enough $\vep$ (depending on $L$, $\lambda$  and $\bar a$).  

\noindent(b) Assume $3b_1+b_2<0$ and let $1-a'$ denote the largest root of $f_{1}(u)=0$ in $(0,1)$.  If $\eta>0$, $L>\vep_1(\eta,\bar a)^{-1}$, $0<\lambda<\vep_1(\eta,\bar a)$, $0<\vep<\vep_0(\eta,L,\lambda,\bar a)$ and $\nu$ is any stationary distribution satisfying $\nu(\xi\equiv 0\hbox{ or }\xi\equiv 1)=0$, then 
$$\sup_x\Bigl|\nu(\xi(x)=1)-(1-a')\Bigr|\le \eta.$$
\end{thm} 

See Section~\ref{ssec:nlv2} for the proof.  For a concrete example, consider $a(1)=a(2)=1$, $a(3)=a(4)=3$, which is a version of the majority vote
plus random flipping. Then $b_1=4a(1)-a(4)=1$, $b_2=6a(2)-4a(3)=-6$, and $3b_1+b_2=-3$, and so the hypotheses of (b)
hold for small $\lambda >0$. As $\lambda\downarrow 0$, the density of any invariant measure approaches $1-a$ the largest root
of $f_1(u)=0$, while symmetric reasoning shows that if $\lambda\uparrow 0$, the densities will approach $a$. Of course
the closer $\lambda$ gets to 0, the smaller we must make $\vep$ to obtain the conclusion of
Theorem \ref{thm:nlv2}, so we are not able to prove anything about the case $\lambda=0$.

\subsection{General Coexistence and Extinction Results}\label{gensys}

Our results for the three examples will be derived from general results with hypotheses concerning properties of the limiting PDE. In this subsection, we state those results and give an overview of the contents of the rest of the paper. We work with a voter model perturbation $\xi^\vep_t(x), x\in\vep\Zd,t\ge 0$ throughout. 

In Section \ref{sec:construction} we first introduce a family of Poisson processes (``graphical representation") which we use to define our process $\xi^\vep_t(x)$ on $x\in\vep \Zd$. Using this and working backwards in time we define a ``dual process'' $X$ which is a branching coalescing random walk with particles jumping at rate $\vep^{-2}$ according to $p^\vep(x) = p(x/\vep)$ and with a particle at $x$ giving birth to particles at $x+ \vep Y^1, \ldots x+ \vep Y^{N_0}$ when a reaction occurs at $x$. The ideas in the definition of the dual are a combination of those of Durrett and Neuhauser \cite{DN94} for systems with fast stirring and those of Durrett and Z\"ahle \cite{DZ} for biased voter models that are small perturbations of the voter model. 

Duality allows us to compute the value at $z$ at time $T$ by running the dual process $X^{\vep}$ backwards from time $T$ to time 0 starting with one particle at $z$ at time $T$. Most of the work in Section \ref{sec:construction} is to use coupling to show that for small $\vep$, $X^\vep$ is close to a branching random walk $\hat X^\vep$.
Once this is done, it is straightforward to show (in Section \ref{sec:hydroproofs}) that as $\vep\to 0$ the dual converges to a branching Brownian motion $\hat X^0$, and then derive Theorem \ref{conv} which includes convergence of $P(\xi^\vep_t(x)=1)$ to the solution $u(t,x)$ of a PDE.

%Here and in what follows initial conditions for the PDE are assumed to be continuous and take values in $[0,1]$.

%We also assume the monotonicity condition

%\note{Adjust this as for positivity to include the LV models}
%\begin{equation}
%\label{hmono}
%\xi\le \xi'\hbox{ implies }h^\vep_1(\xi)\le h^\vep_1(\xi')\hbox{ and }h^\vep_0(\xi)\ge h^\vep_0(\xi').
%\end{equation}
%\note{Strengthened Assumption 1 to the linear increasing set of sites and added the rates assumption
%on the $g^\vep$'s to Thm. \ref{thm:exist}.}

Our general coexistence result will be based on the
following assumption about solutions to the PDE. The
coexistence results for the models discussed in the previous
section are obtained by verifying this assumption in
the particular cases.

\begin{assum} \label{a1}
Suppose that there are constants $0< v_0 < u_* \le u^* < v_1< 1$, and $w,L_i>0$, so that

\noindent
(i) if $u(0,x) \ge v_0$ when $|x| \le L_0$, then 
$\liminf_{t\to\infty} \inf_{|x| \le wt} u(t,x) \ge u_*$.

\noindent
(ii) if $u(0,x) \le v_1$ when $|x| \le L_1$, then 
$\limsup_{t\to\infty} \sup_{|x| \le wt} u(t,x) \le u^*$.
\end{assum}

\noindent We also will need a rate of convergence in \eqref{gcvgce}, namely for some $r_0>0$, 
\beq\label{grate}
\sum_{i=0}^1\Vert g_i^\vep-g_i\Vert_\infty\le c_{\ref{grate}}\vep^{r_0}.
\eeq

Assumption~\ref{a1} shows that the limiting PDE in
Theorem~\ref{thm:strongconv} will have solutions which stay
away from $0$ and $1$ for large $t$.  A ``block
construction'' as in \cite{Dur95} will be employed in
Section 6 to convert this information about the PDE into
information about the particle systems. In effect, this
allows us to  
interchange limits as $\vep\to 0$ and $t\to\infty$ and
conclude the existence of a nontrivial stationary
distribution, and also show that any stationary distribution
will have particle density restricted by the asymptotic
behavior of the PDE solutions at $t=\infty$.

Both Theorem \ref{thm:exist} and \ref{thm:nonexist} below will be formulated for the voter model perturbations on $\vep\Z^d$, but the conclusions then follow immediately for the original unscaled particle systems in \eqref{origrates}.

\begin{thm} \label{thm:exist} Suppose Assumption \ref{a1} and \eqref{grate}.  If $\vep>0$ is small enough, then coexistence holds and the nontrivial stationary distribution $\nu$ may be taken to be translation invariant.

If $\eta>0$ and $\vep>0$ is small enough, depending on $\eta$, then 
any stationary distribution $\nu$ such that  
\begin{equation}\label{nucond}\nu\Bigl( \sum_x\xi(x)=0\hbox{ or }\sum_x(1-\xi(x))=0\Bigr)=0
\end{equation}
satisfies $\nu( \xi(x)=1) \in (u_*-\eta, u^*+\eta)$ for all $x$.
\end{thm}

Note that in the second assertion we do not require that $\nu$ be translation invariant.

% The proof of Theorem \ref{thm:exist} is relatively easy
% using the block construction. Details are given in Section
% \ref{sec:exist}. 

Results that assert 0's will take over will require a stronger pde input:

\begin{assum} \label{a2}
There are constants $0<u_1<1$, $c_2, C_2, w>0$, $L_0\ge 3$ 
so that for all $L\ge L_0$, if $u(0,x)\le u_1$ for $|x|\le L$ then for all $t\ge 0$
$$
u(t,x)\le C_2 e^{-c_2 t}\hbox{ for all $|x|\le L+2wt$.} 
$$
\end{assum}  
Finally we need to assume that the constant configuration of all $0$'s is a trap for our voter perturbation, that is,

\beq\label{staydead} g_1^\vep(0,\dots,0)=0,\hbox{ or equivalently } h_1^\vep(0,\underline 0)=0,\hbox{ for }0<\vep\le \vep_0\, ,
\eeq
where $\underline 0$ is the zero configuration in $\{0,1\}^{\Zd}$.
This clearly implies $f(0)=0$ and is equivalent to
it if $g_1^\vep$ does not depend on $\vep$, as is the case
in some examples. Recall the definition of ``$i$''s take
over" from \eqref{takeover} and that $q$ is the law of $(Y^1,\dots, Y^{N_0})$.

\note{finrange}
\begin{thm} \label{thm:nonexist}
Suppose Assumption \ref{a2}, \eqref{grate}, \eqref{staydead}, $p(\cdot)$ and $q(\cdot)$ have finite support, and $f'(0)<0$. Then for $\vep$ small the $0$'s take over.
\end{thm}

We believe the theorem holds without the finite range assumptions on $p$ and $q$.  By Proposition~\ref{finiterangeh} in the above finite support setting, it suffices to assume \eqref{finrangerep} in place of \eqref{hgrepn}, and also assume \eqref{grate} holds for the $\hat g_i^\vep,\hat g_i$ appearing in \eqref{finrangerep} instead of the $g_i^\vep,g_i$.

In order to show that $0$'s take over, say, we will need several additional arguments.  The pde results will only ensure we can get the particle density down to a low level (see Section~\ref{lockill}) but clearly we cannot expect to do better than the error terms in this approximation. To then
drive the population to extinction on a large region with high probability we will need to refine some coalescing random walk
calculations from the low density setting in Cox and Perkins \cite{CP-2}--see Section~\ref{death} and especially Lemma~\ref{survbnd}. 
This is then used as input for another percolation argument of Durrett \cite{Dur92} to guarantee that there are no particles in a linearly growing region. \note{finrange}
Since there was an error in the original proof of the latter we give all the details here in Sections~\ref{sec:perc} and  \ref{ssec:thmextinct}.

Quantifying the outline above,
the first step in the proof of Theorem \ref{thm:nonexist}, taken in Section \ref{lockill} 
is to use techniques of Durrett and Neuhauser \cite{DN94}
to show that if $\xi^\vep_0$ has density at most $u_1$ on $[-L,L]^d$ then at time $T_1=c_1 \log(1/\vep)$ the
density is at most $ \vep^\beta$ on $[-L-wT_1,L+wT_1]^d$ (see Lemma~\ref{lem:gdlow}). Here $\beta$ is a small parameter.  The second step, taken in Section \ref{death},
is to show that if one waits an additional $T_2=c_2\log(1/\vep)$ units of time then there will be no particles
in $[-L-wT_1+AT_2,L+wT_1-A_2T_2]^d$ at time $T_1+T_2$. The first step here (Lemma~\ref{survbnd}) is to show that if 
we start a finite (rescaled) block of ones of density at most $\vep^\beta$ then with probability at least $1-\vep^{\beta/2}$ it will be extinct by time $C\log(1/\vep)$.  Here it is convenient that arguments for the low density regime of \cite{CP-1} (density $\vep^{d-2}$) continue to work all the way up to $\vep^\beta$ and also that the PDE arguments can be used to reduce the density down to $\vep^\beta$. In short,
although the precise limit theorems, Theorem~\ref{conv} and Corollary~1.8 in \cite{CP-1} 
apply in disjoint regimes (particle densities of $1$ and $\vep^{d-2}$, respectively) the methods underlying these results apply in overlapping regimes which together allow us to control the underlying
particle systems completely. Of course getting $1$'s to be extinct in a large block does not give us what we want.  \note{finrange} 
The block construction in \cite{DN94} is suitably modified to establish
complete extinction of $1$'s on a linearly growing set. A comparison result of \cite{LSS}, suitably modified to accommodate our percolation process, is used to simplify this construction.  
%In addition to the fix mentioned above, an additional argument 
%given in Section \ref{sec:perc} is needed to deal with the possibility that particle may jump over blocks of 
%vacant sites produced by the block construction.   

\subsection{Application to the examples}\label{ssec:ApplEg}

\subsubsection{ Lotka-Volterra systems}
 
 \noindent{\it Proof of Theorem \ref{thm:ext}.} (i) Let $0<\eta<1$ and consider first
 \[\alpha_0=\alpha_0^\vep=1-\vep^2,\ \ \alpha_1=\alpha_1^\vep=1-m_0(1-\eta)\vep^2,\]
 so that in the notation of Section~\ref{ss:LV} we have set $\theta_0=-1$, $\theta_1=-m_0(1-\eta)$.  The rescaled Lotka-Volterra process $\xi^\vep$ is a voter model perturbation and from \eqref{lvf} we have 
 \[f(u)=-u(1-u)[\eta p_2+up_3(1+m_0(1-\eta))]<0\hbox{ on }(0,1).\]
 Proposition~\ref{prop:pde4} verifies Assumption~\ref{a2} in Theorem~\ref{thm:nonexist} and $f'(0)<0$ is obvious.  \eqref{grate} is trivial ($g^\vep_i=g_i$) and \eqref{staydead} is immediate from \eqref{LVgdef}. 
\note{finrange} 
The finite range assumption on $q=p\times p$ is immediate from that on $p$. 
Theorem~\ref{thm:nonexist} implies $0$'s take over for $\vep$ small.  Therefore when $0<1-\alpha_0\le r_0(\eta)$ and $\alpha_1=1+m_0(1-\eta)(\alpha_0-1)$, then $0$'s take over for $LV(\alpha)$.  The monotonicity in \eqref{mona} shows this is also the case for $\alpha_1\ge  1+m_0(1-\eta)(\alpha_0-1)$ and $\alpha_0$ as above.

Next consider 
\[\alpha_0=\alpha_0^\vep=1+\vep^2,\ \alpha_1=\alpha_1^\vep=1+(1+\eta)\vep^2\ \ \hbox{ that is, }\theta_0=1,\ \theta_1=1+\eta.\]
In this case we have 
\[f(u)=u(1-u)[p_2-(1+\eta)(p_2+p_3)+up_3(2+\eta)],\]
and so, assuming without loss of generality (by \eqref{mona}) $1+\eta<m_0^{-1}$, from \eqref{u*def} $f$ has a zero, and an unstable fixed point for the ODE, at 
\[ u^*=\frac{(1+\eta)(p_2+p_3)-p_2}{p_3(2+\eta)}\in\Bigl(\frac{1}{2},1\Bigr).\] 
It follows that $\int_0^1 f(u)du<0$ and Proposition~\ref{prop:pde3} establishes Assumption~\ref{a2}.  As above, Theorem~\ref{thm:nonexist} and \eqref{mona} show that $0$'s take over if $0<\alpha_0-1$ is sufficiently small and $\alpha_1\ge 1+(1+\eta)(\alpha_0-1)$.
\medskip

\noindent (ii) Interchange the roles of $0$ and $1$ in (i).\qed

\medskip

 \noindent{\it Proof of Theorem \ref{thm:coex}.} We slightly extend the setting in Section~\ref{ss:LV} and for $\eta\in(0,1-m_0)$ consider
 \begin{equation}\label{alphachoice} \alpha_0^\vep=1-\vep^2, \alpha_1^\vep=1+\vep^2\theta_1^\vep, \hbox{ where }-\theta_1^\vep\in\Bigl[\frac{m_0}{1-\eta},\frac{1-\eta}{m_0}\Bigr],\, \lim_{\vep\downarrow 0}\theta_1^\vep=\theta_1.
 \end{equation}
 Then the rescaled Lotka-Volterra model $\xi^\vep$ in \eqref{resclv} remains a voter model perturbation but now $g_0^\vep$ may now depend on $\vep$.  From \eqref{lvf} we have 
 \[f(u)=u(1-u)[-p_2-\theta_1(p_2+p_3)+up_3(-1+\theta_1)],\]
 which has a zero, and attracting fixed point for the ODE, at 
 \beq\label{ustar}u^*(-\theta_1)=\frac{\theta_1(p_2+p_3)+p_2}{p_3(\theta_1-1)}\in(0,1).
 \eeq
 Proposition~\ref{prop:pde1} and its mirror image, with $0$ and $1$ reversed, establish Assumption~\ref{a1} with $u^*=u_*=u^*(-\theta_1)$ (see \eqref{pde5}). Theorem~\ref{thm:exist} therefore shows that for $0<\vep<\vep_0(\eta)$
 \begin{align}\label{exconc}&\hbox{coexistence holds, and if $\nu$ is a stationary distribution satisfying}\\
\nonumber&\hbox{$\nu(\xi\equiv 0\hbox{ or }\xi\equiv 1)=0$, then }
 \sup_x|\nu(\xi(x)=1)-u^*(-\theta_1)|<\eta.
 \end{align}
 
 Suppose first that (ii) of Theorem~\ref{thm:coex} fails.  Then there is a sequence $\vep_n\downarrow 0$,
 $(\alpha_0^{\vep_n},\alpha_1^{\vep_n})$ and $\kappa>0$ so that \eqref{alphachoice} holds with $\vep=\vep_n$, and there is a stationary measure $\nu_n$ for $\xi^{\vep_n}$ satisfying $\nu_n(\xi\equiv 0\hbox{ or }\xi\equiv 1)=0$ and such that 
 \[\sup_x|\nu_n(\xi(x)=1)-u^*(-\theta^{\vep_n}_1)|> \kappa.\]
 Since $u^*(-\theta^{\vep_n}_1)\to u^*(-\theta_1)$, if we choose $\eta<\kappa$ this contradicts \eqref{exconc} for large $n$, and so proves (ii).
 The proof of (i) is similar using the first part of \eqref{exconc}. That is, if (i) fails, there is a sequence $\vep_n\downarrow 0$ so that coexistence fails for $\alpha_i^{\vep_n}$ as in \eqref{alphachoice}, contradicting the first part of \eqref{exconc}.  \qed

\medskip
\subsubsection{Evolution of cooperation}

Here is the result on coalescing probabilities which will
help us simplify the formula \eqref{cubicf} for the reaction
function $f$.  The notation is as in
Section~\ref{ssec:evolcoop}.

\begin{lem}\label{coalp} (a) $p(e_1|e_2)=p(0|e_1)$.

\noindent(b) $p(e_1|e_2+e_3)=\Bigl(1+\frac{1}{k}\Bigr)p(0|e_1)$.
\end{lem}
\begin{proof} Let $\tilde B^x_t$ denote a rate $2$ random
  walk with kernel $p$ starting at $x$, and note that for
  $x\ne 0$, $\hat B^{x}-\hat B^0$ has the same law as
  $\tilde B^x$ until it hits 0. Note also that for $x\ne 0$, $P(\tilde B^x_t\ne 0\
  \forall\ t\ge 0)=\sum_y p(y)P( \tilde B^{x+y}_t\ne 0\
  \forall\ t\ge 0)$.

For (a), 
\begin{align*}
p(0|e_1)&=\sum_{x_1}p(x_1)P(\tilde B^{x_1}_t\neq 0\hbox{ for all }t\ge 0)\\
&=\sum_{x_1}\sum_{x_2}p(x_1)p(x_2)P(\tilde B^{x_1+x_2}_t\neq 0\hbox{ for all }t\ge 0)\quad\hbox{(use $x_1\neq 0$)}\\
&=\sum_{x_1}\sum_{x_2}p(x_1)p(x_2)P(\tilde B^{x_1-x_2}_t\neq 0\hbox{ for all }t\ge 0)\quad\hbox{(by symmetry)}\\
&=p(e_1|e_2).
\end{align*}

For (b), let $T_j(x)$ be the time of the $j$th jump
of $\tilde B^{x}$. Then using symmetry,
\begin{align*}
p(e_1|e_2+e_3)&=\sum_{x_1,x_2,x_3}p(x_1)p(x_2)p(x_3)
P(\tilde B^{x_1+x_2+x_3}_t\neq 0\hbox{ for all }t\ge 0)\\
&=\sum_{x_1}p(x_1)
P(\tilde B^{x_1}_t\neq 0\hbox{ for all }t\ge T_2(x_1)) .
\end{align*}
Now using the above and first equality in the proof of (a), 
\begin{align*}
p(e_1|e_2+e_3)-p(0|e_1)
&=\sum_{x_1}p(x_1)P(\tilde B_{T_1}^{x_1}=0,\ \tilde B^{x_1}_t\neq 0\hbox{ for all }t\ge T_1(x_1))\\
&= k^{-1}\sum_{x_1}p(x_1) P(\tilde B^0_t\neq 0\hbox{ for all }t\ge T_1)\\
&=k^{-1}p(0|e_1).
\end{align*}
The result follows.
\end{proof}
\noindent{\it Proof of Lemma~\ref{evalf}.}
We first rewrite \eqref{cubicf} as 
\begin {align}
\nonumber \frac{f(u)}{k}=&(\beta-\gamma)\langle\hat\xi(e_1)\xi(e_2)\rangle_u+(\gamma-\delta)\langle\xi(e_1)\hat\xi(e_2)\hat\xi(e_2+e_3)+\hat\xi(e_1)\xi(e_2)\xi(e_2+e_3)\rangle_u\\
\nonumber&+((\alpha-\beta)-(\gamma-\delta))\langle\hat\xi(e_1)\xi(e_2)\xi(e_2+e_3)\rangle_u\\
\label{cubicf2}=&I+II+III.
\end{align}
Some elementary algebra shows that
\begin{equation}\label{rom2}
II=(\gamma-\delta)\langle\xi(e_1)-\xi(e_1)\xi(e_2)-\xi(e_1)\xi(e_2+e_3)+\xi(e_2)\xi(e_2+e_3)\rangle_u.
\end{equation}
Note that \eqref{prodform} and Lemma~\ref{coalp}(a) imply
\begin{align*}\langle\xi(e_1)\xi(e_2)\rangle_u=&u^2p(e_1|e_2)+u(1-p(e_1|e_2))\\
=&u^2p(0|e_1)+u(1-p(0|e_1))\\
=&\langle\xi(0)\xi(e_1)\rangle_u=\langle\xi(e_2)\xi(e_2+e_3)\rangle_u,
\end{align*}
the last by translation invariance.  Using this in
\eqref{rom2} and again applying \eqref{prodform}, we get

\begin{align}\label{rom12}
I+II&=(\beta-\gamma)u(1-u)p(e_1|e_2)+(\gamma-\delta)[u-\langle\xi(e_1)\xi(e_2+e_3)\rangle_u]\\
\nonumber&=(\beta-\gamma)u(1-u)p(e_1|e_2)\\
\nonumber&\phantom{=(}+(\gamma-\delta)[u-u^2p(e_1|e_2+e_3)-u(1-p(e_1|e_2+e_3))]\\
\nonumber&=u(1-u)[(\beta-\gamma)p(e_1|e_2)+(\gamma-\delta)p(e_1|e_2+e_3)]\\
\nonumber&=u(1-u)p(0|e_1)[(\beta-\gamma)+(1+k^{-1})(\gamma-\delta)],
\end{align}
where Lemma~\ref{coalp} is used in the last equality. A straightforward application of \eqref{prodform} allows us to find the coefficients of the cubic III in \eqref{cubicf2} and we obtain the required expression for $f(u)/k$.\qed

\medskip

\noindent{\it Proof of Theorem~\ref{thm:evol1}.} This is now an easy application of Theorem~\ref{thm:nonexist}.  Assume $\gamma-\delta<k(\delta-\beta)$.  Then Lemma~\ref{evalf} shows that 
$f(u)=c_1u(1-u)$ for $c_1<0$. Proposition~\ref{prop:pde4} shows that Assumption~\ref{a2} of Theorem~\ref{thm:nonexist} is valid for any $u_1\in(0,1)$. The condition \eqref{grate} holds with $r_0=2$ by \eqref{evolgrates} (recall $\Vert g_i^\vep-g_i\Vert_\infty=\Vert\hat g_i^\vep-\hat g_i\Vert_\infty$ by Proposition~\ref{finiterangeh}).  The condition \eqref{staydead} is clear from the expression for $h_1^\vep$ in \eqref{evolh}.  Since $f'(0)<0$ is clear from the above, and $w=\vep^2$, Theorem~\ref{thm:nonexist} completes the proof in this case.  The case where the inequality is reversed follows by a symmetrical argument, or, if you prefer, just reverse the roles of $0$ and $1$.
\qed
\medskip
\subsubsection{Nonlinear voter models}\label{ssec:nlv2}

\noindent{\it Proof of \ref{thm:nlv}.} Consider Case 4A first. As pointed out in this Case in Section~\ref{ssec:nlv}, for $L$ sufficiently large we may employ the mirror image of Proposition~\ref{prop:pde3} on $[0,1/2]$ with  $\rho=a$, the unique root of $f$ in $(0,1/2)$, and Proposition~\ref{prop:pde3} on $[1/2,1]$ with $\rho=1-a$, along with the comparison principle (Proposition~2.1 in \cite{AW78}), to see that Assumption~\ref{a1} holds for $\vep<\vep_0(\eta)$ with $u_*=\frac{1}{2}-\eta$, $u^*=\frac{1}{2}+\eta$, $v_0=\delta$, and $v_1=1-\delta$.  \eqref{grate} is trivial because $g_i^\vep=g_i$.  Theorem~\ref{thm:exist} now implies (a) and (b) in this case.  The proofs in Cases 1 and 2 are similar using Proposition~\ref{prop:pde1} (note all the zeros are simple in these cases) to verify Assumption~\ref{a1} (see the discussion in these cases in Section~\ref{ssec:nlv}).\qed

\medskip
\noindent{\it Proof of \ref{thm:nlv2}.} (a) The conditions on $b_1$ and $\lambda$ imply that $f_{1,\lambda}'(0)>0$ and $f_{1,\lambda}'(1)>0$.  Coexistence for large $L$ and small $\vep$ is now established as in Case 2 (or 1) of Theorem~\ref{thm:nlv}--see the discussion in Section~\ref{ssec:nlv}.

\noindent(b) By taking $\lambda$ and $L^{-1}$ small, depending on $(\eta,\bar a)$, we see from \eqref{flambda1}, \eqref{flambda2}, \eqref{flambdader}, and our conditions on the $b_i$ that 
$f_{(\lambda)}(u)=0$ will have $3$ simple roots in $(0,1)$, $p_1(\lambda)<p_2(\lambda)<p_3(\lambda)$, within $\eta/4$ of the respective roots
\[a'<1/2<1-a'\]
 of $f_1(u)=0$.  As \eqref{grate} is again obvious, we now verify Assumption~\ref{a1} of Theorem~\ref{thm:exist} with $u^*=1-a'+\frac{\eta}{2}$, $u_*=1-a'-\frac{\eta}{2}$, $v_0\in(p_2,u_*)$, and $v_1\in(u^*,1)$ ($\eta$ is small so these intervals are non-empty).  The result would then follow by applying Theorem~\ref{thm:exist}. The upper bound (ii) in Assumption~\ref{a1} is an easy application  of Proposition~\ref{prop:pde4}, with the interval $(p_3,1)$ in place of $(0,1)$, and the comparison principle. 

For the lower bound (i) in Assumption~\ref{a1}, we use a result of Weinberger \cite{W82}. To state the result we need some definitions.
His habitat ${\cal H}$ will be $\R^d$ in our setting and his space $B$ is the set of continuous functions from ${\cal H}$ to $[0,\pi_+]$. In our case $\pi_+=1$.
His result is for a discrete iteration $u_{n+1} = Q(u_n)$, where in our case $Q(u)$ is solution to the PDE at time 1 when the initial data is $u$.
His assumption (3.1) has five parts: 

\mn
(i) if $u \in B$ then $Q(u) \in B$.

\mn
(ii) If $T_y$ is translation by $y$ then $Q(T_yu) = T_yQ(u)$.

\mn
(iii) Given a number $\alpha$, let $Q(\alpha)$ be the constant value of $Q(u_\alpha)$ for $u_\alpha \equiv \alpha$. There are $0 \le \pi_0 < \pi_1 \le \pi_+$ so that if $\alpha \in (\pi_0,\pi_1)$ then $Q(\alpha) > \alpha$. $Q(\pi_0)=\pi_0$ and $Q(\pi_1)=\pi_1$. 

\mn
(iv) $u \le v$ implies $Q(u) \le Q(v)$. 

\mn
(v) If $u_n \in B$ and $u_n \to u$ uniformly on bounded sets then $Q(u_n)(x) \to Q(u)(x)$. 

\mn
Clearly (i) and (ii) hold in our application. For (iii) we let $\pi_0=p_2(\lambda)$ and $\pi_1=p_3(\lambda)$. (iv) a consequence of PDE comparison principles, see, e.g., Proposition 2.1 in Aronson and Weinberger (1978). (v) follows from the representation of solutions of the PDE in terms of the dual branching Brownian motion (see Lemma~\ref{lem:ident}).

The next ingredient for the result is  
$$
{\cal S} = \{ x \in \R^d : x\cdot \xi \le c^*(\xi) \hbox{ for all $\xi \in S^{d-1}$}\},
$$
where $S^{d-1}$ is the unit sphere in $\R^d$. $c^*(\xi)$ is the wave speed in direction $\xi$ defined in Section 5 of \cite{W82}. Due to the invariance of the PDE under rotation, all our speeds are the same, $c^*=\rho$, and ${\cal S}$ is a closed ball of radius $\rho$ or the empty set.  Here is Theorem~6.2 of \cite{W82}. 

\mn
\begin{thm}\label{weinb}Suppose (i)--(v) and that the interior of ${\cal S}$ is nonempty. Let ${\cal S}''$ be any closed and bounded subset of the interior of ${\cal S}$. For any $\gamma > \pi_0$, there is an $r_\gamma$ so that if $u_0(x) \ge \gamma$ on a ball of radius $r_\gamma$ and if $u_{n+1} = Q(u_n)$ then
\begin{equation}\label{propspeed}
\liminf_{n\to\infty} \min_{x \in {n\cal S}''} u_n(x) \ge \pi_1.
\end{equation}
\end{thm}
\medskip
To be able to use this result, we have to show that $\rho>0$. Note that here we require lower bounds on the wave speed of solutions to the reaction diffusion equation in one spatial dimension.   This is because traveling wave solutions in the direction $\xi$ of the form $w(x\cdot\xi-\rho t)$ correspond to traveling waves $w$ in one spatial dimension.  Recall that in the decomposition \eqref{nlvfdef2} $f(u)$ is odd about $u=1/2$, and for large $L$ has $f'(1/2)>0$ by $3b_1+b_2<0$.  The latter shows $f$ has $3$ simple zeros in $(0,1)$ at $a<1/2<1-a$.  The strict inequality in \eqref{nlvfdef2} on $(0,1)$ now easily implies (compare the negative and positive humps separately)
\begin{equation}\label{intcomp}
\int_{p_1}^{p_3}f_{(\lambda)} (u)du>\int_a^{1-a}f(u)du=0.
\end{equation}
So by the discussion in part II(ii) of Section~\ref{ss:PDE} there is a one-dimensional decreasing traveling wave solution to \eqref{rdpde} (with $f=f_{(\lambda)}$) over $(p_1,p_3)$ with positive wave speed $r_2(\lambda)$.

To consider traveling waves over $(0,p_1(\lambda))$, we note that Kolmogorov, Petrovsky, and Piscounov \cite{KPP37} have shown that if we consider
$$
\frac{\partial u}{\partial t} = \frac{\sigma^2}{2} \Delta u + f(u)
$$
in one dimension where $f$ satisfies 
\beq
f(0)=f(1)=0, \quad f(u)>0 \hbox{ for $0<u<1$}, \quad 
f'(u) \le f'(0) \hbox{ for $0 < u \le 1$}
\label{KPPcond}
\eeq
then there is a traveling wave solution with speed $\sqrt{2\sigma^2 f'(0)}$ and this is the minimal wave speed. For this fact one can consult Bramson \cite{Br83} or Aronson and Weinberger \cite{AW75}. However, the intuition behind the answer is simple: the answer is the same as for the linear equation
$$
\frac{\partial u}{\partial t} = \frac{\sigma^2}{2} \Delta u + f'(0)u
$$
which gives the mean of branching Brownian motion.
For more on this connection, see McKean \cite{McK75}. 
%In particular we see that  \[r_1(\lambda)=\sqrt{2f'_\lambda(0)}.\]

%Turning to the proof of \eqref{NLVbdsp}, the case $r_1(\lambda) < r_2(\lambda)$ is covered by Theorem~2.7 of  \cite{FM77} -- note that their traveling waves are increasing so in their notation, $c_{0,\alpha} = -r_1$ and $c_{\alpha,1} = -r_2$. To prove the result when $r_1(\lambda) \ge r_2(\lambda)$ we will reduce this case to the previous one by using  by using a comparison result for wave speeds from \cite{W82}.
Now let $g_1\le g_2$ be $C^1$ functions on $[0,1]$ such that
\[ 0<g_2\le f_{(\lambda)}\hbox{ on }(0,p_1),\ g_1=g_2=f_{(\lambda)}\hbox{ on }[p_1,1],\]
 \beq\label{g2props}g_2'(0)\in\left(0,\frac{r_2(\lambda)^2}{2\sigma^2}\right), \ g'_2(u)\le g_2'(0)\hbox{ on }[0,p_1],
 \eeq
 and for some $0<p_0<p_1$,
\begin{align}
g_1(0)=0,\ g_1<0 \hbox{ on }(0,p_0), g_1>0\hbox{ on }(p_0,p_1),
  \label{g1props} \int_0^{p_1} g_1(t)\,dt>0,\ g_1'(0)<0.
\end{align}
The existence of such functions is elementary. 
By the KPP result above, the minimal wave speed over $(0,p_1(\lambda))$ for the $g_2$ equation is \beq\label{speedcomp0}c_2=\sqrt{2\sigma^2g_2'(0)}<r_2(\lambda). 
\eeq
By Theorem~2.4 and Corollary~2.3 of \cite{FM77} (or the discussion in part II(ii) of Section~\ref{ss:PDE}) there is a unique traveling wave solution $u(t,x)=w(x-c_1t)$ ($w$ decreasing) to the $g_1$ equation with unique wave speed $c_1>0$ (since the integral in \eqref{g1props} is positive) and range $(0,p_1)$.  Note here and elsewhere that the traveling waves $w$ in \cite{FM77} are increasing and so our wave speeds have the opposite sign. By a comparison theorem for wave speeds (Proposition~5.5 of \cite{W82}) we may conclude that
\beq\label{speedcomp}c_2\ge c_1.\eeq
The hypothesis of the above comparison result is easily verified using $g_1\le g_2$ and the standard comparison principle (e.g. Proposition~2.1 of \cite{AW78}). It follows from \eqref{speedcomp0}
and \eqref{speedcomp} that $c_1<r_2(\lambda)$ which are the wave speeds of the $g_1$ equation
over $(0,p_1)$ and $(p_1,p_3)$, respectively.  We can therefore apply Theorem 2.7 of \cite{FM77} to conclude the existence of a traveling wave over $(0,p_3(\lambda))$ for the $g_1$ equation with speed
$r_1(\lambda)\in (c_1,r_2(\lambda))$.  The wave and its speed are both unique by Corollary~3.3 of \cite{FM77}.  Since $f_{(\lambda)}\ge g_1$ on $[0,p_3]$, another application of Proposition~5.5 of \cite{W82}
shows that $\rho\ge r_1(\lambda)$ and in particular $\rho>0$. 

 Using \eqref{propspeed}, we have proved that for $0<2w=\rho$, 
 \[\liminf_{n\to\infty}\inf_{|x|\le 2w n}u(n,x)\ge p_3(\lambda)\ge 1-a'-\frac{\eta}{4},
 \]
 providing that $u(0,x)\ge v_0$ for $|x|\le r_{v_0}$.  The same reasoning gives the same conclusion with $n\tau$ in place of $n$ for any $\tau>0$.  Taking $\tau$ small enough, a simple interpolation argument (use the weak form of the reaction diffusion equation and smoothing properties of the Brownian semigroup) now gives Assumption~\ref{a1}(i)  with $u_*=1-a'-\frac{\eta}{2}$ where the $2w$ in the above helps a bit in this last interpolation step. 
\qed

\subsection{Comparison with low density superprocess limit theorem}\label{ssec:CPcomp}
\noindent To make a comparison between our hydrodynamic limit theorem (Theorem~\ref{thm:strongconv}) and the superprocess limit theorem of Cox and Perkins~\cite{CP-1} we will write our perturbation terms in a different form, which will also be useful in Section~\ref{death}. Define
\[\Xi_S(\eta)=\prod_{i\in S}\eta_i\hbox{ for }\eta=(\eta_1,\dots,\eta_{N_0})\in\{0,1\}^{N_0}, S\in\hat\cP_{N_0}=\{\hbox{subsets of }\{1,\dots,N_0\}\},\]
and
\begin{align*}\chi(A,x,\xi)=\prod_{y\in A}\xi(x+ y)&,\ x\in \Z^d, \xi\in\{0,1\}^{\Z^d}, \\
&A\in\cP_{N_0}=\{\hbox{subsets of }\Z^d\hbox{ of cardinality at most }N_0\}.
\end{align*}
By adding an independent first coordinate to $Y$ we may assume $Y^1$ has law $p$.  If
\begin{equation}\label{gtildedef}\tilde g_i^\vep(\xi_1,\dots,\xi_{N_0})=-\vep_1^{-2}1(\xi_1=i)+g_i^\vep(\xi_1,\dots,\xi_{N_0}),
\end{equation}
and $\tilde g_i$ is as above without the superscript $\vep$, then 
\begin{equation}\label{tildegcvgce} \lim_{\vep\downarrow 0}\Vert \tilde g^\vep_i-\tilde g_i\Vert_\infty=0,
\end{equation} and
we may rewrite \eqref{hgrepn} as 
\begin{equation}\label{htildegrepn}h_i^\vep(x,\xi)=E_Y(\tilde g_i^\vep(x+Y^1,\dots,x+Y^{N_0})),\  i=0,1,\end{equation}
and similarly without the $\vep$'s.
It is easy to check that $\{\Xi_S(\cdot):S\in\hat\cP_{N_0}\}$ is a basis for the vector space of functions from $\{0,1\}^{N_0}$ to $\R$ and so there are reals $\hat\beta_\vep(S),\hat\delta_\vep(S)$, $S\in\hat\cP_{N_0}$, such that 
\beq\label{grepn}\tilde g_1^\vep(\eta)=\sum_{S\in\hat\cP_{N_0}}\hat\beta_\vep(S)\Xi_S(\eta),\quad \tilde g_0^\vep(\eta)=\sum_{S\in\hat\cP_{N_0}}\hat\delta_\vep(S)\Xi_S(\eta),
\eeq
and similarly without the $\vep$'s. If $S\in\hat\cP_{N_0}$, let $Y^S=\{Y^i:i\in S\}$, where $Y\in \Z^{dN_0}$ has law $q$ as usual. Let $E_Y$ denote expectation with respect to $Y$.  It is easy to use \eqref{htildegrepn} to check that
\begin{align}\label{hhatrepn}
 h_1^\vep(x,\xi)&=\sum_{S\in\hat\cP_{N_0}}\hat \beta_\vep(S)E_Y(\chi(Y^S,x,\xi))=\sum_{A\in\cP_{N_0}}\beta_\vep(A)\chi(A,x,\xi)\\
\label{hrepn}h_0^\vep(x,\xi)&=\sum_{S\in\hat\cP_{N_0}}\hat \delta_\vep(S)E_Y(\chi(Y^S,x,\xi))=\sum_{A\in\cP_{N_0}}\delta_\vep(A)\chi(A,x,\xi),
\end{align}
where for $A\in\cP_{N_0}$,
\beq\label{betadeltadef}\beta_\vep(A)=\sum_{S\in\hat \cP_{N_0}}\hat \beta_\vep(S)P(Y^S=A), \ \ \delta_\vep(A)=\sum_{S\in\hat \cP_{N_0}}\hat \delta_\vep(S)P(Y^S=A).
\eeq
Analogous equations to \eqref{hhatrepn}, \eqref{hrepn} and \eqref{betadeltadef} hold without the $\vep$'s.

 Now use \eqref{hhatrepn} and \eqref{hrepn} without the  $\vep$'s, and \eqref{prodform} to see that 
\begin{align}
\nonumber f(u)&\equiv\langle (1-\xi(0)h_1(0,\xi)-\xi(0)h_0(0,\xi)\rangle_u\\
\label{fformula}&=\sum_{A\in\cP_{N_0}}\Bigl[\beta(A)\Bigl[\sum_{j=1}^{|A|}u^j(1-u)P(|\hat\xi^A_\infty|=j,\tau(A,\{0\})=\infty)\Bigr]\\
&\phantom{=\sum_{A\in\cP_{N_0}}\Bigl[}+\beta(\emptyset)(1-u)-\delta(A)\Bigl[\sum_{j=1}^{|A\cup\{0\}|}u^jP(|\hat\xi_\infty^{A\cup\{0\}}|=j)\Bigr]\Bigr],
\end{align}
which is a polynomial of degree at most $N_0+1$ as claimed in Section~\ref{ssec:hydro}.  If $\beta(\emptyset)=0$, then $f(0)=0$ and
\beq\label{f'rep}
f'(0)
%=\lim_{u\to 0} f(u)/u
=\sum_{A\in\cP_{N_0}}\beta(A)P(\tau(A)<\infty,\tau(A,\{0\})=\infty)-\delta(A)P(\tau(A\cup\{0\})<\infty).
\eeq

From \eqref{grepn} one easily derives
\beq\label{grepinv}\hat\beta_\vep(S)=\sum_{V\subset S}(-1)^{|S|-|V|}\tilde g^\vep_1(1_V), \ \ \hat\delta_\vep(S)=\sum_{V\subset S}(-1)^{|S|-|V|}\tilde g^\vep_0(1_V),
\eeq
and similarly without the $\vep$'s.  Therefore
\beq\label{hatbetarate}|\hat\beta_\vep(S)-\hat\beta(S)|+|\hat\delta_\vep(S)-\hat\delta(S)|\le 2^{N_0}(\Vert \tilde g_1^\vep-\tilde g_1\Vert_\infty+\Vert \tilde g_0^\vep-\tilde g_0\Vert_\infty),\eeq
and so
\beq\label{betarate}
\sum_{A\in\cP_{N_0}}|\beta_\vep(A)-\beta(A)|+|\delta_\vep(A)-\delta(A)|\le 2^{2N_0}(\Vert \tilde g_1^\vep-\tilde g_1\Vert_\infty+\Vert \tilde g_0^\vep-\tilde g_0\Vert_\infty).
\eeq

Our spin-flips are now recast as 
\[c_\vep(\vep x,\xi_\vep)=\vep^{-2}c^v(x,\xi)+\sum_{A\in\cP_{N_0}}\chi(A,x,\xi)[\beta_\vep(A)(1-\xi(x))+\delta_\vep(A)\xi(x)],\]
which is precisely (1.17) of \cite{CP-1} with $\vep=N^{-1/2}$.  If we assume
\beq\label{betazero}g^\vep_1(0)=0 \hbox{ (and hence }\tilde g_1^\vep(0)=\hat\beta_\vep(\emptyset)=0)\hbox{ for small }\vep,
\eeq
and the voter kernel $p$ has finite support, then using the fact that the right-hand side of \eqref{betarate} approaches $0$ as $\vep\to 0$ (by \eqref{tildegcvgce}), it is easy to check that all the hypotheses of Corollary~1.8 of \cite{CP-1} hold. Alternatively, in place of the finite support assumption on $p$ one can assume the weaker hypothesis (P4) of Corollary~1.5 of \cite{CP-2}, and then apply that result. These results state that for $\vep$ as above if
$X_t^\vep=\vep^2\sum_{x\in\vep\Z^d}\xi^\vep_t(x)\delta_x$ and $X_0^\vep\to X_0$ weakly in the space $M_F(\R^d)$ of finite measures on $\R^d$,
then $X^\vep$ converges weakly in the Skorokhod space of $M_F(\R^d)$-valued paths to a super-Brownian motion with drift $\theta=f'(0)$ (as in \eqref{f'rep}).  
In this result we are starting $O(\vep^{-2})$ particles on a grid of $\vep^{-d}$ ($d\ge 3$) sites per unit volume, so it is a low density limit theorem producing a random limit, whereas Theorem~\ref{thm:strongconv} is a high density limit theorem producing a pde limit.  The latter result gives a natural explanation for the drift $\theta$ in the super-Brownian limit which was defined by the right-hand side of \eqref{f'rep} in \cite{CP-1}.  Namely, under \eqref{betazero}, in the low density limit we would expect a drift of $\lim_{u\to 0}f(u)/u=f'(0)$, which of course happens to equal the summation in \eqref{f'rep}.

\clearpage

\section{Construction, Duality and Coupling}

\label{sec:construction} 

In this section, we first introduce a family of Poisson processes which we use to
define $\xi_t$ on $\vep \Zd$, a ``dual process'' $X$ and a
``computation process'' $\zeta$. The duality equation
\eqref{dualityeq} below gives a representation of
$\xi_t(x)$ in terms of $(X,\zeta)$. Next we
show that for small $\vep$, $(X,\zeta)$ is
close to the simpler $(\hat X, \hat\zeta)$,
where $\hat X$ is a branching random walk system with
associated computation process $\hat\zeta$. Finally we show by a strong invariance principle
that for small $\vep$, $(\hat X,\hat\zeta)$ is close to a branching Brownian motion and its associated
computation process. 

However, our first task will be to prove Proposition~\ref{finiterangeh} and reduce to the case where $\vep_1=\infty$ in \eqref{hgrepn}. 

\subsection{Preliminaries} \label{ssec:prelimred}

{\it Proof of Proposition \ref{finiterangeh}.} Let $\underline p=\min\{p(y_i):p(y_i)>0\}$, choose $\vep_0>0$ so that $M=\sup_{0<\vep\le \vep_0}\Vert\hat g_0^\vep\Vert_\infty\vee\Vert\hat g_1^\vep\Vert_\infty<\infty$ and then choose $\vep_1>0$ so that 
\begin{equation}\label{vep1choice} \vep_1^{-2}\underline p>M.
\end{equation}
For $0<\vep<\vep_0$ define $g_i^\vep$ on $\{0,1\}^{N_0}$ by
\begin{equation}\label{gidefn}
g_i^\vep(\xi_1,\dots,\xi_{N_0})=\vep_1^{-2}\sum_1^{N_0}1(\xi_j=i)p(y_j)+\hat g^\vep_i(\xi_1,\dots,\xi_{N_0}),\ i=0,1,
\end{equation}
and define $g_i$ by the same equation without the $\vep$'s.  Clearly $\Vert g_i^\vep-g_i\Vert_\infty=\Vert \hat g_i^\vep -\hat g_i\Vert_\infty\to 0$ as $\vep\to 0$.  We may assume $y_1=0$.  By replacing $\hat g_i^\vep$ with $\hat g_i^\vep1(\xi_1=1-i)$ and redefining $h_i^\vep$ analogously (this will not affect \eqref{xirates}), we may assume
\begin{equation}\label{hatgzero}
\hat g_i^\vep(\xi_1,\dots,\xi_{N_0})=0\hbox{ if }\xi_1=i.
\end{equation}
We now show that $g_1^\vep\ge 0$.  Assume first 
\begin{equation}
\sum_1^{N_0}\xi_ip(y_i)=0.
\end{equation}
Choose $\xi\in\{0,1\}^{\Z^d}$ so that $\xi(y_i)=\xi_i$.  If $\xi(0)=0$, then by \eqref{xirates}, \eqref{finrangerep}, and \eqref{gidefn},
$$0\le c_\vep(0,\xi_\vep)=\hat g_1^\vep(\xi(y_1),\dots,\xi(y_{N_0}))=g_1^\vep(\xi_1,\dots,\xi_{N_0}).$$
If $\xi(0)=1$, then $\xi_1=\xi(0)=1$ and by \eqref{hatgzero}, $g_1^\vep(\xi_1,\dots,\xi_{N_0})=0$. Assume next that 
$$\sum_1^{N_0}\xi_ip(y_i)>0.$$
Then the above sum is at least $\underline p$ and so
$$g_1^\vep(\xi_1,\dots,\xi_{N_0})\ge \vep_1^{-2}\underline p-\Vert\hat g_1^\vep\Vert_\infty\ge \vep_1^{-2}\underline p -M>0,$$
the last by \eqref{vep1choice}.  This proves $g_1^\vep\ge 0$ and a similar argument shows $g_0^\vep\ge 0$. Finally \eqref{hgrepn} with $Y^i=y_i$ is immediate from \eqref{finrangerep} and the definition of $g_i^\vep$.
\qed

\medskip

We claim we may assume without loss of generality that $\vep_1=\infty$ in \eqref{hgrepn}, that is, the first term in the right-hand side of \eqref{hgrepn} is absent.  To see why, let $\tilde \vep^{-2}=\vep^{-2}-\vep_1^{-2}$ for $\vep<\vep_1$, and use \eqref{hgrepn} in \eqref{xirates} to rewrite the spin-flip rates of $\xi^\vep$ as 
$$c^\vep(\vep x,\xi_\vep)=\tilde \vep^{-2}c^v(x,\xi)+(1-\xi(x))\tilde h^\vep_i(x,\xi)+\xi(x)\tilde h_0^\vep(x,\xi),$$
where 
\begin{equation}\label{tildehrepn}
\tilde h_i^\vep(x,\xi)=E_Y(g_i^\vep(\xi(x+Y^1),\dots,\xi(x+Y^{N_0}))).
\end{equation}
So by working with $\tilde h_i^\vep$ in place of $h^\vep_i$ throughout, we may use \eqref{tildehrepn} in place of \eqref{hgrepn} and effectively set $\vep_1=\infty$.  Note first that this does not affect the definition of the reaction term $f(u)$ in the PDE \eqref{rdpde} since the terms involving $\vep^{-2}f_i(x,\xi)$ cancel in \eqref{fdef}.  The only cost is that $\vep^{-2}$ is replace with $\tilde \vep^{-2}$.  
The ratio of these terms approaches $1$ and so not surprisingly this only affects some of the proofs
in a trivial manner. Rather than carry this $\tilde \vep^{-2}$ with us throughout, we prefer to use $\vep$ and so
\begin{equation}
\hbox{\bf{henceforth set $\vep_1=\infty$ in \eqref{hgrepn}}.}
\label{vep1convention}\end{equation}

\subsection{Construction of $\xi_t$}\label{ssec:proc}

Define $c^*=c^*(g)$ by 
\begin{equation}\label{c*def}
c^*= \sup_{0<\vep\le \vep_0/2}\Vert g_1^\vep\Vert_\infty + \Vert g_0^\vep\Vert_\infty+1.
\end{equation}
To construct the process, we use a graphical representation.
For $x\in\vep\Zd$, introduce independent Poisson processes
$\{ T^{x}_n, n \ge 1 \}$ and $\{ T^{*,x}_n, n \ge 1\}$ with rates
$\vep^{-2}$ and $c^*$, respectively. Recall $p_\vep(y) = p(y/\vep)$ for $y \in \vep \Z^d$
and let $q_\ep(y) = q(y/\vep)$ for $y \in \vep \Z^{dN_0}$.
For $x\in\vep\Zd$ and $n\ge 1$, define independent random variables  $Z_{x,n}$ 
with distribution $p_\vep$, $Y_{x,n}=(Y^1_{x,n}, \ldots, Y^{N_0}_{x,n})$
with distribution $q_\vep$, and $U_{x,n}$ uniform on $(0,1)$. These random variables are 
independent of the Poisson processes and all are independent of an initial condition
$\xi_0\in \{0,1\}^{\vep\Z^d}$.
 
At times $t=T^x_n, n \ge 1$ (called voter times), we set $\xi_t(x)=\xi_{t-}(x+Z_{x,n})$.  To facilitate the definition of the dual, we draw an arrow from $(x,T^x_n) \to (x+Z_{x,n},T^x_n)$.
At times $t=T^{*,x}_n$, $n\ge 1$ (called reaction times), if $\xi_{t-}(x)=i$ we set $\xi_t(x)=1-i$ if
$$
U_{x,n}  < g^\vep_{1-i}(\xi_{t-}(x+Y^1_{x,n}), \ldots, \xi_{t-}(x+Y^{N_0}_{x,n}))/c^*,\hbox{ and otherwise }\xi_t(x)=\xi_{t-}(x).
$$
At these times, we draw arrows from $(x,T^{*,x}_n) \to (x+Y^{i}_{x,n}, T^{*,x}_n)$
for $1\le i \le N_0$. We write a * next to $(x,T^{*,x}_n)$ and call these *-arrows.
It is not hard to use ideas of Harris \cite{Har72} to show that under the exponential
tail conditions on $p$ and $q$, \eqref{expbd1} and \eqref{expbd2}, 
this recipe defines a pathwise unique process. This reference assumes finite range interactions but the proof applies in our infinite range setting as there are still finitely many sites that need to be checked at each reaction time. To verify this construction and to develop a useful dual process we now show how to compute the state of $x$ at
time $t$ by working backwards in time. It is easy to verify that $\xi$ is the unique in law $\{0,1\}^{\Z^d}$-valued Feller process with rates given by \eqref{voterrates} and \eqref{xirates}, or more precisely has generator as in \eqref{gendesc}.  For example one can recast the graphical representation in terms of SDE's driven by Poisson point processes and use stochastic calculus as in Proposition 2.1(c) of \cite{CP-2} (it is easy to verify condition (2.3) of that reference in our current setting).

We use $B^{\vep,x}$ to denote a continuous time random walk with jump rate $\vep^{-2}$ and jump distribution $p_\vep$ starting at $x\in\vep\Z^d$ and drop dependence on $x$ if $x=0$. We also assume
\beq\label{indrws} \{B^{\vep,x}:x\in\vep\Z^d\}\hbox{ are independent random walks distributed as above.}\eeq
It will be convenient to extend the Poisson times to the negative time line indexed by non-positive integers, and hence have
$\{T^x_n,n\in\Z\}$, $\{T^{*,x}_n,n\in\Z\}$ with the associated $\{Z_{x,n},n\in\Z\}$ and $\{(Y_{x,n},U_{x,n}),n\in\Z\}$, respectively.  At times it is useful to work with the associated independent Poisson point processes of reaction events $\Lambda^x_r(dt,dy,du)$ $(x\in\vep\Z^d)$ on $\R\times \vep\Z^{dN_0}\times [0,1])$ with points $\{(T_n^{*,x},Y_{x,n},U_{x,n})\}$ and intensity $c^*dt\times q_\vep\times du$, and also the independent Poisson point processes of walk steps $\Lambda^x_w(dt,dz)$ ($x\in\vep\Z^d$) on $\R\times\vep\Z^d$ with points $\{(T_n^x,Z_{x,n})\}$ and intensity $\vep^{-2}dt\times p_\vep$.

\subsection{The Dual $X$}
\label{ssec:Xdual}

Fix $T>0$ and a vector of $M+1$ distinct sites
$z=(z^0,\dots,z^M)$, each $z_i\in \vep\Zd$. Our
dual process $X=X^{z,T}$ starts from these sites at time $T$ and works backwards in time 
to determine the values $\xi_T(z_i)$. $X$ will be a {\it coalescing branching
random walk} with $X_0=(z_0,\dots,z_M,\infty,\dots)$ taking values in
\begin{align*}
 \cD=\{(X^0,X^1,\dots) & \in D([0,T],\R^d\cup\{\infty\})^{\Z_+}:\ \\
& \exists K_0\in\Z_+\hbox{ s.t. }X^k_t=\infty \ \forall t\ \in[0,T]\ \hbox{ and }k>K_0\}.
\end{align*}
Here $\infty$ is added to $\R^d$ as a discrete point,
$D([0,T],\R^d\cup\{\infty\})$ is given the Skorokhod $J_1$
topology, and $\cD$ is given the product topology.  

For $X=(X^0,X^1,\dots)\in\cD$, let $K(t)=\max\{i:X_t^i\neq\infty\}$, define 
$i\sim_t i'$ iff $X^i_t=X^{i'}_t\neq\infty$, and choose the minimal index
$j$ in each equivalence class in $\{ 0, \ldots K(t) \}$ to form the set $J(t)$.  We
also introduce  
$$
I(t)=\{X^i_t:i\in J(t)\}=\{X^i_t:X_t^i\neq\infty\}.
$$  
Durrett and Neuhauser \cite{DN94} call $I(t)$ the
influence set because it gives the locations 
of the sites we need to know at time $T-t$, to compute the values at
$z^0,\dots,z^M$ at time $T$.

\begin{figure}
\begin{center}
\begin{picture}(340,300)
\put(50,30){\line(0,1){240}}
\put(50,30){\line(1,0){230}}
\put(280,30){\line(0,1){240}}
\put(50,270){\line(1,0){230}}
\put(270,275){dual $X_t $}
\put(292,260){$\downarrow$}
\put(290,215){$R_1$}
\put(290,165){$R_2$}
\put(290,125){$R_3$}
\put(290,65){$R_4$}
\put(290,25){$T$}
\put(10,265){$T$}
\put(10,215){$T-R_1$}
\put(10,165){$T-R_2$}
\put(10,125){$T-R_3$}
\put(10,65){$T-R_4$}
\put(177,270){\line(-1,-3){17}}
\put(160,220){\line(0,-1){190}}
\put(175,275){0}
\put(149,218){$\bullet$}
\put(166,218){$\bullet$}
\put(173,218){$\bullet$}
\put(149,225){1}
\put(166,225){2}
\put(173,225){3}
\put(152,220){\line(-1,-1){50}}
\put(169,220){\line(1,-5){30}}
\put(176,220){\line(2,-3){60}}
\put(109,168){$\bullet$}
\put(92,168){$\bullet$}
\put(84,168){$\bullet$}
\put(84,175){4}
\put(92,175){5}
\put(117,168){6}
\put(111,170){\line(0,-1){28}}
\put(102,170){\line(1,-3){40}}
\put(142,50){\line(0,-1){20}}
\put(94,170){\line(0,-1){140}}
\put(86,170){\line(1,-5){8}}
\put(225,128){$\bullet$}
\put(242,128){$\bullet$}
\put(249,128){$\bullet$}
\put(217,130){7}
\put(242,135){8}
\put(249,135){9}
\put(227,130){\line(-1,-1){33}}
\put(236,130){\line(1,-3){33}}
\put(244,130){\line(0,-1){24}}
\put(252,130){\line(0,-1){48}}
\put(189,68){$\bullet$}
\put(205,68){$\bullet$}
\put(213,68){$\bullet$}
\put(180,75){10}
\put(200,75){11}
\put(220,68){12}
\put(191,70){\line(-1,-3){13}}
\put(199,70){\line(0,-1){40}}
\put(207,70){\line(-1,-2){8}}
\put(215,70){\line(1,-3){13}}
\put(23,30){$\uparrow$}
\put(20,15){$\zeta_t$}
\put(48,15){$J(T)=\{ 4,$}
\put(139,15){1,}
\put(157,15){0,}
\put(172,15){10,}
\put(196,15){2,}
\put(222,15){12,}
\put(266,15){$3\}$}
\end{picture}
\caption{An example of the dual with $N_0=3$.}
\label{fig:dual}
\end{center}
\end{figure}
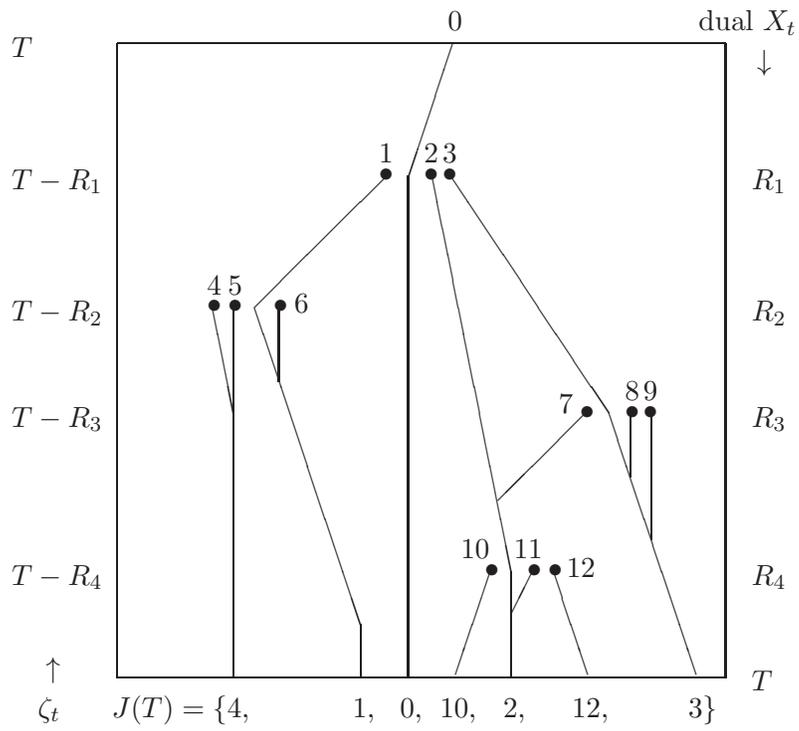

To help digest the definitions, the reader should consult Figure \ref{fig:dual}, which shows 
a realization of the dual starting from a single site when $N_0=3$.
If there were no reaction times $T^{*,x}_n$ then 
the coordinates $X^j_t,j\in J(t)$ follow the system of coalescing random walks dual to the voter part of
the dynamics. Coalescing refers to the fact that if $X^{j}_s=X^{j'}_s$ for some $s$ and $j,j'$, 
then $X^{j}_t=X^{j'}_t$ for all $t\in[s,T]$.
Jumps occur when a particle in the dual encounters the tail of an arrow in the graphical representation.
That is, if $j\in J(s-)$ and $x=X^j_{s-}$ has $T-s = T^x_n$ then $X^j_s=x+Z^x_n$. It coalesces with
$X^i_s=x+Z^x_n$ if such an $i$ exists, meaning that $i\vee j$ is removed from $J(s-)$ to form $J(s)$. 
If $B^\vep$ is a rate $\vep^{-2}$ random walk on $\vep\Z^d$ with step distribution $p_\vep$ then
the coalescing random walks  in the dual $X$ follow coalescing copies of $B^\vep$. 

To complete the definition, we have to explain what happens at the reaction times.
Put $R_0=0$, and for $m\ge 1$ let $R_{m}$ be the first time $t> R_{m-1}$ that a particle
in the dual encounters the tail of a *-arrow. If 
\begin{equation}\label{Rmdefn}
\hbox{$j\in J(R_m-)$ and $x=X^j_{R_m-}$ has $T-R_m = T^{x,*}_n$ for some n,}
\end{equation}
we let $\mu_{m}=j$ denote the parent site index.  In the example in Figure \ref{fig:dual} $\mu_1=0$, $\mu_2=1$, $\mu_3=3$, and $\mu_4=2$.

We create $N_0$ new walks by setting
$Y^i_m=Y^{i}_{x,n}$, $1\le i\le N_0$, 
\begin{equation}\label{dualdef}
\begin{aligned}
& K(R_{m})=K(R_{m-1})+N_0\,, \text{ and }\\
&X_{R_{m}}^{K(R_{m-1})+i}=x+Y^{i}_m,\, i=1,\dots,N_0\,.
\end{aligned}
\end{equation}
The values of the other coordinates $X^{j'}$, $j'\in
J(R_{m}-)$, $j'\ne \mu_{m}$ remain unchanged. 
Each ``new'' particle immediately coalesces with any particle
already at the site where it is born, and we make the resulting changes to
$J(R_{m}-)$ to construct $J(R_{m})\supset J(R_{m}-)$. 
To compute $\xi_T(z^i)$, we will also need the random variables 
\begin{equation}
\label{Umdef}
U_{m} = U_{x,n}\hbox{ where $x$, $m$, and $n$ are as in }\eqref{Rmdefn}.
\end{equation}
This computation is described in the next subsection.

$K(s)$ changes only at reaction times and always increases by exactly $N_0$, so 
\begin{equation}\label{Kform}K(s)=M+mN_0,\hbox{ for }s \in [R_m, R_{m+1}). 
\end{equation} 
Let $\cF_t$ be the right-continuous (time reversed) filtration generated by
the graphical representation restricted to $[T-t,T)$, but excluding the $\{U_{x,n}\}$.  More precisely
$\cF_t$ is the right-continuous filtration generated by
\begin{align}\label{defFt}&\{\Lambda_w^x([T-s,T)\times A):s\le t, x\in\vep\Z^d,A\subset \vep\Z^d\}, \\
\nonumber&\{\Lambda^x_r([T-s,T)\times B\times[0,1]):s\le t, B\subset\vep\Z^{dN_0},x\in\vep\Z^d\}.
\end{align}
The $\{R_m\}$ are then $(\cF_t)$-stopping
times and $X$ is $(\cF_t)$-adapted.  Since\break
 $P(R_{m+1}-R_m\in\cdot|\cF_{R_m})$ is
stochastically bounded below by an exponential random variable with mean
$(c^*(M+mN_0))^{-1}$, $R_m\uparrow\infty$ a.s. (recall our graphical variables were extended to negative values of time) and the
definition of $X$ is complete.  

Note that 
\begin{equation}\label{YmUm}
\mu_m \hbox{ is } \cF_{R_m}-\hbox{measurable and }\delta_{Y_m,U_m}=\Lambda_r^{X^{\mu_m}_{R_m-}}(\{T-R_m\}\times\cdot).
\end{equation}
As the above time reversed Poisson point processes are also Poisson point processes, one may easily see that 
\begin{equation}\label{Ymlaw}
\{Y_m\} \hbox{ are iid with law }q_\vep\hbox{ and }Y_m \hbox{ is }\cF_{R_m}-\hbox{measurable},
\end{equation}
and
\begin{equation}\label{Umlaw}
\{U_m\} \hbox{ are iid uniform on }[0,1]\hbox{ and are independent of }\cF_\infty.
\end{equation}

\subsection{The computation process $\zeta$}
\label{ssec:Xcomp}

Given an initial time $t_0\in[0,T)$, the coalescing
branching random walk $\{X_s,s\in[0,T-t_0]\}$, the sequence
of parent indices $\{\mu_m\}$, the sequence of uniforms
$\{U_m\}$, and a set of initial values in $\{0,1\}$,
$\zeta_{t_0}(j)=\xi_{t_0}(X^j_{T-t_0})$, $j\in J(T-t_0)$, we
will define $\{\zeta_{r}(k),r\in[0,T], 0\le k\le
K((T-r)-)\}$ so that
\begin{equation}\label{dualityeq}
\zeta_{r}(k)=\xi_{r}(X^k_{T-r})
\text{ for all }r\in[t_0,T]\text{ and }k\le K((T-r)-)\,. 
\end{equation}
The left hand limits here reflect the fact that we have reversed the direction of time from that of $X$.

In general we consider a general initial state $\zeta_{t_0}(j)\in\{0,1\}$, $j\in J(T-t_0)$. First we complete this initial state  by
setting $\zeta_{t_0}(k)=\zeta_{t_0}(j)$ if $k\sim_{T-t_0} j\in J(T-t_0)$. 
Suppose that for some $m\ge 1$, $R_m$ is the largest
reaction time smaller than $T-t_0$. The values $\zeta_{r}(k)$ do not change 
except at times $T-R_n$, so $\zeta_{r}=\zeta_{t_0}$ for $r<T-R_m$.
We decide whether or not to flip the value of $\zeta$ at $\mu_m$ 
at time $t-R_m$ as follows.  Define $V_m\in\{0,1\}^{N_0}$ by
\begin{equation}\label{Vdefn}
V_m^j=\zeta_{(T-R_m)-}(M+(m-1)N_0+j)\,, \quad j=1,\dots,N_0.
\end{equation}
Letting $i=\zeta_{(T-R_m)-}(\mu_m)$ we set
\begin{equation}\label{mu-flip}
\zeta_{(T-R_m)}(\mu_m)=\begin{cases} 1-i &\text{ if } U_m \le g_{1-i}(V_m)/c^*  \\
                                      i &\text{ otherwise.}
\end{cases}
\end{equation}
To update the dual now, for $k\le M+(m-1)N_0=K((T-(T-R_m))-)$ and $k\ne\mu_m$,  
\begin{equation}\label{zetadef-2}
\text{if } k \sim_{R_m} \mu_m 
\text{ set }\zeta_{T-R_m}(k)= 
\zeta_{T-R_m}(\mu_m) \,.
\end{equation}
Otherwise we keep $\zeta_{(T-R_m)}(k)=\zeta_{(T-R_m)-}(k)$.

The values $\zeta_{r}(k)$ remain constant for
$r\in[T-R_m,T-R_{m-1})$. Coming to $r=T-R_{m-1}$, if $m-1\ge 1$
we proceed as above. When we reach $r=T-R_0=T$ we end by setting
$\zeta_{T}=\zeta_{T-}$. If $\xi_{t_0}(j)=\xi(X^j_{T-t_0})$ for $j\in J(T-t_0)$, the verification of \eqref{dualityeq} is an easy exercise 
from the definitions of $X$ and $\zeta$.  

\subsection{Branching random walk approximation $\hat X$}
\label{ssec:brw}

Due to the transience of random walk in dimensions
$d\ge 3$, and the fact that the random walk steps are occurring at a
fast rate in $X$ when $\vep$ is small, any coalescing in $X$
will occur soon after a branching event and close to the
branching site.  As in Durrett and Z\"ahle \cite{DZ}, such births followed
quickly by coalescing are not compatible with weak
convergence of the dual. Thus we need a way to excise these
events from $X$. As in \cite{DZ} we define a
{\it (non-coalescing) branching random walk} 
$\hat X$ and associated computation process $\hat\zeta$. Later we will
couple $(X,\zeta)$ and $(\hat X,\hat\zeta)$ so that
they are close when $\vep$ is small.

For $m\in\NN$,  $\Pi_m$ denotes the set of
partitions of $\{0,\dots, m\}$ and for each $\pi\in\Pi_m$,
$J_0(\pi)$ is the subset of $\{0,\dots,m\}$ obtained by
selecting the minimal element of each cell of $\pi$. 
Write $i\sim_\pi j$ if $i$ and $j$ are in the same cell of $\pi$. 
Let $\{\hat B^{Y^i},i=0,\dots N_0\}$ be the rate one coalescing random
walk system on $\Zd$ with step distribution $p$ and initial points at
$Y^0=0, Y^1, \ldots,Y^{N_0}$ where $(Y^1,\dots,Y^{N_0})$ has law
$q$. Let $\nu_0$
denote the law on $\Pi_{N_0}$ of the random partition
associated with the equivalence relation $i\sim j$ iff $\hat
B^{Y^i}(t)=\hat B^{Y^j}(t)$ for some $t\ge 0$.  For $\vep>0$ let $\nu_\vep$ denote the
law on $\Pi_{N_0}$ of the random partition associated with
the equivalence relation $i\sim^\vep j$ iff $\hat
B^{Y^i}(\vep^{-3/2})=\hat
B^{Y^j}(\vep^{-3/2})$.  Note that $\vep^{-3/2}=\vep^{1/2}\vep^{-2}$ so this is a short 
amount of time for the sped up process.  For later use when we
define the branching Brownian motion $Z$ we note that since $\vep^{-3/2}\to\infty$, 
\begin{equation}\label{nuconv}
\nu_\vep \text{ converges weakly to } \nu_0
\text{ as }\vep\downarrow 0 \,.
\end{equation}

As before we will have a fixed $T>0$ and distinct sites
$z_0,\dots,z_M$ in $\vep\Z^d$. Our branching random walk
$\hat X$ will have paths in $\cD$ and an associated set of
indices $\hat J(t)=\{j: \hat X^j_t\ne\infty\}$.  Let
$\pi_0\in\Pi_M$ be defined by the equivalence relation
$i\sim j$ iff $\hat
B^{\vep^{-1}z_i}(\vep^{-3/2})=\hat
B^{\vep^{-1}z_j}(\vep^{-3/2})$. In words, $\pi_0$
will be used to ``mimic'' the initial coalescence in $X$ of
the particles starting at $z_i$ before any reaction events
occur. 

For $n\ge 1$ let $\pi_n\in \Pi_{N_0}$ be iid with law
$\nu_\vep$ and independent of $\pi_0$.
From $\{\pi_n\}$ we inductively define a
sequence of nonempty subsets $\{\hat J_n\}$ of $\Z_+$ by $\hat J_0 = J_0(\pi_0)$
and for $n \ge 0$
\begin{align}
  \label{hatJn}\hat J_{n+1}=\hat J_n\cup\{M+nN_0+j:j\in
  J_0(\pi_{n+1})\setminus\{0\}\}.
\end{align} Set $\hat R_0=0$ and
conditional on $\{\pi_n\}$ let $\{\hat R_{n+1}-\hat R_n:n\ge
0\}$ be independent exponential random variables with means
$(c^*|\hat J_n|)^{-1}$, and let $\{\hat \mu_n\}$ be an
independent sequence of independent random variables where
$\hat \mu_n$, $n\ge 1$, is uniformly distributed over $\hat
J_{n-1}$. $\hat\mu_n$ is the index of the particle that gives birth at time $\hat R_n$.

To define $\hat X$ inductively we start with 
\begin{equation}\label{hatX0def}\hat X^j_0=z_j\hbox{ if }j\in\hat J_0=\hat J(0)
\hbox{ and }\infty\hbox{ otherwise}. 
\end{equation}
On $[\hat R_n,\hat R_{n+1})$, the $\hat X^j:j\in\hat J_n$
follow independent copies of $B^\vep$ starting at $\hat X^j_{\hat R_n}$. At $\hat R_{n+1}$ we define
\begin{equation*}
  \hat X^j_{\hat R_{n+1}}=
  \begin{cases}\hat X^j_{\hat R_{n+1}-}&\text{ if }j\in \hat J_n=\hat J(\hat R_n) \,,\\
    \hat X^{\hat\mu_{n+1}}_{\hat R_{n+1}-}&\text{ if }j\in \hat J_{n+1}-\hat J_n\, ,\\
    \infty&\text{ otherwise.}
  \end{cases}
\end{equation*}
Note that offspring are no longer displaced from their parents and that coalescence reduces the number of particles born at time $\hat R_{n+1}$,
but otherwise no coalescence occurs as $\hat J(t) = \hat J_n$ on $[\hat R_n, \hat R_{n+1})$.
Thus, conditional on the sequence $\{\pi_n\}$, $\hat X$ is
a branching random walk starting with particles at $z_j$,
$j\in J_0(\pi_0)$, with particle branching rate $c^*$ and giving birth
to $|\pi_n|-1$ particles on top of the parent $\hat X^{\hat
  \mu_n}_{\hat R_n}$ (who also survives) at the $n$th branch time $\hat
R_n$.

\subsection{Computation process $\hat\zeta$}
\label{ssec:hatXcomp}

As we did for $X$, for $t_0\in[0,T)$ we now define an computation process
$\{\hat\zeta_r(k):0\le k\le \hat K((T-r)-), r\in[t_0,T]\}$ for $\hat X$. Here $\hat K(s)=M+mN_0$ if $s\in[\hat R_m,\hat R_{m+1})$.  Given are the branching random walks $\{ \hat X_s, s \in [0,T-t_0]\}$, the
associated sequence $\{\pi_n,\hat R_n,\hat \mu_n\}$, 
a sequence of iid random variables $\{\hat U_n\}$, uniformly distributed on $[0,1]$ and independent of $(\hat X,\{\pi_n,\hat R_n,\hat \mu_n:n\in\NN\})$, and a set of initial values $\hat\zeta_{t_0}(j), j \in \hat J(T-t_0)$.
In the next section when we
couple $(X,\zeta)$ and $(\hat X,\hat \zeta)$ we will set
$\hat U_n$ equal to $U_n$ defined in \eqref{Umdef}. Define an equivalence relation $\approx_{\hat R_n}$ on $\{0,\dots, M+nN_0\}$ by
\begin{align}\nonumber M+(m-1)N_0+j&\approx_{\hat R_n}M+(m-1)N_0+i\ (1\le i,j\le N_0,\,1\le m\le n)  \hbox{ iff }j\sim_{\pi_{m}} i,\\
\label{approxdef}M+(m-1)N_0+j&\approx_{\hat R_n}\hat \mu_m \phantom{+mN\ }\ (1\le j\le N_0,\,1\le m\le n)\ \ \hbox{ iff }j\sim_{\pi_{m}} 0,\\
\nonumber j&\approx_{\hat R_n}i,\phantom{+mN_0+} (0\le i,j\le M)\ \ \ \hbox{ iff } j\sim_{\pi_0} i.
\end{align}
Finally if $\hat R_n\le t< \hat R_{n+1}$ define $i\approx_t j$ iff $i\approx_{\hat R_n}j$ for $0\le i,j\le M+nN_0$.
To prepare for the proof of
Lemma \ref{lem:compeq}, note that the definition of $\hat\zeta$ that follows is just
the definition of $\zeta$ with hats added and $\approx_t$ used in place of $\sim_t$.

First we complete the initial state $\hat\zeta_{t_0}$ by setting $\hat\zeta_{t_0}(k)=\hat\zeta_{t_0}(j)$
if $k \approx_{T-t_0} j\in \hat J(T-t_0)$, $k\le K(T-t_0)=K((T-t_0)-)$ a.s. Suppose that for some $m\ge 1$, $\hat R_m$ is the largest
branch time smaller than $T-t_0$. The values $\hat\zeta_r(k)$ do not change except at times
$T-\hat R_n$, so $\hat\zeta_r = \hat\zeta_{t_0}$ for $t < T - \hat R_m$.
We decide whether or not to flip the value of $\hat\zeta$ at $\hat\mu_m$ 
at time $t-\hat R_m$ as follows.  Define $\hat V_m\in\{0,1\}^{N_0}$ by
\begin{equation}\label{hatVdefn}
\hat V_m^j=\hat\zeta_{(T-\hat R_m)-}(M+(m-1)N_0+j)\,, \quad j=1,\dots,N_0.
\end{equation}
Letting $i=\hat\zeta_{(T-\hat R_m)-}(\hat\mu_m)$ we set
\begin{equation}\label{mu-flip-hat}
\hat\zeta_{(T-\hat R_m)}(\hat\mu_m)=\begin{cases} 1-i &\text{ if } \hat U_m \le g^\vep_{1-i}(\hat V_m)/c^*  \\
                                      i &\text{ otherwise.}
\end{cases}
\end{equation}
To update $\hat\zeta$ now, for $k\le M+(m-1)N_0$ and $k\ne\hat \mu_m$, 
\begin{equation}\label{hatzetadef}
\text{if } k \approx_{\hat R_m} \hat \mu_m 
\text{ set }\hat\zeta_{T-\hat R_m}(k)= \hat\zeta_{(T-\hat R_m)}(\hat\mu_m) \,,
\end{equation}
and for the remaining values of $k\le M+(m-1)N_0$ keep
$$\hat\zeta_{T-\hat R_m}(k)= \hat\zeta_{(T-\hat R_m)-}(k).$$

The values $\hat\zeta_{r}(k)$ remain constant for
$r\in[T-\hat R_m,T-\hat R_{m-1})$. Coming to $r=T-\hat R_{m-1}$, if $m-1\ge 1$
we proceed as above. When we reach $r=T-\hat R_0=T$ we end by setting
$\hat \zeta_{T}=\hat \zeta_{T-}$.

\subsection{Coupling of $(X,\zeta)$ and $(\hat X,\hat\zeta)$}
\label{ssec:XhatXcoupling}
We now give a construction of $\hat X, \hat \zeta$ which
will have the property that with high probability for small $\vep$,
(i) $X$ and $\hat X$  are close and (ii) given identical inputs,
$\zeta$ and $\hat\zeta$ will compute the same result. As before,
$T>0$ and $z=(z_0,\dots,z_M)$, $z_i\in\vep\Zd$ are
fixed. Recall the reaction times $R_m$, the uniform
random variables $U_m$ from \eqref{Umdef}, and the natural
time-reversed filtration $\cF_t$ used in the construction of the dual $X$
given in \eqref{dualdef}.

%Let $N_T=\min\{m:R_m>T\}$. 
The following general definition will be used to construct the partitions 
$\{\pi_n:n\in\Z_+\}$ needed to define $\hat X$,(distributed as in Section~\ref{ssec:brw}) in terms of the graphical representation.
Let $V$  be an $\cF_t$-stopping time (think of $V=R_m$),
and let $\gamma_0 \dots,\gamma_{M'}\in\vep\Z^d\in\vep\Z^d$ be $\cF_V$-measurable.
Let $\{\hat B^{\vep,\gamma_i}:i=0,\dots,M'\}\subset\vep\Z^d$ be the rescaled coalescing
random walk system, starting at time $V$ at locations $\gamma_0,\dots,\gamma_{M'}$, determined by the $\{T_n^x\}$ in the graphical representation. That is, $\{\hat B^{\vep,\gamma_i}:i=0,\dots,M'\}$ are as described in Figure~\ref{fig:dual} but now starting at time $T-V$ at sites $\gamma_0,\dots,\gamma_M'$. 
For each $t>0$ let $\pi_{V,\gamma}(t)\in
\Pi_{M'}$ be the random partition of $\{0,\dots,M'\}$
associated with the equivalence relation $i\sim i'$ iff $\hat
B^{\vep,\gamma_i}(t)=\hat B^{\vep,\gamma_{i'}}(t)$.  We call
$\pi_{V,\gamma}(t)$ the random partition at time $V+t$ with initial condition
$\gamma=(\gamma_0,\dots, \gamma_{M'})$ at time $V$.

Let  $\pi_0=\pi_{0,z}(\sqrt\vep)\in \Pi_M$ be the
random partition of $\{0,\dots,M\}$ at time $\sqrt{\vep}$ with initial
condition $z=(z_0,\dots,z_M)$ at time 0, and note that its law is the
same as the law of the $\pi_0$ described just before
\eqref{nuconv}.  For $m\ge 1$ let
$$
\gamma_m=(X^{\mu_m}_{R_m},X^{\mu_m}_{R_m}+Y^1_m,\dots, X^{\mu_m}_{R_m}+Y^{N_0}_m)
$$
and $\{\pi'_m, m\in\NN\}$ be an iid sequence with law $\nu_\vep$ and chosen independent of $\cF_\infty$.  
For $m\in\NN$, define
\begin{equation}\label{pimdef}
\pi_m=\begin{cases} \pi_{R_m,\gamma_m}(\sqrt\vep) &\text{ if }R_n>R_{n-1}+\sqrt\vep\hbox{ for all }1\le n\le m\\
\pi'_m &\text{ otherwise.}
\end{cases}
\end{equation}
By the translation invariance and independent increments
properties of the Poisson point processes used in the graphical representation and also \eqref{Ymlaw}, 
$\pi_m$ is independent of ${\cal F}_{R_{m-1}+\sqrt\vep}\vee\sigma(\pi'_n,n<m)\equiv\bar\cF_{m-1}$, and has law
$\nu_\vep$ defined just before \eqref{nuconv}.  It is also easy to check that $\pi_m$ is
$\bar\cF_m$-measurable ($m\ge 0$) and so $\{\pi_m,m\ge 0\}$ are independent and distributed as in
Section~\ref{ssec:brw}.  

For $m\in\NN$ let
\begin{align*}
 \tau'_{m}  & = \inf\{s\ge R_{m-1}:  \exists i \neq j \hbox{ both in } J(R_{m-1}-), \hbox{ or} \\
& i\in J(R_{m-1}-)\setminus\{\mu_{m-1}\}, j\in J(R_{m-1})\setminus J(R_{m-1}-),
\hbox{ so that $X_s^i=X_s^j$} \}, \\
 \tau_{m} & = \inf\{s\ge R_{m-1}+\sqrt\vep:
\inf_{i\neq j\in J(s)} |X_s^i-X_s^j|\le \vep^{7/8}\},\\
Y^*_m&=\max\{|Y^i_m|:i=1,\dots,N_0\}.
\end{align*}
We introduce the time, $T_b$, that one of four possible
``bad events'' occurs:
\begin{align*} T_b=&\min\{R_m:m\ge 1, R_m\le R_{m-1}+\sqrt\vep\hbox{ or }Y^*_m\ge\frac{\vep}{\kappa} \log(1/\vep)\}\\
\nonumber&\wedge\min\{\tau_m:m\ge 1, \tau_m<R_m\}\wedge\min\{\tau'_m:m\ge 2, \tau'_m\le R_{m-1}+\sqrt\vep\}.
\end{align*}
Here $\min\emptyset=\infty$
To see why the last two minima should be large, note that after a birth of $N_0$ particles from particle $\mu_m$ at time $R_m$, we expect
some coalescence to occur between the parent and its children.  After time $\sqrt\vep$, particles should all be separated by at least $\vep^{7/8}$ and remain that way until the next reaction time when again there may be coalescing within the family producing offspring but no other coalescing events.  The qualifier $m\ge 2$ is needed in the last minimum because we have no control
over the spacings between particles at time 0.
The collision of particles 2 and 7 in Figure \ref{fig:dual} is an example of
a bad event that enters into the definition of $\tau'_4$.
We assume throughout that
\begin{equation}\label{epskappa}
0<\vep<\vep_{1}(\kappa)\hbox{ so that }\frac{\vep}{\kappa}\log(1/\vep)<\vep^{7/8}/2.
\end{equation}

Given $\{\pi_m\}$ we now construct $\hat X$ and $\hat A(s)=((\hat \mu_n,\hat R_n)1(\hat R_n\le s))_{n\in\NN}$
(with the law described in Section~\ref{ssec:brw}) initially up to time $\hat T=T_b\wedge\hat T_b$, where
$$\hat T_b=\min\{\hat R_m:m\ge 1, \hat R_m-\hat R_{m-1}\le \sqrt\vep\}.$$
Once one of the five bad events (implicit in the definition of $\hat T$) occurs, we will give up and continue the definition of the
branching random walk using independent information.
The coupling of $X$ and $\hat X$ will be through our definition of $\{\pi_n\}$ and also through the use of the random walks steps of $X^j$ to define corresponding random walk steps in $\hat X^j$ whenever possible, as will be described below.

We begin our inductive construction by setting $\hat R_0=0$, $\hat J(0)=J_0(\pi_0)$, and define $\hat X_0$ as in \eqref{hatX0def}. Note that
\begin{equation}\label{hatJ0}
\hat J(0)=J(\sqrt\vep)=J_0(\pi_0)\hbox{ if }R_1>\sqrt\vep.
\end{equation}
Assume now that $(\hat X,\hat A)$ has been defined on $[0,R_m\wedge\hat T]$.  Assume also that $R_m<\hat T$ implies the following for all $1\le i\le m$:
\begin{equation}\label{Rmu} \hat R_i=R_i,\ \ \hat\mu_i=\mu_i,
\end{equation}
\begin{equation}\label{hatJdef}
\hat J(R_i) = \hat J(R_{i-1}) \cup\{M+(i-1)N_0+j:j\in J_0(\pi_i)\setminus\{0\}\}.
\end{equation}
\begin{equation}\label{hatJinc}
\hat J(R_{i-1})=\hat J(s)\subset J(s)\hbox{ for all }s\in[R_{i-1},R_i).
\end{equation}
\begin{equation}\label{Jeq}
\hat J(s)=J(s)=J(R_{i-1}+\sqrt\vep)\hbox{ for all }s\in[R_{i-1}+\sqrt\vep,R_i),
\end{equation}

The $m=0$ case of the induction is slightly different, due for example to the special nature of $\pi_0$, so let us assume $m\ge 1$ first.  To define $(\hat X,\hat A)$ on\break 
$(R_m\wedge\hat T,R_{m+1}\wedge\hat T]$ we may assume $R_m(\omega)<\hat T(\omega)$ and so \eqref{Rmu}-\eqref{Jeq} hold by induction.
On $(R_m,(R_m+\sqrt\vep)\wedge R_{m+1}\wedge\hat T]$ let $(\hat X,\hat A)$ evolve as in Section~\ref{ssec:brw} conditionally independent of $\cF_\infty$ given $\{\pi_n\}$.  Here it is understood that the unused partitions $\{\pi_i:i>m\}$ are used to define the successive branching events as in \eqref{hatJn}.

Next, to define $(\hat X,\hat A)$ on $((R_m+\sqrt\vep)\wedge R_{m+1}\wedge\hat T,R_{m+1}\wedge\hat T]$ we may assume $R_m(\omega)+\sqrt\vep<R_{m+1}\wedge\hat T(\omega)$.
By the definition of $\hat T_b$ this implies $\hat R_{m+1}>R_m+\sqrt\vep$ and so for all $s\in[R_m,R_m+\sqrt\vep]$,
\begin{align}
\nonumber \hat J(s)=\hat J(R_m)&=\hat J(R_{m-1})\cup\{M+(m-1)N_0+j:j\in J_0(\pi_m)\setminus\{0\}\}\\
\label{hatJJI}&=J(R_{m-1}+\sqrt\vep)\cup\{M+(m-1)N_0+j:j\in J_0(\pi_m)\setminus\{0\}\}.
\end{align}
In the first equality we used \eqref{hatJdef} and in the second we used \eqref{hatJinc} and \eqref{Jeq} with $s=R_{m-1}+\sqrt\vep$.  The fact that $\tau_m\le R_m$ (since $T_b>R_m+\sqrt\vep$) shows there are no coalescings of $X$ on $[R_{m-1}+\sqrt\vep,R_m)$ and so 
\begin{equation}\label{JeqII}
J(R_{m-1}+\sqrt\vep)=J(R_m-). 
\end{equation}
Again use $T_b>R_m+\sqrt\vep$ together with \eqref{epskappa} to see that $Y^*_m\le \frac{\vep}{\kappa}\log(1/\vep)\le \frac{\vep^{7/8}}{2}$, and so the spacings of the previously existing particles at time $R_m\le \tau_m$ ensures that none of the $N_0$ new particles at time $R_m$ will land on a previously occupied site. Therefore if
$$J_1(Y_m)=\{1\le j\le N_0:Y_m^j\notin \{Y_m^i:0\le i<j\}\},$$
then
$$J(R_m)=J(R_m-)\cup\{M+(m-1)N_0+j:j\in J_1(Y_m)\}.$$
The fact that $R_{m+1}\wedge\tau'_{m+1}>R_m+\sqrt\vep$ means that $X$ has no branching events in $(R_m,R_m+\sqrt\vep]$ and $X$ has no particles coalescing on $[R_m,R_m+\sqrt\vep]$ except those involving $X^{\mu_m}_{R_m}+Y^i_m, i=0,\dots,N_0$. 
Therefore, the definition of $\pi_m$ ensures that
\begin{align}
\nonumber J(R_m+\sqrt\vep)&=J(R_m-)\cup\{M+(m-1)N_0+j:j\in J_0(\pi_m)\setminus\{0\}\}.\\
\label{hatJJIII}&=\hat J(s)\hbox{ for all }s\in[R_m,R_m+\sqrt\vep],
\end{align}
where in the last line we have used \eqref{hatJJI} and \eqref{JeqII}.
For $s\in[R_m+\sqrt\vep,R_{m+1}\wedge\hat T)$ we have $s<\tau_{m+1}$ and so
\begin{equation}\label{sepn}
|X^j_s-X^k_s|>\vep^{7/8}\hbox{ for all }j\neq k \hbox{ both in }J(s), \hbox{ for all }s\in[R_m+\sqrt\vep,R_{m+1}\wedge\hat T).
\end{equation}
In particular $X$ can have no coalescings on the above interval and so $J(s)=J(R_m+\sqrt\vep)$ for $s\in[R_m+\sqrt\vep,R_{m+1}\wedge\hat T)$. On $(R_m+\sqrt\vep,R_{m+1}\wedge \hat T]$ let $(\hat X^j_s,j\in\hat J(s))$ follow the random walk steps and branching events of $\{X^j:j\in J(s)\}$ (of course there is at most one of the latter at time $R_{m+1}$ providing $R_{m+1}\le \hat T$). In particular we are setting
\begin{equation}\label{hatJdefII}
\hat J(s)=J(s)=J(R_m+\sqrt\vep)\hbox{ for }s\in[R_m+\sqrt\vep,R_{m+1}\wedge\hat T)\hbox{ or }s=\hat T<R_{m+1}.
\end{equation}
\eqref{sepn} shows that the random walk steps and branching events for distinct particles  of $X$ on $(R_m+\sqrt\vep, R_{m+1}\wedge\hat T]$
are independent.  In addition, these steps and branching events are independent of the random walk increments used to define $\{\pi_n\}$.  This shows that $\hat X$ evolves like the branching random walk  described in Section~\ref{ssec:brw} on $(R_m,R_{m+1}\wedge \hat T)$, and on $(R_m,R_{m+1}\wedge \hat T]$ if either $\hat T<R_{m+1}$, or $R_m+\sqrt\vep\ge R_{m+1}\wedge\hat T$.
(In the latter case the first part of the above construction did the job and in the former case there is no reaction event to define at $\hat T\wedge R_{m+1}=\hat T$.) So to complete the construction at $t=R_{m+1}\wedge \hat T$ we may assume
 \begin{equation}\label{weakhyp}R_m+\sqrt\vep<R_{m+1}\le \hat T.
 \end{equation}
The above definition shows that $\hat R_{m+1}=R_{m+1}$, we use \eqref{hatJdef} with $i=m+1$ to define $\hat J(R_{m+1})$ and we set $\hat \mu_{m+1}=\mu_{m+1}$.  Clearly $\hat \mu_{m+1}$ is uniform on $\hat J(R_m)=J(R_m+\sqrt\vep)$ (given $\{\pi_n\}$) and is independent of $\{\hat\mu_n:n<m\}$.  In addition the branching events used to define $\{\hat \mu_n\}$ are independent of the random walk steps used to define $\{\pi_n\}$.  This completes our inductive definition of $(\hat X,\hat A)$ on $[0,R_{m+1}\wedge\hat T]$.  
 
 Next we complete the inductive step of the derivation of \eqref{Rmu}-\eqref{Jeq} for $m+1$ under \eqref{weakhyp} which is in fact weaker than the $R_{m+1}<\hat T$ condition.  \eqref{hatJdefII} implies 
\eqref{Jeq} for $i=m+1$, and \eqref{Rmu} and \eqref{hatJdef} hold by definition.  On $\{R_m+\sqrt\vep<R_{m+1}\le \hat T\}$ $J$ can only decrease on $[R_m,R_m+\sqrt\vep]$ due to coalescings of the random walks, while $\hat J$ is constant on this interval by \eqref{hatJJIII}.  The inclusion \eqref{hatJinc} therefore follows from the equality in \eqref{Jeq}.  

To complete the inductive construction of $(\hat X,\hat A)$ on each $[0,R_m\wedge\hat T]$ and proof of \eqref{Rmu}=\eqref{Jeq} it remains to give the $m=0$ step of the construction and verify the $m=1$ case of the induction.  Both follow by making only minor changes in the above induction step. For example,  \eqref{hatJ0} is used in place of the (now non-existent) induction hypothesis \eqref{hatJdef} both in defining $\hat X$ on the initial interval and in obtaining \eqref{hatJJIII} for $m=0$.

Since $R_m\uparrow\infty$ a.s. we have defined $(\hat X,\hat A)(s)$ on $[0,\hat T]$ and to complete the definition we let it evolve conditionally independently (given $\{\pi_n\}$) for $s\ge \hat T$.  

The above construction and \eqref{Umlaw} show that 
\begin{equation}\label{Umind} (X,\{\pi_n\},\{\mu_n\},\hat X,\{\hat \mu_n\},\{\hat R_n\}) \hbox{ is independent of }\{U_n\},
\end{equation}
where $\{U_n\}$ are the uniforms from \eqref{Umdef}.  Therefore the computation process $\hat \zeta$ for the above $\hat X$ may be defined as in Section~\ref{ssec:hatXcomp} but with $\hat U_n=U_n$.
%We may enlarge $(\F_t)$ (e.g. by using a sequence of iid copies of our Poisson point processes $(\Lambda_w,\Lambda_r)$) so that $X$ and $\hat X$ are both $\F_t$-Markov processes.

\begin{lem}\label{GmdefOK}
(a) For all $m\in\Z_+$, $R_m<T_b$ and $\hat R_m<\hat T_b$ imply $R_m=\hat R_m<\hat T$.

\noindent (b) For all $m\in\NN$, if 
\begin{align*} G_m=\{\omega:&\Bigl(\wedge_{i=1}^m R_i-R_{i-1}\Bigr)\wedge\Bigl(\wedge_{i=2}^{m+1}\tau'_i-R_{i-1}\Bigr)>\sqrt\vep,\\
&R_i\le \tau_i\ \forall i\le m, \max_{i\le m}Y^*_i<\frac{\vep}{\kappa}\log(1/\vep),
\wedge_{i=1}^m \hat R_i-\hat R_{i-1}>\sqrt\vep\},
\end{align*}
then $G_m\subset\{\hat R_m=R_m<\hat T\}$.
\end{lem}
\begin{proof} (a) The implication is trivial for $m=0$ so
  assume it for $m$ and assume also $R_{m+1}<T_b$, $\hat
  R_{m+1}<\hat T_b$.  By induction we have $R_m=\hat
  R_m<\hat T$.  Since $R_{m+1}\wedge \hat
  R_{m+1}>R_m+\sqrt\vep$, we also know $\hat
  T>R_m+\sqrt\vep$.  The construction of $\hat X$ on
  $(R_{m}+\sqrt\vep,R_{m+1}\wedge\hat T]$ shows that the
  next reaction time of $\hat X$ on this interval must be
  $R_{m+1}$ (if it exists) and so $\hat T_b\ge
  R_{m+1}$. Since $T_b> R_{m+1}$ by hypothesis we get
  $(R_{m}+\sqrt\vep,R_{m+1}\wedge\hat
  T]=(R_m+\sqrt\vep,R_{m+1}]$.  Hence our construction of
  $\hat X$ on this interval shows $\hat R_{m+1}=R_{m+1}$ and
  so the result follows for $m+1$.

\noindent (b) The first four conditions in the definition of $G_m$ imply
$$T_b\ge R_{m+1}\wedge\tau'_{m+2}\wedge \tau_{m+1}>R_m.$$
The last condition implies $\hat T_b>\hat R_m$. Now apply (a).

\end{proof}
As an immediate consequence of the above and our inductive proof of \eqref{Rmu}-\eqref{Jeq} we get the following:

\begin{lem}\label{lem:hatagree}
$$G_m\Rightarrow R_m<\hat T\Rightarrow \hbox{ for all }1\le i\le m\ \eqref{Rmu}-\eqref{Jeq} \hbox{ hold.}$$
\end{lem}

On $G_m$ and on the intervals $[R_{m-1}+\sqrt{\ep},R_{m})$  our definition of $\hat X$ and Lemma~\ref{GmdefOK}(b) shows that the movement of particles
in $X$ and $\hat X$ are coupled (they take identical steps) but on $[R_{m-1},R_{m-1}+\sqrt\ep)$ they move
independently. To bound the discrepancies that accumulate during these
intervals we use:

\begin{lem} \label{lem:Xclose}
If $\omega\in G_m$, then 
\begin{align} \label{Xclose}
\sup&\{|\hat X^j_s-X^j_s|: j\in \hat J(s),\ s\in[0,R_m)\}\\
    \nonumber &\le (m-1)\frac{\vep}{\kappa} \log(1/ \ep) +\sum_{l=0}^{m-1} \sup_{j\in
      \hat J(R_{l}), s\in[R_l,R_l+\sqrt\vep]}|\hat
    X^j_{s}-\hat X^j_{\hat R_l}|+|X^j_{s}-X^j_{R_l}|.
  \end{align}
\end{lem}

\begin{proof}
Suppose first that $m> 1$ and we are on $G_m$. By the coupling of the spatial motions noted above, for $j\in \hat J(R_{m-1})$
\begin{align*}
    &\sup_{s\in[R_{m-1},R_{m})}|\hat X^j_s-X^j_s|
= \sup_{s\in[R_{m-1},R_{m-1}+\sqrt\vep]}|\hat X^j_{s}-X^j_{s}|\\ 
&\quad \le |\hat X^j_{R_{m-1}}-X^j_{R_{m-1}}| + \sup_{s\in[R_{m-1},R_{m-1}+\sqrt\vep]}|\hat X^j_{s}-\hat X^j_{R_{m-1}}|\\
&\phantom{\quad \le |\hat X^j_{R_{m-1}}-X^j_{R_{m-1}}|}  +\sup_{s\in[R_{m-1},R_{m-1}+\sqrt\vep]}|X^j_s- X^j_{R_{m-1}}|.
\end{align*}
On $G_m$, a newly born particle to $X^j_{R_{m-1}-}$ may jump a distance at most $\frac{\vep}{\kappa} \log(1/\ep)$ from its
parent, while for $\hat X^j_{R_{m-1}-}$ it will be born on its
parent site, so the above is at most
\begin{align*}
& \sup_{k\in J(R_{m-1}-)}|\hat X^k_{R_{m-1}-}-X^k_{R_{m-1}-}|+ \frac{\vep}{\kappa} \log(1/\vep) \\ 
&\qquad+\sup_{s\in[R_{m-1},R_{m-1}+\sqrt\vep]}|\hat X^j_{s}-\hat X^j_{R_{m-1}}|
+\sup_{s\in[R_{m-1},R_{m-1}+\sqrt\vep]}|X^j_s-X^j_{R_{m-1}}|.
\end{align*}
Things are simpler when $m=1$ because there are no initial jumps to worry about and so the second term in the above is absent. The required bound now follows by induction in $m$ and the fact that $G_m$ is decreasing in $m$.
\end{proof}
 \subsection{Bounding the probability of bad events}\label{ssec:pbad}

Here and in what follows it is useful to dominate $X$ with a branching random walk 
$\bar X$, also with paths in $\cD$ and with the same
initial state. Particles in $\bar X$ follow
independent copies of $ B^\vep$ and with rate $c^*$
give birth to $N_0$ particles  located at $
B^\vep_t+Y_m^i,\, i=1,\dots,N_0$, where $
B^\vep_t$ is the
location of the parent particle.  At the $m$th birth time
$\bar R_m$ we use $X^{M+(m-1)N_0+i},\, i=1,\dots,N_0$ to
label the new particles, so that if $\bar
J(t)=\{j:X^j_t\neq\infty\}$, then $\bar J(\bar
R_m)=\{0,\dots,M+mN_0\}$. Coalescence is avoided in $\bar
X$ by having the coalescing particle with the larger index
have its future steps and branching events dictated by an
independent copy of the graphical representation. 
This will ensure that $J(t)\subset \bar J(t)$ and
$\{X^j(t):j\in J(t)\}\subset\{\bar X^j(t):j\in \bar J(t)\}$
for all $t\ge 0$.

Let $N_T=\min\{m:R_m>T\}$ and define $\bar N_T$ in the same way, using the branching times
$\{\bar R_m\}$. Let
\begin{equation}\label{cbdef} 
c_b= c^* N_0\ge 1 \,.
\end{equation}
 We will also need to separate the particles in $\hat X$ and so define
 \begin{equation}
 \hat\tau_m=\inf\{s\ge \hat R_{m-1}+\sqrt\vep:\inf_{i\neq j\in \hat J(s)} |\hat X^i_s-\hat X^j_s|\le \vep^{7/8}\}, m\in\NN.
 \end{equation}

\begin{lem}\label{lem:good1} There is a constant
  $c_{\tref{lem:good1}}$ so that for all $T>0$ and $n\in\NN$ 

  \noindent(a) $P(N_T>n)\le P(\bar N_T>n)\le e^{c_bT}(M+1)(nN_0)^{-1}$.

  \noindent(b) $P(\min_{1\le m\le N_T} R_m-R_{m-1}\le \sqrt\vep \hbox{ or }\min_{1\le m\le N_T} \hat R_m-\hat R_{m-1}\le \sqrt\vep)$\hfil\break
 $\phantom{P(\min_{1\le m\le N_T}}   \le c_{\tref{lem:good1}}e^{c_bT}(M+1) \vep^{1/6}$.
\end{lem}

\begin{proof} (a) The first inequality follows from the
domination of $X$ by $\bar X$.  For the second one note that
$E(\bar J(T))=(M+1)e^{c_bT}$ and conclude
\begin{align*} 
P(\bar N_T>n)&\le P(|\bar J(T)|\ge M+1+nN_0)\\
&\le (M+1+nN_0)^{-1}(M+1)e^{c_bT}.
\end{align*}

  (b) Let $Z$ be a mean one exponential random variable. The
  domination of $X$ by $\bar X$ shows that for any $n\ge 1$,
  \begin{align*}
    P&\left(\min_{1\le m\le N_T}R_m-R_{m-1}\le \sqrt\vep\right)
    \le P\left(\min_{1\le m\le \bar N_T}\bar R_m-\bar R_{m-1}\le \sqrt\vep\right)\\
    &\le P(\bar N_T>n)+\sum_{m=1}^nP\Bigl(\frac{Z}{M+1+(m-1)N_0}\le \sqrt\vep\Bigr)\\ 
    &\le
    e^{c_bT}(M+1)(nN_0)^{-1}+\sum_{m=1}^n(M+1+(m-1)N_0)\sqrt\vep,
  \end{align*}
  by (a).  Now set $n=\lceil\vep^{-1/6}\rceil$ and note that the sum is at most 
$ (M+1)n \sqrt{\vep} + n^2 N_0 \sqrt{\vep}$.
A similar calculation gives the same upper bound for the $\hat R_m$'s.
\end{proof}

\begin{lem} \label{lem:good4}
There is a constant $c_{\tref{lem:good4}}$ so that for all $T>0$ 
$$
P( Y^*_m > \frac{\vep}{\kappa}\log(1/\vep) \hbox{ for some $m\le N_T$}) 
\le c_{\tref{lem:good4}}e^{c_bT}(M+1) \vep^{1/2}
$$
\end{lem}

\begin{proof}
By (a) of Lemma \ref{lem:good1}, $P(N_T>n)\le e^{c_bT}(M+1)(nN_0)^{-1}$.
Using \eqref{expbd2} gives
$$
P( Y^*_m >\frac{\vep}{\kappa} \log(1/ \vep) \hbox{ for some $m\le n$}) \le n C \vep.
$$
Taking $n=\lceil\vep^{-1/2}\rceil $ now gives the desired result. 
\end{proof}

The following facts about random walks will be used frequently.

\begin{lem} \label{lem:notcrowd}
Let $Z_s$ denote a continuous time rate $2$ random walk on $\Z^d$ jumping
with kernel $p$, and starting at $x\in\Z^d$ under $P^{x}$, and $B^\vep$ be our continuous time rescaled copy of $Z$, starting at $z\in\vep\Z^d$ under $P_z$.

\noindent(a) For any $t_0\ge 0$, $r_0\ge 1$, $x\in\Z^d$ and $p\ge 2$,
$$P^x(|Z_s|\le r_0\hbox{ for some }s\ge t_0)\le c_{\ref{lem:notcrowd}}\int_{t_0}^\infty\Bigl[[(|x|-r_0)^+]^{-p}(s^{p/2}\vee s)\Bigr]\wedge \Bigl[(s\vee 1)^{-d/2}r_0^d \Bigr]ds.$$

\noindent (b)
$
\sup_{x}P^{x}(|Z_s|\le \vep^{-1/8}\hbox{ for some $s\ge \vep^{-3/2}$}) \le c_{\ref{lem:notcrowd}} \ep^{3/8}.
$

\noindent (c) For any $z\in\vep\Z^d$, $r_0\ge 1$
$$P_z(|B^\vep_s|\le r_0\vep\hbox{ for some }s\ge 0)\le c_{\ref{lem:notcrowd}} (|z|\vep^{-1})^{-(2/3)(d-2)}r_0^{2(d+1)/3}.$$
\end{lem}

\begin{proof}  (a)
Use $T(t_0,y)\le \infty$ to denote the time of the first
visit of $Z$ to $y$ after time $t_0$, and let
$$
G=\int_0^\infty P^0(Z_s=0)\,ds
$$ 
be the expected time at 0 (which is finite since $d\ge 3$). Then
  \begin{align*}
    \int_{t_0}^\infty P^{x}(Z_s=y) \, ds
    &=E^{x}\Bigl(1\{T(t_0,y)(\omega)<\infty\}
\int_{T(t_0,y)(\omega)}^\infty P^y(Z_{s-T(t_0,y)(\omega)}=y) \, ds\Bigr)\\
    &= GP^{x}(T(t_0,y)<\infty).
  \end{align*}
Summing over $|y|\le r_0$ for $r_0\ge 1$ and rearranging, we get
  \begin{align}
    \nonumber P^{x}(|Z_s|\le r_0 &\hbox{ for some }s\ge t_0) 
\le G^{-1}\sum_{|y|\le r_0}\int_{t_0}^\infty P^{x}(Z_s=y)\, ds\\ 
&=G^{-1}\int_{t_0}^\infty P^x(|Z_s|\le r_0)ds.
 \label {retbnd}
\end{align}
%&\le c\sum_{|y|\le r_0}\int_{t_0}^\infty s^{-d/2}\,ds \le c r_0^d t_0^{1-\frac{d}{ 2}},\quad r_0\ge 1
 A martingale square function inequality shows that for $p\ge 2$,  
\beq\label{martineq}
  P^x(|Z_s|\le r_0)\le P^0(|Z_s|\ge (|x|-r_0)^+)\le c((|x|-r_0)^+)^{-p})(s^{p/2}\vee s).
\eeq
A local central limit theorem (see, e.g. (A.7) in \cite{CDP}) 
shows that 
\beq\label{uniflclt}
P^x(|Z_s|\le r_0)\le c (s\vee1)^{-d/2}r_0^d.
\eeq
Use the above two inequalities to bound the integrand in \eqref{retbnd} and derive (a).
\medskip

\noindent(b)
Set $r_0=\vep^{-1/8}$ and $t_0=\vep^{-3/2}$ in (a) and use only the second term in the infimum inside the integral. The right-hand side is $c \vep^{-(d/8)-(3/2)+ (3d/4)}$.
To complete the proof we note that exponent is smallest when $d=3$.
\medskip

\noindent (c) We may assume without loss of generality that $r_0\le |z|\vep^{-1}/2=M/2$ (or the bound is trivial) and so $t_1=M^{4/3}r_0^{2/3}\ge 1$.  Apply (a) with $p=2d$ and break the integral at $t_1$ to see that the probability in (c) is 
\begin{align*}
P^{z\vep^{-1}}(|Z_s|\le r_0\hbox{ for some }s\ge 0)
&\le c\Bigl[\int_0^{t_1}M^{-2d}(s^d\vee s)ds+\int_{t_1}^\infty s^{-d/2}r_0^d ds\Bigr]\\
&\le c(M^{-2d}t_1^{d+1}+t_1^{1-(d/2)}r_0^d)\\
&\le cM^{-(2/3)(d-2)} r_0^{2(d+1)/3}.
\end{align*}
\end{proof}

\begin{lem} \label{lem:good2}
$P(\tau_m<R_m\hbox{ or }\hat \tau_m<\hat R_m\hbox{ for some }1\le m\le N_T)$\hfil\break
$\phantom{P(\tau_m<R_m\hbox{ or }\hat \tau_m<\hat R_m}\le c_{\tref{lem:good2}}e^{c_bT}(M+1)^2\vep^{3/32}$.
\end{lem}

\begin{proof}
To bound $P(\tau_m<R_m\hbox{ for some }1\le m\le N_T)$, we start with 
  \begin{align*}
    P(\tau_m<R_m|\cF_{R_{m-1}})
    &\le P(R_m>R_{m-1}+\sqrt\vep,\ \exists i\neq j\hbox{ both in } J(R_{m-1}+\sqrt\vep),\hbox{ s.t. }\\
    &\phantom{\le P(R_m}\inf_{\sqrt\vep+R_{m-1}\le s\le
      R_m}|X_s^i-X_s^j|\le\vep^{7/8}|\cF_{R_{m-1}}).
  \end{align*}
Now $i\neq j\hbox{ both in } J(R_{m-1}+\sqrt\vep)$ and $R_m>R_{m-1}+\sqrt\vep$ imply $i,j\in J(R_{m-1})$ and
$X^i_s\neq X^j_s$ for all $s\in[R_{m-1},R_{m-1}+\sqrt\vep]$. Therefore, the above is at most
  \begin{align}
    \label {ctbnd}& \sum_{i\neq j\in J(R_{m-1})} P(X^i_s-X^j_s\neq 0,\ \forall s\in[R_{m-1},R_{m-1}+\sqrt\vep],\\
    \nonumber&\phantom{\le \sum_{i\neq j\in J(R_{m-1})} P(}
    |X^i_s-X^j_s|\le \vep^{7/8}\ \exists s\ge R_{m-1}+\sqrt\vep|\cF_{R_{m-1}}).
  \end{align}
If $Z$ as in Lemma \ref{lem:notcrowd}, we may use (b) of that result to bound the above by
  \begin{align*}
    & |J(R_{m-1})|^2 \, \sup_{z_0\neq 0}P^{z_0}(|Z_s|\le \vep^{-1/8}\ \exists s\ge \vep^{-3/2})\\
    &\le (M+1+(m-1)N_0)^2 \cdot c \vep^{3/8}
\end{align*}
Using Lemma \ref{lem:good1}(a), we conclude
  \begin{align*}
    &P(\tau_m<R_m\hbox{ for some }1\le m\le N_T)\\
    &\le e^{c_bT}(M+1)(nN_0)^{-1}+\sum_{m=1}^n(M+1+(m-1)N_0)^2 c\vep^{3/8}.
\end{align*}
To bound the sum we note that for $a,b\ge 1$,
$$
\sum_{m=1}^n(a+(m-1)b)^2 \le \int_0^n (a+xb)^2 \, dx = \frac{1}{3b} [(a+nb)^3 - a^3] \le ca^2(nb)^3.
$$
Taking $n=\lceil \vep^{-3/32}/N_0 \rceil$ gives the desired bound. A similar calculation (in fact there is some simplification) gives the same upper bound for $$P(\hat \tau_m<\hat R_m\hbox{ for some }1\le m\le N_T.)$$
\end{proof}

\begin{lem}\label{lem:good3} 
$$
P(\min_{1\le m\le N_T}\tau'_{m+1}-R_{m}\le \sqrt\vep) \le c_{\tref{lem:good3}}e^{c_bT}(M+1)^2\vep^{1/40}.
$$
\end{lem}

\begin{proof} Define $S_m\supset G_m$ as $G_m$ (in Lemma~\ref{GmdefOK}) but without the lower bounds on $\wedge_{i=2}^{m+1}\tau_i'-R_{i-1}$ or $\wedge_{i=1}^m\hat R_i-\hat R_{i-1}$. Note that $S_m\in\cF_{R_m}$ and
if $\omega\in S_m$, then
  \begin{equation}\label{sep}
    |X^i_{R_m}-X^j_{R_m}|\ge \vep^{7/8}\hbox{ for all distinct }i,j\in J(R_m-).
  \end{equation}
In addition, since $Y^*_m \le\frac{\vep}{\kappa} \log(1/\vep)$ we have that for all $i\in J(R_m-)-\{\mu_m\}$, $j\in J(R_m)-J(R_m-)$
\beq\label{sep2}
|X^i_{R_m}-X^j_{R_m}|\ge \vep^{7/8}- \frac{\vep}{\kappa} \log(1/\vep) \ge \vep^{7/8}/2
\eeq
since $\vep<\vep_1(\kappa)$ (recall \eqref{epskappa}).

If $T_0$ is the return time to zero of the random walk $Z$
in Lemma \ref{lem:notcrowd}, we have (see P 26.2 in \cite{Spi}
for $d=3$ and project down for $d>3$)
\begin{equation}\label{pot}
P^{z_0}(T_0<\infty)\le c|z_0|^{-1}.
\end{equation}
Use \eqref{sep}, \eqref{sep2}, and \eqref{pot} with scaling, and the bound
$$
|J(R_m-)|\le M+1+(m-1)N_0
$$ 
to see that on $S_m\in\cF_{R_m}$,
\begin{align}
    \nonumber P&\left(\tau'_{m+1}-R_m\le \sqrt\vep|\cF_{R_m}\right)\\
    \nonumber &\le c[(M+1+(m-1)N_0)^2\vep^{1/8}+(M+1+(m-1)N_0)N_0\vep^{1/8}]\\
    \label{condbnd}&\le c(M+1)^2m^2N_0^2\vep^{1/8}.
\end{align}
Using the bound in Lemma \ref{lem:good1}(a), we conclude
  \begin{align*}
    P&\left(\min_{1\le m\le N_T} \tau'_{m+1}-R_m\le \sqrt\vep, S_{N_T} \right)\\
    &\le P(N_T>n)+\sum_{k=1}^nP(N_T=k,S_k,\min_{1\le m\le k}\tau'_{m+1}-R_m\le \sqrt\vep)\\
    &\le c_{\tref{lem:good1}} e^{c_bT}(M+1)(nN_0)^{-1}
       +\sum_{k=1}^n\sum_{m=1}^{k}P(S_m,\, \tau'_{m+1}-R_m\le \sqrt \vep).
\end{align*}
Using \eqref{condbnd} now, the above sum is at most
$$
  cN_0^2(M+1)^2n\vep^{1/8}\sum_{m=1}^{n}m^2  \le c (M+1)^2 n^4 \ep^{1/8}.
$$
Take $n=\lceil\vep^{-1/40}\rceil$ and use Lemmas \ref{lem:good1}, \ref{lem:good4}, and \ref{lem:good2} to bound
$P(S^c_{N_T})$ to get the desired result.
\end{proof}

\subsection{When nothing bad happens, $(X,\zeta)$ and $(\hat X,\hat\zeta)$ are close} 

The next result gives a concrete bound on the difference between $X$ and $\hat X$ and
deals with the final interval $[R_m,R_m \wedge T]$. Let
$$\bar G_m=G_m\cap \{\hat T\ge \hat R_m\},\ m\in\NN,$$
and for $0<\beta\le 1/2$, define 
\begin{align}
\tilde G^\beta_{T} = \bar G_{N_{T}} & \cap\{\sup_{s\le T}\sup_{j\in\hat J(s)} |X^j_s-\hat X^j_s|\le \vep^{1/6}\}
\nonumber \\
& \cap \{ T \notin \cup_{m=0}^{N_{T}-1}[R_m,R_m+2\vep^\beta]\}
\label{def:tildeG}.
\end{align}
Allowing smaller $\beta$ values will be useful in Sections \ref{sec:hydroproofs} and \ref{lockill}, but for now the reader may take $\beta=1/2$.

\begin{lem}\label{lem:tildeG} There is a $c_{\tref{lem:tildeG}}$ and $\vep_{\tref{lem:tildeG}}(\kappa)>0$ so that for any $T \ge 2\vep^\beta$, $0<\vep<\vep_{\tref{lem:tildeG}}(\kappa)$, 
$$
P(({\tilde G}^\beta_T)^c) \le c_{\tref{lem:tildeG}}e^{c_bT}(M+1)^2\vep^{\frac{1}{40}\wedge \frac{\beta}{3}}.
$$
On $\tilde G^\beta_{T}$ we have $\hat J(s)=J(s)$ for all $s\in[T-\vep^\beta,T]$, and $|\hat X^i_{T}-\hat X^j_{T}|\ge \vep^{7/8}$ for all $i\neq
j$ in $\hat J(T)$.
\end{lem}

\begin{proof}  Dependence on $\beta$ will be suppressed.  For $s$ as above, Lemma \ref{lem:hatagree} implies $\hat J(s)=J(s)$ on $\tilde G_T$ since $s\in[R_{N(T)-1}+\sqrt\vep,R_{N(T)})$ on $\tilde G_T$. The last assertion of the Lemma holds on $\tilde G_T$ because on $\tilde G_T$, $\hat\tau_{N(T)}\ge \hat R_{N(T)}$ and $$T\in[R_{N(T)-1}+\sqrt\vep,R_{N(T)})=[\hat R_{N(T)-1}+\sqrt\vep,\hat R_{N(T)}).$$

Lemmas~\ref{lem:good1}, \ref{lem:good4}, \ref{lem:good2}, and \ref{lem:good3} imply
\begin{equation}\label{GTbnd} 
P({\bar G}^c_{N_T})\le c e^{c_b T}(M+1)^2\vep^{1/40}.
\end{equation}
To deal with the first additional good event in $\tilde G_T$, we note that by Lemma \ref{lem:Xclose}
\begin{align*}
P &(G_{N_T},\sup_{s\le T}\sup_{j\in\hat J(s)}|X^j_s-\hat X^j_s|>\vep^{1/6}) \le P(N_T>n)\\
&+P\Bigl((n-1)\frac{\vep}{\kappa}\log(1/\vep) +\sum_{i=0}^{n-1} \sup_{j\in\hat J(R_i)}\sup_{s\in[R_i,R_i+\sqrt\vep]} 
|\hat X^j_s-\hat X^j_{R_i}|+|X^j_s-X^j_{R_i}|>\vep^{1/6}\Bigr)
\end{align*}
By (a) in Lemma~\ref{lem:good1} the first term is at most $ e^{c_bT}(M+1)(nN_0)^{-1}$. If
\beq
(n-1)\frac{\vep}{\kappa}\log(1/\vep) < \vep^{1/6}/2,
\label{ncond}
\eeq
then it enough to bound 
\begin{align*}    
&P\left(\sum_{i=0}^{n-1}\sup_{j\in J(R_i)} \sup_{s\in[R_i,R_i+\sqrt\vep]} 
|\hat X^j_s-\hat X^j_{R_i}|+|X^j_s-X^j_{R_i}|>\frac{\vep^{1/6} }{2}\right) \\
&\quad \le \sum_{i=0}^{n-1}(M+1+iN_0)2P\Bigl(\sup_{s\le \se}|B_s^\vep|>\frac{\vep^{1/6}}{ 4n}\Bigr)
\le c(M+1)n^2N_0 \cdot n^2\vep^{-2/6}\vep^{1/2},
\end{align*}
by the $L^2$ maximal inequality for martingales. If $n=\lceil\vep^{-1/40}\rceil$ (so that \eqref{ncond}
holds for $\vep<\vep_{{\tref{lem:tildeG}}}(\kappa)$) the above gives
\begin{equation}\label{tGbnd} 
P\left( {\bar G}_{N_T},\ \sup_{s\le T} \sup_{j\in\hat J(s)} |X_s^j-\hat X_s^j|>\vep^{1/6} \right)
\le ce^{c_{b}T}(M+1)\vep^{1/15}.
\end{equation}

The domination of $X$ by $\bar X$ ensures that
$$
\cup_{m=0}^{N_T-1}[R_m,R_m+2\vep^\beta]\subset
\cup_{m=0}^{\bar N_T-1}[\bar R_m,\bar R_m+2\vep^\beta].
$$
Therefore (recall $T>2\vep^\beta$) for any $\ell\in\NN$
\begin{align*}
P&(T \in\cup_{m=0}^{N_{T}-1}[R_m,R_m+2\vep^\beta])\\
    &\le P(\bar N_T>\ell ) + P( T \in\cup_{m=1}^{\ell-1}[\bar R_m,\bar R_m+2\vep^\beta]).
\end{align*}
Lemma~\ref{lem:good1}(a) shows that the first term is at most $e^{c_bT}(M+1)(\ell N_0)^{-1}$. Conditional on ${\cF}_{\bar R_{m-1}}$, 
$\bar R_m-\bar R_{m-1}$ is an exponential random variable with rate $(M+1+(m-1)N_0)c^*$, so the second term is at most
\begin{align*}
& E \left( \sum_{m=1}^{\ell} P(T-2\vep^\beta-\bar R_{m-1}\le\bar R_m-\bar R_{m-1}\le T-\bar R_{m-1}|
 {\cF}_{\bar R_{m-1}}) \right)\\
& \le 2\vep^\beta  \sum_{m=1}^\ell ((M+1+(m-1)N_0)c^*) \le c e^{c_bT} (M+1) \ell^2 \ep^\beta.
\end{align*}
Taking $\ell=\lceil\vep^{-\beta/3}\rceil$ then using \eqref{GTbnd} and \eqref{tGbnd}  
gives the desired bound on $P(\tilde G_T^c)$.
\end{proof}

The next ingredient required for the convergence theorem is:
  
\begin{lem}\label{lem:compeq}Assume $T>2\vep^\beta$, $t_0\in[0,\vep^\beta]$,  and $\omega\in\tilde G^\beta_T$. If
$\hat\zeta_{t_0}(j)=\zeta_{t_0}(j)$ for all $j\in \hat J(T-t_0)$, then
$\hat\zeta_T(i)=\zeta_T(i)$, $i=0,\dots, M$. In particular if $\hat\zeta _{t_0}(j)=\xi_{t_0}(X_{T-t_0}^j)$ for $j\in J(T-t_0)$, then $\hat\zeta_T(i)=\xi_T(z_i)$ for $i=0,\dots,M$.
\end{lem}

\begin{rem} By Lemma~\ref{lem:tildeG}, $J(T-t_0)=\hat
  J(T-t_0)$ on $\tilde G^{\beta}_T$, and so
  all the necessary inputs required for both
  computations are prescribed in the above result.
\end{rem}
\begin{proof} 
The last statement is immediate from the first and \eqref{dualityeq} with $r=T$. 

By the definition of $G_{N_T}\supset \tilde G^\beta_T$ and Lemma~\ref{lem:hatagree} there is a unique $n<N_T$ so that 
\beq \label{mgap}R_n+\sqrt\vep\le T-\vep^\beta\le t-t_0< T<R_{n+1}
\eeq 
and
\begin{equation} \hat R_m=R_m\hbox{ and }\hat \mu_m=\mu_m\hbox{ for }m\le n+1,\ \hat K(s)=K(s)\hbox{ for }s\in[0,T].
\end{equation}
As was noted in Section~\ref{ssec:hatXcomp} the inductive
definitions of $\zeta$ and $\hat \zeta$ are identical except
the latter has hats on the relevant variables and uses
$\approx_t$ in place of $\sim_t$.  The above shows that in
our current setting the relevant variables are the same with
or without hats (recall we are using $\hat U_n=U_n$ in our
coupled construction of $\hat \zeta$) and so it remains to
show the equivalence relations are the same and we do this
now for the initial extensions.  That is, we extended $\hat
\zeta_{t_0}$ to $\{0,\dots, K(T-t_0)\}$ by
$\hat\zeta_{t_0}(k)=\hat\zeta_{t_0}(j)$ if
$k\approx_{T-t_0}j\in \hat J(T-t_0)=J(T-t_0)$ (see the above
Remark) and extended $\zeta_{t_0}$ in the same way but if
$k\sim j\in J(T-t_0)$ which means $X^j_{T-t_0}=X^k_{T-t_0}$,
and so we now show these equivalencies are the same and
hence so are the extensions. Note that in applying
\eqref{approxdef} to extend $\hat \zeta_{t_0}$ we are using
$\pi_m=\pi_{R_m,\gamma_m}(\sqrt\vep)$ for $m>0$ and
$\pi_0=\pi_{0,z}(\sqrt\vep)$.  This means two indices $j,k$
in a family which has branched at time $\hat R_m=R_m$, $0\le
m\le n$ (if $m=0$ this means two initial indices) are
equivalent (in the $\approx$ sense) at time $T-t_0$ if their
corresponding $X$ paths coalesce by time $R_m+\sqrt\vep$.
Lemma~\ref{lem:hatagree} implies that on $\tilde G^\beta_T$
there are no coalescing events in $[0,T-t_0]$ (in fact on
$[0,T]$) except for those in $[R_m,R_m+\sqrt\vep]$,
involving a common family born at $R_m$, for $m\le
n$. Therefore, the above condition is equivalent to
$X^j_{T-t_0}=X^k_{T-t_0}$ and the required result is proved.

The Lemma now follows easily by induction up the tree of $X$.  In place of the above we must show equivalence of the equivalencies used in \eqref{zetadef-2} and \eqref{hatzetadef} at times $R_m$. Note here that for the indices of interest in \eqref{zetadef-2} and \eqref{hatzetadef} this is equivalent to the corresponding equivalencies at times $R_{m-1}+\sqrt\vep$ and this follows as above for $m\ge 1$. 
\end{proof}

\subsection{The branching Brownian motion and 
  computation process} 
We now define a branching Brownian motion $\hat X^0$
starting at $x\in\Rd$ with paths in $\cD$.  Let $\{\pi^0_n,n\ge 1\}$ be an iid sequence of
partitions with law $\nu_0$ (defined in the second paragraph of Section~\ref{ssec:brw}).   Particles
in $\hat X^0$ branch at rate $c^*$ and at the $n$th branching time, $|\pi_n^0|-1$ particles are born at the location of the parent who also remains alive. After birth,
particles in $\hat X^0$ move as independent Brownian motions
in $\R^d$ with variance parameter $\sigma^2$. 
To couple $\hat X^0$ with the branching random walk $\hat X^\vep$ from Section~\ref{ssec:brw} we need two preliminary lemmas which allow us to couple the corresponding particle motions and offspring numbers, respectively, of the two branching processes. 

\begin{lem}\label{brcoup} 
We may define our scaled random walk $B^\vep$ and a
$d$-dimensional Brownian motion $B$ with variance
$\sigma^2$, starting at $0$, on the same space so that for some constant $c_{\tref{brcoup}}$
$$
P\left(\sup_{t\le T} |B^\vep_t-B_t|\ge \se\right) \le c_{\tref{brcoup}}T \vep.
$$
\end{lem}

\begin{proof} Apply Theorem 2.3(i) of Chapter 1 of \cite{CH}
with $H(x)=x^6$, to see we may define the unscaled random
walk $B^1$ (rate $1$, step distribution $p$) and a
Brownian motion as above, $B'$ on the same space so that
for all $S>0$ and $r\ge 1$
\begin{equation}\label{CH}
P\left(\sup_{s\le S}|B^1_s-B'_s|\ge r\right)\le cS r^{-6}.
\end{equation}
Although the above reference applies to discrete time
random walks, we apply it to the step distribution
$\sum_{i=1}^{N(1)} X_i$, where $\{X_i\}$ are iid
$p(\cdot)$ and $N(1)$ is an independent Poisson$(1)$
random variable.  We arrive at the above after a short
interpolation calculation for $B^1$.

To get the desired result from \eqref{CH} we set $B^\vep_t=\vep B^1_{\vep^{-2}t}$, 
$B_t=\vep\tilde B'_{\vep^{-2}t}$ and use $r=\vep^{-1/2}$ to conclude that
\begin{align*}
P(\sup_{t\le T} |B^\vep_t-B_t|\ge \se)
& \le P(\sup_{t\le T} |\vep B^1_{\vep^{-2}t}- \vep B'_{\vep^{-2} t}|\ge \se) \\ 
&\le P(\sup_{t\le \vep^{-2}T} |B^1_t-B'_t|\ge \vep^{-1/2}) \\
&\le c\vep^{-2}T \vep^3= cT\vep
\end{align*}
which proves the desired result.
\end{proof}

\begin{lem}\label{pincoupl}
For each $\vep>0$ we may construct the sequence
$\{\pi^0_n : n \ge 1 \}$ on the same space as
$\{\pi_n^\vep : n \ge 1 \}$ so that
$$
P(\pi_n^\vep\neq \pi_n^0)\le c_{\tref{pincoupl}} \vep^{3/4}.
$$
\end{lem}
\begin{proof} The obvious way to couple $\pi^\vep_n$ and $\pi_n^0$ is to use the same system of rate one coalescing random walks $\{\hat B^{Y^i}:i=0,\dots,N_0\}$.  If $Z$ is as in Lemma~\ref{lem:notcrowd}, 
then by \eqref{retbnd} and \eqref{uniflclt}
\begin{align*} 
P(\pi_n^\vep\neq\pi^0_n)&\le\sup_x P^x(Z_s=0\hbox{ for some }s\ge \vep^{-3/2})\\
&\le c(\vep^{-3/2})^{-1/2}=c\vep^{3/4}.
\end{align*}
\end{proof}

Let $x_\vep\in\vep\Zd$ for $\vep>0$ and assume $x_\vep\to
x\in\Rd$. Our goal now is a joint construction of $(\hat
X^\vep,\hat X^0)$ started from $(x_\vep,x)$, and associated computation processes
$(\hat\zeta^\vep,\hat\zeta^0)$ with the property that if
$\hat\zeta^\vep,\hat\zeta^0$ have the same inputs then they
will have the same outputs with probability close to one.

The branching random walk $\hat X^\vep$ starting with a
single particle at $x_\vep$, along with the associated index
sets $\hat J^\vep(\cdot)$, branch times $\{\hat R^\vep_m\}$,
and parent variables $\hat \mu^\vep_m\}$, are constructed as
in Section~\ref{ssec:brw} using the sequence
$\{\pi^\vep_m\}$ in Lemma~\ref{pincoupl}.  There is no
initial coalescing step now as we are starting with a single
particle.  We use the coupled sequence $\{\pi_n^0:n\ge 1\}$
to define the offspring numbers, branching times
$\{R_n^0:n\ge 1\}$, index sets $J^0(\cdot)$ and parent
variables $\{\mu_n^0:n\ge 1\}$ with the same conditional
laws (given $\{\pi^0_m\}$) as in the definition of $\hat
X^\vep$.  We may couple these two constructions so that for
all $n\in \Z_+$, on the set
\[G_n^{0,\vep}=\{\pi^0_m=\pi^\vep_m\hbox{ for all }0\le m<n\},\]
we have
\begin{equation}\label{rematch} \hat R^\vep_m=R^0_m, \hat
  \mu^\vep_m=\mu^0_m,\hbox{ and }J^0(s)=\hat J^\vep(s)
  \hbox{ for all }s<R^0_m,\hbox{ for all }m\le n.
\end{equation} 
Define $N^0_t=\inf\{m:R^0_m>t\}$.  Using these sequences we follow the
prescription in Section~\ref{ssec:brw} for constructing
$\hat X$ but substituting Brownian motion paths for random
walk paths. Couple these random walks
and  Brownian motions as in
Lemma~\ref{brcoup} at least as long as the branching
structure of the two are the same.  Note that if there are $n$ branching 
events up to time $T$ there are at most $1+nN_0$ independent random walk segments 
and Brownian motions of length at most $T$ to couple  (recall our labeling scheme
from Section ~\ref{ssec:brw}).  In addition to the errors in Lemma~\ref{brcoup}
there will be a small error from the difference in initial
positions at time $0$, and so we get
\begin{align}\label{ZXmatch}
  &P(G^{0,\vep}_{N^0_T}, \sup_{s\le T}\sup_{j\in J^0(s)}
  |\hat X^{0,j}_s-\hat X^{\vep,j}_s|\ge |x_\vep-x|+\se)\\
  \nonumber&\le
  P(N^0_T>n)+(1+nN_0)c_{\tref{brcoup}}T\vep+c_{\ref{pincoupl}}n\vep^{3/4}.
\end{align}

The first time $\pi_n^\vep\neq \pi_n^0$ we declare the coupling a failure
and complete the definition of $\hat X^0_t$ for $t\ge  R_n^0$ using random variables independent of $\hat X^\vep$. 

Fix $T>0$ and $t_0\in[0,T)$. Given $\hat X^0_t, J^0(t), 0\le t\le T$,
the sequences $\{\pi^0_n\}$, $\{R^0_n\}$, $\{\mu^0_n\}$ an independent sequence of iid uniform $[0,1]$ random variables $\{U^0_m\}$,
and initial inputs $\{\zeta_{t_0}(j):j\in J^0(T-t_0)\}$, we define a computation process
$\hat\zeta^0_t, t_0\le t\le T$. The definition is analogous
to that of $\hat\zeta_t$ given in
Subsection~\ref{ssec:hatXcomp} for $\hat X^\vep$ started at a
single point, but we use $g_{1-i}$ in place of  $g^\vep_{1-i}$ in \eqref{mu-flip-hat}. 
That is, as in \eqref{hatVdefn}, \eqref{mu-flip-hat}, we have 
\[V_m^{0,j}=\hat\zeta^0_{(T-R^0_m)-}((m-1)N_0+j), \ j=1,\dots,N_0,\]
and if $i=\hat \zeta^0_{(T-R^0_m)-}(\mu^0_m)$, we have 
\begin{equation}\label {mu-flip-0}\hat\zeta^0_{(T-\hat R_m)}(\mu^0_m)=\begin{cases} 1-i &\text{ if } U^0_m \le g_{1-i}(\hat V^0_m)/c^*  \\
                                      i &\text{ otherwise.}
\end{cases}
\end{equation}
We further couple $\hat\zeta^0$ and $\hat\zeta^\vep$ by
using the same sequence of independent uniforms:
$\{U_m^0\}=\{U_m\}$ in their inductive definitions.  Just as
in \eqref{Umind} we can show that this sequence is
independent of all the other variables used to define $\hat
X^0$ and $\hat\zeta^0$, as required. We let $\hat\F^0_t$
denote the right-continuous filtration generated by $\hat
X^0$, $\hat X^\vep$ and $\hat A^0(t)=((R_m^0, \mu_m^0,
\pi_m^0,U_m)1(R_m^0\le t))_{m\in\\N}$ as well as its
counterpart for $\hat X^\vep$.

\bigskip

\noindent{\bf Notation.} $\tilde G^{0,\vep}_T=G_{N^0_T}^{0,\vep}\cap\{\sup_{s\le T}\sup_{j\in J^0(s)}|\hat X^{0,j}_s-\hat X^{\vep,j}_s|\le |x_\vep-x|+\sqrt\vep\}$,\\
$\bar G^{0,\vep}_T=\tilde G^{0,\vep}_T\cap\Bigr\{U_m\notin \Bigl[\frac{g_i(\xi)\wedge g_i^\vep(\xi)}{c^*}, \frac{g_i(\xi)\vee g_i^\vep(\xi)}{c^*}\Bigr]\hbox{ for all }\xi\in\{0,1\}^{N_0}, m<N^0_T, i=0,1\Bigl\}$.

\begin{lem} \label{poscoupl}
(a) On $\tilde G^{0,\vep}_T$, we have 
$$\hat R^\vep_m=R^0_m, \hat\mu^\vep_m=\mu^0_m, \pi^\vep_m=\pi^0_m,\hbox{ for all }m\le N^0_T, \hbox{ and } \hat J^\vep(s)=J^0(s)\hbox{ for all }s\le T.$$

\noindent (b) $P((\tilde G^{0,\vep}_T)^c)\le c_{\ref{poscoupl}}e^{c_bT}\vep^{3/8}$.

\noindent (c) On $\bar G_T^{0,\vep}$ we also have for any $t_0\in[0,T)$, if $\hat\zeta^0_{t_0}(j)=\hat\zeta^\vep_{t_0}(j)$ for all $j\in J^0(T-t_0)$, then $\hat\zeta^0_T(0)=\hat\zeta^\vep_T(0)$.

\noindent (d) $P((\bar G^{0,\vep}_T)^c)\le c_{\ref{poscoupl}}e^{c_bT}\Bigl[\vep^{3/8}+\sqrt{\sum_{i=0}^1 \Vert g_i^\vep-g_i\Vert_\infty}\Bigr]$. 
\end{lem}

\begin{proof} (a) is immediate from \eqref{rematch} and the definition of $\tilde G_T^{0,\vep}$.

\noindent (b) follows from \eqref{ZXmatch} and the now familiar bound $P(N^0_T>n)\le \frac{e^{c_bT}}{nN_0}$, by setting $n=\lceil \vep^{-3/8}\rceil$.

\noindent (c) On $\bar G_T^{0,\vep}$, we see from (a) and the inductive definitions of $\hat\zeta^0$ and  $\hat\zeta^\vep$, that all the variables used to define $\hat\zeta^0_T(0)$ and $\hat\zeta^\vep_T(0)$ coincide.  Therefore these outputs can only differ due to the use of $g_{i-1}$ in \eqref{mu-flip-0} and the use of $g^\vep_{i-1}$ in \eqref{mu-flip-hat}.  By induction we may assume $\hat V_m=V^0_m$ and the additional condition defining $\bar G^{0,\vep}_T$ now ensures that these two steps produce the same outputs. 

\noindent (d) The additional condition defining $\bar G_T^{0,\vep}$ fails with probability at most (recall $c^*\ge 1$)
\begin{align*}P(N_T>n)+n2^{N_0}\Bigl[\sum_{i=0}^1 \Vert g_i-g_i^\vep\Vert_\infty\Bigr]\le \frac{e^{c_bT}}{nN_0}+n2^{N_0}\Bigl[\sum_{i=0}^1 \Vert g_i-g_i^\vep\Vert_\infty\Bigr].
\end{align*}
Now let $n=\lceil \sum_{i=0}^1 \Vert g_i-g_i^\vep\Vert_\infty\rceil^{-1/2}$ and use (b) to complete the proof.
\end{proof}

\clearpage

\section{Proofs of Theorems~\ref{conv} and \ref{thm:strongconv}}
\label{sec:hydroproofs}
\subsection{Proof of Theorem \ref{conv}}
\label{ssec:convproof}

We start with a key estimate giving the product structure in 
Theorem~\ref{conv}. This relies on the fact that duals starting
at distant points with high probability will not collide.  For some
results we will need a quantitative estimate.  Let $x^k_\vep\in\vep\Z^d$, $y_i\in\Z^d$ and 
$z^\vep_{ik}=z_{ik}=x_\vep^k+\vep y_i$, for $0\le i\le L$ and $1\le k\le K$.   Set 
$$\Delta_\vep=\min_{1\le i,i'\le L, 1\le k\neq k'\le K}|z_{ik}-z_{i'k'}|\vep^{-1}.$$
The notation is taken
to parallel that in Theorem~\ref{conv} and hypothesis~\eqref{xcond} of that result implies
\begin{equation}\label{deltato0}\lim_\vep \Delta_\vep=\infty.\end{equation}
Let $X=X^{z,T}$ be the dual
process starting at $z$ for the time period $[0,T]$, with
associated computation process $\zeta_t$ which has initial
inputs $\zeta_0(j) = \xi^\vep_0(X^j_T)$, $j\in J(T)$. 
Let $z_k=(z_{ik},i=0,\dots,L)$ and consider the duals $X^{z_k,T}$, $1\le k\le K$
defined as in Section~\ref{sec:construction} with their associated uniforms $\{U^k_m\}$ and parent
variables $\{\mu^k_m\}$. These duals are naturally embedded in $X^{z,T}$, and although the numbering of the particles may differ, we do have
\begin{equation}\label{partlist}\{X^{z,j}_t:j\in J(t)\}=\cup_{k=1}^K\{X_t^{z_k,j}:j\in J^{z_k}(t)\},\ t\in[0,T]. 
\end{equation}
Define
\begin{align}\label{Vdefnep}V_{z,T,\vep}=\inf\{t\in[0,T]:&X_t^{z_k,T,j}=X_t^{z_{k'},T,j'}\\
\nonumber&\hbox{ for some }1\le k\neq k'\le K, j\in J^{z_k}(t),j'\in J^{z_{k'}}(t)\},
\end{align}
where $\inf\emptyset=\infty$.

\begin{lem}\label{lem:Vprob}
$P(V_{z,T,\vep}<\infty)\le c_{\ref{lem:Vprob}}(K,L)e^{c_bT}(\Delta_\vep)^{-(d-2)/(d+3)}$.
\end{lem}
\begin{proof} We may dominate $X^{z_k,T}$ by the branching random walks $\bar X^{z_k,T}$ from Section~\ref{ssec:pbad}.  By Lemma~\ref{lem:good1}(a), if $\{Y^*_m\}$ are iid, equal in law to $Y^*$, and independent of $B^\vep$ in what follows, then for $R\ge 1$,
\begin{align*}
&P(V< \infty)\\
&\le P(\max_{k\le K}\bar N^{z_k}_T>n)\\
&\ +\sum_{1\le k\neq k'\le K}P(|\bar X_t^{z_k,j}-\bar X_t^{z_{k'},j'}|=0\ \exists j\in\bar J^{z_k}(t), j'\in \bar J^{z_{k'}}(t), t\le T, \bar N_T^{z_k}\vee \bar N_T^{z_{k'}}\le n).
\end{align*}
The first term is bounded by $ Ke^{c_bT}n^{-1}$ and the second term is at most
\begin{align*}
&\sum _{1\le k\neq k'\le K; 0\le i,i'\le L}(1+nN_0)^2P_{z_{k,i}-z_{k'i'}}(|B^\vep_{2t}|\le \sum_{m=1}^n \vep|Y^*_m|\ \exists t\le T)\\
&\le \sum _{1\le k\neq k'\le K; 0\le i,i'\le L}(1+nN_0)^2\Bigl[nCe^{-\kappa R}+P_{z_{ki}-z_{k'i'}}(|B^\vep_{2t}|\le n\vep R\ \exists t\le T)\Bigr],
\end{align*}
where we used \eqref{expbd2} in the last line.  By Lemma~\ref{lem:notcrowd}(c), the probability in the last term is at most $c_{\ref{lem:notcrowd}}\Delta_\vep^{-(2/3)(d-2)} (nR)^{2(d+1)/3}$, and so if 
$\delta=\Delta_\vep^{-(2/3)(d-2)}$,
\[P(V<\infty)\le cK^2(L+1)^2e^{c_bT}N_0^2\Bigl[n^{-1}+n^3 e^{-\kappa R}+\delta(nR)^{2(d+1)/3}\Bigr].
\]
Now, optimizing over $n$ and $R$, set $c_d=\frac{12}{2d+5}$, $\kappa R=c_d\log(1/\delta)$ and $n=\lceil e^{\kappa R/4}\rceil$.  Here we may assume without loss of generality that $\Delta_\vep\ge M(\kappa)$ so that $R\ge 1$.  A bit of arithmetic now shows the the above bound becomes
\[P(V<\infty)\le c(K,L)e^{c_bT}\delta^{3/(2d+6)},\]
and the result follows.
\end{proof}

We suppose now that the assumptions of Theorem~\ref{conv}
are in force.  That is, $T>0$ is fixed, $\xi^\vep_0$ has law $\lambda_\vep$ 
satisfying the local density condition \eqref{densdef} for a
fixed $r\in(0,1)$, and \eqref{xcond} holds.  It is intuitively
clear that the density hypothesis is weakened by reducing $r$.  To
prove this, note that the boundedness of the density and uniformity in $x$
of the convergence in \eqref{densdef} shows that the contribution to
the density on larger blocks from smaller blocks whose density is
not near $v$ is small in $L^1$. We may therefore approximate the mass in a large block by 
the mass in smaller sub-blocks of density near $v$, and use the fact that the contributions close to the boundary of the large block is negligible to derive the density condition \eqref{densdef} for the larger blocks.  As a result we may assume that $r<1/4$.   

By
inclusion-exclusion, it suffices to prove for $-1\le L_k\le L$,
 \begin{multline}\label{convergence}
    \lim_{\vep\to 0}P(\xi^\vep_T(x_\vep^k+\vep
    y_{i_j})=1,\ \ j=0,\dots,L_k,\, k=1,\dots K)\\ 
    =\prod_{k=1}^K\langle 1\{\xi(y_{i_j})=1,\ \
    j=0,\dots,L_k\}\rangle_{u(T,x^k)}.
  \end{multline}
Allowing $k$-dependence in $L_k$ and general subsets of the $y_i$'s is needed for the inclusion-exclusion, but to reduce
eyestrain we will set $i_j=j$ and $L_k=L$ in what follows.  The general case requires
only notational changes.
By the duality
equation \eqref{dualityeq},
\begin{align}\label{dualeq2}
P(\xi^\vep_T(z_{ik})=1,\ & i=0,\dots,L, k=1,\dots,K ) \\
\nonumber& = P(\zeta_T(i,k)=1, i=0,\dots,L,  k=1,\dots,K ) ,
\end{align}
so \eqref{convergence} is then equivalent to
\begin{multline}\label{convgoal1}
P(\zeta_T(i,k)=1, i=0,\dots,L,  k=1,\dots,K )\\
\to \prod_{k=1}^K \langle 1\{ \xi(y_i)=1, i=0,\dots, L\}
\rangle_{u(T,x^k)} \text{ as }\vep\to 0.
\end{multline}
The proof of \eqref{convgoal1} uses the approach of
  \cite{DN94}, pp. 304-306.

To work with the left-hand
side of \eqref{convgoal1} we need the following preliminary
result to simplify the initial inputs $\zeta_0$.  Define
\begin{equation}\beta=1.9r\hbox{ and }t_\vep = \vep^{\beta}.\label{betaeps}
\end{equation}

\begin{lem}\label{lem:rev3}  Assume $\xi^\vep_0$ is
independent of the rescaled random walks $\{
B^{\vep,w}:w\in\vep\Zd\}$ as in \eqref{indrws}. Then for any $n\in\NN$ and $k>0$,  
\[
\lim_{\vep\to 0}\sup_{\substack{{|w_1|,\dots,|w_n|\le k,w_i\in\vep\Z^d}\\
w_i\ne w_j\text{ for }i\ne j}}
\Bigl|E\Bigl(\prod_{i=1}^n\xi^\vep_0( B^{\vep,w_i}_{t_\vep})\Bigr)
-\prod_{i=1}^n v(w_i)\Bigr| = 0 \,.
\]
\end{lem}
\begin{proof}
For $z_1,\dots,z_n\in a_\vep\Zd$ define
\[
\Gamma(z_1,\dots, z_n) = 
\{ B^{\vep,w_i}_{t_\vep} \in z_i+Q_\vep 
\text{ for }1\le i\le n \} \,,
\]
and $\gamma(w_i,z_i)=P( B^{\vep,w_i}_{t_\vep}\in
z_i+Q_\vep)$, so that $P(\Gamma(z_1,\dots,z_n))=\prod_{i=1}^n
\gamma(w_i,z_i)$. Let $G$ be the union of the events $\Gamma(z_1,\dots,z_n)$ 
over distinct $z_1,\dots,z_n\in a_\vep \Zd$ such that $|z_i-w_i| \le
k\sqrt{t_\vep}$ for $1\le i\le n$. 
We claim that $P(G^c)$ is small for $k$ large enough. To see this, 
fix $\delta>0$ and choose $k$ large enough so that
\begin{equation}\label{tc1}
P(|B^{\vep,w_i}_{t_\vep}-w_i|\ge (k-1)\sqrt{t_\vep})=P(| B^{\vep,0}_{t_\vep}| >(k-1) \sqrt{t_\vep})
<\delta/n\,. 
\end{equation}
By a standard estimate (and also since $r<1/4$), for $w_i$ as above and $i\ne j$,
\[
P( | B^{\vep,w_i}_{t_\vep} -
B^{\vep,w_j}_{t_\vep}|\le 2 a_\vep)
\le c |Q_\vep| P( B^{\vep,0}_{2t_\vep}=0)
\le c |Q_\vep| (\vep^{-2}t_\vep)^{-d/2} \le c \vep^{(d-1)3/4}\,,
\]
which implies 
\begin{equation}\label{tc0}
P( | B^{\vep,w_i}_{t_\vep} -
B^{\vep,w_j}_{t_\vep}|\le 2 a_\vep \text{ for some }1\le
i<j\le n) \le C n^2\vep^{(d-1)3/4}\,.
\end{equation}
By \eqref{tc0} and
\eqref{tc1},
\begin{equation}\label{tc2}
P(G^c) \le Cn^2\vep^{(d-1)3/4} + \delta\ .
\end{equation}

Now consider the decomposition
\begin{equation}\label{tc3}
E\Bigl(\prod_{i=1}^n\xi^\vep_0(
B^{\vep,w_i}_{t_\vep}); G\Bigr)  =
\sum_{z_1,\dots,z_n} E\Bigl(\prod_{i=1}^n\xi^\vep_0(
B^{\vep,w_i}_{t_\vep}); \Gamma(z_1,\dots,z_n)\Bigr) 
\end{equation}
where the sum is taken over only those $(z_1,\dots,z_n)$
used in the definition of $G$. A typical term in this sum
takes the form
\begin{multline*}
\sum_{e_1,\dots,e_n \in Q_\vep} E\Bigl(\prod_{i=1}^n\xi^\vep_0(z_i+e_i)
1(B^{\vep,w_i}_{t_\vep}=z_i+e_i)\Bigr)\\
= \sum_{e_1,\dots,e_n \in Q_\vep} E\Bigl(\prod_{i=1}^n\xi^\vep_0(z_i+e_i)\Bigr)
\prod_{i=1}^nP(B^{\vep,w_i}_{t_\vep}=z_i+e_i) \ .
\end{multline*}
Since $\sqrt{t_\vep} \gg a_\vep$, the probabilities 
$P(B^{\vep,w_i}_{t_\vep}=z_i+e_i)
=P(B^{\vep,0}_{t_\vep}=z_i-w_i+e_i)$ 
are almost constant over $e_i\in Q_\vep$.
In fact, a calculation, using the version of the local central limit
theorem in the Remark after P7.8 in \cite{Spi} to
expand\break
 $P( B^{\vep,0}_{t_\vep}=z_i-w_i+e_i) =
P(B^{0}_{\vep^{-2}t_\vep}=(z_i-w_i+e_i)/\vep)$, shows that
\begin{equation}\label{ratio}
\lim_{\vep\to 0} \sup_{\substack{
e,e'\in Q_\vep\\
|z_i-w_i|\le k\sqrt{t_\vep}}
} 
\dfrac{P(B^{\vep,0}_{t_\vep}=z_i-w_i+e)}
{P(B^{\vep,0}_{t_\vep}=z_i-w_i+e')} 
=1
\end{equation}
The continuous
time setting is easily accommodated, for example by
noting that along multiples of a fixed time it becomes a
discrete time random walk.

Consequently, for all
sufficiently small $\vep>0$, 
we have $k\sqrt{t_\vep}<1$ and uniformly in $|w_i|\le k$,
$|z_i-w_i|\le k\sqrt{t_\vep}$ and $e\in Q_\vep$,
\begin{equation}\label{densebound}
1-\delta\le
\dfrac{|Q_\vep| P(B^{\vep,w_i}_{t_\vep}=z_i+e)} 
{\gamma(w_i,z_i)}
\le 1+\delta \ .
\end{equation}
Using this bound and the fact that the
$z_i$ are distinct we have
\begin{align*}
\sum_{e_1,\dots,e_n \in Q_\vep}
E\Bigl(\prod_{i=1}^n\xi^\vep_0(z_i+e_i)\Bigr) 
&\prod_{i=1}^nP( B^{\vep,w_i}_{t_\vep}=z_i+e_i)\\
&\le \sum_{e_1,\dots,e_n \in Q_\vep}
E\Bigl(\prod_{i=1}^n\xi^\vep_0(z_i+e_i)\Bigr) 
\dfrac{(1+\delta)^n}{|Q_\vep|^n}\prod_{i=1}^n \gamma(w_i,z_i)\\
&= (1+\delta)^n 
E\Bigl(\prod_{i=1}^nD(z_i,\xi^\vep_0)\Bigr)
P(\Gamma(z_1,\dots,z_n))
 \,.
\end{align*}

The continuity of $v$ implies that for small enough $\vep$,
for all $|w|\le k$ and $|z-w|\le k\sqrt{t_\vep}$,
$|v(w)-v(z)|< \delta$. Also for sufficiently small $\vep$ and
$z\in a_\vep\Z^d$, $|z|\le k+1$,  we have
$P(D(z,\xi^\vep_0)>v(z)+\delta)\le \delta/n$. Thus  
\[
E\Bigl(\prod_{i=1}^n D(z_i,\xi^\vep_0)\Bigr)
\le \delta + \prod_{i=1}^n (v(z_i)+\delta) 
\le  \delta+\prod_{i=1}^n (v(w_i)+2\delta)  \ .
\]
Returning to the decomposition \eqref{tc3}, the above bounds imply
that for sufficiently small $\vep$,
\begin{align*}
E\Bigl(\prod_{i=1}^n \xi^\vep_0( B^{\vep,w_i}_{t_\vep}) ; G\Bigr)&
\le (1+\delta)^n\Bigl[\delta + \prod_{i=1}^n (v(w_i)+2\delta)  \Bigr]
\sum_{z_1,\dots,z_n} P(\Gamma(z_1,\dots,z_n)) \\
&\le (1+\delta)^n\Bigl[\delta + \prod_{i=1}^n
(v(w_i)+2\delta) \Bigr] \,. \\
\end{align*}
Let $\vep\to 0$ and then $\delta\to 0$ above
and in \eqref{tc2} to obtain 
\[
\limsup_{\vep\to 0} \sup_{|w_1|,\dots,|w_n|\le
  k}
\Bigl(E\Bigl(\prod_{i=0}^n\xi^\vep_0( 
B^{\vep,w_i}_{t_\vep})\Bigr)
- \prod_{i=1}^n v(w_i)\Bigr) \le 0 \,.
\]
A similar argument gives a reverse  inequality needed
to complete the proof.
\end{proof}

\bigskip We break the proof of \eqref{convgoal1} into three
main steps. Introduce
$$S=T-t_\vep=T-\vep^\beta.$$

\bigskip\noindent \emph{Step 1. Reduction to Bernoulli inputs and $K=1$.} 

\noindent Let $\tilde X=\tilde X^{z,T}$ be the
modification of the dual in which particles ignore reaction
and coalescing events on $[S,T]$, and let 
$\tilde\zeta_t$ be the associated computation
process with inputs $\tilde\zeta_0(j)= \xi^\vep_0(\tilde X^j_T)$. 
That is,  $\tilde X_t =X_t$ for
$t\in[0,S]$, and during the time period $[S,T]$, 
$\tilde X^j_t$, $j\in J(S)$ follows the same path as $X^j_t$
until the first time a reaction or coalescence occurs, at
which time all the $\tilde X^j_t$ switch to following 
completely independent $B^\vep$ random walks.

On the event $\tilde G_T^\beta$
defined in \eqref{def:tildeG}
there are no reaction or coalescing events during $[S,T]$.
Thus, $\tilde X_t=X_t$ for all $t\in[0,T]$ on 
$\tilde G_T^\beta$, so it follows from
Lemma~\ref{lem:tildeG} that 
\begin{equation}\label{ignoreonST}
P( \zeta_t\ne \tilde\zeta_t  \text{ for some } 
t\in[0,T])
\le c_{\tref{lem:tildeG}}[(L+1)K]^2e^{c_bT}\vep^{\frac{1}{40}\wedge
  \frac{\beta}{3}} \,.
\end{equation}
%Consequently,
%\begin{multline}\label{red1}
%P(\zeta_T(i,k)=1, 
%i=0,\dots,L,  k=1,\dots,K )\\
%-P(\tilde\zeta_T(i,k)=1, 
%i=0,\dots,L,  k=1,\dots,K ) \to 0
%\text{ as } \vep\to 0\,.
%\end{multline}

Let $\psi_\vep(x)=P^\vep_{t_\vep}\xi_0^\vep(x)$, where 
\begin{equation}\label{rwsgroup}P_t^\vep
  f(x)=E(f(x+B_t^\vep)),\ x\in\vep\Z^d,\hbox{ is the
    semigroup of }B^\vep,
\end{equation}
and let $W_1,W_2,\dots$ be an iid sequence of uniforms on
the interval $[0,1]$, independent of $\xi^\vep_0$ and the
random variables used in Section 2. We will use this
sequence throughout the rest of this section and also in
Section~\ref{lockill}.  Define a second computation process
$\zeta^*_{t},t_\vep\le t\le T$, for $\tilde X$, with inputs
\begin{equation}\label{starinputs}
\zeta^{*}_{t_\vep}(j)= 1 \{W_j\le v( \tilde X^j_S )\},
\quad j\in J(S). 
\end{equation}
It is clear that conditional on $\sigma(\xi^\vep_0)\vee\F_\infty$,
the variables $\zeta^*_{t_\vep}(j)$, $j\in J(S)$, respectively $\tilde \zeta_{t_\vep}(j)$, $j\in J(S)$, are independent
Bernoulli with means $v(\tilde X^j_S)$, respectively $\psi_\vep(\tilde X_S^j)$.  Let $\bar X=\bar X^{z,T}$ be the branching random walk dominating $X$ which was introduced in Section~\ref{ssec:pbad}.
If we  fix $\delta>0$, then  
using Lemma~\ref{lem:good1}(a) it is not 
hard to see that there exist $n,k$ such that for 
all $\vep$ sufficiently small,  
\begin{equation}\label{compact}
P( |\bar J(S)|\le n \text{ and }
|\bar X^\vep_j(S)|\le k \text{ for all }j\in J(S)) > 1-\delta.
\end{equation}

\noindent It now follows from \eqref{compact},
Lemma~\ref{lem:rev3} and the definitions of $\tilde X$,
$\tilde \zeta_0$ and $\zeta^*_{t_\vep}$ that for any
$b:\Z^+\to\{0,1\}$,
\begin{align*}
|&
P(\tilde\zeta_{t_\vep}(j)=b_j, j\in J(S)) - 
P(\zeta^{*}_{t_\vep}(j)=b_j, j\in J(S))|\\
&\le E(|P(\tilde\zeta_{t_\vep}(j)=b_j, j\in J(S)|\F_S\vee\sigma(\xi_0^\vep))-P(\zeta^{*}_{t_\vep}(j)=b_j, j\in J(S))|\F_S\vee\sigma(\xi_0^\vep))|)\\
& \to 0
\text{ as } \vep\to 0.
\end{align*}
As a consequence, since both $\tilde \zeta_t,\zeta^*_t,
t_\vep\le t\le T$ are defined relative to $\tilde X$ with identical 
$\{U_m\}$, $\{\mu_m\}$ and $\{R_m\}$, by conditioning on the input values,
 the above implies 
\begin{multline}\label{red2}
P(\tilde \zeta_T(i,k)=1, 
i=0,\dots,L,  k=1,\dots,K )\\
- P(\zeta^{*}_T(i,k)=1, 
i=0,\dots,L,  k=1,\dots,K )\to 0 \text{ as }\vep\to 0.
\end{multline}

Let $\zeta^{*,z_k}_t$ be the computation process associated with $X^{z_k,T}$, $1\le k\le K$ 
with inputs as in
\eqref{starinputs}. That is, for $j\in J^{z_k}(S)$ 
there exists a $j'\in J(S)$ with $X^{j'}_S= X^{z_k,j}_S$ (by \eqref{partlist})
and we set $\zeta^{*,z_k}_{t_\vep}(j) = 1\{W_{j'}\le
v(X^{z_k,j}_S)\}$. 
Up to time $V=
V_{z,T,\vep}$
the duals $X^{z_k,T}$, $k\le K$, use independent random walk steps and branching mechanisms, and on $\{V=\infty\}$ the computation processes $\zeta^{*,z_k}$ also use independent uniforms and parent variables as well as independent inputs at time $t_\vep$. It follows that (see below)
\begin{align}
\nonumber|P&(\zeta^*_T(i,k)=1,
i=0,\dots,L,  k=1,\dots,K ) \\ 
\nonumber&\phantom{(\zeta^*_T(i,k)=1,
i=0,\dots,L,}-\prod_{k=1}^K P(\zeta^{*,z_k}_T(i)=1,
0\le i\le L )|\\
\label{K1}&\le P(V<\infty) \to 0
\text{ as }\vep\to 0.
\end{align}
The last limit follows from Lemma~\ref{lem:Vprob} and
\eqref{deltato0}. Perhaps the easiest way to see the first
inequality is to extend $X_t^{z_k,T}$ to $t\in[V,T]$ by
using independent graphical representations and define the
corresponding computation processes $\zeta^{' *,z_k}$ using
independent collections of $\{W_j\}$'s for the inputs at
time $t_\vep$. The resulting computation processes
$\zeta^{'*,z_k}$ are then independent, each $\zeta^{'
  *,z_k}$ is equal in law to $\zeta^{*,z_k}$, and the two
are identical for all $k$ on $\{V=\infty\}$.  On this set we
also have
\[\{\zeta_T^*(i,k): i, k\}=\{\zeta_T^{*,z_k}(i):i,k\},\]
and so \eqref{K1} follows.
It is therefore enough to set $K=1$ and  drop the
superscript $k$. Altering our notation to 
$z=(z_i)$, $z_i=x_\vep+\vep y_i$ where $x_\vep\to x$, 
it suffices now to prove
\begin{equation}\label{convgoal2}
P(\zeta^*_T(i)=1, 0\le i\le L ) \to
\langle 1\{\xi(y_i)=1, 0\le i\le L 
\}\rangle_{u(T,x)} \text{ as }\vep\to 0 .
\end{equation}

\bigskip\noindent\emph{Step 2. Reduction to $L=0$.} 
Let $\hat X = \hat X^{z,T}, 0\le t\le T$ be the branching random walk
started at $z$, with associated computation 
process $\hat\zeta_t,t_\vep \le t\le T$. We suppose that $X$ and $\hat X$
are coupled as in Subsection~\ref{ssec:XhatXcoupling}, and that $\hat\zeta_t$ has
initial inputs
\[
\hat\zeta_{t_\vep}(j) = 1\{ W_j \le v(\hat X^j_S) \}, j\in
\hat J(S) .
\]
On the
event $\tilde G^\beta_{T}$, $J(S)=\hat J(S)$ and all the
differences $|X^j_S- \hat X^j_S|$, $j\in J(S)$ are small. 
It therefore follows from \eqref{compact} (if we take $Y^i=0$ $\bar X$ will stochastically dominate $\hat X$), the continuity of
$v$, the definitions of $\zeta^*_{t_\vep}$ and $\hat \zeta_{t_\vep}$, and
Lemma~\ref{lem:tildeG} that 
\begin{equation}\label{comp-coup}
P(\tilde G^\beta_T, \zeta^{*}_{t_\vep}(j) =\hat\zeta_{t_\vep}(j)
\text{ for all }j
\in J(S)) \to 1
\text{ as } \vep \to 0.
\end{equation}
By Lemma~\ref{lem:compeq}, on the event in
\eqref{comp-coup}, the outputs $\zeta^*_T$ and
$\hat\zeta_T$ agree, and consequently
\begin{equation}\label{red3}
P(\zeta^*_T(i)=1, 0\le i\le L)- P(\hat \zeta_T(i)=1, 0\le i\le L) \to 0
\text{ as }\vep\to 0.
\end{equation}

Using the branching structure we can now
reduce to the case $L=0$. To see this,  
let $\hat X^{z_i,T}$ be the branching random walk started
from $z_i=x_\vep+\vep y_i$, with associated computation
process $\hat\zeta^{z_i}_t,t_\vep\le t\le T$ 
with initial inputs $\hat \zeta^{z_i}_{t_\vep}(j)$ which,
conditional on $\hat X^{z_i}_t, 0\le t\le S$ are independent with means
$v(\hat X^{z_i,j}_S)$. The branching property and definition of $\hat X_0$ in \eqref{hatX0def} imply  (recall $\nu_\vep$ from just
above \eqref{nuconv}) 
\[
P(\hat \zeta_T(i)=1,  i=0,\dots, L) =
\sum_{\pi\in \Pi_L}\nu_\vep(\pi)\prod_{j\in J(\pi)}
P(\hat\zeta^{z_j}_T(0) = 1 ) .
\]
Since $z_j\to x$ as $\vep\to 0$  for $i=0,\dots,L$,  
if we can establish
\begin{equation}\label{convgoal3}
P(\hat\zeta^{x_\vep}_T(0) = 1 ) \to \hat u(T,x)
\text{ as }\vep\to 0,
\end{equation}
for some $\hat u:\R_+\times\R^d\to [0,1]$, 
then the convergence $\nu_\vep\To\nu_0$ implies
\begin{multline}\label{Lconv}
P(\hat\zeta_T(i)=1, i=0,\dots,L) 
\to 
\sum_{\pi\in \Pi_L}\nu_0(\pi)( \hat u(T,x))^{|\pi|}
\\ = \langle 1\{ \xi(y_i)=1, 0\le i\le L
\}\rangle_{\hat u(T,x)} \text{ as }\vep\to 0\,,
\end{multline}
where \eqref{prodform} is used in the last line.
Combining this with \eqref{red3} gives the desired result \eqref{convgoal2} but
with $\hat u$ in place of $u$, that is, we get
\begin{equation}\label{convgoal2'}
P(\zeta^*_T(i)=1, 0\le i\le L ) \to
\langle 1\{\xi(y_i)=1, 0\le i\le L 
\}\rangle_{\hat u(T,x)} \text{ as }\vep\to 0 .
\end{equation}
We first turn now to the proof of \eqref{convgoal3}.

\bigskip
\noindent\emph{Step 3. Convergence and identification of the limit.} Let $\hat X^{0}$ be the
branching Brownian motion started at $x\in \R^d$ run over the time
period $[0,T]$,  with 
associated computation process $\hat \zeta^0_t,t_\vep\le
t\le T$ with inputs 
\[
\hat \zeta^0_{t_\vep}(j) = 1\{ W_j\le v(\hat X^{0,j}_S)\}, j\in
J^0(S) .
\]
Using the obvious analogue of \eqref{compact} for $\hat
X^0$, the continuity of $v$ and the definitions of
$\hat\zeta_{t_\vep}$ and $\hat\zeta^0_{t_\vep}$,
Lemma~\ref{poscoupl} (and the uniform convergence of $g^\vep_i$ to $g_i$) implies
\begin{equation}\label{comp-coup2}
P(\bar G^{0,\vep}_T, \hat \zeta_{t_\vep}(j) =\hat\zeta^0_{t_\vep}(j)
\text{ for all }j
\in J^0(S)) \to 1
\text{ as } \vep \to 0.
\end{equation}
By Lemma~\ref{poscoupl}(c), on the event in \eqref{comp-coup2},
$\hat\zeta_T(0)=\hat\zeta^0_T(0)$, and thus 
\begin{equation}\label{red4}
P(\hat \zeta^{x_\vep}_T(0) =1 ) -
P(\hat \zeta^{0}_T(0) =1 )\to 0\hbox{ as }\vep\to 0,
\end{equation}
where we note that both quantities in the above depend on $\vep$.
If we take the initial inputs for the
computation process $\hat \zeta^0=\hat \zeta^{0,*}$ at time $0$ to be
\beq\label{hatinputs}
\hat\zeta^{0,*}_0(j) = 1\{W_j\le v(X^{0,j}_T)\}, j\in J^0(T),
\eeq
it is now routine to see that $P(\hat\zeta^0_{t_\vep}=
\hat\zeta^{0,*}_{t_\vep})\to 1$ as $\vep\to 0$, and so by \eqref{red4}
\begin{equation}\label{red5}
\lim_{\vep\to 0}P(\hat \zeta^{x_\vep}_T(0) =1 ) =P(\hat \zeta_T^{0,*}(0)=1)\equiv\hat u(T,x).
\end{equation}
This proves \eqref{convgoal3}, hence \eqref{Lconv} and so to complete the proof of 
\eqref{convgoal2}, and hence Theorem~\ref{conv}, we only need show $\hat u=u$:

\begin{lem} \label{lem:ident} 
Let $\hat X^0_t,0\le t\le T$ be the
branching Brownian motion started at $x\in\R^d$, with associated
computation process $\hat\zeta^{0,*}_t,0\le t\le T$ with initial inputs
as in \eqref{hatinputs}. Then
$$
P( \hat\zeta^{0,*}_T(0)=1 ) = u(T,x),
$$
where $u$ is the solution of the PDE \eqref{rdpde}.
\end{lem}
\begin{proof} This is very similar to the proof in Section 2(e) of
  \cite{DN94}.  Recall $P^\vep_t$ is the semigroup of $B^\vep$.  Let $x\in\R^d$ and $x_\vep\in\vep\Z^d$ satisfy $|x-x_\vep|\le \vep$ and let $\xi^\vep$ be our rescaled particle system where $\{\xi^\vep_0(\vep y):y\in\Z^d\}$ are independent Bernoulli random variables with means
  $\{v(\vep y):y\in\Z^d\}$.  If
\beq\label{dvepdef}d_\vep(\vep y,\xi_\vep) =-\xi(y)h_0^\vep(y,\xi)+(1-\xi(y))h_1^\vep(y,\xi),\ y\in\Z^d,\ \xi\in\{0,1\}^{\Z^d},
\eeq
then the martingale problem for $\xi^\vep$ shows that (cf. (2.25) of \cite{DN94})
\beq\label{meanequ}
E(\xi^\vep_T(x_\vep))=E(P^\vep_T\xi_0^\vep(x_\vep))+\int_0^TE_{x_\vep}\times E(d_\vep(B^\vep_{T-s},\xi^\vep_s))\,ds,
\eeq
where $B^\vep_0=x_\vep$ under $P_{x_\vep}$.  Our hypotheses on $\xi_0^\vep$ imply 
\[E(P_T^\vep\xi_0^\vep(x_\vep))=P^\vep_Tv(x_\vep)\to P_Tv(x)\hbox{ as }\vep\to0,\]
where $P_t$ is the $d$-dimensional Brownian semigroup with variance $\sigma^2$.  
Recall we have proved (\eqref{Lconv} and the preceding results) that
\[\lim_{\vep\to 0} P(\xi_T^\vep(x_\vep+\vep y_i)=\eta_i,\ i=0,\dots,L)=\langle1\{\xi(y_i)=\eta_i,\  i=0,\dots,L\}\rangle_{\hat u(T,x)}.\]
Now use the above with Fubini's theorem, the uniform convergence of $g_i^\vep$ in \eqref{gcvgce} and the coupling of $B^\vep$ and $B$ in Lemma~\ref{brcoup} to conclude that
\begin{align*}
&\lim_{\vep\to 0}E_{x_\vep}\times E(d_\vep(B^\vep_{T-s},\xi^\vep_s))\\
&=\lim_{\vep\to 0} E_{x_\vep}\times E\times E_Y\Bigl(-\xi^\vep_s(B^\vep_{T-s})g_0^\vep(\xi_s^\vep(B^\vep_{T-s}+\vep Y^1),\dots,\xi_s(B^\vep_{T-s}+\vep Y^{N_0}))\\
&\phantom{=\lim_{\vep\to0} E_{x_\vep}\times E\times E_Y(}+(1-\xi^\vep_s(B^\vep_{T-s}))g_1^\vep(\xi_s^\vep(B^\vep_{T-s}+\vep Y^1),\dots,\xi^\vep_s(B^\vep_{T-s}+\vep Y^{N_0})\Bigr)\\
&=E_x\Bigl(\langle -\xi(0)h_0(0,\xi)+(1-\xi(0))h_1(0,\xi)\rangle_{\hat u(s,B_{T-s})}\Bigr)\\
&=E_x(f(\hat u(s,B_{T-s})),
\end{align*}
the last by \eqref{fdef}.
Now use the above to take limits in \eqref{meanequ} to show that $\hat u$ solves the weak form of \eqref{rdpde}.  As in Lemma 2.21 of \cite{DN94} it follows that $\hat u$ solves \eqref{rdpde} and so equals $u$. 
\end{proof}
The following asymptotic independence result follows easily from Step 1 in the above argument.

\begin{prop}\label{prop:xiind} If $K\in\NN$, there is a
    $c_{\tref{prop:xiind}}(K)$ so that if
    $z_1,\dots,z_K\in \vep\Z^d$ satisfy $\inf_{j\neq
      k}|z_j-z_k|\ge \vep^{1/4}$ and $\xi_0^\vep$ is
    deterministic, then
    \[|E(\prod_{k=1}^K\xi^\vep
    _T(z_k))-\prod_{k=1}^KE(\xi^\vep _T(z_k))|\le
    c_{\tref{prop:xiind}}(K) e^{c_bT}\vep^{1/8}.\]
\end{prop}
\begin{proof} Define $V=V_{z,T,\vep}$ as in \eqref{Vdefnep}
  but now with $z_k\in\vep\Z^d$, that is $L=0$. Use the dual
  equation \eqref{dualeq} and argue just as in the
  derivation of \eqref{K1} to see that
\[|E(\prod_{i=1}^K\xi^\vep
    _T(z_i))-\prod_{i=1}^KE(\xi^\vep _T(z_i))|\le P(V<\infty).\]
The fact that $\xi^\vep_0$ is deterministic makes the independence argument simpler in this setting.  
Now use Lemma~\ref{lem:Vprob} and the separation hypothesis on the $z_k$'s to bound the right-hand side of the above by 
$$c_{\ref{lem:Vprob}}(K,0)e^{c_bT}\vep^{(3/4)(d-2)/(d+3)}\le c_{\ref{lem:Vprob}}(K,0)e^{c_bT}\vep^{1/8}.$$
\end{proof}
\subsection{Proof of Theorem \ref{thm:strongconv}}
\label{ssec:corproof}

\begin{proof} Let $t>0$ and choose $\eta(\vep)\downarrow 0$ so that
$\eta(\vep)/\vep \to \infty$ and $\eta(\vep)/\delta(\vep)\to 0$. Recall $I_\delta(x)$ is the semi-open cube containing $x$ defined prior to Theorem~\ref{thm:strongconv}.
Write $E((\tilde u^\delta(t,x)-u(t,x))^2) =$
$$ =\Bigl(\frac{\vep}{\delta(\vep)}\Bigr)^{2d}\sum_{x_1,x_2\in I_\delta(x)}
E(\xi_t^\vep(x_1)\xi^\vep_t(x_2)-u(t,x)(\xi^\vep_t(x_1)+\xi_t^\vep(x_2))+u(t,x)^2).
$$
The contribution to the above sum from $|x_1-x_2|\le
\eta(\vep)$ is trivially asymptotically small, uniformly in
$x$, as $\vep\to0$.  Theorem~\ref{conv} shows that the
expectation in the above sum goes to zero uniformly in
$x_1,x_2\in I_{\delta}(x),\, |x_1-x_2|\ge \eta(\vep),\, x\in
[-K,K]^d$ as $\vep\to 0$. The result follows.
\end{proof}

\clearpage
  
\section{Achieving low density}
\label{lockill}

The first step in the proof of Theorem~\ref{thm:nonexist} is to use the convergence to the partial differential equation in Theorem~\ref{conv}, and more particularly the estimates in the proof, to get the particle density in \eqref{densdef0} low on a linearly growing region.  

As we will now use the partial differential equation results in Section~\ref{ss:PDE}, we begin by giving the short promised proofs of Propositions~\ref{prop:pde1}
, \ref{prop:pde4} and \ref{prop:pde3}.

\medskip

\noindent{\it Proof of Proposition \ref{prop:pde3}.}
Set $\eta=|r|/3$ and let $L^0_\delta$, $C_0$ and $c_0$ be the constants in
Proposition~\ref{prop:pde2}, and define 
\[
L_\delta=L^0_{\delta}, \quad  c_1=c_0,\quad
t_\delta = L^0_{\delta} \cdot 3\sqrt{d}/|r|, \quad 
C_1= (C_0\vee 1)e^{c_0t_\delta} \,.
\]
Suppose $t\ge t_\delta$, $L\ge L_{\delta}$ and $|x|\le L + (|r|/3)t/\sqrt{d}$. Then we may 
write $x=x_0+y$, where 
\begin{equation}\label{xdecomp}
      |y|\le \frac{2|r|}{3} \frac{t}{\sqrt d}
\quad\text{ and }\quad |x_0|\le L-\frac{|r| t}{3\sqrt d} 
\le L-\frac{|r| t_\delta}{3\sqrt d}  = L-L^0_{\delta}. 
\end{equation}
For $t\ge 0$ and $z\in\Rd$ define $\tilde u(t,z)=u(t,x_0+z)$.  If
$|z|\le L^0_{\delta}$, then $|x_0+z|\le |x_0|+L^0_{\delta}\le L$, which implies that 
$\tilde u(0,z)\le \rho-\delta$.  
Applying Proposition~\ref{prop:pde2} to $\tilde u$, and recalling the bound on $|y|$ in
\eqref{xdecomp}, which implies $|y|_2 \le \frac{2|r|}{3}t$ we have that for $t\ge t_\delta$,
and $|x|\le L + (|r|/3)t/\sqrt{d}$
$$
u(t,x)=\tilde u(t,y)\le C_0 e^{-c_0 t} \le C_1 e^{-c_1 t} \,.
$$
Since the right-hand side above is at least 1 if $t\le t_\delta$, 
the above bound follows for all $t\ge 0$, and we have proved the result with $w=|r|/6\sqrt{d}$. \qed

\medskip

\noindent {\it Proof of Proposition~\ref{prop:pde4}.} Extend $f|_{[0,1]}$ to a smooth function $\tilde f$ on $[0,1+\delta_0]$ so that $\tilde f>0$ on $(1,1+\delta_0)$, $\tilde f(1+\delta_0)=0$, $\tilde f'(1+\delta_0)<0$ and $\int_0^{1+\delta_0}\tilde f(u)du<0$.  The situation is now as in Proposition~\ref{prop:pde3} with $0$, $1$ and $1+\delta_0$ playing the roles of $0$, $\rho$ and $1$. As the solutions take values in $[0,1]$ the extension will not affect the solutions and the Theorem follows from Proposition~\ref{prop:pde3}. \qed

\medskip
\noindent{\it Proof of Proposition~\ref{prop:pde1}.}  
%We may assume that $f(\alpha)=0$ by increasing $\alpha$ to the smallest positive zero of $f$.  Assume first that $f'(\alpha)<0$.  
The version of Proposition~\ref{prop:pde4} with the roles of $0$ and $1$ reversed, applied on the interval $(0,\alpha)$ shows there are positive constants $L$, $c$, and $C$ so that if $u(0,x)\ge \alpha/2$ for $|x|\le L$, then 
\[u(t,x)\ge \alpha-Ce^{-c t}\hbox{ for }|x|\le L+2wt.\]
It is here that we need $f'(\alpha)<0$, corresponding to $f'(0)<0$ in Proposition~\ref{prop:pde4}.  By Theorem~3.1 of Aronson and Weinberger \cite{AW78} there is a $T_0$ so that 
\[u(T,x)\ge \alpha/2\hbox{ for }|x|\le L \hbox{ and }T\ge T_0.\]
Therefore we have 
\[u(t+T_0,x)\ge \alpha-Ce^{-ct}\hbox{ for }|x|\le L+2w(t+T_0)-2wT_0,\]
and so for $t\ge 2T_0$,
\[u(t,x)\ge \alpha-Ce^{cT_0}e^{-ct}\hbox{ for }|x|\le L+wt.\]
The result follows as we may replace $w$ with $2w$.
% $f'(\alpha)$ is not strictly negative it must be zero.  In this case for any $0<\underline \alpha<\alpha$ it is easy to construct a $C^1$ function $\underline f\le f$ on $[0,1]$ such that $\underline f>0$ on $(0,\underline\alpha)$, $\underline f(0)=\underline f(\underline \alpha)=0$, $\underline f'(0)>0$ and $\underline f'(\underline\alpha)<0$.  By the comparison principle (Theorem~2.1 of \cite{AW78}) if $\underline u$ is the solution of \eqref{rdpde} with $\underline f$ in place of $f$ and the same initial data as $u$, then $\underline u(t,x)\le u(t,x)$.  Applying the above to $\underline u$ we get the required lower bound with $\underline \alpha$ in place of $\alpha$. Now let $\underline\alpha\uparrow\alpha$ to complete the proof.
\qed

\medskip

Recall the parameter $r\in(0,1)$, and definitions of $a_\vep$, $t_\vep$, $Q_\vep$, and $D(x,\xi)$ in
\eqref{densdef0}. We first show the density $D(x,\xi^\vep_T)$ is close to its mean.

\begin{lem}\label{lem:meanD} Let $T>0$ and assume $\xi_0^\vep$ is deterministic. 

\noindent(a) If $0<r<\frac{5}{24}$, then  for all $x\in a_\vep\Z^d$, 
\[E((D(x,\xi^\vep_T)-E(D(x,\xi^\vep_T)))^2)\le C_{\ref{lem:meanD}}e^{c_bT}\vep^{1/8}.\]

\noindent(b) If $0<r\le 1/(16d)$ and $C=y+[-L,L]^d$ for $y\in \R^d$, then for all $\eta>0$,
$$P\Bigl(\sup_{x\in C\cap a_\vep\Z^d}|D(x,\xi^\vep_T)-E(D(x,\xi^\vep_T))|\ge \eta\Bigl)\le C_{\ref{lem:meanD}}\vep^{1/16}L^d e^{c_bT}\eta^{-2}.$$
\end{lem}
\begin{proof} (a)  Note that 
\begin{align}\label{zcount}
 |&\{(z_1,z_2)\in(x+Q_\vep)^{2}:\,|z_1-z_2|\le\vep^{1/4}\}|\\ 
    \nonumber&\le (2\vep^{-3/4}+1)^d|Q_\vep|\le
    c_d|Q_\vep|^{2}(\vep^{{1\over 4}-r})^d.
  \end{align}
If $\Sigma^x_z$ denotes the sum over
\begin{equation}\label{zdom}
    z\in\{(z_1, z_2)\in
    (x+Q_\vep)^{2}:\,|z_1-z_2|>
    \vep^{1/4}\}, 
\end{equation}
then by \eqref{zcount} and Proposition~\ref{prop:xiind} with $K=2$,
\begin{align}\nonumber
    E&((D(x,\xi^\vep_T)-E(D(x,\xi^\vep_T)))^{2})\\
  \nonumber  &\le
    |Q_\vep|^{-2}\Bigl[c_d|Q_\vep|^{2}(\vep^{{1\over
        4}-r})^d+\hbox{$\sum_z^x$}[E(\prod_{k=1}^2\xi_T^\vep(z_k))-\prod_{k=1}^2E(\xi_T^\vep(z_k))]\Bigr]\\  
 \nonumber  &\le c_d\vep^{({1\over
        4}-r)d}+4c_{\ref{prop:xiind}}(2)e^{c_bT}\vep^{1/8}\\
 \label{Dmom}        &\le C_{\ref{lem:meanD}}e^{c_bT}\vep^{1/8},
\end{align}
where our condition on $r$ is used in the last line.

\noindent(b) Note that 
$$| C\cap a_\vep\Z^d|\le c_dL^da_\vep^{-d}\le c_dL^d\vep^{-rd}\le c_dL^d\vep^{-1/16}.$$
The result now follows from (a) and Chebychev's inequality.
\end{proof}

We recall  the following hypothesis from Section~\ref{gensys}:

\mn
{\bf Assumption \ref{a2}.}
There are constants $0<u_1<1$, $c_2, C_2, w>0$, $L_0\ge 3$ 
so that for all $L\ge L_0$, if $u(0,x)\le u_1$ for $|x|\le L$ then for all $t\ge 0$
$$
u(t,x)\le C_2 e^{-c_2 t}\hbox{ for all $|x|\le L+2wt$.} 
$$ 
\medskip

We also recall the following condition from  the same Section: For some $r_0>0$, 

\beq\label{grate2}
\sum_{i=0}^1\Vert g_i^\vep-g_i\Vert_\infty\le c_{\ref{grate}}\vep^{r_0}.
\eeq

  We say that $\xi\in\{0,1\}^{\vep\Z^d}$ has density at most $\kappa$ (respectively, in $[\kappa_1,\kappa_2]$) on $A\subset\R^d$ iff $D(x,\xi)\le \kappa$ (respectively $D(x,\xi)\in[\kappa_1,\kappa_2]$) for all $x\in(a_\vep\Z^d)\cap A$.  We set (recall \eqref{betaeps})
\beq \label{Tdefn}r=\frac{1}{16d},\hbox{ hence }\beta=\frac{1.9}{16d},\,t_\vep=\vep^{1.9/(16d)}, T=A_{\ref{lem:gdlow}}\log(1/\vep),\hbox{ and }S=T-t_\vep,\eeq
where $A_{\ref{lem:gdlow}}=c_b^{-1}\Bigl(\frac{1}{100d}\wedge\frac{r_0}{4}\Bigr)$.

\begin{lem} \label{lem:gdlow} Suppose Assumption~\ref{a2} and \eqref{grate2} hold. 
Let $u_2\in(0,u_1)$ and \break 
$\gamma_{\ref{lem:gdlow}}=\Bigl(\frac{c_2}{c_b}\wedge 1\Bigr)\Bigl(\frac{1}{120d}\wedge\frac{r_0}{5}\Bigr)$. There is an $\vep_{\ref{lem:gdlow}}>0$, depending on $(u_1,u_2,w,c_2,C_2)$ and satisfying
\beq\label{veplemgdlow} \vep_{\ref{lem:gdlow}}^{\gamma_{\ref{lem:gdlow}}}\le u_2,
\eeq 
so that if 
$0<\vep\le \vep_{\ref{lem:gdlow}}$ and $2+L_0\le L\le \vep^{-.001/d}$,  then
whenever $\xi_0^\vep$ has density at most $u_2$ in $[-L,L]^d$, 
\begin{equation*}
P(\xi_T^\vep\hbox{ has density at most }\vep^{\gamma_{\ref{lem:gdlow}}}
\hbox{ in }[-L-wT,L+wT]^d|\xi_0^\vep)
\ge 1 - \vep^{.05}.
\end{equation*}
\end{lem}
\noindent
Note that \eqref{veplemgdlow} allows us to iterate this result and obtain the conclusion on successively larger spatial regions at multiples of $T$.

The proof of the above Lemma will require some preliminary lemmas.

\begin{lem}\label{LCLTI} If $p^\vep_t(y)=\vep^{-d}P(B^\vep_t=y)$, $y\in\vep\Z^d$, then for $0<\vep\le 1$, 
\[|p^\vep_t(x)-p^\vep_t(x+y)|\le c_{\ref{LCLTI}}|y|t^{-(d+1)/2}\hbox{ for all }x,y\in\vep\Z^d\hbox{ and }t>0.\]
\end{lem}
\begin{proof} This is a standard local central limit theorem; for $d=2$ this is Lemma 2.1 of \cite{CMP} and  the same proof applies in higher dimensions. 
\end{proof}

Recall (from \eqref{rwsgroup}) that $P^\vep_t$ is the semigroup associated with $B^\vep$. 

\begin{lem}\label{lem:GLCLT} There is a $c_{\ref{lem:GLCLT}}$ such that if $1>\alpha>\beta/2$, then for $0<\vep\le 1$,
\begin{align*}|P^\vep_{t_\vep}\xi(x)-P^\vep_{t_\vep}\xi(x')|\le c_{\ref{lem:GLCLT}}\vep^{(2\alpha-\beta)/(2+d)}
\end{align*}
for all $x,x'\in\vep\Z^d$ such that $|x-x'|\le 2\vep^\alpha$ and all $\xi\in\{0,1\}^{\vep\Z^d}$.
\end{lem}
\begin{proof} Let $-\infty<\delta\le \alpha$, $\Delta=x-x'$ and assume $|\Delta|\le 2\vep^\alpha$.  Apply Lemma~\ref{LCLTI} to see that
\begin{align*}
|&P^\vep_{t_\vep}\xi(x)-P^\vep_{t_\vep}\xi(x')|\\
&\le \sum_{z\in\vep\Z^d}|P(B^\vep_{t_\vep}=z)-P(B^\vep_{t_\vep}=z+\Delta)|\\
&\le \sum_{|z|\le 3\vep^\delta}c_{\ref{LCLTI}}\vep^d|\Delta|t_\vep^{-(d+1)/2}+P(|B^\vep_{t_\vep}|>3\vep^\delta)+P(|B^\vep_{t_\vep}|\ge 3\vep^\delta-\Delta)\\
&\le c\vep^{(\delta-1)d}\vep^{d+\alpha}\vep^{-\beta(d+1)/2}+2P(|B^\vep_{t_\vep}|>\vep^\delta).
\end{align*}
If we use Chebychev to bound the last summand by $ct_\vep
\vep^{-2\delta}=c\vep^{\beta-2\delta}$ and optimize over
$\delta$ (setting
$\delta=\frac{\beta}{2}-\frac{\alpha-(\beta/2)}{2+d}<\frac{\beta}{2}<\alpha$),
we obtain the required upper bound.
\end{proof}
\begin{lem}\label{lem:ICUD}  For any $\eta>0$ there is an
  $\vep_{\tref{lem:ICUD}}(\eta)>0$ so that if $0<\vep\le
  \vep_{\tref{lem:ICUD}}$, $u\in[0,1], \ L>1 $, and
$\xi\in\{0,1\}^{\vep \Zd}$ has density at most $u$ in
$[-L,L)^d$, then
\begin{equation}
P^\vep_{t_\vep}\xi(x)\le u+\eta \text{ for all } x\in [-L+1,L-1]^d\cap\vep\Z^d \,.
\end{equation}
\end{lem}
\begin{proof} By translation invariance it suffices to prove that for small enough $\vep>0$ 
and all $x\in[-a_\vep,a_\vep]^d\cap\vep Z^d$, if $\xi$ has density at most $u$
in $[-1,1)^d$ then $P^\vep_{t_\vep}\xi(x)\le u+\eta$.
(This addresses the uniformity in $L$.) 
Argue as in the upper bound in \eqref{densebound} to see that for $\vep<\vep_0(\eta)$,
\[\frac{|Q_\vep|P(B_{t_\vep}=z+e)}{P(B_{t_\vep}^{\vep,x}\in z+Q_\vep)}\le 1+\frac{\eta}{2}\hbox{ for all }z\in a_\vep\Z^d,\ |z-x|\le 1\hbox{ and }e\in Q_\vep.\]
We therefore have
\begin{align*}
P_{t_\vep}^\vep\xi(x)&\le P(|B_{t_\vep}^{x,\vep}|\ge 1/2)\\
&+\sum_{z\in a_\vep\Z^d}1(|z-x|\le 3/4)\sum_{e\in Q_\vep} \xi(z+e) \frac{1+(\eta/2)}{|Q_\vep|}P(B_{t_\vep}^{\vep,x}\in z+Q_\vep)\\
&\le 4\sigma^2dt_\vep+\sum_{z\in a_\vep\Z^d}1(|z-x|\le 3/4)u(1+\frac{\eta}{2})P(B_{t_\vep}^{\vep,x}\in z+Q_\vep)\\
&\le 4\sigma^2dt_\vep+u+\frac{\eta}{2}\le u+\eta,
\end{align*}
for $\vep<\vep_1(\eta)$.
\end{proof}

We are ready for the Proof of Lemma~\ref{lem:gdlow}.

\begin{proof} By the conditioning we may fix a deterministic $\xi^\vep_0$ as in the statement of Lemma~\ref{lem:gdlow}. In light of Lemma~\ref{lem:meanD} our first and main goal is to bound $E(D(x,\xi^\vep_T))$ for a fixed $x\in a_\vep\Z^d\cup[-L-wT,L+wT]^d$.  Let $z\in x+Q_\vep$.  
Let $\tilde X^{z,T}$ be the modification of the dual $X^{z,T}$, starting with a
single particle at $z$, in which particles ignore branching
and coalescing events on $[S,T]$ by following their own
random walk and switching to independent random walk mechanisms when a
collision between particles occurs.  Hence 
$X^i_S=\tilde X^i_S$ for all $i\in J(S)$, and 
on $[S,T]$ the particles in $\tilde X^{z,T}$ follow independent copies of
$B^\vep$.  Let $\tilde\zeta^\vep$ be the
associated computation process, defined just as $\zeta^\vep$
is for $X^{z,T}$, with initial values $\tilde
\zeta^\vep_0(j)=\xi^\vep_0(\tilde X^j_T),j\in J(S)$ (for
$\tilde X^{z,T}$ the index set is constant on $[S,T]$).
On $\tilde G^\beta_T$, $T\notin\cup_{m=0}^{N_T-1}[R_m,R_m+2\vep^\beta]$, with $\beta<1/2$, and so
\beq\label{noST} [S,T]\cap(\cup_{m=0}^{N_T-1}[R_m,R_m+\sqrt\vep])=\emptyset.
\eeq
Therefore on $\tilde G^\beta_T$, $X^{z,T}$ has no branching or coalescing events on $[S,T]$, and so $\tilde X^{z,T}=X^{z,T}$ on $[0,T]$.  This also means (by \eqref{dualityeq}) that, given the common inputs $\zeta^\vep_0(j)=\tilde\zeta_0^\vep(j)=\xi_0^\vep(X^j_T)$, $j\in J(T)=J(S)$ we have 
\beq\label{tildezetaeq} \tilde\zeta_{T}^\vep(0)=\zeta^\vep_{T}(0)=\xi_T^\vep(z)\hbox{  on $\tilde G^\beta_T$}.
\eeq
Let $\psi_\vep(x)=P^\vep_{t_\vep}\xi^\vep_0(x)$. Conditional on $\F_S$, $\{\tilde X_T^j-\tilde X_S^j:j\in \hat J(S)\}$ are iid with law $P_0(B^\vep_{t_\vep}\in\cdot)$, and so, conditional on $\F_S$, 
$\{\tilde\zeta^\vep_{T-S}(j)=\xi^\vep_0(\tilde X^j_T):j\in J(S)\}$ 
are independent Bernoulli rv's with means 
$\{\psi_\vep(X_S^j):j\in J(S)\}$.  Recall $\{W_j\}$ is an iid sequence of uniform $[0,1]$ rv's independent of $\F_\infty$ (that is of our graphical construction).  Let $\{\tilde\zeta^{\vep,*}_t(j):j\in J(T-t), T-S\le t\le T\}$ be the computation process associated with $\tilde X^{z,T}$ but with initial inputs $\tilde\zeta_{T-S}^{\vep,*}(j)=1(W_j\le \psi_\vep(X_S^j):j\in J(S)\}$. Then $\{\tilde\zeta^\vep_t:T-S\le t\le T\}$ and $\{\tilde\zeta^{\vep,*}_t:T-S\le t\le T\}$ have the same law because the joint law of their Bernoulli inputs and the processes $\tilde X^{z,T}_t, t\le S$ and $((\mu_m,U_m)1(R_m\le t), t\le S)$ used to define them are the same. Therefore by \eqref{tildezetaeq}
\begin{align}\nonumber|P&(\xi_T^\vep(z)=1)-P(\tilde\zeta^{\vep,*}_T(0)=1)|\\
\label{tildestar}&=|P(\xi_T^\vep(z)=1)-P(\tilde\zeta^{\vep}_T(0)=1)|\le P((\tilde G_T^\beta)^c).
\end{align}

Consider now the branching random walk $\hat X$ starting with a single particle at $z$ and coupled with $X^{z,T}$ as in Section~\ref{ssec:XhatXcoupling}, together with its computation process $\{{\hat \zeta}^\vep:t\in[T-S,T]\}$ with initial inputs $\hat\zeta_{T-S}(j)=1(W_j\le \psi_\vep(\hat X_S^j)), j\in\hat J(S)$.  
Conditional on $\F_\infty$, these inputs are independent Bernoulli rv's with means $\{\psi(\hat X^j_S):j\in \hat J(S)\}$.  The computation processes $\tilde\zeta^{\vep,*}$ and $\zeta^\vep$ are identical on $[T-S,T]$ if given the same inputs at time $T-S$.  Therefore Lemma~\ref{lem:compeq} shows that on $\tilde G_T^\beta$ $\hat\zeta^\vep_T(0)$ and $\tilde \zeta^{\vep,*}_T(0)$ will coincide if given the same inputs at time $T-S$.  Therefore
\begin{align}
\nonumber|P&(\hat\zeta_T(0)=1)-P(\tilde\zeta^{\vep,*}_T(0)=1)|\\
\nonumber &\le P((\tilde G^\beta_T)^c)+E(P(\hat\zeta^\vep_{T-S}(j)\neq \tilde\zeta^{\vep,*}_{T-s}(j)\ \exists j\in\hat J(S)|\F_\infty)1(\tilde G_T^\beta))\\
\label{hatzetaeq}&\le P((\tilde G^\beta_T)^c)+E\Bigl(\sum_{j\in \hat J(S)}|\psi_\vep(\hat X^j_S)-\psi_\vep(X_S^j)|1(\sup_{j\in\hat J(S)}|X^j_S-\hat X^j_S|\le \vep^{1/6})\Bigr).
\end{align}
Use Lemma~\ref{lem:tildeG} to bound the first term above and Lemma~\ref{lem:GLCLT} with $\alpha=1/6>\beta/2$ to bound the second, and combine this with \eqref{tildestar} to conclude that (use $d\ge 3$)
\begin{align}
\nonumber|P(\xi^\vep_T(z)=1)-P(\hat\zeta^\vep_T(0)=1)|
&\le 2 c_{\ref{lem:tildeG}}e^{c_bT}\vep^{\beta/3}+E(|\hat J(S)|)c_{\ref{lem:GLCLT}}\vep^{(1/3-\beta)/(2+d)})\\
\nonumber &\le 2c_{\ref{lem:tildeG}}\vep^{1/(40d)}+e^{c_bS}c_{\ref{lem:GLCLT}}\vep^{(1/3-\beta)/(2+d)}\\
\label{ICcoupleI}&\le (2c_{\ref{lem:tildeG}}+c_{\ref{lem:GLCLT}})\vep^{1/(40d)}.
\end{align}

To prepare with the coupling with the branching Browian motion we must extend $\psi_\vep(x)=P^\vep_{t_\vep}\xi_0^\vep(x)$ from $\vep\Z^d$ to $\R^d$ in an appropriate manner.  
%Let $\phi_{L}(x)=\tilde \phi_L(|x|)$ where $\tilde\phi_L(y)=u_1$ for $0\le y\le L$, $\tilde \phi_L(y)=1$ for $y\ge L+1$ and $\tilde\phi_L$ is linear on $[L,L+1]$.  
Since $\xi^\vep_0$ has density at most $u_2$ in $[-L,L]^d$, Lemma~\ref{lem:ICUD} shows that for $\vep\le \vep_{\ref{lem:ICUD}}(u_1-u_2)$ we may extend $\psi_\vep$ in a piecewise linear manner so that
\beq\label{psiep1}
\psi_\vep(x)\le u_11(|x|\le L-2)+1(|x|>L-2) \hbox{ for all }x\in\R^d.
\eeq
In addition, using Lemma~\ref{lem:GLCLT} with $\alpha=1/6$, we may assume the above extension also satisfies
\beq\label{psiep2}
|\psi_\vep(x)-\psi_\vep(x')|\le c_{\ref{lem:GLCLT}}\vep^{\frac{(1/3)-\beta}{2+d}}\hbox{ for }x,x'\in\R^d \hbox{ such that }|x-x'|\le \vep^{1/6}.
\eeq

Now consider the branching Brownian motion $\hat X^0$ starting with a single particle at $z$, coupled with $\hat X^\vep$ as in Section~\ref{sec:hydroproofs}.  Consider also its associated computation process $\hat\zeta^0$ on $[T-S,T]$ starting with conditionally independent Bernoulli inputs $\{\hat \zeta_{T-S}^0(j)=1(W_j\le \psi_\vep(\hat X^{0,j}_S)):j\in \hat J^0(S)\}$.  We may argue as in the derivation of \eqref{ICcoupleI}, but now using Lemma~\ref{poscoupl}, \eqref{grate2}, and \eqref{psiep2} in place of Lemmas~\ref{lem:tildeG} and \ref{lem:GLCLT}, to conclude after some arithmetic using $d\ge 3$,
\begin{equation}\label{ICcoupleII}|P(\hat \zeta_T^\vep(0)=1)-P(\hat \zeta^0_T(0)=1)|\le e^{c_bT}c_{\ref{poscoupl}}(\vep^{3/8}+c_{\ref{grate}}\vep^{r_0/2})+c_{\ref{lem:GLCLT}}\vep^{1/(40d)}.
\end{equation}
By Lemma~\ref{lem:ident} (we have shifted time by $T-S$), $P(\hat \zeta^0_T(0)=1)=u_\vep(S,z)$, where $u_\vep$ is the solution of the PDE \eqref{rdpde} with initial condition $u_\vep(0,\cdot)=\psi_\vep$.  Now combine this with \eqref{ICcoupleI} and \eqref{ICcoupleII} to see that for small enough $\vep$ as above
\beq\label{DEbound} 
E(\xi^\vep_T(z))\le c\vep^{(1/(40d))\wedge(r_0/4)}+u_\vep(S,z).
\eeq
Now use the bound on the initial condition \eqref{psiep1} and Assumption~\ref{a2} in the above to 
conclude that for $|z|\le L-2+2wS$, and small $\vep$
\beq\label{meanxibnd}
E(\xi^\vep_T(z))\le c\vep^{(1/(40d))\wedge(r_0/4)}+C_2e^{-c_2 S}\le \vep^{\gamma'},
\eeq
where $\gamma'=\Bigl(\frac{c_2}{c_b}\wedge 1\Bigr)\Bigl(\frac{1}{110d}\wedge\frac{r_0}{4}\Bigr)>\gamma_{\ref{lem:gdlow}}$ and we used the definition of $S$ and some arithmetic.
By taking $\vep$ smaller if necessary we may assume $2wS-3\ge wT$ and so the above holds for $|z|\le L+1+wT$.  This shows that
\beq\label{meanDbound} E(D(x,\xi^\vep_T))\le \vep^{\gamma'}\hbox{ for }x\in a_\vep\Z^d\cap[-L-wT,L+wT]^d.
\eeq

Finally apply the above and Lemma~\ref{lem:meanD} to conclude that for small enough $\vep$
\begin{align*}
P&\Bigl(\sup_{x\in [-L-wT,L+wT]^d\cap a_\vep\Z^d}D(x,\xi^\vep_T)\ge \vep^{\gamma_{\ref{lem:gdlow}}}\Bigr)\\
&\le P(\sup_{x\in [-L-wT,L+wT]^d\cap a_\vep\Z^d}|D(x,\xi^\vep_T)-E(D(x,\xi^\vep_T))|\ge \vep^{\gamma_{\ref{lem:gdlow}}}/2)\\
&\le C_{\ref{lem:meanD}}\vep^{1/16}(L+wT)^de^{c_bT}4\vep^{-2{\gamma_{\ref{lem:gdlow}}}}\\
&\le C\vep^{1/16}\vep^{-.001}\vep^{-1/100d}\vep^{-1/120d}\le \vep^{.05}.
\end{align*}
\end{proof}

Our next goal is to show that the dual process only expands linearly in time. The first ingredient is
a large deviations result. Recall the
dominating branching random walk $\{\bar
X^{\vep,j}(t):j\in\bar J^\vep(t)\}$ introduced at the beginning of
Section~\ref{ssec:pbad} which satisfies
\[\{X_s^{\vep,j}:j\in J^\vep(s)\}\subset\{\bar X_s^{\vep,j}:j\in \bar J^\vep(s)\}.\]
If $\Vert X^\vep_s\Vert_\infty=\sup\{|X^{\vep,j}(s)|:j\in J^\vep(s)\}$ and similarly for $\Vert \bar X^\vep_s\Vert_\infty$, then the above domination implies
\begin{equation}\label{normdom} \Vert \bar X^\vep_s\Vert_\infty\ge \Vert X^\vep_s\Vert_\infty\hbox{ for all }s,\vep.
\end{equation}
Recall $c^*$ is as in \eqref{c*def}.

\begin{lem} \label{lem:ldbd}
Assume $\bar X^\vep$ starts from one particle at $0$.  For each $R>0$ there is an $\vep_{\ref{lem:ldbd}}(c^*,R)>0$, nonincreasing in each variable, so that for $0<\vep\le \vep_{\ref{lem:ldbd}}$ and $t>0$, 
$$P( \| \bar X^\vep_s \|_\infty \ge 2\rho t \hbox{ for some $s\le t$}) 
\le (4d+1)\exp(- t(\gamma(\rho) - c_b))\quad\hbox{for all }0<\rho\le R,
$$
where $\gamma(\rho) = \min \{\rho/2, \rho^2/3\sigma^2\}$. 
Moreover, if $\rho \ge \max\{4c_b, 2\sigma^2 \}$, then the above bound is at most $(4d+1)\exp(-t\rho/4)$.\end{lem}

\begin{proof} The last assertion is trivial.  Let $S^\vep_t$ be a random walk that starts at 0, jumps according to $p_\vep$ at rate $\vep^{-2}$,
and according to $q_\vep$ at rate $c^*$. Since $E|\bar X^\vep_t| = \exp(c_b t)$ by summing over the branches of the tree, it suffices to show
\beq
P( \| S^\vep_s \|_\infty \ge \rho t \hbox{ for some $s\le t$}) 
\le (4d+1)\exp\left( - \gamma(\rho)t \right).
\label{ldrw}
\eeq 
As usual, $B^\vep_t$ is the random walk that jumps according to $p_\vep$ at rate $\vep^{-2}$.
By the reflection principle
$$
P\left( \sup_{s\le t} B^{\vep,i}_s \ge \rho t \right)
\le 2P\left( B^{\vep,i}_t \ge \rho t \right)
\le 2 e^{-\theta\rho t} E\left(\exp\left( \theta  B^{\vep,i}_t \right)\right)
$$
for any $\theta>0$. If $\phi_\vep(\theta) = \sum_x e^{\theta x^i} p_\vep(x)$ then a standard Poisson calculation gives
$$
E\left(\exp\left(\theta  B^{\vep,i}_t \right)\right) = \exp( t\vep^{-2} (\phi_\vep(\theta)-1)).
$$
By scaling $\phi_\vep(\theta) = \phi_1(\vep\theta)$. Our assumptions imply $\phi_1'(0)=0$
and $\phi_1''(0) = \sigma^2$ so
$$
\vep^{-2} (\phi_1(\vep\theta) - 1 ) \to \sigma^2\theta^2/2\hbox{ as }\vep\to 0.
$$
If $0<\rho\le R$ and $\theta = \rho/\sigma^2$ in the above, it follows that for $\vep<\vep_0(R)$, 
$$
e^{-(\rho^2/\sigma^2) t} E\left(\exp\left( (\rho/\sigma)  B^{\vep,i}_t \right)\right)
\le \exp( -\rho^2t/3\sigma^2 ),
$$
and so,
\beq\label{brbnd1} P(\sup_{s\le t}|B^{\vep,i}_s|\ge \rho t)\le 4\exp\left(-\rho^2t/3\sigma^2\right).
\eeq

Let $J^\vep_t$ be the one dimensional random walk that jumps according to the law of $\vep Y^*=\vep\max_{i\le N_0}|Y^i|$ at rate $c^*$,
and notice that this will bound the $L^\infty$ norm of the sum of the absolute values of the 
jumps according to $q_\vep$ in $S_\vep$ up to time $t$. If we let $\phi_J(\theta) = E(\exp(\theta Y^*))$, then, arguing as above, we obtain
$$
P( J^\vep_t \ge \rho t ) \le \exp( - \rho \theta t + c^*t (\phi_J(\vep\theta) - 1)).
$$
The exponential tail of $Y^*$ (from \eqref{expbd2}) shows that $(\phi_J(\vep\theta) -1 )/\vep\theta \to EY^*$ as $\vep \to 0$, and so, if we set $\theta=1$,
then for small $\vep$, $c^* (\phi_J(\vep) - 1) \le \rho/2$.   (The choice of $\vep$ here works for all $\rho$ because we may assume without loss of generality that $\rho\ge \rho_0>0$ as the Lemma is trivial for small $\rho$.) Therefore 
\beq\label{brbnd2} P(J^\vep_t\ge \rho t)\le \exp(-\rho t/2).\eeq
 To derive \eqref{ldrw}, write
 $$P(\sup_{s\le t}\Vert S_s^\vep\Vert_\infty\ge 2\rho t)\le \sum_{i=1}^dP(\sup_{s\le t}|B^{\vep, i}_s|\ge \rho t)+P(J^\vep_t\ge \rho t),$$
 and use \eqref{brbnd1} and \eqref{brbnd2}.
\end{proof}  

Our next result uses the large deviation bound in Lemma \ref{lem:ldbd} to control the movement
of all the duals that start in a region. Recall that $X^{x,U}$ is the dual for $\xi^\vep$ starting with
one particle at $x$ from time $U$. For $x\in\R^d$ and $r>0$
let $Q(x,r)=[x-r,x+r]^d$ and $Q^\vep(x,r) = 
Q(x,r)\cap\vep\Zd$. Write $Q(r)$ for $Q(0,r)$ and $Q^\vep(r)$
for $Q^\vep(0,r)$. 

\mn
\begin{lem} \label{lem:contain2} For $c>0$,
  $b\ge 4c_d\vee 2\sigma^2$, $L\ge 1$ and
  $U\ge T'=c\log (1/\vep)$, let
\begin{align*}\bar p_\vep(b,c,L,U)=P(&X_t^{x,u}\hbox{ is not contained in }Q(L+2bT')\\
&\hbox{ for some }u\in[U-T',U], t\le T'\hbox{ and some }x\in Q^\vep(L)).
\end{align*}
Let $c'_d=12(4d+1)3^d$. There exists
$\vep_{\ref{lem:ldbd}}(c^*,b)>0$ such that if  $0<\vep\le \vep_{\ref{lem:ldbd}}$
\[
\bar p_\vep(b,c,L,U)\le
c'_dL^d(c\log(1/\vep)+1)\vep^{q-d}
\]
where $q=(\frac{bc}{4} -2 )\wedge \vep^{-2}$.
\end{lem} 
\begin{proof} By translation invariance it suffices to take
  $U=T'$. 
  For $x\in\vep\Z^d$ let $\{T_i(x):i\ge 0\}$ be the
  successive jump times of the reversed Poisson process,
  starting at time $T'$, determined by the
  $T^x_n,T^{*,x}_n$.  Also let $N_x$ be the number of such
  jumps up to time $T'$, so that $N_x$ is Poisson with mean
  $(c^*+\vep^{-2})T'$.  The process $\xi_t^\vep(x)$ is constant for
  $t\in(T'-T_{i+1}(x),T'-T_{i}(x)]$ and for such $t$ the dual
  $X^{x,t}(v)$ is $$X^{x,T-T_i(x)}(v+(T'-T_i(x)-t)),$$ that
  is, one is a simple translation of the other.  This means
  for $t$ as above
\begin{equation}\label{Xinter} \cup _{v\le
t}X^{x,t}(v)\subset \cup_{v\le T'-T_i(x)}X^{x,T'-T_i(x)}(v),
\end{equation} (in fact equality clearly holds).  As a
result, in $\bar p_{\vep}(b,c,L,T') $ we only need consider
$t$ to be one of 
the times $T'-T_i(x)$ for $0\le i\le N_x$ and we may bound
$1-\bar p_{\vep}(b,c,L,T')$ by
\begin{align*} &P(\exists x\in Q^\vep(L) \hbox{ s.t. }N_x\ge
  3T'(\vep^{-2}+c^*))\\ &\qquad+P(\exists x\in
  Q^\vep(L),\ 0\le T_i(x)\le 3T'(\vep^{-2}+c^*)\hbox{ s.t. }\\
  &\qquad\qquad\sup_{v\le T'-T_i(x)}\Vert
  X^{x,T'-T_i(x)}(v)\Vert_\infty>2bT')\\ &\le
  (2L\vep^{-1}+1)^d\exp\{-3T'(\vep^{-2}+c^*)\}E(e^{N_x})\\
  &\qquad+(2L\vep^{-1}+1)^d(3T'(\vep^{-2}+c^*)+1)(4d+1)
  \exp(-T'b/4).
\end{align*} 

\noindent
Here we are using Lemma~\ref{lem:ldbd} and
the strong Markov property at $T_i(x)$ for the filtration
generated by the reversed Poisson processes $\cF_t$. Some arithmetic
shows the above is at most
\begin{align*} & 3^d(L\vee \vep)^d\vep^{-d}\Bigl[
\exp(-3T'(\vep^{-2}+c^*))\exp((\vep^{-2}+c^*)T'(e-1))\\
&\phantom{ 3^d(L\vee \vep)^d\vep^{-d}\Bigl[}
+ (4d+1)(3T'(\vep^{-2}+c^*)+1)\vep^{bc/4}\Bigr]\\   &\le
3^d(L\vee \vep)^d\vep^{-d}\Bigl[ 
\exp(-T'(\vep^{-2}+c^*)) 
+ 6(4d+1)(c\log(1/\vep) \vep^{-2}+1)
\vep^{bc/4}\Bigr]\\   
&\le 3^d(L\vee \vep)^d\vep^{-d}\Bigl[ \vep^{(c\vep^{-2})} +
6(4d+1)(c\log(1/\vep)+1)\vep^{bc/4-2}\Bigr]\\
&\le c'_d (L\vee \vep)^d(c\log(1/\vep)+1)
 \vep^{-d} \vep^{(bc/4-2)\wedge \vep^{-2}}
\end{align*}
\end{proof} 

\clearpage 

\section{Percolation results}
\label{sec:perc}

To prove Theorems~\ref{thm:exist} and (especially) \ref{thm:nonexist} we will use block arguments that
involve comparison with oriented percolation.  Let $D=d+1$, where for now we allow $d\ge1$,
and let $\A$ be any $D \times D$ matrix so that (i) if $x$
has $x_1+ \cdots + x_D = 1$ then $(\A x)_D=1$, and (ii) if $x$
and $y$ are orthogonal then so are $\A x$ and
$\A y$. Geometrically, we first rotate space to take $(1/D,
\ldots 1/D)$ to $(0, \ldots, 0,1/\sqrt{D})$ and then scale $x
\to x\sqrt{D}$. Let $\LD = \{ \A x : x \in \ZD \}$. The reason
for this choice of lattice is that if we let $\cQ = \{ \A x :
x\in [-1/2,1/2]^D \}$, then the collection $\{z + \cQ, z \in
\LD\}$ is a tiling of space by rotated cubes. When $d=1$,
${\cal L}_2 = \{ (m,n) : m+n \hbox{ is even}\}$ is the usual
lattice for block constructions (see Chapter~4 of
\cite{Dur95}).

Let ${\cal H}_k = \{ z \in {\cal L}_D : z_D
= k \}=\{ \A x:x\in\Z^D,\sum_i x_i=k\}$ be the points on ``level'' $k$. We will often
write the elements of $\cH_k$ in the form $(z,k)$ where
$z\in\Rd$.  Let $\cH'_k=\{z\in\Rd: (z,k)\in\cH_k\}$. When $d=2$, the points in
${\cH'}_0$ are the vertices of a triangulation of the
plane using equilateral triangles, and the points in ${\cH'}_1$ 
are obtained by translation. One choice of $\A$ leads to
Figure~\ref{fig:L3}, where $\cH'_1$ and $\cH'_2$ are obtained by translating $\cH'_0$ upward by $\sqrt 2$ and $2\sqrt 2$, respectively, and $\cH'_3=\cH'_0$.  
\begin{figure}[h]
\begin{center}
\begin{picture}(220,140)
\put(17,68){$\bullet$}
\put(47,118){$\bullet$}
\put(47,18){$\bullet$}
\put(77,68){$\bullet$}
\put(107,118){$\bullet$}
\put(107,18){$\bullet$}
\put(137,68){$\bullet$}
\put(167,118){$\bullet$}
\put(167,18){$\bullet$}
\put(197,68){$\bullet$}
\put(20,70){\line(1,0){180}}
\put(50,20){\line(1,0){120}}
\put(50,120){\line(1,0){120}}
\put(50,20){\line(3,5){60}}
\put(50,120){\line(3,-5){60}}
\put(20,70){\line(3,5){30}}
\put(20,70){\line(3,-5){30}}
\put(170,20){\line(-3,5){60}}
\put(170,120){\line(-3,-5){60}}
\put(200,70){\line(-3,-5){30}}
\put(200,70){\line(-3,5){30}}
%\put(47,84){$\circ$}
\put(47,52){$\circ$}
%\put(107,84){$\circ$}
\put(107,52){$\circ$}
\put(77,102){$\circ$}
%\put(77,34){$\circ$}
%\put(167,84){$\circ$}
\put(167,52){$\circ$}
\put(137,102){$\circ$}
%\put(137,34){$\circ$}
%\put(50,86){\line(0,-1){32}}
%\put(110,86){\line(0,-1){32}}
%\put(50,86){\line(5,3){30}}
%\put(110,86){\line(-5,3){30}}
\put(50,55){\line(3,5){30}}
\put(50,54){\line(1,0){120}}
\put(108,54){\line(3,5){30}}
\put(140,104){\line(3,-5){30}}
\put(80,104){\line(1,0){60}}
\put(80,104){\line(3,-5){30}}
%\put(50,54){\line(5,-3){30}}
%\put(110,54){\line(-5,-3){30}}
\end{picture}
\caption{${\cH'}_0$ (black dots) and ${\cH'}_1$ (white dots)
  in ${\cal L}_3$} 
\label{fig:L3}
\end{center}
\end{figure}
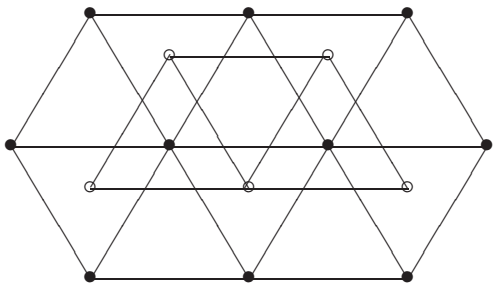

In $d\ge 3$ dimensions (the case we will need for our applications in this work) the lattice is hard to visualize so
we will rely on arithmetic. Let $\{e_1,\dots,e_D\}$ be the
standard basis in $\R^D$, and put $v_i=\A e_i$, $i=1,\dots,D$.
By the geometric description of $\A$ given above,
$v_i\in{\cal H}_1$ has length $\sqrt{D}$, and writing
$v_i=(v'_i,1)$, $v'_i\in\Rd$ has length $\sqrt{D-1}$.  For $i\ne j$, $\|v'_i-
v'_j\|_2=\|v_i- v_j\|_2 =\sqrt{2D}$, the last by orthogonality of $v_i$ and $v_j$.  The definitions easily imply that 
$\cH'_{k+1}=v'_i+\cH'_k\equiv\{v'_i+x:x\in\cH'_k\}$ for each $i$ and $k$. Note that $Dv'_i\in\cH'_0$ because $Dv_i-(0,\dots,0,D)\in\cH_0$. This implies that $\cH'_{k+D}=Dv'_i+\cH'_k=\cH'_k$. 

\begin{figure}[h]
\begin{center}
\begin{picture}(220,140)
\put(17,68){$\bullet$}
\put(47,118){$\bullet$}
\put(47,18){$\bullet$}
\put(77,68){$\bullet$}
\put(107,118){$\bullet$}
\put(107,18){$\bullet$}
\put(137,68){$\bullet$}
\put(167,118){$\bullet$}
\put(167,18){$\bullet$}
\put(197,68){$\bullet$}
\put(20,70){\line(1,0){180}}
\put(50,20){\line(1,0){120}}
\put(50,120){\line(1,0){120}}
\put(50,20){\line(3,5){60}}
\put(50,120){\line(3,-5){60}}
\put(20,70){\line(3,5){30}}
\put(20,70){\line(3,-5){30}}
\put(170,20){\line(-3,5){60}}
\put(170,120){\line(-3,-5){60}}
\put(200,70){\line(-3,-5){30}}
\put(200,70){\line(-3,5){30}}
\put(47,84){$\circ$}
\put(47,52){$\circ$}
\put(107,84){$\circ$}
\put(107,52){$\circ$}
\put(77,102){$\circ$}
\put(77,34){$\circ$}
\put(167,84){$\circ$}
\put(167,52){$\circ$}
\put(137,102){$\circ$}
\put(137,34){$\circ$}
\put(50,86){\line(0,-1){32}}
\put(110,86){\line(0,-1){32}}
\put(50,86){\line(5,3){30}}
\put(110,86){\line(-5,3){30}}
\put(50,54){\line(5,-3){30}}
\put(110,54){\line(-5,-3){30}}
\end{picture}
\caption{${\cH'}_0$ (black dots) and  Voronoi region about $0$ (inside white dots)
  in ${\cal L}_3$} 
\label{fig:L4}
\end{center}
\end{figure}
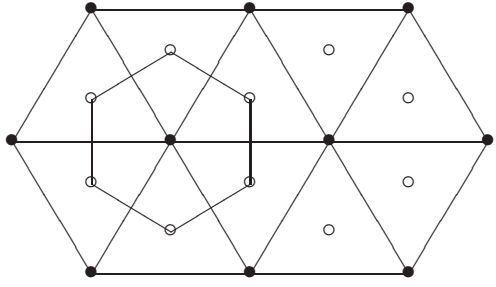
For $x\in\cH'_k$ let
$\cV_x\subset \Rd$ be the Voronoi region for $x$ associated
with the points in $\cH'_k$, i.e., the closed set of points in
$\Rd$ that are closer to $x$ in Euclidean norm  than to all the other points of
${\cH}'_k$ (including ties). 
If $\cV=\cV_0$ (in $d=2$, $\cV_0$ is the hexagon in
Figure \ref{fig:L4} inside the connected six white dots), then the translation invariance of $\cH'_0$ and fact that $\cH'_k=kv'_i+\cH'_0$ show that $\cV_x=x+\cV$ for all $x\in\cup\cH'_k$.  It is
immediate from the definition of Voronoi region that for
each $k$,
\begin{equation}\label{vLd}
\cup_{x\in{\cH}'_k} \cV_x = \Rd.
\end{equation}
Furthermore, $\cV_x$ is contained in the closed ball of
radius $D$ centered at $x$. (To see this we may set $x=k=0$ and transfer the problem to $\Z^D$ via $\A^{-1}$.  It then amounts to noting that if $x\in\R^D$ satisfies $\sum x_i=0$ and $\Vert x\Vert_2>\sqrt D$, then there are $i\neq j$ s.t. $x_i>1,x_j<0$ or $x_i<-1,x_j>0$, and so $\Vert x\pm(e_i-e_j)\Vert_2<\Vert x\Vert_2$.)  From this inclusion we see
that for any $L>0$,
\begin{align} \label{cLd}\text{if }
c_{L} = &L/(2D)\hbox{ then } c_L\cV_x\subset c_Lx+[-L,L]^d,\\
\nonumber&\hbox{ and so }
\cup_{x\in{\cH'_k}} c_{L} x + [-L,L]^d = \Rd .
\end{align} 
The above also holds with $2c_L$ in place of $c_L$ but the above ensures a certain overlap in the union which makes it more robust.
Finally, one can check that for some positive $c_{\tref{vLd1}}(D)$,
\begin{equation}\label{vLd1}
\text{if }a\in\cV_x, b\notin \cV_x\text{ and } |a-b|<
c_{\tref{vLd1}} \text{ then  } b\in \cup_{i\ne j}
\cV_{x+v'_i-v'_j} .
\end{equation}
For this, note that $x+v'_i-v'_j$, $1\le i\neq j\le D$ are the $D(D-1)$ ``neighboring points to $x$" in $\cH'_k$, corresponding to the $6$ black vertices of the hexagonal around $0$ in Figure~\ref{fig:L4} for $x=0$ and $D=3$.  The above states that the $D(D-1)$ corresponding Voronoi regions provide a solid annulus about $\V_x$, as is obvious from Figure~\ref{fig:L4} for $D=3$.

Our oriented percolation process will be constructed from a
family of random variables $\{\eta(z),z\in\LD\}$ taking
values 0 or 1, where 0 means closed and 1 means open.
In the block construction, one
usually assumes that the collection of $\eta(z)$ is ``$M$
dependent with density at least $1-\theta$'' which means that for any $k$,
\begin{align} \label{Mdepdef}P(&
\eta(z_i) = 1|\eta(z_j),j\neq i) \ge (1-\theta),\\
\nonumber&\hbox{ whenever }z_i\in\LD, 1\le
i\le k\hbox{ satisfy $|z_i-z_j|> M$ for all $i\ne j$.}
\end{align}
Our process will satisfy the modified condition
\begin{align}\label{modMdep}
P(&\eta(z_k) = 1|\eta(z_j),j< k) \ge (1-\theta)\hbox{ whenever }z_j=(z'_j,n_j)\in\cL_D,\ 1\le j\le k\\
\nonumber&\hbox{ satisfy }n_j<n_k\hbox{ or }(n_j=n_k\hbox{ and }|z'_j-z'_k|> M)\hbox{ for all }j<k.
\end{align}

It is typically not difficult to prove results for
$M$-dependent percolation processes with $\theta$ small (see
Chapter 4 of \cite{Dur95}), but in Section~\ref{death} we will simplify things
by applying Theorem~1.3 of \cite{LSS} to reduce to the case of independent percolation.  
By that result, under \eqref{Mdepdef}, there is a constant
$\Delta$ depending on $D$ and $M$ such that if

\[
1-\theta' = \Bigr(1- \frac{\theta^{1/\Delta}}{(\Delta-1)^{(\Delta-1)/\Delta}}\Bigl) \bigl(1-(\theta(\Delta-1))^{1/\Delta}\bigr),
\]
we may couple $\{\eta(z),z\in\LD\}$ with a family
$\{\zeta(z),z\in\LD\}$ of iid Bernoulli random variables
with $P(\zeta(z)=1)=1-\theta'$ such that $\zeta(z)\le \eta(z)$ for
all $z\in\LD$. An examination of the proofs of Proposition~1.2 and Theorem 1.3 of \cite{LSS}, shows that the above result remains valid under our condition \eqref{modMdep}.  [In their proof of Theorem~1.3 we can order the vertices of a finite set in $\cL_D$ so that the levels $n$ of the vertices are non-decreasing, and then in the inductive proof of Proposition~1.2 we will only be conditioning $\{\eta(z_0)=1\}$ on vertices whose level is at most that of $z_0$.]

In view of the comparison, and the fact that $\theta'\to 0$ as $\theta\to 0$, we can for the rest of the
section suppose:
\beq
\eta(z) \hbox{ are i.i.d.~with $P( \eta(z)=1) = 1-\theta$}.
\label{OP2}
\eeq
We now define the edge set ${\cal E}_{\uparrow}$ for $\LD$
to be the set of all oriented edges from $z$ to
$z+v_i$, $z\in \LD$, $1\le i\le D$. A sequence of points
$z_0,\dots, z_{n-1},z_n$ in $\LD$ is called an open path
from $z_0$ to $z_n$, and we write $z_0\to z_n$, if there is
an edge in $\mathcal{E}_\uparrow$ from $z_i$ to $z_{i+1}$
and $z_i$ is open for $i=0,\dots,n-1$.  Note that $z_n$ does
not have to be open if $n\ge1$ but $z_0$ does. In Sections \ref{sec:exist} and \ref{death}
we will employ a block construction and determine suitable
parameters so that $(x,n)\in\cH_n$ being open
will correspond to a certain ``good event'' occuring for our
Poisson processes in the space-time block $(c_{L}x +
[-K_1T,K_1T]^d)\times [nJ_1T,(n+1)J_1T]$ for appropriate $L$, $K_1$ and $J_1$.

Given an initial set of ``wet'' sites $W_0 \subset {\cal H}_0$,
we say $z \in {\cal H}_n$ is wet if $z_0\to z$ for some
initial wet site $z_0$.  Let $\bar W_n$ be the set of wet
sites in ${\cal H}_n$ when all the sites in ${\cal H}_0$ are
wet, and let $W^0_n$ be the set of wet sites in ${\cal H}_n$
when only $0\in{\cal H}_0$ is wet. Let 
$\Omega^0_\infty = \{ W^0_n\ne\emptyset \hbox{ for all $n\ge 0$}\}$.

\begin{lem} \label{ballth} 
(i) $\inf_{x\in {\cal H}_n} P( x \in \bar W_n ) \ge P(\Omega^0_\infty)\to 1$ as $\theta\to 0$. 

\noindent
(ii) Let ${\cal H}^r_n = \{ (z,n) \in \LD : z \in [-r,r]^d\}$. Then 
there are $\theta_{\tref{ballth}}>0$ and $r_{\tref{ballth}}>0$ such that 
if $\theta<\theta_{\tref{ballth}}$ and $r\le r_{\tref{ballth}}$ then  as $N\to\infty$.
\begin{equation}
\label{eq:ballthm}
P(\Omega^0_\infty \text{ and } W^0_n\cap \mathcal{H}^{rn}_n \ne  \bar W_n\cap \mathcal{H}^{rn}_n 
\hbox{ for some $n\ge N$}) \to 0.
\end{equation}
\end{lem}

\begin{proof}
  The first result follows from well-known $d=1$ results,
  e.g., see Theorem~4.1 of \cite{Dur84}. The second result
  is weaker than a ``shape theorem'' for $W^0_n$, which
  would say the following, using the notation $A'=\{x':
  (x',n)\in A\}$ for $A\subset\cH_n$.  For $\theta<\theta_c$
  there is a convex set ${\cal D}\subset\Rd$, containing the origin in its interior, so that on
  $\Omega^0_\infty$,
$$
{(W^{0}_n)}'\ \approx n{\cal D} \cap (\bar W_n)'
$$
for all large $n$. More precisely with probability $1$, if $\delta>0$ there is a
random $n_\delta$ such that $({W}^0_n)' \subset
n(1+\delta){\cal D}$ and $({W}^0_n)' \supset
(1-\delta)n{\cal D} \cap (\bar W_n)'$ for all $n \ge
n_\delta$.  The technology exists to prove such a result for
oriented percolation on ${\cal L}_D$, but unfortunately no
one has written down the details.  
The argument is routine but messy, so we content ourselves to remark that 
(ii) can be established by the methods used in Durrett
and Griffeath \cite{DG82} to prove the shape theorem for the
$d$-dimensional contact process with large birth rates: one
uses percolation in two dimensional subspaces
$A(me_i+ne_j)$, $1\le i < j \le n$ and self-duality.  
\end{proof} 

Call sites in $\bar V_n = {\cal H}_n \setminus \bar W_n$ dry. In
Section~\ref{death}, when we are trying to show that
that $\xi^\vep_t$ dies out, the block construction will
imply for appropriate $L$ and $J_1$, 
\beq \hbox{if $(z,n) \in W_n^0$, then $\Bigl(c_{L}z +
  [-L,L]^d\Bigr) \times [(n-1)J_1T,nJ_1T]$ is $\vep$-empty},
\label{blockcon}
\eeq where a region is $\vep$-empty if $\xi^\vep_t(x)=0$ for all
$(x,t)$ in the region.  This will not be good enough for our
purposes because the space-time regions associated with
points in $V_n^0 = \cH_n \setminus W^0_n$ might be occupied by particles. To
identify the locations where there might be 1's in $\xi_t$
we will work backwards in time. However in our coarser grid $\LD$, 1's may
spread sideways through several dry regions and so we need to
introduce an additional set of edges for $\LD$. Let ${\cal
  E}_{\downarrow}$ consist of the set of oriented edges from
$z$ to $z-v_i$ for $1 \le i \le D$, and from $z$ to
$z+v_i-v_j$ for $1\le i\ne j \le D$, $z \in \LD$.

We assume for the rest of this section that
\[d\ge 2,\]
since we will in fact applying these results only for $d\ge 3$.
Our next goal is to prove an exponential bound on the size of clusters of
dry sites.  
Up to this point the
definitions are almost the same as the ones in Durrett \cite{Dur92}.
However, we must now change the details of the contour
argument there, so that it is done on the correct graph. Let
$y\in\LD$ with $y_D=n\ge 0$ (write $y\in\LD^+$). In addition to $P$ as in \eqref{OP2}, for $M>0$ we also work with a probability $\bar P=\bar P_{n,M}$ under which $\eta(z)=1$ for $z=(z',m)\in\LD^+$ satisfying $m\le n$ and $|z'|\ge M$, and the remaining $\eta(z)$'s are as in \eqref{OP2}.  Therefore under $\bar P$ the sets of wet sites $\{\bar W_n\}$ will be larger, although we will use the same notation since their definition is the same under either probability law.   If $y$ is wet put $D_y=\emptyset$, and otherwise
let $D_y$ be the connected component in
$({\cal L}_D, {\cal E}_{\downarrow})$ of dry sites containing
$y$.  That is, $z\in D_y$ iff there are $z_1=y, z_2,\dots, z_K=z$ all in $\LD$ so that the edge from $z_i$ to $z_{i+1}$ is in $\E_{\downarrow}$ and each $z_i$ is dry.  Since all sites in $\cH_0$ are wet, $D_y\subset \{z\in\LD:n\ge z_D>0\}$,  and under $\bar P_{n,M}$, $D_y\subset \{z\in\LD:n\ge z_D>0, |(z_1,\dots,z_{D-1})|<M\}$.  We assume that $\omega$ satisfies
\begin{equation}\label{com}D_y(\omega) \hbox{ is finite.}
\end{equation}
The fact that \eqref{com} holds a.s. under $\bar P_{n,M}$ is the reason this law was introduced.
 To make $D_y$ into a solid object we consider the compact solid
$$
R_y = \cup_{z \in D_y} (z + {\cal Q})\subset \R^d\times\R_+ \,.
$$

 If $R_y^c$ is the complement of $R_y$ in $\R^d\times\R_+$, we claim that both $R_y$ and
$R^c_y$ are path-connected. For $R_y$, suppose for
concreteness that $D=3$ and note that for the diagonally
adjacent points $y(0) = {\cal A}(0,0,0)$ and $y(1) = {\cal
  A}(1,-1,0)$, $D_{y(0)} \cap D_{y(1)}$ contains the edge
$\mathcal{A}(\{1/2\} \times \{-1/2\} \times [-1/2,1/2])$. For $R^c_y$,
if $x\in R^c_y$ then there exists $[x]\in \LD^+\setminus D_y$
such that $x\in [x]+\cQ$ and the line segment from $x$ to
$[x]$ is contained in $R^c_y$. We first assume $[x]\in\cH_k$
for some $k\in\{1,2\dots,n\}$. If $[x]$ is wet then there must
be a path in $R^c_y$ connecting $[x]$ to $\cH_0$. Suppose $[x]$
is dry, and let $z_0,z_1,\dots,z_K$ be a path in
$\mathcal{E}_{\downarrow}$ connecting $z_0=y$ to
$z_K=[x]$. At least one site on this path must be wet (else $[x]\in D_y$), so let
$z_j$ be the first wet site encountered starting at
$z_K$. Then for each $i>j$, $z_i$ is dry and $z_i\notin
D_y$  (or else $[x]$ would in $D_y$). Thus $\cup_{i=j}^K(z_i+\cQ)$ is path-connected, 
contained in $R_y^c$, and $z_j$ is connected to $\cH_0$  by
a path in $R_y^c$. Note that $\cH_0\subset (\R^d\times\{0\})\cap R_y^c\equiv\tilde\cH_0$  which is path-connected because the rotated cubes making up $R_y$ can only intersect $\R^d\times\{0\}$ in a discrete set of points (since $D_y\subset\{z_D>0\}$).  It is here that we use $d\ge 2$. Now suppose $[x]\in \cH_k$ for some $k>n$.  $\tilde\cH_0$ is also connected
to $\cH_{n+1}$ by a path in $R_y^c$ (assuming $\theta<1$). This allows us to connect $[x]$ to $\tilde\cH_0$ and so conclude that $R^c_y$ is path-connected.

Let $\Gamma_y$ be the boundary of $R^c_y$. 
To study $\Gamma_y$ we need some notation.  We define the plus faces of $[-1/2,1/2]^D$ to
be $[-1/2,1/2]^m \times \{1/2\} \times [-1/2,1/2]^{D-m-1}$,
and define the minus faces to be $[-1/2,1/2]^m \times \{-1/2\}
\times [-1/2,1/2]^{D-m-1}$, $m=1,\dots,D$. The images of the
plus and minus faces of $[-1/2,1/2]^d$ under $A$ constitute
the plus and minus faces of $Q=A([-1/2,1/2]^d)$, which are
used to defined the plus and minus faces of $\Gamma_y$ in the
obvious way. Note that the plus faces of $\Gamma_y$ will have outward normal
$v_i$ for some $i$ while the minus faces will have outward normal $-v_i$ for some $i$.

\begin{lem} If \eqref{com} holds, then $\Gamma_y$ is connected and bounded.
\label{gyconn}
\end{lem}

\begin{proof} For $\vep>0$ let
  $R^\vep_y=\{x\in\mathbb{R}^d:|x-w|_\infty<\vep \text{ for
    some }w\in R_y\}$. Since $R_y$ is connected, so is
  $R^\vep_y$. If $\kappa(U)$ denotes the number of path-connected
  components of a set $U$, it is a consequence of  the Mayer-Vietoris exact sequence with $n=0$ that for open sets $U,V\subset \R^D$ with $U\cup V=\R^D$,
\[
\kappa(U\cap V) = \kappa(U) + \kappa(V) -1 .
\]
See page 149 of \cite{Hat02} and also Proposition 2.7 of that reference. 
Applying this to the open connected (hence path-connected) sets $R^\vep_y$ and $R^c_y$ whose
union is $\R^D$, we find that $R^\vep_y\cap R^c_y$ is
path-connected. 

Finally, $R^\vep_y\cap R^c_y$ is homotopic to $\Gamma_y$,
and therefore $\Gamma_y$ is also path-connected.  Boundedness is immediate from \eqref{com}.
\end{proof}

\bigskip
For the next result we follow the proof of Lemma 6 from \cite{Dur92}. 
A {\it contour} will be a finite union of faces in $\LD$ which is connected.  

\begin{lem}\label{kesten} There are constants
  $C_{\tref{kesten}}$ and $\mu_{\tref{kesten}}$ which only depend on the
  dimension $D$ so that  
  the number of possible contours with $N$ faces,  containing a fixed face, is at most
  $ C_{\tref{kesten}} (\mu_{\tref{kesten}})^N$.
\end{lem}

\begin{proof} Make the set of faces of $\LD$ into a graph by
  connecting two if they share a point in common.  Note that by the above
  definition a contour corresponds to a finite connected subset of this graph. Each point in the graph has a constant
  degree $\nu=\nu(D)$.  An induction argument shows that any
  connected set of $N$ vertices has at most $N(\nu -2) + 2$
  boundary points. (Adding a new point removes 1 boundary
  point and adds at most $\nu-1$ new ones.)  Consider
  percolation on this graph in which sites are open with
  probability $a$ and closed with probability $1-a$. Let 0
  be a fixed point of the graph corresponding to our fixed face, and ${\cal C}_0$ be the
  component containing 0. If $B_N$ is the number of
  components of size $N$ containing 0, then
$$
1 \ge P( |{\cal C}_0| = N ) \ge B_N a^N (1-a)^{N(\nu-2)+2}.
$$
Rearranging, we get $B_N \le C \mu^N$ with $C=(1-a)^{-2}$ and $\mu = a^{-1} (1-a)^{-(\nu-2)}$.
Taking the derivative of $-\log a - (\nu-2) \log(1-a)$ and setting it equal to 0,
we see that $a=1/(\nu-1)$ optimizes the bound, and gives constants that only depend on
the degree $\nu$.
\end{proof}  

\begin{lem}\label{expcontbd} If  
  $\theta_{\tref{expcontbd}}=(2\mu_{\tref{kesten}})^{-2D}$, then
 $\theta\le\theta_{\tref{expcontbd}}$ implies 
  that for all $y=(y',n)\in\LD^+$ and all $M>|y'|$, 
$\bar P_{n,M} ( |\Gamma_y| \ge N ) \le 2C_{\tref{kesten}}2^{-N} $ for all $N\in\NN$. 
\end{lem}

\begin{proof} By Lemma~\ref{gyconn} if $D_y\neq\emptyset$ we see that under $\bar P_{n,M}$, $\Gamma_y$ is a contour which by definition contains the plus faces of $y+Q$. 
Given a plus face in $\Gamma_y$ if
we travel the line perpendicular to $\R^d\times \{0\}$ and through the center of the face,
then we enter and leave the set an equal number of times, so
the number of plus faces of $\Gamma_y$ is equal to the number
of minus faces. Thus, if the contour $\Gamma_y$ has size $N$ there are
$N/2$ minus faces. It is easy to see that a point of $\bar
W_j$ adjacent to a minus face associated with a point in
$\bar V_{j+1}$ must be closed for otherwise it
would wet the point in $\bar V_{j+1}$ (recall the outward
normal of a minus face is $-v_i$ for some $i$). The point of
$\bar W_j$ 
that we have identified might be associated with as many as
$D$ minus faces, but in any case for a contour of size $N$  there must be at least $N/2D$ associated closed sites. Taking $\theta\le (2\mu_{\ref{kesten}})^{-2D}$, using Lemma~\ref{kesten} to bound the number of possible contours containing a fixed plus face of $y+Q$, and summing the resulting geometric series
now gives the result.
\end{proof}

It follows from the above and an elementary isoperimetric inequality that there are finite positive
constants $C,c$ such that for all $y=(y',n)\in\LD^+$ and $M>|y'|$,
\begin{equation}\label{expdrybd}
\text{if }\theta\le\theta_{\tref{expcontbd}}\text{ then }
\bar P_{n,M}(|D_y|\ge N ) \le C\exp(-c N^{(D-1)/D})\hbox{ for all }N\in\NN.
\end{equation}
Now fix $r>0$ and let ${\cal B}_n$ be the dry sites in ${\cal H}^{rn/4}_n$
connected to the complement of $\cup_{m=n/2}^n {\cal
  H}^{rm/2}_m$ by a path of dry sites on the graph with edges ${\cal E}_\downarrow$, 
where as for open sites the last site in such a path need not be dry. 

\begin{lem}\label{deadcore}
If $\theta\le\theta_{\tref{expcontbd}}$ then
$$
P({\cal B}_n \neq \emptyset
\hbox{ infinitely often})= 0.
$$
\end{lem}

\begin{proof} Let $M>n(r+\sqrt{2D})$.  We couple the iid Bernoulli random variables $\{\eta(z):z\in\LD\}$ (under $P$) with the corresponding random field $\bar \eta$ (under $\bar P=\bar P_{n,M}$)  so that 
\[\eta(z)=\bar\eta(z)\quad\forall z=(z',m)\ \hbox{ where }|z'|<M\hbox{ or }m>n.\]
We claim that $z\in\cup_{m=n/2}^n\cH^{rn}_n\equiv\hat\cH_n$ is wet for $\eta$ iff it is wet for $\bar \eta$.  
It clearly suffices to fix $z=(z',m)\in\hat\cH_n$ which is wet for $\bar\eta$ and show it is wet for $\eta$.
A path of sites $z_i=(z_i',i),\ i=0,\dots,m$ with edges in $\cE_\uparrow$ from $\cH_0$ to $z$ satisfies
$\max_{i\le m}|z'_i|\le rn+\sqrt{2D}n<M$. This is because
 the edges in
$\cE_{\uparrow}$ have length at most $\sqrt{2D}$.
Therefore if the sites in the path are open in $\bar\eta$, then they will also be open in $\eta$. This proves the claim.

Next note that  if $y\in\cH_n^{rn/4}$, then $y\in\cB_n$ for $\eta$ iff $y\in\cB_n$ for $\bar \eta$. This is because the path of dry sites connecting $y$ to the complement of $\cup_{m=n/2}^n {\cal
  H}^{rm/2}_m$ can be taken to be inside $\hat\cH_n$ and so we may apply the claim in the last paragraph.  It now follows from the above bound on the length of the edges in $\cE_\downarrow$ that 
  \[P(y\in\cB_n)=\bar P_{n,M}(y\in\cB_n)\le \bar P_{n,M}\Bigr(|D_y|\ge\frac{c(r)n}{\sqrt{2D}}\Bigl).\]
   The number of sites in ${\cal H}^{rn/4}_n$ is at most $C
  n^d$, and the bound in \eqref{expdrybd} shows that 
$P(\cB_n\ne\emptyset)\le \sum_{y\in\cH_n^{rn/4}}P(y\in\cB_n)$ is summable over $n$. \end{proof}

\begin{rem} \label{rem:deadzone} We will prove in
  Section~\ref{death} that if wet sites have the property in
  \eqref{blockcon}, and the kernels $p(\cdot)$ and
  $q(\cdot)$ are finite range, then for an appropriate
  $r>0$, ${\cal B}_n = \emptyset$ will imply that on
  $\Omega^0_\infty$ all sites in $[-c_{L,d}rn,c_{L,d}rn]^d$
  will be vacant at times $t\in[(n-1)J_1T,nJ_1T]$.  This
  linearly growing dead zone will guarantee 
  extinction of the 1's.
\end{rem}  
\note{finrange}\note{Long range bond stuff at end of file}
\clearpage

\section{Existence of stationary distributions}
\label{sec:exist}
 
With the convergence of the particle system to the PDE established and the percolation result introduced, we can infer the existence of stationary distributions by using a ``block construction". 
Recall that our voter model perturbations take values in $\{0,1\}^{\vep\Z^d}$ and so our stationary distributions will be probabilities on this space of rescaled configurations.
We begin with a simple result showing that for stationary distributions, having some $1$'s a.s. or infinitely many $1$'s a.s. are equivalent. Let $|\xi|=\sum_x\xi(x)$.

\begin{lem}\label{nuequiv} If $\nu$ is a stationary distribution for a voter perturbation, then
\[|\xi|=\infty\ \ \nu-a.s.\ \ \hbox{ iff }\ \ |\xi|>0\ \ \nu-a.s.\]
\end{lem}
\begin{proof} It suffices to prove
\begin{equation}\label{red1}
\nu(|\xi|<\infty)>0\hbox{ implies }\nu(|\xi|=0)>0.
\end{equation}

Assume first that the $0$ configuration is a trap.  Then if $|\xi_0|=K<\infty$, \eqref{xirates} shows the sum of the flip rates is finite and so it is easy to prescribe a sequence of $K$ flips which occur with positive probability and concludes with the $0$ state.  By stationarity we get the implication in \eqref{red1}.

Assume next that $0$ is not a trap, which means
$g^\vep_1(0,\dots,0)>0$. We claim that
$\nu(|\xi|<\infty)=0$, which implies the required
result. Intuitively this is true because configurations with
finitely many $1$'s have an infinite rate of production of
$1$'s.  One way to prove this formally is through
generators. Let $\Omega^\vep$ be the generator of our voter
perturbation, $\Omega_v$ be the generator of the voter model
in \eqref{voterrates} and for $i=0,1$
\[
\Omega_i\psi(\xi)=\sum_{x\in\Z^d}1(\xi(x)=1-i)E(g_i^\vep(\xi(x+Y^1),\dots,\xi(x+Y^{N_0})))(\psi(\xi^x)-\psi(\xi)).
\]
Here $\psi$ will be a bounded function on $\{0,1\}^{\Z^d}$ depending on finitely many coordinates, and we recall that $\xi^x$ is $\xi$ with the coordinate at $x$ flipped to $1-\xi(x)$.  Recall that $\xi_\vep(\vep x)=\xi(x)$ for $\xi\in\{0,1\}^{\Z^d}, x\in\Z^d$. For $\psi$ as above define $\psi_\vep$ on $\{0,1\}^{\vep\Z^d}$ by $\psi_\vep(\xi_\vep)=\psi(\xi)$.  Then by \eqref{xirates} and \eqref{gendesc},  
\begin{equation}\label{geneq}
\Omega^\vep\psi_\vep(\xi_\vep)=(\vep^{-2}-\vep_1^{-2})\Omega_v\psi(\xi)+\Omega_0\psi(\xi)+\Omega_1\psi(\xi).
\end{equation}
For $0<r<R$, let $A(r,R)=\{x\in\Z^d:r\le |x|\le R\}$ and 
$$
\psi_{r,R}(\xi)=1(\xi|_{A(r,R)}\equiv0), \ \ \xi\in\{0,1\}^{\Z^d}.
$$
Considering two cases $x \in A(r,R)$ and $x \not\in A(r,R)$ we have
\begin{equation}
\label{psimon}
\hbox{ if $\xi(x)=0$ then } \psi_{r,R}(\xi^x)-\psi_{r,R}(\xi)\le 0.
\end{equation}
Since $\psi_{r,R}(\xi^x)-\psi_{r,R}(\xi)=1$ only if $x$ is the only site in $A(r,R)$ where $\xi(x)=1$, we have
\begin{equation}\label{elemid}
\Omega_v\psi_{r,R}(\xi)\le 1,\quad \Omega_0\psi_{r,R}(\xi)\le \Vert g^\vep_0\Vert_\infty.
\end{equation}
Choose $\lambda$ so that $P(Y^*\le \lambda)\ge 1/2$, where $Y^*$ is as in \eqref{expbd2}. 
Flipping a site from 0 to 1 cannot increase $\psi_{r,R}$, and
$\psi_{r,R}(\xi)=1$ implies $\xi(x)=0$ for all $x\in A(r,R)$, so we have 
\begin{align}\label{omega1bnd}
\nonumber\Omega_1\psi_{r,R}(\xi)\le -\sum_{x\in A(r,R)} &
(1-\xi(x))g_1^\vep(0,\dots,0)P(\xi(x+Y^i)=0\hbox{ for }1 \le i \le N_0)\psi_{r,R}(\xi)\\
&\le -\frac{g_1^\vep(0,\dots,0)}{2}\psi_{r,R}(\xi)|A(r+\lambda,R-\lambda)|.
\end{align}
The stationarity of $\nu$ implies, see Theorem B.7 of Liggett \cite{Lig99}, that if $\psi=\psi_{r,R}$ then 
$\int \Omega^\vep\psi_\vep d\nu=0$. Using  \eqref{geneq}, \eqref{elemid} and \eqref{omega1bnd}, and noting that $$\int\psi_\vep\,d\nu=\nu(\xi\equiv 0 \text{ on }A(\vep r, \vep R)),$$ 
we have
$$
0 \le (\ep^{-2} - \ep_1^{-2}) + \| g^\ep_0\|_\infty
-\frac{g_1^\vep(0,\dots,0)}{2}|A(r+\lambda,R-\lambda)| \nu(\xi\equiv 0\text{ on }A(\vep r,\vep R)).
$$
Rearranging this inequality we get
\[
 \nu(\xi\equiv 0\text{ on }A(\vep r,\vep R))\le \frac{2((\vep^{-2}-\vep^{-2}_1)+\Vert g^\vep_0\Vert_\infty)}
{g_1^\vep(0,\dots,0)|A(r+\lambda,R-\lambda)|}
\]
(recall $g_1^\vep(0,\dots,0)>0$).  Letting $R\to\infty$ we conclude that
$\nu(\xi\equiv 0\text{ on }A(\vep r,\infty)) = 0$. In words, for
$\nu$-a.a. configurations there is a 1 outside the ball of
radius $\vep r$. As this holds for all $r < \infty$, there are infinitely many ones with probability 1 under $\nu$. 
\end{proof}

Assumption~\ref{a1} and \eqref{grate} are in force throughout the rest of this section and we drop dependence on the parameters $w$, $v_i$, $L_0$, $L_1$, $r_0$, etc. arising in those hypotheses in our notation. We continue to work with the particle densities $D(x,\xi)$ using the choice of $r$ in \eqref{Tdefn}.  We start with a version of Lemma~\ref{lem:gdlow} which is adapted for proving coexistence.  We let 
\[
L_2=3+L_0\vee L_1.
\]

\begin{lem}\label{lem:gdmed} There is a $C_{\ref{lem:gdmed}}>0$ and for every $\eta>0$, there are $T_\eta\ge 1$ and $\vep_{\ref{lem:gdmed}}(\eta)>0$ so that for $t\in [T_\eta,C_{\ref{lem:gdmed}}\log(1/\vep)]$ and $0<\vep<\vep_{\ref{lem:gdmed}}$, if 
\[\xi_0^\vep\hbox{ has density in }[v_0+\eta,v_1-\eta]\hbox{ on }[-L_2,L_2]^d,\]
then
\[P(\xi_t^\vep\hbox{ has density in }[u_*-\eta,u^*+\eta]\hbox{ on }[-wt,wt]^d|\xi_0^\vep)\ge 1-\vep^{.05}.\]
\end{lem}

The proof is derived by making minor modifications to that of Lemma~\ref{lem:gdlow} and so is omitted.
We will always assume $\eta>0$ is small enough so that 
\[0<v_0+\eta\le u_*-\eta<u^*+\eta\le v_1-\eta<1.\]  The
one-sided versions of the above Lemma also hold (recall
Lemma~\ref{lem:gdlow} on which the proof is based is a
one-sided result), that is, with only one-sided bounds on
the densities in the hypothesis and conclusion.  

\mn
{\bf Theorem \ref{thm:exist}.} {\it Suppose Assumption \ref{a1} and \eqref{grate}, and let $\eta>0$.  If $\vep>0$ is small enough, depending on $\eta$, then 
coexistence holds for the voter model perturbation, the nontrivial stationary distribution $\nu$ may be taken to be translation invariant,
and any stationary distribution such that  
\begin{equation}\label{nucond2}\nu\Bigl( \sum_{x\in\vep\Z^d}\xi(x)=0\hbox{ or }\sum_{x\in\vep\Z^d}(1-\xi(x))=0\Bigr)=0
\end{equation}
satisfies $\nu( \xi(x)=1) \in (u_*-\eta, u^*+\eta)$ for all $x\in\vep\Z^d$.}

\mn{\it Proof.} 
We use the block construction in the form of Theorem~4.3 of
\cite{Dur95}. This result is formulated for $D=2$ but it is easy
to extend the proof to $D \ge 3$, and we use this extension
without further comment.  
Recall $Q(r) = [-r,r]^d$ and $Q^\ep(r) = Q(r) \cap \ep\Z^d$. Let $U=(C_{\ref{lem:gdmed}}/2)\log(1/\vep)$, 
$L=wU/(\alpha_0D+1)$, where $\alpha_0>0$ is a parameter to be chosen below, and $I^*_\eta = [u_* - \eta/4, u^* + \eta/4]$. Next we define the sets $H$ and $G_\xi$ which appear in the above Theorem.
Let 
\[
H=\{\xi\in\{0,1\}^{\vep\Z^d}:\xi\hbox{ has density in }I^*_\eta\hbox{ on } Q(L) \},
\]
that is, if $Q_\ep = [0,a_\ep)^d \cap \vep\Z^d$ then the fraction of occupied sites in $x + Q_\ep$ is in $I^*_\eta = [u^* - \eta/4, u^* + \eta/4]$ whenever $x \in a_\ep \Z^d \cap [-L,L]^d$.  
If $L'=L+1$, then $\{\xi\in H\}$ depends on $\xi|_{[-L',L']^d}$. Here we need to add $1$ as the cubes of side $a_\vep$ with ``lower left-hand corner" at $x\in[-L,L]^d$ will be contained in $[-L',L']^d$. This verifies the measurability condition in Theorem~4.3 of \cite{Dur95} with $L'=L+1$ in place of $L$ which will affect nothing in the proof of Theorem~4.3.

Let $G_\xi$ be the event on which (a) if $\xi_0^\vep=\xi$, then $\xi^\vep_U$ has density in $I^*_\eta$ on $Q(wU)$ and (b) for all $z\in Q^\vep(wU+1)$ and all $t\le U$, $X_t^{z,U}\subset Q((w+b_0)U+1)$, where $b_0=16(3+d)/C_{\ref{lem:gdmed}}$.
Note that
\begin{align}\label{Gxidef}
G_\xi\in&\sigma\Bigl(\Lambda_r^y|_{[0,U]\times\vep\Z^{dN_0}\times[0,1]},\Lambda^y_w|_{[0,U]\times\vep\Z^d}:y\in Q^\vep((w+b_0)U)+1\Bigr)\\
\nonumber\equiv&\cG(Q((b_0+w)U+1))\times[0,U])
\end{align}
Informally, $\cG(R)$ is the $\sigma$-field of generated by the points in the graphical representation that lie in $R$. The above measurability is easy to verify using the duality relation \eqref{dualityeq}.

Consider now the Comparison Assumptions prior to Theorem~4.3 of \cite{Dur95}. In our context we need to show 

\begin{lem} \label{lem:CA}
For $0 < \ep < \ep_{\ref{lem:CA}}(\eta)$:

\mn
(i) if $\xi_0^\vep\in H$, then on $G_{\xi_0^\vep}$, $\xi^\vep_U$ has density in $I^*_\eta$ 
on $\alpha_0Lv_i'+[-L,L]^d$, $1\le i \le D$,

\mn
(ii) if $\xi\in H$, then $P(G_\xi)\ge 1-\vep^{0.04}$.
\end{lem}

\begin{proof}
By assuming $\ep<\ep_1(\eta)$ we have $U\ge T_{\eta/4}$ and $L\ge L_2$. Using the definition of $L$ and the fact that $|v_i'|\le \Vert v_i'\Vert_2=\sqrt{D-1}$ one easily checks that 
\begin{equation}\label{percin}
\alpha_0 Lv'_i+[-L,L]^d\subset[-wU,wU]^d\hbox{ for }i=1,\dots,D.
\end{equation}
Part (a) of the definition of $G_\xi$ now gives (i). 
By Lemma~\ref{lem:contain2} with parameters $L=wU+1$, $2b=b_0$, $c=C_{\ref{lem:gdmed}}/2$ and $T'=U$, and Lemma~\ref{lem:gdmed}, for $\xi\in H$ we have for $\vep<\vep_{\ref{lem:gdmed}}(\eta)$,
\begin{align*}
P(G_\xi^c)&\le \vep^{.05}+c'_d(wU+1)^d (U+1)\vep^{((b_0 C_{\ref{lem:gdmed}})/16)-2-d}\\
&\le \vep^{.05}+c(\log(1/\vep))^{d+1}\vep\le \vep^{.04},
\end{align*}
where the last two inequalities hold for small $\ep$. We may reduce $C_{\ref{lem:gdmed}}$ to ensure that $b=b_0/2$ satisfies the lower bound in Lemma~\ref{lem:contain2}.  This proves (ii).
\end{proof}

Continue now with the proof of Theorem~\ref{thm:exist}.  Let $\vep<\vep_{\ref{lem:CA}}$ and define
$$V_n=\{(x,n)\in\cH_n:\xi^\vep_{nU}\hbox{ has density in }I^*_\eta\hbox{ on }\alpha_0Lx+[-L,L]^d\}.$$
(To be completely precise in the above we should shift $\alpha_0Lx$ and $\alpha_0Lv_i'$ to the point in $\vep\Z^d$ ``below and to the left of it" but the adjustments become both cumbersome and trivial so we suppress such adjustments in what follows.) If we let
 \[
R_{y,n}=(y\alpha_0L,nU)+Q((b_0+w)U+1)\times[0,U],\hbox{ for }(y,n)\in\cL_D
\]
and 
\[
M=\left\lceil\frac{2(b_0+w)(\alpha_0 D+1)}{\alpha_0w}\right\rceil,
\]
then  $R_{y_1,m}\cap R_{y_2,n}=\emptyset$ if $|(y_1,m)-(y_2,n)|>M$.
Since $\cG(R_i)$, $1 \le i \le k$ are independent for disjoint $R_i$'s, Lemma~\ref{lem:CA} allows us to apply the proof of Theorem~4.3 of \cite{Dur95}.  This shows there is an $M$-dependent (in the sense of \eqref{modMdep}) oriented percolation process $\{W_n\}$ on $\cL_D$ with density at least $1-\vep^{.04}$ such that $W_0=V_0$ and $W_n\subset V_n$ for all $n\ge 0$. We note that although a weaker definition of $M$-dependence is used in \cite{Dur95} (see (4.1) of that reference), the proof produces $\{W_n\}$ as in \eqref{modMdep}.  By Lemma~\ref{ballth} with $r=r_{\ref{ballth}}$ and $\theta=\vep^{.04}$, if $\vep<\vep_1(\eta)$, then
\begin{align}\label{percprop}
\lim_{n\to\infty}\inf_{(x,n)\in\cH_n^{rn}} &P(\xi^\vep_{nU}\hbox{ has density in $I^*_\eta$ on }\alpha_0Lx+[-L,L]^d)\\
\nonumber&\ge \Bigl(1-\frac{\eta}{4}\Bigr)P(0\in V_0).
\end{align}

We will choose different values of $\alpha_0$ to first prove the existence of a stationary law, and then to establish the density bound for any stationary distribution.  For the first part, set $\alpha_0=3$ and take $\{\xi_0^\vep(x):x\in\vep\Z^d\}$ to be iid Bernoulli variables with mean $u=(u_*+u^*)/2$. 
 The weak law of large numbers implies that if $\vep$ is small enough
\begin{equation}\label{vee0}
P(\xi^\vep_0\hbox{ has density in }I^*_\eta\hbox{ on }[-L,L]^d) \ge \frac{1}{2}.
\end{equation}
Since $\alpha_0=3$, $L\ge 3$ and $|x-y|\ge\Vert x-y\Vert_2/\sqrt D\ge 1$ for all $x\neq y\in\cH_n'$,
$\{\alpha_0Lx+[-L',L']^d:x\in\cH_n'\}$ is a collection of
disjoint subsets of $\R^d$ for each $n$. This and the measurability property of  $\{\xi\in H\}$ noted above shows that  
if $0 < \ep < \ep_{0}(\eta)$ then $\{V_n\}$ is bounded below by an $M$-dependent (as in \eqref{modMdep}) oriented 
percolation process, $\{W_n^{1/2}\}$, with density $\ge 1 - \vep^{.04}$ starting with an 
iid Bernoulli ($1/2$) field.
Having established that our process dominates oriented percolation, it is now routine to show the existence of a nontrivial stationary distribution. We will spell out the details for completeness. 

\begin{lem}\label{lem:lotsa} Assume $\alpha_0=3$ and $\{\xi_0^\vep(x):x\in\vep\Z^d\}$ are as above.  There is an $\vep_{\ref{lem:lotsa}}(\eta)>0$ so that 
for any $\vep\in(0,\vep_{\ref{lem:lotsa}}(\eta))$  and any $k\in\NN$ there are $t_1(k,\vep)$, $M_1(k,\vep)>0$ so that for $t\ge t_1$, 
$$P\Bigl(\sum_{|x|\le M_1}\xi_t^\vep(x)\ge k\hbox{ and }\sum_{|x|\le M_1} 1-\xi^\vep_t(x)\ge k\Bigr)\ge 1-\frac{2}{k}.$$ 
\end{lem}

\begin{proof} As in Theorem A.3 of \cite{Dur95} for $k\in \NN$ there are $n_0,\ell_0,M_0\in\NN$ and $z_1,\dots,z_{4k}\in Q(M_0)$ satisfying $|z_i-z_j|>3M+2\ell_0+1$ for $i\neq j$, such that for $n\ge n_0$ with probability at least $1-k^{-1}$
\begin{equation}\label{Wlotsa}W_n^{1/2}\cap Q(z_j,\ell_0)\neq\emptyset\hbox{ for }j=1,\dots,4k.
\end{equation}
The above implies there are
$\sigma(\xi^\vep_{nU})$-measurable $y_j\in
Q^\vep(z_j,\ell_0)$ such that 
\begin{equation}\label{discdens}
\xi_{nU}^\vep\hbox{ has density in }I^*_\eta\hbox{ on }3Ly_j+[-L,L]^d,\ \ j=1,\dots,4k.
\end{equation}
This proves the result for $t=nU$. Intermediate times can be easily handled using Lemma~\ref{lem:gdmed} and the finite speed of the dual (Lemma~\ref{lem:contain2}).  Those results show that for a fixed $\vep<\vep_{\ref{lem:gdmed}}$ and $t\ge (n_0+1)U$, if we choose $n\ge n_0$ so that $t\in[(n+1)U,(n+2)U]$ (use $T_{\eta/4}\le 2U= C_{\ref{lem:gdmed}}\log(1/\vep)$ in applying Lemma~\ref{lem:gdmed}), then on the event in \eqref{discdens} we have
\begin{align*}
P(&\xi_t^\vep\hbox{ has density in }I^*_\eta \hbox{ on }3Ly_j+[-L,L]^d,\\
&\hbox{and }X_s^{x,t}\in Q(3Ly_j,L'+b_0U)\hbox{ for all }x\in 3Ly_j+[-L',L']^d \hbox{ and }s\in [0,t]|\xi_{nU}^\vep)\\
&\ge 1-\vep^{.05}-c_1(\log (1/\vep))^{d+1}\vep\ge \frac{1}{2}.
\end{align*}
where in the last we may have needed to make $\vep$ smaller.

Our separation condition on the $\{z_j\}$ and $L\ge 3$ implies that $Q(3Ly_j,L'+b_0U)$, $j=1,\dots,4k$ are disjoint and so the events on the left-hand side are conditionally independent as $j$ varies.      
Therefore a simple binomial calculation shows that
\begin{align*}P&(|\{j\le 4k:\xi^\vep_t\hbox{ has density in }I^*_\eta\hbox{ on }3Ly_j+[-L,L]^d\}|
\ge k)\\
&\ge\Bigl(1-\frac{1}{k}\Bigr)\Bigl(1-\frac{1}{k}\Bigr)\ge1-\frac{2}{k}.
\end{align*}
Here the first $1-\frac{1}{k}$ comes from establishing \eqref{discdens} and the second $1-\frac{1}{k}$ comes from the binomial error in getting fewer than $k$ points with appropriate density at time $t$. Since the above event implies the required event with $M_1=3L(M_0+\ell_0)+L$ we are done.
\end{proof}

Fix $\vep <\vep_{\ref{lem:lotsa}}$.  By Theorem I.1.8 of \cite{Lig} there is a sequence $t_n\to\infty$ s.t. $t_n^{-1}\int_0^{t_n}1(\xi_s\in\cdot)\,ds\rightarrow \nu$ in law where $\nu$ is a translation invariant stationary distribution for our voter perturbation.  Lemma \ref{lem:lotsa} easily shows that there are infinitely many $0$'s and $1$'s $\nu$-a.s., proving the first part of Theorem~\ref{thm:exist}. 

Turning to the second assertion, by Lemma~\ref{nuequiv} and symmetry it suffices to show that for $\vep<\vep_2(\eta)$ and any given stationary $\nu$ with infinitely many $0$'s and $1$'s a.s. then
$$\sup_x\mu(\xi(x)=1)\le u^*+\eta.$$
Start the system with law $\nu$.  We claim that 

\begin{lem} \label{lem:hole} There is a $\sigma(\xi_0^\vep)$-measurable r.v.~$x_0\in\vep\Z^d$ such that $\xi^\vep_0\equiv 0$ on $Q^\vep(x_0,L)$ a.s. More generally w.p.~1 there is an infinite sequence $\{x_i:i\in\Z_+\}$ of such random variables satisfying $|x_i-x_j|\ge 2L+3$ for all $i\neq j$.
\end{lem}

\begin{proof}
To see this condition on $\xi_0^\vep$, choose $x_0$ so that $\xi_0^\vep(x_0)=0$ and note that if $R^x_1$ is the first reaction time of the dual $X^{x,\vep}$, the event ``$\xi^\vep_0\equiv 0$ on $Q^\vep(x_0,L)$"
occurs if for all $x\in x_0+[-L,L]^d$, $R^x_1>1$, $X^{x,\vep}_1=x_0$,
and $\sup_{s\le 1} |X^{x,\vep}_s-x| \le 1$. Call the last event $A(x_0)$. The last condition has been 
imposed so that if $|x_0-x_1| \ge 2L+3$ then the events $A(x_0)$ and $A(x_1)$ are (conditionally) independent. Clearly they have positive probability.  
Given our initial configuration with $|\{ y : \xi^\vep_0(y) = 1 \}| = \infty$ a.s., we can pick an infinite sequence
$x_i$, $i\in\NN$, with $\xi_0^\vep(x_i)=0$ and $|x_j-x_i| \ge 2L+3$ when $j>i$, so the strong law of large 
numbers implies that at time 1 there will be infinitely many $x_i$ with $\xi_1^\vep(x) = 0$ 
for all $x \in Q^\vep(x_i,L)$.  By stationarity this also holds at time $0$. 
\end{proof}

Now condition on $\xi_0^\vep$, shift our percolation construction in space by $x_0$, set $\alpha_0=(2D)^{-1}$ and only require the density to be at most $u^*+\eta/4$ in our definition of $V_n$ which now becomes
\[V_n=\{(x,n)\in\cH_n:\xi^\vep_{nU}\hbox{ has density at most }u^*+\eta/4\hbox{ on } x_0+c_Lx+[-L,L]^d\},\]
where we recall from \eqref{cLd} that $c_L=L/(2D)$.  (Here we are using the one-sided version of Lemma~\ref{lem:gdmed} mentioned above, after its statement.) Then $0\in V_0$ and the one-sided analogue of \eqref{percprop} shows that if $\vep<\vep_3(\eta)$, then
$$
\lim_{n\to\infty}\inf_{(x,n)\in\cH_n^{r n}} P( x \in V_n )\ge 1-\frac{\eta}{4}.
$$
Recall from \eqref{cLd} that $\cup_{x\in\cH_n'}x_0+c_Lx+[-L,L]^d=\R^d$, so this implies for any $x\in\R^d$ and $n$ large enough, 
\[P(\xi_{nU}^\vep\hbox{ has density at most }u^*+\frac{\eta}{4}\hbox{ on }x+[-L,L]^d)\ge 1-\frac{\eta}{3},\]
and so by stationarity
\[\nu(\xi^\vep\hbox{ has density at most }u^*+\frac{\eta}{4}\hbox{ on }x+[-L,L]^d)\ge 1-\frac{\eta}{3}\hbox{ for all }x\in\R^d.\]
To complete the proof, run the dual for time $t_\vep$ ($t_\vep$ as in \eqref{Tdefn}) and apply Lemma~\ref{lem:ICUD} with $u=u^*+\frac{\eta}{4}$ to see that for $x\in\vep\Z^d$ and 
$\vep<\vep_{3}(\eta)\wedge\vep_{\ref{lem:ICUD}}(\eta/3)$,
\begin{align*}
\nu(\xi(x)=1)&=P(\xi^\vep_{t_\vep}(x)=1)\\
&\le P(R_1\le t_\vep)+E(P(R_1>t_\vep,\xi_{t_\vep}^\vep(x)=1|\xi_0^\vep))\\
&\le (1-e^{-c^*t_\vep})+E(P(\xi_0^\vep(B^{\vep,x}_{t_\vep})=1|\xi_0^\vep))\\
&\le c^*t_\vep+\frac{\eta}{3}+u^*+\frac{\eta}{4}+\frac{\eta}{3}\le u^*+\eta,
\end{align*}
where $\vep$ is further reduced, if necessary, for the last inequality.
\qed
\clearpage

\section{Extinction of the process}
\label{death}

\subsection{Dying out}

Our goal in this section is to show that if $f'(0)<0$ and
$|\xi^\vep_0|$ is $o(\vep^{-d})$, then with high probability
$\xi^\vep_t$ will be extinct by time $O(\log (1/\vep))$.
Throughout this Section we assume that $0<\vep\le \vep_0$ and 
that \eqref{staydead} holds, i.e.,
$g^\vep_1(0,\dots,0)=0\hbox{ for }0<\vep\le \vep_0$

Recall from \eqref{dvepdef} the drift at $\vep x$ in the rescaled state $\xi_\vep\in\{0,1\}^{\vep\Z^d}$ (recall the notation prior to \eqref{locdens}) is
\begin{equation*}
d_\vep(\vep x,\xi_\vep) = (1-\xi(x))h_1^\vep(x,\xi)  - \xi(x)h_0^\vep(x,\xi),
\end{equation*}
and define the total drift for $|\xi_\vep|<\infty$ by
\begin{equation}
\psi_\vep(\xi_\vep) = \sum_x d_\vep(\vep x,\xi_\vep).
\label{psiform} 
\end{equation}
Recall from \eqref{hgrepn} and \eqref{vep1convention} that 
\beq\label{hivepdefn}
h_i^\vep(x,\xi) = E_Y( g^\vep_i(\xi(x+Y^1), \ldots \xi(x+Y^{N_0}) )),
\eeq
where $E_Y$ denotes the expected value over the distribution of $(Y^1,\ldots Y^{N_0})$, and
also that 
\beq\label{cdefns}c^*=c^*(g)=\sup_{0<\vep\le \vep_0/2}\Vert g^\vep_1\Vert_\infty+\Vert g^\vep_0\Vert_\infty +1,\quad c_b=c^*N_0. \eeq
It will be convenient to write
$$
\xi_\vep(\vep x+\vep \bar Y) = (\xi(x+\vep Y^1), \ldots \xi(x+\vep Y^{N_0})).
$$
 
If  $\cH_t$ is the right-continuous filtration generated by the graphical representation, then
\beq
\label{totmass}
|\xi^\vep_t|=|\xi^\vep_0|+M^\vep_t+\int_0^t\psi_\vep(\xi^\vep_s)\,ds,
\eeq
where $M^\vep$ is a zero mean $L^2$-martingale. This is easily seen by writing $\xi^\vep_t(x)$ as a solution of a stochastic differential equation driven by the Poisson point processes in the graphical representation and summing over $x$. The integrability required to show $M^\vep$ is a square integrable martingale is readily obtained by dominating $|\xi^\vep|$ by a pure birth process (the rates $c_\vep$ are uniformly bounded for each $\vep$) and a square function calculation.

\begin{lem}\label{lem:mombd}
For any finite stopping time $S$ 
$$
e^{-c_{b}t}|\xi^\vep_{S}| 
\le E(|\xi^\vep_{S+t}| |{\cH}_S )\le e^{c_{b}t}|\xi^\vep_S|.
$$
\end{lem}

\begin{proof} By the strong Markov property it suffices to prove the result when $S=0$. 
The fact that $d_\vep(\vep x,\xi_\vep)\ge -\Vert g_0^\vep\Vert_\infty$ implies $\psi_\vep(\xi^\vep_s) \ge - \|g^\vep_0\|_\infty|\xi^\vep_s|$.  It follows from \eqref{staydead} and \eqref{hivepdefn} that 
$$
d_\vep(\vep x,\xi_\vep)\le \|g^\vep_1\|_\infty \sum_{y: \xi(y)=1} \sum_{i=1}^{N_0} P( Y^i = y-x ).
$$
Summing over $x$ and then $y$, we get $\psi_\vep(\xi^\vep_s) \le N_0 \|g^\vep_1\|_\infty |\xi^\vep_s|$
and (recalling \eqref{cdefns}) the desired result follows by taking means in \eqref{totmass} and using Gronwall's Lemma.
\end{proof}
 
Let $\xi^{\ep,0}$ be the voter model constructed from the same graphical representation as $\xi^\vep$ by only considering the voter flips. We always assume $\xi_0^{\vep,0}=\xi_0^\vep$.

\begin{lem}\label{lem:voterapprox}  
If $c_{\tref{lem:voterapprox}} = 4(2N_0+1)c^*$ then
\[
E(|\psi_\vep(\xi^\vep_s)-\psi_\vep(\xi^{\vep,0}_s)|)
\le c_{\tref{lem:voterapprox}}[e^{c^*(N_0+1)s}-1]|\xi_0^\vep| \,.
\]
\end{lem}
\begin{proof} 
Let $\xi^\vep_s(\vep x+\vep \tilde Y)=(\xi^\vep_s(\vep x+\vep Y^0),\dots,\xi^\vep_s(\vep x+\vep Y^{N_0}))$,
where $Y^0=0$, $\tilde Y$ is independent of $\xi^\vep$, and note that in contrast to $\bar Y$, $\tilde Y$ contains $0$. Let
$$
D_\vep(\eta_0,\eta_1, \ldots \eta_{N_0}) = - \eta_0 g^\vep_0(\eta_1, \ldots, \eta_{N_0})
+ (1- \eta_0) g^\vep_1(\eta_1, \ldots, \eta_{N_0}),
$$
and note that 
\begin{align}
\nonumber E(|\psi_\vep(\xi^\vep_s)  -\psi_\vep(\xi^{\vep,0}_s)|)
&\le E\Bigl(\sum_x| D_\vep(\xi^\vep_s(\vep x+\vep \tilde Y) - D_\vep(\xi^{\vep,0}_s(\vep x+\vep \tilde Y))|\Bigr)\\
&\le 2\Vert D_\vep \Vert_\infty
E\Bigl(\sum_x[\max_{0\le i\le  N_0}\xi^\vep_s(\vep x+\vep Y_i)\vee\xi^{\vep,0}_s(\vep x+\vep Y_i)]
\label{thbnd}\\ 
\nonumber&\phantom{\le 2\Vert\hat D\Vert_\infty\sum_x}
\times 1\{\xi_s^\vep(\vep x+\vep\tilde Y)\neq \xi^{\vep,0}_s(\vep x+\vep\tilde Y)\}\Bigr), 
\end{align}
because for fixed $x$ if the latter summand is zero, so is the former, and if the
latter summand is 1, the former is at most $ 2\|D_\vep\|_\infty$.
 
Let $X_t=X^{z,s}_t$, $t\in [0,s]$ be the dual of $\xi^\vep$ starting at 
$(z_0,\dots,z_{N_0})=\vep x+\vep\tilde Y$ at time $s$ and let $R_m$, $m \ge 1$ be the 
associated branching times.  We claim that 
\begin{align}\label{sumlvbound}
E\Bigl(\Bigl[\max_{0\le i\le N_0}\xi_s^\vep(\vep x+\vep Y^i)\Bigr] & 
1\{\xi^\vep_s(\vep x+\vep\tilde Y)\neq\xi_s^{\vep,0}(\vep x+\vep\tilde Y)\}\Bigr)\\ 
\nonumber&\le E\Bigl(\sum_{\ell\in J(s)}\xi_0^\vep(X^\ell_s)1\{R_1\le s\}\Bigr).
\end{align}
To see this, note that: 

\mn
(i) if $R_1>s$, then there
are no branching events and so $(X_t,t\le s)$ is precisely the
coalescing dual used to compute the the rescaled voter model
values $\xi^{\vep,0}_s(\vep x+\vep\tilde Y)$. 

\mn
(ii) In the case $R_1 \le s$, if $\xi_0^\vep(X^\ell_s)=0$ for all
$\ell\in J(s)$ then $\xi^\vep_s(\vep x+\vep Y^i)=0$ for $0 \le i \le N_0$ 
because working backwards from time 0 to time $s$, we see that no
site can flip due to a reaction, and again we have 
$\xi^\vep(\vep x+\vep\tilde Y)=\xi^{\vep,0}(\vep x+\vep\tilde Y)$.  

Similar reasoning and the fact that the dual $(X^{0,j}_t,j\in
J^0(t))$ of the voter model $\xi^{\vep,0}$ with the same initial condition $z$
satisfies $J^0(t)\subset J(t)$ for all $t\le s$ a.s., shows that 
\begin{align}\label{sumvoterbound} 
E\Bigl(\Bigl[\max_{0\le i\le N_0}\xi^{\vep,0}_s(\vep x+\vep Y^i)\Bigr] 
& 1\{\xi_s^\vep(\vep x+\vep\tilde Y)\neq \xi_s^{\vep,0}(\vep x+\vep\tilde Y)\}\Bigr)\\ 
\nonumber & \le  E\Bigl(\sum_{\ell\in J(s)}\xi_0^\vep(X^\ell_s)1\{R_1\le s\}\Bigr).
\end{align}

If $E_0$ denotes expectation with respect to the law of $X^{z,s}_t$
when $x=0$ then, using \eqref{sumlvbound} and
\eqref{sumvoterbound}, we may bound \eqref{thbnd} by 
$$
4c^*E_0\Bigl(\sum_{\ell\in J(s)}\sum_{x}\xi^\vep_0(\vep x+X^\ell_s)1\{R_1\le s\}\Bigr)
$$
Bounding by the dominating branching random walk $\bar X$, and using 
$|J(\bar R_1)|=2N_0+1$ and $P(\bar R_1 \le s) = 1-e^{-c^*(N_0+1)s}$, we see the expected value 
in the last formula is at most
\begin{align*} 
& |\xi^\vep_0| E(|\bar J(s)|1\{\bar R_1\le s\}) \le |\xi^\vep_0| e^{c^*N_0s}E(|\bar J(\bar R_1)|1\{\bar R_1\le s\})\\ 
&\le (2N_0+1)|\xi^\vep_0|e^{c^*N_0s}(1-e^{-c^*(N_0+1)s})
\le (2N_0+1)|\xi^\vep_0|(e^{c^*(N_0+1)s}-1),
\end{align*}
which proves the desired result.
\end{proof}

For the next step in the proof we recall the notation from
Section~\ref{ssec:CPcomp}.  We assume $Y$ is independent
from the coalescing random walk system $\{\hat
B^x:x\in\Z^d\}$ used to define $\tau(A)$ and $\tau(A,B)$.
Recall from \eqref{f'rep} and \eqref{betadeltadef} that
under \eqref{staydead}
$$\theta\equiv f'(0)= \sum_{S\in\hat\cP_{N_0}}\hat\beta(S) P(\tau(Y^S)<\infty ,\tau(Y^S,\{0\})=\infty)
-\hat\delta(S)P(\tau(Y^S\cup\{0\}) < \infty).$$
For $M>0$ define
$$
\theta^\vep_M=\sum_{S\in\hat\cP_{N_0}}\hat\beta_\vep(S)P(\tau(Y^S)\le M <\tau(Y^S,\{0\}))
-\hat\delta_\vep(S)P(\tau(Y^S\cup\{0\})\le M) \ .
$$
It follows from \eqref{grepinv} that (with or without the $\vep$'s) 
\beq\label{hatbetabnd} \sum_{S\in \hat\cP_{N_0}}|\hat\beta_\vep(S)|+|\hat\delta_\vep(S)|\le 2^{2N_0}(\Vert g_1^\vep\Vert_\infty+\Vert g_0^\vep\Vert_\infty)\le 2^{2N_0}c^*(g)
\eeq
(recall here that $\tilde g^\vep_i=g^\vep_i$ by our $\vep_1=\infty$ convention).
It is clear that $\lim_{M\to\infty, \vep\to0}\theta^\vep_M=\theta$, 
but we need information about the rate.

\begin{lem}\label{lem:thetarate} There is a $\vep_{\tref{lem:thetarate}}(M)\downarrow0$ (independent of the $g^\vep_i$) so that
$$
|\theta^\vep_M-\theta|\le 2^{2N_0}\Bigl[\Vert g^\vep_1-g_1\Vert_\infty+\Vert g^\vep_0-g_0\Vert_\infty+c^*(g)\vep_{\tref{lem:thetarate}}(M)\Bigr].
$$
\end{lem}

\begin{proof}
 Define 
$$\vep_{\tref{lem:thetarate}}(M) =\frac{1}{2}\sup_{S\in\hat\cP_{N_0}} \{ P(M< \tau <\infty) : \tau = \tau(Y^S), \tau(Y^S\cup\{0\}), \tau(Y^S,\{0\})\},$$
and note that $\vep_{\tref{lem:thetarate}}(M)\downarrow 0$ as $M\uparrow\infty$. Using \eqref{hatbetarate} and \eqref{hatbetabnd} (the latter without the $\vep$'s) we have
\begin{align*}
|\theta^\vep_M-\theta|\le& \sum_{S\in\hat\cP_{N_0}}|\hat\beta_\vep(S)-\hat\beta(S)|+|\hat\delta_\vep(S)-\hat\delta (S)|\\
&\ +\sum_{S\in\hat\cP_{N_0}}\Bigl[|\hat\beta(S)||P(\tau(Y^S)\le M<\tau(Y^S,\{0\}))-P(\tau(Y^S)<\infty=\tau(Y^S,\{0\}))|\\
&\ \phantom{+\sum_{S\in\hat\cP_{N_0}}\Bigl[}+|\hat\delta(S)|P(M<\tau(Y^S\cup\{0\})<\infty)\\
\le& 2^{2N_0}[\Vert g^\vep_1-g_1\Vert_\infty+\Vert g^\vep_0-g_0\Vert_\infty]+2^{2N_0}c^*(g)\vep_{\tref{lem:thetarate}}(M).
\end{align*}
The result follows.
\end{proof}

To exploit the inequality in Lemma~\ref{lem:voterapprox} we need
a good estimate of $E(\psi_\vep(\xi^{\vep,0}_s))$ for small $s$.
\begin{lem}\label{voterest} 
There is a constant $c_{\tref{voterest}}$ (independent of $g_i^\vep$) such that for $\vep,\delta>0$, 
\begin{equation}\label{voterasym}
  E(\psi_\vep(\xi_\delta^{\vep,0}))=\theta^\vep_{\delta{\vep}^{-2}}|
\xi_0^\vep|+\eta_{\ref{voterest}}(\vep,\delta),
\end{equation}
where $|\eta_{\ref{voterest}}(\vep,\delta)|\le
c_{\tref{voterest}}c^*(g)\delta^{-d/2}|\xi^\vep_0|^2\vep^d$. 
\end{lem}

\begin{proof}  As usual we assume $Y$ is independent of $\xi^{\vep,0}$. Summability issues in what follows are handled by Lemma~\ref{lem:mombd} (and its proof) with $g^\vep_i\equiv0$.  The representation \eqref{hhatrepn} and \eqref{psiform} imply that
\begin{align}\nonumber
E(\psi_\vep(\xi^{\vep,0}_\delta))&=\sum_{S}\hat \beta_\vep(S)E^\vep_{0}(S) - \hat\delta_\vep(S) E^\vep_{1}(S) \\
\label{psiexp}
E^\vep_0(S) & =\sum_{x\in\Z^d} E\Bigl( (1-\xi_\delta^{\vep,0}(\vep x))\prod_{i\in S}\xi_\delta^{\vep,0}(\vep x+\vep Y^i)\Bigr)
\nonumber \\ 
E^\vep_1(S) & =\sum_{x\in\Z^d} E\Bigl(\xi_\delta^{\vep,0}(\vep x)\prod_{i\in S}\xi_\delta^{\vep,0}(\vep x+\vep Y^i)\Bigr). 
\end{align}
%Let $B^{x,i}$, $0\le i \le N_0$ be a collection of rate $1$ coalescing 
%random walks with step distribution $p$ starting from $x+Y^i$, and
%let $\tau_x(A)$ and $\tau_x(A,B)$ be the corresponding quantities for these walks,   
We will use duality between $\xi^{\vep,0}$ and $\{\hat B^x\}$ (see (V.1.7) of \cite{Lig}) to argue that
\begin{align*}
E^\vep_0(S) & \approx |\xi^\vep_0|P(\tau(Y^S)\le\delta\vep^{-2}<\tau(Y^S,\{0\})),\ \emptyset \neq S\subset\{1,\dots,N_0\}\\
E^\vep_1(S) & \approx |\xi^\vep_0|P(\tau(Y^S,\{0\})\le\delta\vep^{-2})\hbox{ all }S\subset\{1,\dots,N_0\}.
\end{align*}
Beginning with the first of these, note that duality implies (recall $Y^0\equiv 0$)
\begin{align}
E^\vep_0(S) &=\sum_{x\in \Z^d}E\Bigl((1-\xi_0^\vep(\vep\hat B^{x}_{\delta{\vep}^{-2}}))
\prod_{i\in S}\xi^\vep_0(\vep\hat B^{x+Y^i}_{\delta{\vep}^{-2}})
1\{\tau(x+Y^S,\{x\})>\delta{\vep}^{-2}\}\Bigr)
\nonumber \\ 
&=\sum_{x\in \Z^d}E\Bigl(\prod_{i\in S} \xi_0^\vep(\vep\hat B^{x+Y^i}_{\delta{\vep}^{-2}})
1\{\tau(x+Y^S,\{x\})>\delta{\vep}^{-2}\}\Bigr)
\nonumber \\ 
\label{dualeq}&\quad-\sum_{x\in\Z^d} E\Bigl(\prod_{i\in S\cup\{0\}} \xi_0^\vep(\vep\hat B^{x+Y^i}_{\delta{\vep}^{-2}})
1\{\hat\tau(x+Y^S,\{x\})>\delta{\vep}^{-2}\}\Bigr)\\
\nonumber&\equiv \Sigma_1-\Sigma_2. 
\end{align}

If $\tau(x+Y^S)>\delta\vep^{-2}$ there are  $i \neq  j\in S$ so that
$\tau(\{x+Y^i\},\{x+Y^j\})>\delta\vep^{-2}$. If we condition on the values of the $Y^i, Y^j$ in the next to last line below, 
\begin{align}
  \nonumber &\sum_{x\in\Z^d}E\Bigl(\prod_{i\in S}
\xi_0^\vep(\vep\hat B^{x+Y^i}_{\delta{\vep}^{-2}})1\{\tau(x+Y^S)>\delta{\vep}^{-2}\}\Bigr)\\ 
  \nonumber &\le\sum_{x\in\Z^d}\sum_{1\le i<j\le N_0}
E(\xi_0^\vep(\vep\hat B_{\delta\vep^{-2}}^{x+Y^i})\xi_0^\vep(\vep\hat B_{\delta\vep^{-2}}^{x+Y^j})
1\{ \tau(\{x+Y^i\},\{x+Y^j\}) > \delta\vep^{-2}\}) \\ 
  \nonumber &\le \sum_{w\in\Z^d}\sum_{z\in\Z^d}\xi_0^\vep(\vep w)\xi_0^\vep(\vep z)
\sum_{1\le i<j\le N_0}\sum_{x\in\Z^d} P(\hat B^{x+Y^i}_{\delta\vep^{-2}}=w,\,\hat B^{x+Y^j}_{\delta\vep^{-2}}=z,\\
\nonumber &\phantom{le \sum_{w\in\Z^d}\sum_{z\in\Z^d}\xi_0^\vep(\vep w)\xi_0^\vep(\vep z)
\sum_{1\le i<j\le N_0}\sum_{x\in\Z^d} P(} \tau(\{x+Y^i\},\{x+Y^j\}) > \delta\vep^{-2}\})\\  
  \nonumber &\le \sum_{w\in\Z^d}\sum_{z\in\Z^d}\xi_0^\vep(\vep w)\xi_0^\vep(\vep z)
\sum_{1\le i<j\le N_0} P(\hat B^0_{2\delta\vep^{-2}}=w-z-Y_i+Y_j)\\
\label{coal1}&\le N_0(N_0-1)|\xi_0^\vep|^2c(1+2\delta\vep^{-2})^{-d/2},
\end{align}
where the local central limit theorem (e.g. (A.7) in
\cite{CDP}) is used in the last line.  A similar calculation
shows that 
\beq\label{coal2}
\Sigma_2\le |\xi_0^\vep|^2c(1+2\delta\vep^{-2})^{-d/2}.
\eeq
To see this, note that $\tau(x+Y^S,\{0\})>\delta\vep^{-2}$ implies that for $i_0\in S$ (this is where we require $S$ non-empty)
$\tau(\{x+Y^{i_0}\},\{x\})>\delta\vep^{-2})$ and we
may repeat the above with $i=i_0$ and $j=0$.   
Returning to the study of $\Sigma_1$, taking any $i_0\in S$ we have
\begin{align}
\nonumber &\sum_{x\in\Z^d}E\Bigl(\prod_{i\in S}\xi_0^\vep(\vep B_{\delta\vep^{-2}}^{x+Y^i})
1\{\tau(x+Y^S)\le\delta\vep^{-2}<\tau(x+Y^S,\{x\})\}\Bigr)\\ 
\nonumber&=\sum_{x\in\Z^d}E\Bigl(\xi_0^\vep(\vep x+\vep B_{\delta\vep^{-2}}^{Y^{i_0}})
1\{\tau(Y^S)\le \delta\vep^{-2}<\tau(Y^S,\{0\})\}\Bigr)\\ 
\label{coal3}&=|\xi_0^\vep|P(\tau(Y^S)\le \delta\vep^{-2}<\tau(Y^S,\{0\})).
\end{align}

Together \eqref{coal1} and \eqref{coal3} bound $\Sigma_1$. Using this with \eqref{coal2} in \eqref{dualeq}, we conclude that  
\beq \label{betaeq}
E^\vep_0(S)=|\xi_0^\vep|P(\tau(Y^S)\le \delta\vep^{-2}<\tau(Y^S,\{0\}))+\eta_1(\vep,\delta,S),
\eeq
where $|\eta_1(\vep,\delta,S)|\le cN_0^2|\xi_0^\vep|^2\delta^{-d/2}\vep^d$.
A similar, and simpler, argument shows that for $S\subset\{1,\ldots,N_0\}$,
\beq\label{deltaeq}
E^\vep_1(S) =|\xi_0^\vep|P(\tau(Y^S\cup\{0\}) \le \delta\vep^{-2})+\eta_2(\vep,\delta,S),
\eeq
where $|\eta_2(\vep,\delta,S)|\le c N_0(N_0+1)|\xi^\vep_0|^2\delta^{-d/2}\vep^d$.

Now use \eqref{betaeq}, \eqref{deltaeq} and the fact that
$\hat\beta_\vep(\emptyset)=0$ (by \eqref{staydead}), to
obtain \eqref{voterasym} with
\begin{equation*}|\eta_{\ref{voterest}}(\vep,\delta)|\le
  \sum_{S}(|\hat\beta_\vep(S)|+|\hat\delta_\vep(S)|)cN_0(N_0+1)\delta^{-d/2}|\xi_0^\vep|^2\vep^d. 
\end{equation*}
Finally use \eqref{hatbetabnd} to complete the proof.
\end{proof}

\medskip
For $0<\eta_1<1$, let
$T(\eta_1)=T_\vep(\eta_1)=\inf\{t\ge \vep^{\eta_1}:|\xi^\vep_{t-\vep^{\eta_1}}|\ge (\vep^{-1+{\eta_1\over 2}})^d\vep^{\eta_1}\}$
and note that $T(\eta_1) - \vep^{\eta_1}$ is an $(\cH_t)$-stopping time.

\begin{lem}\label{lem:submartstuff}
There is a $c_{\tref{lem:submartstuff}}$ so that if $\eta_1\in(0,1)$, then for all $s\ge \vep^{\eta_1}$
$$
E(\psi(\xi^\vep_s)|\cH_{s-\vep^{\eta_1}})\le
[\theta^\vep_{\vep^{\eta_1}\vep^{-2}}+c_{\tref{lem:submartstuff}}\vep^{\eta_1}]
\,|\xi^\vep_{s-\vep^{\eta_1}}| \quad\hbox{a.s. on $\{T(\eta_1)>s\}$.}
$$
\end{lem}

\begin{proof} Let $\delta=\vep^{\eta_1}$. If $|\xi^\vep_0|\le (\vep^{-1+{\eta_1\over 2}})^d\vep^{\eta_1}$, then Lemmas~\ref{lem:voterapprox} and
\ref{voterest} imply $E(\psi(\xi_\delta^\vep))=|\xi_0^\vep|\theta^\vep_{\delta\vep^{-2}}+\eta'(\vep)$ with
\begin{align*}
|\eta'(\vep)|&\le
c_{\tref{lem:voterapprox}}[e^{(N_0+1)c^*\vep^{\eta_1}}-1]|\xi^\vep_0|
+c_{\tref{voterest}}\vep^{-\eta_1d/2+d}|\xi^\vep_0|^2\\ 
&\le
c_{\tref{lem:voterapprox}}c^*(N_0+1) e^{(N_0+1)c^*}\vep^{\eta_1}|\xi^\vep_0|
+c_{\tref{voterest}}|\xi_0^\vep|\vep^{\eta_1}.
\end{align*}
For the second term we used the bound on $|\xi^\vep_0|$.
The result now follows from the above by the Markov property
and the definition of $T(\eta_1)$. 
\end{proof}

\begin{lem}\label{survbnd}  
Let $\beta,\eta_2\in(0,1]$.  There is an $\vep_{\tref{survbnd}}(\beta, \eta_2)\in(0,1)$,
so that if $0<\vep\le \vep_{\tref{survbnd}}$ and $\theta=f'(0)\le -\eta_2$, then 
$|\xi^\vep_0|\le \vep^{-d+\beta}$ implies 
$$
P(|\xi^\vep_t|>0)\le 6 e^{2c_b}\vep^{\beta/2}
\quad\hbox{for all $t\ge \frac{2d}{\eta_2}\log(1/\vep)$}. 
$$
\end{lem}

\begin{proof}Let $\lambda\le \eta_2/2$, $\eta_1=\beta(2+d)^{-1}$, $T=T_\vep(\eta_1)$ and
  $\delta=\vep^{\eta_1}$.  An integration by parts using
  \eqref{totmass} shows that for $t\ge \delta$, 
$$
e^{\lambda(t\wedge T)}|\xi_{t\wedge T}^\vep|=e^{\lambda\delta}|\xi^\vep_\delta|
+\int_\delta^t1\{r<T\}[\lambda e^{\lambda r}|\xi_r^\vep|+e^{\lambda r}\psi_\vep(\xi_r^\vep)]dr +N^\vep_t, 
$$
where $N^\vep$ is a mean $0$ martingale. Since $\{r<T\}\in\cH_{r-\delta}$,
we have for $\delta\le s\le t$
\begin{align*}
\nonumber E&(e^{\lambda (t\wedge T)}|\xi^\vep_{t\wedge
  T}|-e^{\lambda(s\wedge T)}|\xi^\vep_{s\wedge
  T}||\cH_{s-\delta})\\ 
\nonumber
&=\int_s^t E(1\{r<T\}e^{\lambda r}E(\lambda|\xi_r^\vep|+\psi_\vep(\xi^\vep_r)|\cH_{r-\delta})|\cH_{s-\delta})dr.
\end{align*}
Using Lemmas \ref{lem:mombd} and \ref{lem:submartstuff} the above is at most
\beq\label{condexp}
E\Bigl(\int_s^t 1\{r<T\}e^{\lambda r}\gamma(\vep)|\xi^\vep_{r-\delta}|dr|\cH_{s-\delta}\Bigr), 
\eeq
where $\gamma(\vep)=\lambda e^{c_b\delta}+
\theta^\vep_{\delta \vep^{-2}}+c_{\tref{lem:submartstuff}}\delta$.
Recall $\delta=\vep^{\eta_1}$ and $\theta = f'(0)\le -\eta_2$. By Lemma~\ref{lem:thetarate} and the uniform convergence of the $g^\vep_i$ to $g_i$ there is a
$\vep_1(\beta,\eta_2)>0$ so that if $0<\vep\le\vep_1$, then  
\begin{align*}
\gamma(\vep) &\le{\eta_2\over 2} e^{c_b\vep^{\eta_1}}
- \eta_2+2^{2N_0}\Bigl[\sum_{i=0}^1\Vert g_i^\vep-g_i\Vert_\infty\Bigr]+2^{2N_0}c^*\vep_{\tref{lem:thetarate}}(\vep^{\eta_1}\vep^{-2})
+c_{\tref{lem:submartstuff}}\vep^{\eta_1}\\
& \le {-\eta_2/4}<0.
\end{align*}
We assume $0<\vep\le\vep_1$ in what follows.
Since the bound in \eqref{condexp} is therefore non-positive and our assumption on $|\xi^\vep_0|$ implies $T>\delta$, we may use Lemma~\ref{lem:mombd} and the fact that $\delta\le 1$ to see that
for $t\ge \delta$,
\beq\label{expbnd1}
E(|\xi_{t\wedge T}^\vep|e^{\lambda(t\wedge T)}) \le e^{\lambda\delta}E(|\xi_\delta^\vep|)
\le e^{(\lambda+c_b)\delta}|\xi_0^\vep|
\le e^{\eta_2+c_b}\vep^{\beta-d}.
\eeq

Now $|\xi^\vep_t|\ge 1$ if it is positive so
\beq
P(|\xi^\vep_t|>0)\le E(|\xi^\vep_{t\wedge T}|e^{\lambda(T\wedge t)}1\{T\ge t\})e^{-\lambda t}+P(T< t).
\label{surv0}
\eeq
Let
$t\ge (2d/\eta_2)\log(\vep^{-1})$ and use \eqref{expbnd1} with $\lambda=\eta_2/2$ to 
see that the first term is at most
\beq
e^{\eta_2+c_b}\vep^{\beta-d}\vep^d
=e^{\eta_2+c_b }\vep^\beta.
\label{surv1}
\eeq
To bound $P(T<t)$,  we note that  
$|\xi^\vep_{T-\delta}|\ge (\vep^{-1+{\eta_1\over 2}})^d\vep^{\eta_1}$ if $T<\infty$, so
$$
P( T < t ) \le 
E(|\xi^\vep_{T-\delta}|1\{T< t\}) (\vep^{-1+{\eta_1\over 2}})^{-d}\vep^{-\eta_1}.
$$
By making $\vep_1$ smaller, depending on $\beta$, we can assume that $(2d/\eta_2)\log(\vep^{-1}) \ge \vep^{\eta_1}=\delta$.
Let $S=(T-\delta)\wedge(t-\delta)$, note $\{T< t\}\in\cH_{S}$, and use the lower bound in
Lemma \ref{lem:mombd} with $\lambda=0$ to conclude the first inequality in
$$
E(|\xi^\vep_{T-\delta}|1\{T< t\}) \le e^{c_{b} \delta}E(|\xi^\vep_{T\wedge t}|)
\le e^{2c_{b}+\eta_2}\vep^{\beta-d}.
$$
The second inequality comes from \eqref{expbnd1} with $\lambda=0$ (recall that $t\ge2d/\eta_2)\log(\vep^{-1}) \ge \vep^{\eta_1}=\delta$) and $\delta\le 1$.
Using the last two equations with \eqref{surv1} in \eqref{surv0}, we conclude that
\begin{align*}
P(|\xi^\vep_t|>0) & \le e^{\eta_2+c_{b} }\vep^{\beta}
+ e^{2c_{b}+\eta_2}\vep^{\beta-d}(\vep^{-1+{\eta_1\over 2}})^{-d}\vep^{-\eta_1}
\\
&\le e^{\eta_2+2c_{b}}[\vep^\beta+\vep^{\beta-\eta_1(1+{d\over 2})}]
\le 2e^{1+2c_b}\vep^{\beta/2},
\end{align*}
where the definition of $\eta_1$ is used in the last line.  The result follows.
\end{proof}

\subsection{The Dead Zone}

For the remainder of this Section we suppose 
\eqref{grate}, \eqref{staydead} and Assumption~2 are in 
force and $-f'(0)\ge \eta_2\in(0,1]$. We also assume that $p(\cdot)$ and $q(\cdot)$ have 
finite supports. More specifically, $R_0\in\NN$ satisfies
\note{finrange}
\beq\label{finsupp1} \{x\in \Z^d:p(x)>0\}\subset[-R_0,R_0]^d\hbox{ and }\{x\in \Z^{dN_0}:q(x)>0\}\subset[-R_0,R_0]^{dN_0}.\eeq
In order to connect with the percolation results
from Section~\ref{sec:perc} we need certain space-time regions
suitable for applying Lemma~\ref{lem:gdlow} to
decrease particle density, Lemma~\ref{lem:contain2} to
control the spread of duals, and Lemma~\ref{survbnd} 
to actually kill off particles. Recall that $Q^\vep(r)=[-r,r]^d\cap(\vep\Z^d)$.   For
$J_0<J_1\in\NN,0<w<1,A,K>1$, and $T>0$ define 
regions
$\mathcal{D}(J_0,J_1,w,A,K)=\mathcal{D}_0\cup\mathcal{D}_1$,
where
\begin{align*}
\mathcal{D}_0 &=\cup_{j=1}^{J_0}(Q^\vep((K-jA)T)\times[(j-1)T,jT]), 
\\
\mathcal{D}_1&=\cup_{j=J_0}^{J_1-1}(Q^\vep((K+jw-(w+A)J_0)T)\times[jT,(j+1)T]). 
\end{align*}
For help with the definition consult the following picture:

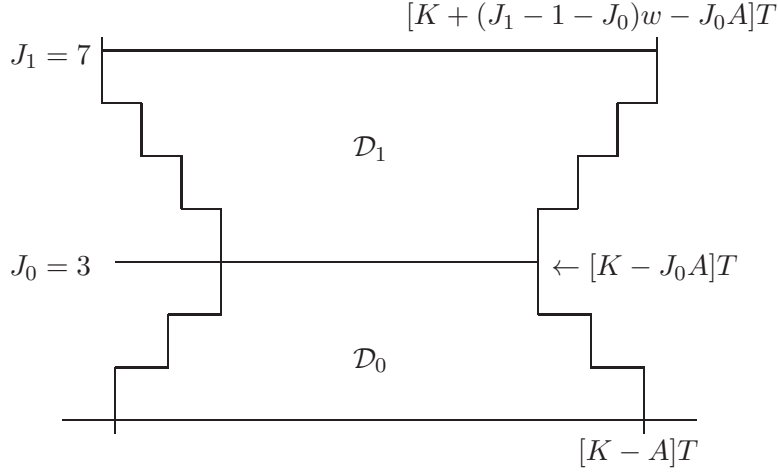
\begin{center}
\begin{picture}(300,220)
\put(30,40){\line(1,0){240}}
\put(50,35){\line(0,1){25}}
\put(250,35){\line(0,1){25}}
\put(225,25){$[K-A]T$}
\put(50,60){\line(1,0){20}}
\put(250,60){\line(-1,0){20}}
\put(70,60){\line(0,1){20}}
\put(230,60){\line(0,1){20}}
\put(70,80){\line(1,0){20}}
\put(230,80){\line(-1,0){20}}
\put(90,80){\line(0,1){40}}
\put(210,80){\line(0,1){40}}
\put(50,100){\line(1,0){160}}
\put(10,95){$J_0=3$}
\put(215,95){$\leftarrow [K-J_0A]T$}
\put(90,120){\line(-1,0){15}}
\put(210,120){\line(1,0){15}}
\put(75,120){\line(0,1){20}}
\put(225,120){\line(0,1){20}}
\put(75,140){\line(-1,0){15}}
\put(225,140){\line(1,0){15}}
\put(60,140){\line(0,1){20}}
\put(240,140){\line(0,1){20}}
\put(60,160){\line(-1,0){15}}
\put(240,160){\line(1,0){15}}
\put(45,160){\line(0,1){25}}
\put(255,160){\line(0,1){25}}
\put(160,190){$[K+(J_1-1-J_0)w - J_0A]T$}
\put(45,180){\line(1,0){210}}
\put(10,175){$J_1=7$}
\put(140,60){${\cal D}_0$}
\put(140,140){${\cal D}_1$}
\end{picture}
\end{center}

The speed $w>0$ is as in Assumption~2 (and may be assumed to
be $<1$), and $T=A_{\tref{lem:gdlow}}\log(1/\vep)$ is the
same as in \eqref{Tdefn}.  For the regions $\cD_0,\cD_1,\cE$
we take 
\begin{multline}\label{constants}
J_0= \left\lceil \dfrac{2d}{\eta_2 A_{\tref{lem:gdlow}}}
\right\rceil+1, \quad
A=\dfrac{8(2d+3))}{A_{4.2}}\vee (2c_b \vee 2\sigma^2) \\
K = 2 + AJ_0, \text{and }
J_1= J_0 + 1+ \left\lceil\dfrac{K+AJ_0}{w}
\right\rceil .
\end{multline}

The choice $K=2+AJ_0$ implies that
$Q^\vep(2T)\times\{J_0T\}$ is the ``top'' of $\cD_0$ and the
``bottom'' of $\cD_1$. The choice of $J_1$ implies
\begin{align}\label{Dbig}&\hbox{the top of $\cD$ contains
$Q^\vep(2KT)\times \{J_1T\}$ and is contained in}\\
\nonumber&\hbox{$Q^\vep((2K+1)T)\times \{J_1T\}$, and $\cD$ contains the region
$Q^\vep(2T)\times[0,J_1T]$.}
\end{align} 

Recall from Section~\ref{sec:perc} that a region $\mathcal{C}$ in $ \R^d\times\R_+$ is $\vep$-empty iff $\xi^\vep_t(x)=0$ for all $(t,x)\in \mathcal{C}$, where
$\xi^\vep$ is our voter model perturbation as usual.  If $A\subset\R^d$ let $\xi_0^\vep(A)=\sum_{x\in A\cap\vep\Z^d}\xi_0^\vep(x)$.

\begin{lem}\label{lem:percest} There exist
  $\vep_{\tref{lem:percest}}, c_{\tref{lem:percest}}>0$ depending on 
  $u_1,u_2,w,c_2,C_2$ (from Assumption~\ref{a2}) and
$r_0,\gamma_{\tref{lem:gdlow}},\eta_2$ such that 
such that if $0<\vep\le \vep_{\tref{lem:percest}}$
and  
\begin{equation}\label{0massic}
\xi_0^\vep(Q(KT))=0,
\end{equation}
then 
\begin{equation}\label{percprob}
P\Bigl(\mathcal{D}(J_0,J_1,w,A,K) \hbox{ is
  $\vep$-empty}\Bigr) 
\ge 1-c_{\tref{lem:percest}}\vep^{{.05}\wedge{\gamma_{\tref{lem:gdlow}}\over
    4}}, 
\end{equation}
and with probability at least
$1-c_{\tref{lem:percest}}\vep^d$,
\begin{align}\label{lightcone}
&\text{for all }j=1,\dots,J_1, 
(x,u)\in (Q^\vep((K+J_1w+A(J_1-j))T)\times [(j-1)T,jT])
\nonumber \\
&\text{and } t\in[0,u-(j-1)T], X^{x,u}_t \subset Q^\vep((K+J_1w+A(J_1-j+1))T).
\end{align}
\end{lem}

\begin{proof} We begin with some notation for describing
  events in which the dual process is confined to certain
  space-time regions. For $j\ge 1$ and $0<r<s$ let
  $\Gamma_T(j,r,s)$ be the event
\[
X^{x,u}_t \subset Q^\vep(sT) \ \forall \
x\in Q^\vep(rT) , \
u\in[(j-1)T,jT], \text{ and }
t\in [0,u-(j-1)T] .
\]
On $\Gamma_T(j,r,s)$, duality and \eqref{staydead} imply
\begin{equation}\label{gambnd2}
\xi^\vep_{(j-1)T}(Q^\vep(sT))=0 \text{ implies } 
Q^\vep(rT)\times[(j-1)T,jT]
\text{ is $\vep$-empty}. 
\end{equation}

\bigskip
\emph{Step~1.} We first check that $\cD_0$ is empty with
high probability. For $j\in\{1,\dots J_0\}$ we bound the
probability of $\Gamma_T(j,K-jA,K-(j-1)A)$ by using 
Lemma~\ref{lem:contain2}. If we set $c=A_{\tref{lem:gdlow}}$, 
$U=jT$, $L=(K-jA)T$ and $2b=A$, then evaluating $q$ in the
lemma we obtain
\[
q = (\frac{AA_{\tref{lem:gdlow}}}{8}-2)\wedge \vep^{-2} \ge
2d+1
\]
if $\vep^{-2}\ge (2d+1)$. Hence the bound on $\bar p_\vep$ in
Lemma~\ref{lem:contain2} gives us
\begin{equation}\label{gambnd3}
P(\Gamma_T(j,K-jA,K-(j-1)A) ) \ge 1 - c'_d(K-jA)^d
\vep^{d}
\end{equation}
for $\vep<\vep_{\tref{lem:ldbd}} (A/2)$ such that 
$\vep^{-2}\ge 2d+1$ and $\vep
(A_{\tref{lem:gdlow}}\log(1/\vep)+1)^d\le 1$.

By \eqref{gambnd2}, on the intersection
\[
\cap_{j=1}^{J_0} \Gamma_T(j,K-jA,K-(j-1)A),
\]
for each
$j\in\{1,\dots J_0\}$,
if $\xi^\vep_{(j-1)T}(Q^\vep((K-(j-1)A)T))=0$ then 
$Q^\vep((K-jA)T)\times[(j-1)T,jT]$ is $\vep$-empty. Iterating this, \eqref{0massic} and \eqref{gambnd3}
imply that for some positive $\vep_0$,
\begin{equation}\label{deadstep1}
P\left(
  \cD_0\text{ is $\vep$-empty}
\right) \ge 1-c'_dJ_0K^d\vep^{d} 
\text{ if } \vep<\vep_0.
\end{equation}
Here, and throughout the proof, $\vep_0$ will denote a positive constant
depending only on our fixed parameters including $r_0$. 

\emph{Step~2.} 
By taking $\vep$ small enough we may assume that (recall $L_0$ is as in Assumption~\ref{a2})
\begin{equation}\label{vep1}
2+L_0\le KT\le (K+wJ_1)T\le \vep^{-.001/d}.
\end{equation}
For $j\in\{1,\dots,J_1-J_0\}$, on account of
\eqref{0massic}, we may apply Lemma~\ref{lem:gdlow} and the Markov property
$J_1-J_0$ times and conclude that for
$\vep<\vep_{\tref{lem:gdlow}}$, 
\begin{align*}
P( &\xi_{jT}^\vep \text{ has
density at most }\vep^{\gamma_{\tref{lem:gdlow}}}
\text{ in }
Q^\vep((K+wj)T)\hbox{ for }j=1,\dots,J_1-J_0) )\\
& \ge 1-(J_1-J_0)\vep^{.05} .
\end{align*}
When the above event occurs, for any 
$j\in\{1,2,\dots,J_1-J_0\}$, (recall that $Q_\vep=[0,a_\vep)^d\cap(\vep\Z^d)$)
\begin{align} 
\nonumber \xi^\vep_{jT}&(Q^\vep((K+wj)T))\\
&\le \sum_{x\in
  a_\vep\Z^d\cap Q^\vep((K+wj)T)}\xi^\vep_{jT}(x+Q_\vep)+\sum_{x\in Q^\vep((K+wj)T)-Q((K+wjT-a_\vep)}\xi^\vep_{jT}(x)\\ 
\nonumber&\le
|Q_\vep|\vep^{\gamma_{\tref{lem:gdlow}}}\hbox{card}(a_\vep\Z^d
\cap Q^\vep((K+wj)T)+c_d\vep^{-d}[(K+wj)T]^{d-1}a_\vep\\
\nonumber&\le
c_d|Q_\vep|\vep^{\gamma_{\tref{lem:gdlow}}}a_\vep^{-d}((K+wj)T)^d+c_d\vep^{-d}[(K+wj)T]^{d-1}a_\vep\\ 
\nonumber&\le
c_d(K+wJ_1)^dA_{\tref{lem:gdlow}}^d(\log(\vep^{-1}))^d[\vep^{\gamma_{\tref{lem:gdlow}}-d}+\vep^{(1/16d)-d}]\\ 
\label{tomass}&\le \vep^{\gamma_{\tref{lem:gdlow}}/2-d},
\end{align}
for small enough $\vep$, where we have used $\gamma_{\tref{lem:gdlow}}\le (16d)^{-1}$ in the last line.  We have shown that for all $\vep$
smaller than some 
positive $\vep_0$,
\begin{equation}\label{deadstep2}
P(\xi^\vep_{jT}(Q^\vep((K+wj)T))\le 
\vep^{\gamma_{\tref{lem:gdlow}}/2-d}\hbox{ for }j=1,\dots,J_1-J_0) \ge 1-(J_1-J_0)\vep^{.05}.
\end{equation}

\emph{Step~3.} 
Fix $j\in\{1,\dots,J_1-J_0\}$, and define
$(\hat\xi^{j,\vep}_t,t\ge jT)$ by setting 
\begin{equation*}\hat
  \xi^{j,\vep}_{jT}(x)=\begin{cases}\xi^\vep_{jT}(x)& \text{ if
    }x\in Q^\vep((K+wj)T),\\ 
0&\text{otherwise,}\end{cases}
\end{equation*}
and then using our Poisson processes $\{\Lambda^x_r,\Lambda^x_w:x\in\vep\Z^d\}$ to continue
constructing $\hat \xi^{j,\vep}_t$ in the same way as $\xi^\vep_t$ is
constructed. 
By Lemma~\ref{survbnd}, if
$\vep<\vep_{\tref{survbnd}}
(\gamma_{\tref{lem:gdlow}}/2,\eta_2)$,
\begin{multline}\label{hatdead}
\xi^\vep_{jT}(Q^\vep((K+jw)T)\le
\vep^{(\gamma_{\tref{lem:gdlow}}/2)-d}
\text{ implies }\\
P(|\hat\xi^{j,\vep}_t|>0|\xi^\vep_{jT}) \le 6e^{2c_b} \vep^{\gamma_{\tref{lem:gdlow}}/4} \text{ for
  all }t\ge (j+J_0-1)T.
\end{multline}
Using $\hat\xi^{j,\vep}_t$, we will show that with high probability,
\begin{multline}\label{diecouple}
\xi^\vep_{jT}(Q^\vep((K+jw)T)\le
\vep^{(\gamma_{\tref{lem:gdlow}}/2)-d} \text{ implies}\\
Q^\vep( (K+jw-J_0A)T )\times
[(j-1+J_0)T,(j+J_0)T]\text{ is $\vep$-empty.}
\end{multline}

To do this, define the event 
\[
\Gamma_T(j) = \cap_{i=1}^{J_0} \Gamma_T(j+i,
K+wj-iA, K+wj-(i-1)A).
\] 
Using Lemma~\ref{lem:contain2} as in Step~1
we have for small enough $\vep$
\[
P(\Gamma_T(j+i, K+wj-iA, K+wj-(i-1)A) ) \ge 1
-c'_d(K+wj-iA)^d\vep^{d} 
\]
and thus 

\begin{equation}\label{deadstep3}
P(\Gamma_T(j)) \ge 1-c'_d(3K)^dJ_0\vep^{d}
\end{equation}
for small enough $\vep$.

Observe that on the event $\Gamma_T(j)$ we have 
\begin{multline}\label{contain3}
X^{x,u}_t\subset Q^\vep( (K+jw)T) \ \forall \
x\in Q^\vep( (K+jw - J_0A)T ), 
\\ u\in[(j-1+J_0)T,(j+J_0)T],  \text{ and }
t\in[0,u-jT].
\end{multline}
Therefore, by duality, on 
$\Gamma_T(j)$,
\begin{multline*}
\xi^\vep_t(x) = \hat\xi^{j,\vep}_t(x) \text{ for all }\\
(x,t)\in Q^\vep((K+jw -J_0A)T)\times[ (j-1+J_0)T, (j+J_0)T].
\end{multline*}
Combining this observation with \eqref{hatdead} and
\eqref{deadstep3} we see that
the event in \eqref{diecouple} has probability at least
\[
1-6e^{2c_b}\vep^{\gamma_{\tref{lem:gdlow}}/4} - c'_d(3K)^dJ_0\vep^d
\]
for $\vep$ smaller than some $\vep_0$.

\emph{Step~4} We
can now sum the last estimate over $j=1,\dots,J_1-J_0$ and use \eqref{deadstep2} to
obtain
\begin{equation}\label{deadstep4}
P(\cD_1 \text{ is $\vep$-empty}) \ge 1 - J_1(\vep^{.05}
+ 6e^{2c_b}\vep^{\gamma_{\tref{lem:gdlow}}/4} + c'_d(3K)^dJ_0\vep^d )
\end{equation}
for small enough $\vep$. (Actually we get a slightly larger set than $\cD_1$.)
Combine \eqref{deadstep1} and \eqref{deadstep4} to obtain 
\eqref{percprob}.

\emph{Step 5} Finally, using the notation from Step~1, the
event in \eqref{lightcone} is just
\[
\cap_{j=1}^{J_1}\Gamma_T(j,K+wJ_1+(J_1-j)A,K+wJ_1+(J_1-j+1)A) .
\]
As in Step~1, we can use Lemma~\ref{lem:contain2} to bound
the probability of this intersection by
$1-c_d'J_1(K+J_1(w+A))^d\vep^{d}$ for small enough $\vep$, so we
are done.
\end{proof}
Let $K_1=K+J_1(w+A)$.  For $\xi\in\{0,1\}^{\vep\Z^d}$, let $G^\vep_\xi$ be the
event, depending on our graphical representation, on which $\cD=\cD(J_0,J_1,w,A,K)$ is $\vep$-empty if $\xi_0^\vep=\xi$, and on which \eqref{lightcone} holds. Note that \eqref{lightcone} implies all the duals starting at $(x,u)\in\cD$ and run up until time $u$ remain in $Q(K_1T)$.  Hence duality implies that $G^\vep_\xi$ is $\cG(Q(K_1T)\times[0,J_1T])$-measurable, where we recall from \eqref{Gxidef} that $\cG(R)$ is the $\sigma$-field generated by the Poisson points in the graphical representation in the region $R$.  By the inclusion \eqref{Dbig} we have
\begin{equation}\label{g0good}
\hbox{on }G^\vep_{\xi_0^\vep},\ Q^\vep(2T)\times[0,J_1T] \text{ is }\vep-\hbox{ empty, and }\xi^\vep_{J_1T}(Q^\vep(2KT))=0,
\end{equation}
providing that
$\xi_0^\vep\in H=\{\xi\in\{0,1\}^{\vep\Z^d}:\xi(Q^\vep(KT))=0\}$.
Adding the bounds in Lemma~\ref{lem:percest} we see that
\begin{equation}\label{g0goodprob}
\hbox{if }\xi\in H,\hbox{ then }P(G^\vep_\xi)\ge 1 - 2 c_{\tref{lem:percest}}\vep^{.05\wedge
{\gamma_{\tref{lem:gdlow}}\over
    2}}\  \text{ if } \vep<\vep_{\tref{lem:percest}} .
\end{equation}

\subsection{Proof of Theorem \ref{thm:nonexist}}\label{ssec:thmextinct}

\begin{proof}[Proof of Theorem~\ref{thm:nonexist}.]
We continue to take $T=A_{\tref{lem:gdlow}}\log(1/\vep)$, and
with $K,J_1$ from \eqref{constants} we define 
\[
L=T, \quad T'=J_1T, 
\]
and set $c_L=L/(2D)$ as before. We set $\bar\xi^\vep(y)=1(|y|>L)$, $y\in\vep\Z^d$, and $\sigma_z$, $z\in\vep\Z^d$ denote the translation
operators on $\{0,1\}^{\vep\Z^d}$.
For $(x,n)\in\LD$ let 
$$\xi^{x,n}=\begin{cases} \sigma_{-c_Lx}(\xi^\vep_{nT'} )&\text{ if }\sigma_{-c_Lx}(\xi^\vep_{nT'})\in H\\
\bar\xi^\vep&\text{ otherwise,}
\end{cases}$$
and define the percolation
variables
\begin{align}
\nonumber\eta(x,n) = 1(&G^\vep_{\xi^{x,n}}\hbox{ occurs in the graphical representation in which the Poisson}\\
\label{percvars}&\hbox{processes are translated by $-c_Lx$ in space and $-nT'$ in time}).
\end{align}
In the percolation argument which follows it is the first
part of the definition of $\xi^{x,n}$ that will matter; the
$\bar\xi^\vep$ is really only a place-holder which allows us
to define $\eta$ when the translated configuration is not in
$H$. As in the proof of Theorem~\ref{thm:exist} in
Section~\ref{sec:exist}, we are actually translating in
space by the ``lower left hand corner" in $\vep\Z^d$
associated with $-c_Lx$ and as before suppress this in our
notation.  In Section~\ref{sec:exist} we used Theorem~4.3 of
\cite{Dur95}; here we copy the key definition in its proof.
Using the measurability of $G_\xi^\vep$, the independence of
$\cG(R)$ for disjoint regions $R$, and \eqref{g0goodprob}
one can check that for any any $M>4DK_1$, the family
$\{\eta(z),z\in\LD\}$ satisfies the modified $M$-dependent
condition \eqref{modMdep} with $\theta=2
c_{\tref{lem:percest}}\vep^{.05\wedge
  {\gamma_{\tref{lem:gdlow}}\over 4}}$. To see this argue
exactly as in the proof of Theorem~A.4 of \cite{Dur95}.
  
Using the
percolation results from Section~\ref{sec:perc}, we will show 

\begin{lem} \label{deadgoal}
There exists $\bar r>0$ such that for $\vep$ small
enough $\xi^\vep_0(Q^\vep(KT))=0$ implies
$$
P(\Lambda)\equiv P( \xi^\vep_t(Q^\vep(\bar rt)) =0 \text{ for all large $t$} )>1/2.
$$
\end{lem}

Intuitively this is an immediate consequence of Lemmas \ref{ballth} and \ref{deadcore}. The first result implies that
on $\Omega^0_\infty$ then for large $n$, the wet sites satisfy $W^0_n \cap {\cal H}^{rn}_n = \bar W_n \cap {\cal H}^{rn}_n$.
The second result shows that if ${\cal B}_n$ is the collection of dry sites in ${\cal H}_n^{rn/4}$ connected to the complement of $\cup_{m=n/2}^n {\cal H}_m^{rm/2}$ by a path of dry sites on the graph with edge set ${\cal E}_\downarrow$ then ${\cal B}_n = \emptyset$ eventually. Wet sites in ${\cal H}^{rn}_n$ will correspond to space-time blocks that are empty of $1$'s while dry sites (i.e. not wet sites) in ${\cal H}^{rn}_n$ correspond to space-time blocks which {\it may} contain a $1$. If a dry site  in ${\cal H}_n^{rn/4}$ corresponds to a block containing a $1$ there must be a dual path of $1$'s leading from this $1$ to a site outside of $\cup_{m=n/2}^n {\cal H}_m^{rm/2}$. This corresponds to a path of dry sites in ${\cal E}_\downarrow$ and so cannot happen for large $n$ since ${\cal B}_n = \emptyset$ for large $n$. Thinking of the corresponding space time regions are being filled with concrete, and the dry sites as air spaces, we see that there cannot be a 1 in ${\cal H}_n^{rn/4}$ unless some air space reaches outside of $\cup_{m=n/2}^n {\cal H}_m^{rm/2}$. We now give a formal proof.   

\begin{proof}
Fix $\xi_0^\vep$ as above and recall $W^0_n, \bar W_n, \cB_n$ and $\Omega_\infty^0$ from Section~\ref{sec:perc}. In particular $W^0,\bar W$ are constructed from an iid Bernoulli field which is bounded above by $\eta(z,n),\ z\in\cH_n,n\ge 1$.  By \eqref{g0goodprob} and our condition on $\xi_0^\vep$ (which implies $\xi^{0,0}=\xi_0^\vep$ in the definition of $\eta(0,0)=1$) we see that
\[P(\eta(0,0)=1)\ge 1-2c_{\ref{lem:percest}}\vep^{.05\wedge(\gamma_{\ref{lem:gdlow}}/4)}\ge 3/4,\]
for $\vep$ small enough.  By working with $P(\cdot|\eta(0,0)=1)$ in place of $P$ we may assume $\eta(0,0)\equiv1$ at a cost of proving (under our new $P$) that
\begin{equation}\label{dead34}P(\Lambda)>3/4.
\end{equation}

Assume $n\ge 1$ and
$(y,n)\in W^0_n$.  Then for some $i$, letting $y'=y-v'_i$, 
$(y',n-1)\in W^0_{n-1}$
with $\eta(y',n-1)=1$ (if $n=1$ we use $\eta(0,0)\equiv 1$ here). Continue to trace back the set of open sites $y'=y'_{n-1}, \dots, y'_0=0$.  
Proceeding through the $y'_i$ values, using the second part of \eqref{g0good} and $c_Lv'_i+[-L,L]^d\subset [-2L,2L]^d$ for $i=1,\dots,D$, we see that $\xi^{y'_i,i}=\sigma_{-c_Ly'_i}(\xi^\vep_{iT'})$ in the definition of $\eta(y'_i,i)=1$.  Therefore \eqref{g0good} and translation invariance, show that $\eta(y',n-1)=1$ implies
\[\xi^\vep_t(c_Ly'+Q^\vep(2L))=0\hbox{ for all }t\in[(n-1)T',nT'].\]
Since $c_{L}y+Q^\vep(L)\subset c_Ly' +
Q^\vep(2L)$ we obtain
\begin{equation}\label{wetfact1}
(y,n)\in W^0_n\text{ implies } \xi^\vep_t( c_{L}y + Q^\vep(L))
 = 0 \text{ for all }t\in[(n-1)T',nT'].
\end{equation}
This confirms \eqref{blockcon} in Section~\ref{sec:perc}.

Next by Lemma~\ref{ballth} we may assume $\vep$ is small enough (independent of the choice of $\xi_0^\vep$) so that
$P(\Omega^0_\infty)>3/4$ and
$\theta<\theta_{\tref{ballth}}\wedge\theta_{\tref{expcontbd}}$.
Let $r=r_{\tref{ballth}}$ and assume $\omega\in\Omega_\infty^0$.  By
Lemma~\ref{ballth} there is an $n_0\in\NN$ so that
\begin{equation}\label{coupled}
W^0_k\cap \cH^{rk}_k = 
\bar W_k\cap \cH^{rk}_k \ \forall \ k\ge n_0 .
\end{equation}
Let $\bar r=\frac{r}{16DJ_1}$ and assume $\omega\notin\Lambda$. The latter implies
that for infinitely many $n>2n_0$ there are $t\in[(n-1)T',nT']$ and $x\in Q^\vep(\bar r t)$
with $\xi_t(x)=1$. We claim that this implies
\begin{equation}\label{deadgoal2}
\cB_n\ne\emptyset \hbox{ for $n$ as above}.
\end{equation}
Lemma~\ref{deadcore} implies the above is a null set, so it follows that $P(\Omega_\infty^0\setminus\Lambda)=0$ and so \eqref{dead34} would be proved (recall $P(\Omega_\infty^0)>3/4$).

To prove \eqref{deadgoal2} fix such an $n$ and trace backward in time a path
of 1's that leads to $\xi^\vep_t(x)=1$.  Here we are using \eqref{staydead} to show if all the inputs are $0$ the dual process will produce a $0$ at a given site.  By \eqref{cLd} there must exist some $(y,n)\in \cH_n$
such that $x\in c_L\cV_y\subset c_Ly+Q^\vep(L)$ and a bit of arithmetic using the definition of $\bar r$
gives 
\[|y|\le \frac{|x|}{c_L}+\frac{L}{c_L}\le \frac{\bar rnT'}{c_L}+2D\le \frac{rn}{4},\]
and we have taken $n_0$ big enough for the last inequality. 
Hence $(y,n)\in\cH_n^{rn/4}$ and so \eqref{wetfact1} and \eqref{coupled} imply
 $(y,n)\notin \bar
W_n$, i.e., $(y,n)$ is dry.  By duality and the finite range
assumption (recall \eqref{finsupp1}), there must exist $x'\in\vep\Z^d$ and
$t'\in[(n-1)T',t)$ such that $|x-x'|\le R_0\vep$ and
$\xi^\vep_{t'}(x')=1$. That is, $t'\ge (n-1)T'$ is the first time below $t$ 
that the dual jumps or $t'=(n-1)T'$ if there is no such time in which case $x'=x$.   We may assume $\vep$ is small enough
so that $R_0\vep/c_L\le c_{\tref{vLd1}}$, in which case by \eqref{vLd1} $x'\in
c_L\V_{y'}$ for some $y'$ of the form $y+v_i'-v'_j$ ($y=y'$
is included).  If $(y',n)\in \cH_n^{rn/2}\subset
\cH^{rn}_n$, it follows from \eqref{wetfact1} that $(y',n)$
must be dry, and thus $(y',n)\in D_{(y,n)}$.

Continue the above construction until either we reach a point $(y'',n)\in (\cH_n^{rn/2})^c$ with all earlier points in our path from $(y,n)$ being dry, or we obtain $x'',y''$ such that
$\xi^\vep_{(n-1)T'}(x'')=1$, $(y'',n)\in D_{(y,n)}\cap
\cH_n^{rn/2}$ and $x''\in c_L\cV_{y''}$.  In the former case $\cB_n\neq\emptyset$ (recall the precise definition prior to Lemma~\ref{deadcore}).  In the latter case if
$(y''-v'_i,n-1)\notin\cH_{n-1}^{r(n-1)/2}$ for some $i$,
then \eqref{deadgoal2} holds.  If not, then as one easily
checks $|c_L(y''-v'_i)-x''|<L$, and so arguing 
as above, we see that $(y''-v'_i,n-1)$ is dry.  Therefore the iteration can be continued until it stops as above or
continues down to time $(\frac n2-1)T'$, again forcing
\eqref{deadgoal2} in either case. 
\end{proof}

Having established Lemma \ref{deadgoal} the rest of the proof of Theorem \ref{thm:nonexist} is
routine. The proof of Lemma~\ref{lem:hole} shows that if we
start from an initial configuration with infinitely many 0's
then at time 1 there will be infinitely many cubes of the
form $c_{L}x+Q^\vep(L)$ with $x \in {\cal H}_0$ that are
$\vep$-empty. By the Markov property this will hold at all times $N\in\NN$ a.s. The above shows that
if $x_0$ is chosen so that $\xi^\vep_1(x_0+Q(L))=1$, then w. p. at least $1/2$, $\xi^\vep_{1+t}\equiv 0$ on $c_Lx_0+Q(\bar r t)$ for all large $t$.  If this fails at some time we can try again at a later time $N$ by the above and after a geometric $(1/2)$ number of trials we will succeed and produce a linearly growing set of $0$'s starting at some space-time location.  Therefore  the 0's
take over.
\end{proof}

\clearpage
\begin{center}
\address{Mathematics Department, Syracuse University,}
\address{215 Carnegie Hall, Syracuse, NY 13244-1150, USA}
\email{jtcox@syr.edu}
\mbox{}\\
\address{Mathematics Department, Duke University,}
\address{Box 90320, Durham, NC 27708-0320, USA}
\email{rtd@math.duke.edu}
\mbox{}\\
\address{ Department of Mathematics, The University of British Columbia,}
\address{ 1984 Mathematics Road, Vancouver, B.C., Canada V6T 1Z2}
\email{perkins@math.ubc.ca}
\end{center}

\end{document}